\documentclass[10.6pt]{amsart}
\usepackage{slashed}
\usepackage{graphicx}
\usepackage{amsfonts}
\usepackage{amssymb}
\usepackage{epstopdf}
\usepackage{amsmath}
\usepackage{amsxtra}
\usepackage{latexsym}
\usepackage{mathrsfs}
\usepackage{accents}

\newtheorem{theorem}{Theorem}[section]
\newtheorem{lemma}[theorem]{Lemma}
\newtheorem{corollary}[theorem]{Corollary}
\newtheorem{proposition}[theorem]{Proposition}
\newtheorem{definition}[theorem]{Definition}

\theoremstyle{remark}
\newtheorem{remark}[theorem]{Remark}

\numberwithin{equation}{section}

\def\bp{\boldsymbol{\partial}}
\def\vb{\underline{h}}
\def\tmin{{\mbox{\tiny{min}}}}
\def\T{\mathcal{T}}
\def\slP{\slashed{P}}
\def\slE{\slashed{E}}

\def\smu{\slashed{\mu}}

\def\tt{{t'}}

\def\bb{{\mathbf{b}}}

\def\er{\mbox{err}}

\def\sn{{\slashed{\nabla}}}
\def\bE{{\bf E}}
\def\bn{{\bf n}}

\def\zb{{\underline{\zeta}}}
\def\bpi{{\bar{\pi}}}

\def\J{{\mathcal{J}}}

\def\Ab{{\underline{A}}}

\def\M{{\mathcal{M}}}
\def\bT{{\textbf{T}}}

\def\bR{{\textbf{R}}}

\def\bd{{\textbf{D}}}
\def\ti{\tilde}

\def\bg{\mathbf{g}}

\def\I{{\mathcal I}}

\def\beaa{\begin{eqnarray*}}
\def\eeaa{\end{eqnarray*}}

\def\ba{\begin{array}}
\def\ea{\end{array}}
\def\d{\delta}
\def\be#1{\begin{equation} \label{#1}}
\def \eeq{\end{equation}}

\newcommand{\nn}{\nonumber}

\def\l{\langle}
\def\r{\rangle}

\def\nn{\nonumber}
\def\S{{\mathcal S}}

\def\cga{\overset\circ{\ga}}

\def\S2{{\Bbb S}^2}

\def\A{\mathcal {A}}

\def\V{{\mathcal V}}

\def\E{{\mathcal E}}

\def\V{{\mathcal V}}
\def\W{{\mathcal W}}
\def\ze{{\zeta}}

\def\ub{\underline{u}}

\def\Lb{\underline{L}}
\def\Lie{{\mathcal L}}
\def\tr{\mbox{tr}}

\def\bA{{\textbf{A}}}
\def\D{{\mathcal D}}
\def\H{{\mathcal H}}

\def\N{{\mathcal N}}
\def\cga{\overset\circ{\ga}}

\def\La{{\Lambda}}

\def\B{{\mathcal B}}
\def\F{{\mathcal F}}
\def\P{{\mathcal P}}
\def\R{{\mathcal R}}
\def\c{\cdot}
\def\hot{\widehat{\otimes}}
\def\sig{\sigma}

\def\a{\alpha}
\def\b{\beta}

\def\ep{\epsilon}
\def\l{\langle}
\def\r{\rangle}
\def\ga{\gamma}
\def\Ga{\Gamma}
\def\la{\lambda}

\def\p{\partial}
\def\P{{\mathcal P}}

\def\nab{\nabla}
\def\F{{\mathcal{F}}}

\def\C{{\mathcal C}}

\def\Lb{{\underline{L}}}

\def\div{\mbox{\,div\,}}
\def\curl{\mbox{\,curl\,}}

\def\tr{\mbox{tr}}
\def\Tr{\mbox{Tr}}

\def\tir{{\tilde r}}

\def\itt{{\mbox{Int}}}

\def\f14{\frac{1}{4}}
\def\f12{{\frac{1}{2}}}

\def\t1a{t^{-\frac{1}{a}}}

\def\bm{{\bf m}}

\def\sl{\slashed}
\def\sD{\slashed{\Delta}}

\def\er{\mbox{err}}

\def\sn{{\slashed{\nabla}}}

\def\zb{{\underline{\zeta}}}
\def\bpi{{\bar{\pi}}}

\def\J{{\mathcal{J}}}

\def\Ab{{\underline{A}}}

\def\M{{\mathcal{M}}}
\def\bT{{\textbf{T}}}

\def\bR{{\textbf{R}}}

\def\bd{{\textbf{D}}}
\def\ti{\tilde}

\def\I{{\mathcal I}}

\def\beaa{\begin{eqnarray*}}
\def\eeaa{\end{eqnarray*}}

\def\ba{\begin{array}}
\def\ea{\end{array}}
\def\be#1{\begin{equation} \label{#1}}
\def \eeq{\end{equation}}
\def\nn{\nonumber}

\def\l{\langle}
\def\r{\rangle}

\def\nn{\nonumber}
\def\S{{\mathcal S}}

\def\cga{\overset\circ{\ga}}

\def\S2{{\Bbb S}^2}

\def\A{\mathcal {A}}

\def\V{{\mathcal V}}

\def\E{{\mathcal E}}

\def\V{{\mathcal V}}
\def\W{{\mathcal W}}
\def\ze{{\zeta}}

\def\ub{\underline{u}}

\def\Lb{\underline{L}}
\def\Lie{{\mathcal L}}
\def\tr{\mbox{tr}}

\def\bA{{\textbf{A}}}
\def\D{{\mathcal D}}
\def\H{{\mathcal H}}

\def\N{{\mathcal N}}
\def\cga{\overset\circ{\ga}}

\def\La{{\Lambda}}

\def\B{{\mathcal B}}
\def\F{{\mathcal F}}
\def\P{{\mathcal P}}
\def\R{{\mathcal R}}
\def\c{\cdot}
\def\hot{\widehat{\otimes}}
\def\sig{\sigma}

\def\a{\alpha}
\def\b{\beta}

\def\ep{\epsilon}
\def\l{\langle}
\def\r{\rangle}
\def\ga{\gamma}
\def\Ga{\Gamma}
\def\la{\lambda}

\def\p{\partial}
\def\P{{\mathcal P}}

\def\nab{\nabla}
\def\F{{\mathcal{F}}}

\def\Lb{{\underline{L}}}

\def\div{\mbox{\,div\,}}
\def\curl{\mbox{\,curl\,}}

\def\tr{\mbox{tr}}
\def\Tr{\mbox{Tr}}

\def\tir{{\tilde r}}

\def\tiw{{\tilde w}}

\def\itt{{\mbox{Int}}}

\def\f14{\frac{1}{4}}
\def\f12{{\frac{1}{2}}}

\def\t1a{t^{-\frac{1}{a}}}

\def\bm{{\bf m}}

\def\sl{\slashed}

\def\sD{\slashed{\Delta}}

\def\ckk{\check}

\def\dum{\mbox{ }}
\newcommand{\bea}{\begin{eqnarray}}
\newcommand{\eea}{\end{eqnarray}}

\def\nn{\nonumber}
\def\be{{(e)}}
\def\bi{{(i)}}

\def\G{\mathcal{G}}

\def\tiR{{\widetilde{\bR}}}

\newcommand{\chih}{\hat{\chi}}
\newcommand{\chib}{\underline{\chi}}

\newcommand{\chibh}{\underline{\hat{\chi}}\,}
\newcommand{\les}{\lesssim}

\def\gac{\stackrel{\circ}\ga}

\begin{document}
\title[]
{A geometric approach for sharp  Local well-posedness of quasilinear wave equations}
\author{Qian Wang}
\address{
Oxford PDE center, Mathematical Institute, University of Oxford, Oxford, UK}
\email{qian.wang@maths.ox.ac.uk}
\date{\today}
\maketitle
\begin{abstract}
The commuting vector fields approach, devised for strichartz estimates in \cite{KCom},  was developed for proving the local well-posedness  in the Sobolev spaces $H^s$ with $s>2+\frac{2-\sqrt{3}}{2}$ for general quasi-linear wave equation in ${\mathbb R}^{1+3}$ by Klainerman and Rodnianski.
Via this approach they obtained the local well-posedness in $H^s$ with $s>2$ for $(1+3)$ vacuum Einstein equations, by taking advantage of
the vanishing Ricci curvature. The sharp, $H^{2+\epsilon}$, local well-posedness result for general quasilinear wave
equation was achieved by Smith and Tataru by constructing a parametrix using wave packets.  Using the vector fields approach, one has to face the major hurdle caused by the Ricci tensor of the metric for the quasi-linear wave equations. This posed a question that if the geometric approach can provide the sharp result for the non-geometric equations.
In this paper, based on geometric normalization and new observations on the mass aspect function, we prove the sharp
local well-posedness of general quasilinear wave equation in ${\Bbb R}^{1+3}$ by a vector field approach.
 \end{abstract}

\section{\bf Introduction}\label{sec_1}

We consider the initial value problem of the second order quasi-linear wave equations of the form \begin{footnote}{Here, little Latin letters represent
indices from 1 to 3, and Greek letters represent indices from 0 to 3 with 0 standing for the time variable. Throughout the paper we use
the Einstein summation convention. We denote by $\p$ the partial differential $\p_{x_i}$. $\bp$ represents $(\p_t, \p).$ }\end{footnote}
\begin{equation}\label{wave1}
\left\{\begin{array}{lll}
-\p_t^2 \phi+g^{ij}(\phi)\p_i\p_j \phi=\N(\phi, \bp \phi),\\
\phi|_{t=0}=\phi_0, \quad \p_t \phi|_{t=0}=\phi_1
\end{array}\right.
\end{equation}
in ${\mathbb R}^{1+3}$, where the positive definite symmetric matrix $(g^{ij}(\phi))$ and its inverse $(g_{ij}(\phi))$
are smooth functions of $\phi$, and the function $\N(\phi, \bp\phi)$ takes the form
$$
\N(\phi, \bp \phi) = \N^{\a\b}(\phi) \bp_\a \phi \bp_\b \phi
$$
with $\N^{\a\b}(\phi)$ being smooth functions of $\phi$. We assume that $g^{ij}$ and $\N^{\a\b}$ are locally bounded in the sense that,
there is a sufficiently large integer $N$ such that for any $\La_0>0$ there exists a constants $C=C(\La_0)>1$ such that
\begin{equation}\label{g9}
\begin{split}
&C^{-1}|\xi|^2\le g^{ij}(y)\xi_i \xi_j \le C |\xi|^2,\qquad \forall \,|y|\le \Lambda_0,\\
&\sup_{|y|\le \Lambda_0} \left|\Big(\frac{d}{dy}\Big)^l g^{ij}(y)\right|\le C,\qquad \forall \,0\le l\le N,\\
&\sup_{|y|\le  \Lambda_0}\left|\Big(\frac{d}{dy}\Big)^l \N^{\a\b}(y)\right|\le C,\qquad \forall\, 0\le l\le N.
\end{split}
\end{equation}
The purpose of this paper is to prove the following sharp local well-posedness result in the Sobolev spaces $H^s$ with $s>2$
for the general quasi-linear wave equation (\ref{wave1}) by the vector field approach.

\begin{theorem}[Main Theorem]\label{main1}
For any $s>2$ and $M_0>0$,  there exist positive constants $T_*$ and $M_1$
such that for any initial data $(\phi_0, \phi_1)$ satisfying $\|\phi_0\|_{H^s}+\|\phi_1\|_{H^{s-1}} \le M_0$,
there exists a unique solution $\phi\in C(I_*, H^s)\times C^1(I_*, H^{s-1})$ to (\ref{wave1})
satisfying the estimates
\begin{equation*}
\|\bp \phi\|_{L^2_{I_*} L_x^\infty}+\|\bp\phi\|_{L_{I_*}^\infty H^{s-1}}\le M_1,
\end{equation*}
where $I_*:=[-T_*, T_*]$.
\end{theorem}

\subsection{Review of history}
Since the work of Choquet-Bruhat \cite{Ch}, there has been extensive work on the well-posedness of quasi-linear
wave equation (\ref{wave1}) in ${\Bbb R}^{1+n}$ for $n\ge 2$. In view of the energy estimate
\begin{equation}\label{eqin1}
\|\bp \phi(t)\|_{H^{s-1}}\le c \|\bp \phi(0)\|_{H^{s-1}}\c
\exp \left(\int_0^t  \|\bp \phi(\tau)\|_{L_x^\infty} d\tau\right),
\end{equation}
the Sobolev embedding and a standard iteration argument, the classical local well-posedness result of Hughes-Kato-Marsden \cite{HKM}
in the Sobolev space $H^s$ follows for any $s>\frac{n}{2}+1$, where the estimate
of $\|\bp \phi\|_{L_t^\infty L_x^\infty}$ is heavily relied on. To improve the classical
result, it is crucial to get a good estimate on $\|\bp \phi\|_{L_t^1 L_x^\infty}$. This is
reduced to deriving the Strichartz estimate for the wave operator $\bg^{\a\b}(\phi)\p_\a\p_\b$ which has rough coefficients  since  the Lorentzian metric
$\bg(\phi)=- d t^2 + g_{ij}(\phi) d x^i d x^j$
depends on the solution $\phi$ and thus is only as smooth as $\phi$.

Strichartz estimate for wave equations  with rough coefficients was first studied by Smith \cite{Sm}.
An important breakthrough was then made by Bahouri-Chemin \cite{BC1, BC2} and by Tataru \cite{T1}
using parametrix constructions. By establishing a Strichartz estimate for solutions to linearized equation
$\p_t^2\phi-g^{ij}\p_i \p_j \phi=0$ of the form
\begin{equation*}
\|\bp \phi\|_{L_t^2 L_x^\infty}\le c\left(\|\phi_0\|_{H^{\frac{n}{2}+\f12+\a}}+\|\phi_1\|_{H^{\frac{n}{2}-\f12 +\a}}\right)
\end{equation*}
with a loss of $\a>\frac{1}{4}$, they obtained the well-posedness of (\ref{wave1}) in $H^s$
with $s>\frac{n}{2}+\f12+\frac{1}{4}$.
This result
was later improved to $s>\frac{n}{2}+\f12+\frac{1}{6}$ in \cite{T3}.

In \cite{KRduke}, Klainerman-Rodnianski improved the local well-posedness  of (\ref{wave1}) in ${\mathbb R}^{1+3}$ to the Sobolev space $H^{s}$
with $s>2+\frac{2-\sqrt{3}}{2}$, where they blended the geometric treatment on the actual nonlinear equation introduced in \cite{KCom},  with the paradifferential calculus ideas initiated in \cite{BC2}, \cite{T2} and \cite{T3}. Thanks to the geometric approach, the feature that\begin{footnote}{
 $\Box_{\bg}$ represents the Laplace-Beltrami operator of the Lorentzian metric $\bg$ in this paper. }\end{footnote} $\Box_{\bg} g^{ij}(\phi)$ is quadratic in $\bp \phi$ exhibited its power and provided the improvement.  According to the counter-examples in \cite{Lind_1,Lind_2}, one can only expect to obtain the
local well-posedness in $H^s$ with $s>2$. 
Einstein vacuum equations under the wave
coordinates can be written as a system of the type (\ref{wave1}).
For this equation system,  Klainerman and Rodnianski achieved the local well-posedness result in $H^s$ for any $s>2$
in \cite{KR1,KR2,KRd}  by using the vector field approach. An extension to Einstein vacuum equation
under the CMC spatial harmonic gauge was established in \cite{Wangrough, Wangricci}.   Vanishing Ricci plays a
key role in these works for Einstein equations. The local well-posedness of (\ref{wave1}) in $H^s$ with $s>2$ was achieved by
Smith and Tataru in \cite{SmTT} where, to prove the Strichartz estimates, they constructed the parametrix by using wave packets.

In \cite{KRduke,KR1}, there adopted a procedure of paradifferential regularization over the spacetime metric $\bg$. Without regularizing the rough Einstein metric, the vector field approach is implemented
in a more direct fashion in \cite{Wangrough, Wangricci}.  This was achieved by reducing the Strichartz estimates to controlling  conformal energy of a very low order. With the needed order of conformal energy lowered, to obtain such energy, the required  regularity on the background geometry was relaxed, and was obtained when the metric is Einstein and belongs to $H^s$ with $s>2$.

\subsection{ Bootstrap assumptions}
For the general quasi-linear wave equation (\ref{wave1}), the metric $\bg=\bg(\phi)$ is no longer Einstein and is rough
due to its dependence on the unknown solution. This presents a great deal of challenging issues in implementing the vector
field approach. It has been a longstanding question whether the vector field approach can  be used to achieve the sharp local well-posedness
result for (\ref{wave1}). In this paper we will give the affirmative answer to this question. We observe that regularizing the metric $\bg$
in phase space does not help much, due to the potential loss of regularity it causes on null hypersurfaces. We take the more direct scheme
in \cite{Wangrough, Wangricci}, and further relax the requirement on conformal energy. To derive the boundedness of such energy, we introduce
a new geometric normalization over the metric, such that ${\bf{Ric}}_{\ti\bg}(L,L)$, the tangential component of Ricci tensor along null
hypersurfaces under the new metric $\ti \bg$, only contains lower order terms. Moreover, we employ an energy method which makes full use of
the hidden structures we discovered on the connection coefficients of the null frame.

According to \cite{KRduke, KR1}, in order to complete the proof of Theorem \ref{main1}, it suffices to show that for any $s>2$
there exist positive constants $C$ and $T$ depending on $\|(\phi_0, \phi_1)\|_{H^s\times H^{s-1}}$ such that there holds the Strichartz
estimate
\begin{equation*} 
\int_0^T \|\bp \phi(t)\|_{L_x^\infty}^2 dt \le C.
\end{equation*}
We will achieve this by using a bootstrap argument. To be more precise, we take
 small positive numbers $\ep_0, \delta_0$ such that
\begin{equation}\label{7.10.6}
 0<\ep_0<\frac{s-2}{5}, \quad \delta_0=\ep_0^2.
\end{equation}
We make the bootstrap assumption with $T$ a small number in $(0,1)$ that
\begin{equation}\label{BA1}
\|\bp \phi\|_{L^2_{[0,T]}L_x^\infty}^2+\sum_{\la\ge 2}\la^{2\delta_0}\|P_\la\bp \phi\|_{L_{[0, T]}^2 L_x^\infty}^2 \le 1. \tag {\bf BA}
\end{equation}
We then improve (\ref{BA1}) by showing that\begin{footnote}
 {By Sobolev embedding on initial slice,  $|\phi(0)|\le \La_1$ with $\La_1$ depending on $\|\phi(0)\|_{H^2}$. Using (\ref{BA1}), we can derive $|\phi|\le \La_1+1$ in $[0,T]\times {\Bbb R}^3$  by using fundamental theorem of calculus.
  This determines the constant $C(\La_0)$ in (\ref{g9}) with $\La_0$ substituted by $\La_1+1$. Throughout this paper, a constant is universal means it  depends on the $C=C(\La_1+1)$  from (\ref{g9}) and $\|(\phi_0, \phi_1)\|_{H^s\times H^{s-1}}$ only.  For two quantities $A$ and $B$,
we use $A\les B$ to represent that there exists a universal constant $M$ such that $A\le M\c B$.}
   \end{footnote}
\begin{equation}\label{p2}
\|\bp \phi\|_{L^2_{[0,T]}L_x^\infty}^2+\sum_{\la\ge 2}\la^{2\delta_0}\|P_\la\bp \phi\|_{L_{[0,T]}^2 L_x^\infty}^2
\les T^{2\delta} \|(\phi_0, \phi_1)\|_{H^s\times H^{s-1}}^2
\end{equation}
for some constant $\delta>0$,
where $P_\la$ is the Littlewood-Paley (LP) projection with frequency $\la =2^k$ defined for any function $f$ by
\begin{equation}\label{LP8.4}
P_\la f (x) = f_\la(x) = \int e^{{\rm i} x\c \xi} \eta(\la^{-1} \xi) \hat f(\xi)d \xi
\end{equation}
for some smooth function $\eta$ supported over $\{\xi: 1/2 \le |\xi|\le 2\}$ with $\sum_{k\in {\mathbb Z}} \eta(2^k \xi)= 1$ for $\xi\ne 0$.

In what follows, we first briefly explain the scheme of reduction for deriving the Strichartz estimate
and then outline the novelty of this paper.

\subsection{Scheme of Reductions}
Our proof of the Strichartz estimate is based on the combination of an abstract $\T\T^*$ argument for wave equations and
the reduction of Strichartz estimate to estimates on weighted energies. This approach was devised in \cite{KCom},
developed and extended in \cite{KRduke, KR1,KR2, KRd,Wangrough, Wangricci}.
Now we outline the reduction scheme for proving (\ref{p2}).

\subsubsection*{{\bf$\bullet$} {\it Reduction to dyadic Strichartz estimates}} \,\,
By using the LP decomposition, it is easy to reduce the proof of (\ref{p2}) to establishing for sufficiently large $\la$ the estimates
\begin{equation}\label{11.3.1}
\|\la^{\delta_0}P_\la \bp \phi\|_{L_{[0, T]}^2 L_x^\infty} \le c_\la T^{\f12-\frac{1}{q}}\|(\phi(0), \p_t \phi(0))\|_{H^s\times H^{s-1}},
\end{equation}
where $q>2$ is sufficiently close to $2$, and $\sum_{\la} c_\la <\infty$.
According to (\ref{BA1}), we may partition $[0, T]$ into disjoint union of sub-intervals $I_k:=[t_{k-1}, t_k]$ of total
number $\les \la^{8\epsilon_0}$ with the properties that
$|I_k|\les \la^{-8\epsilon_0} T$ and $\|\bp \phi\|_{L_{I_k}^2 L_x^\infty}\le \la^{-4\epsilon_0}$.
By considering $\phi$ on $I_k$ and using the Duhamel principle, the proof of (\ref{11.3.1}) can be reduced to establishing
the dyadic Strichartz estimate
\begin{equation}\label{plagen2}
  \|P_\la \bp \psi\|_{L^q_{I_k} L_x^\infty}\les\la^{\frac{3}{2}-\frac{1}{q}}\|\psi[0]\|_{\dot{H}^1}
\end{equation}
for linear equation $\Box_{\bg(\phi)} \psi=0$ on small time intervals, where $q>2$ is sufficiently close to $2$. Here $\psi[t]:=(\psi(t), \p_t \psi(t))$
and  $\|\psi[t]\|_{\dot{H}^1}:=\|\p_t \psi(t)\|_{L^2}+\|\p \psi(t)\|_{L^2}$.
To implement this reduction, the reproducing property of LP projections is used to resolve the issue that
the solution is not necessarily frequency localized even if the initial data is, see Section 3.



\subsection*{{\bf$\bullet$} {\it Physical localization  and reduction to $L^2-L^\infty$ decay estimate}}\,\,
For a large fixed frequency $\la$, by rescaling the coordinates and using a $\T\T^*$ argument we can reduce (\ref{plagen2})
to showing that, for any $\psi$ satisfying $\Box_{\bg} \psi=0$ on $[0,\tau_*]\times {\Bbb R}^3$ with $\tau_*\le \la^{1-8\ep_0}T$,
there holds the $L^1-L^\infty$ decay estimate
\begin{equation}\label{dispst1}
\|P\p_t \psi(t)\|_{L_x^\infty} \les \left((1+t)^{-\frac{2}{q}}+d(t)\right)\left(
\sum_{m=0}^3 \|\p^m \psi(0)\|_{L_x^1}+\sum_{m=1}^3 \|\p^{m-1} \p_t \psi(0)\|_{L_x^1}\right)
\end{equation}
for $t\in I_*:=[0, \tau_*]$, where $P=P_1$ the LP projection with frequency $1$, $\bg = \bg(\phi(\la^{-1} t +t_k, \la^{-1} x))$ and $d(t)$  is a function satisfying
$\|d\|_{L^{\frac{q}{2}}}\les 1$ for  $q>2$ sufficiently close to $2$; see Section 3 and Appendix B.
Let $t_0=1$. By using a suitable partition of unity on ${\Bbb R}^3$, in Section 4 we further reduce the derivation of (\ref{dispst1}) to proving that
\begin{equation}\label{dispst2}
\|P \p_t \psi(t) \|_{L_x^\infty}\les \left((1+|t-t_0|)^{-\frac{2}{q}}+d(t)\right)
(\|\psi[t_0]\|_{\dot{H}^1}+\|\psi(t_0)\|_{L^2})
\end{equation}
for $\psi$ satisfying $\Box_\bg \psi=0$ with $\psi(t_0)$ supported within the ball $B_{1/2}$ centered at the origin.

\subsection*{{\bf$\bullet$} {\it Reduction to bounded conformal energy}}\,\,
Let $\Sigma_t$ be the level set of $t$ and let $\Ga^+$ be the positive time axis. We introduce the optical
function $u$ with $u=t$ on $\Ga^+$ whose level sets are null cones $C_u$ with vertices on $\Ga^+$.
Let $S_{t,u}:=C_u\cap \Sigma_t$ and $N$ be the outward unit normal of $S_{t,u}$ embedded in $\Sigma_t$.  We can show $B_{1/2}\times \{t_0\} \subset \cup_{0\le u\le t_0}  S_{t_0, u}$.
Then the causal future of $B_{1/2}\times \{t_0\}$ is contained in $\D_0^+$, the causal future of $\cup_{0\le u\le t_0} S_{t_0, u}$
within $t_0\le t\le \tau_*$. A null frame $\{L,\Lb,e_1, e_2\}$ over $\D_0^+$ can be naturally defined with
$L=\p_t+N$ tangent to $C_u$, $\Lb=\p_t-N$,  and $\{e_1, e_2\}$ being orthonormal frame on each $S_{t, u}$.
We denote by $\sn$ the covariant differentiation on $S_{t,u}$ and $\tir=t-u$.
We  prove (\ref{dispst2}) by controlling  the conformal energy
$$
\C[\psi](t)=\C[\psi]^\bi(t)+\C[\psi]^\be(t),
$$
where
\begin{align*}
&\C[\psi]^\bi(t)=\int_{\Sigma_t\cap\{u\ge \frac{t}{2}\}} t^2 \left(|\bp\psi|^2+|\tir^{-1} \psi|^2\right) d\mu_g, \displaybreak[0]\\
&\C[\psi]^\be(t)=\int_{\Sigma_t\cap\{0\le u\le \frac{t}{2}\}} \left(\tir^2|L \psi|^2+\tir^2|\sn \psi|^2+|\psi|^2\right) d\mu_g.
 \end{align*}
In Section \ref{sec_4} we will prove that the $L^2-L^\infty$ decay estimate (\ref{dispst2}) can be derived from the
following boundedness result on conformal energy.

\begin{theorem}[Boundedness theorem]\label{thm8.1.1}
Let (\ref{BA1}) hold. Let $\psi$ be any solution of $\Box_\bg \psi=0$ on $I_*=[0, \tau_*]$ with $\psi(t_0)$
supported on $B_{\f12}\times\{t_0\} \subset \D_0^+\cap \Sigma_{t_0}$. Then, for $t\in [t_0, \tau_*]$, the conformal energy of $\psi$
satisfies the estimate
$$
\C[\psi](t)\les(1+t)^{2\ep}\|(\psi(t_0), \p_t\psi(t_0))\|_{H^1\times L^2}^2,
$$
where $\ep>0$ is an arbitrarily small number.
\end{theorem}

Compared with \cite[Theorem 5]{Wangrough}, Theorem \ref{thm8.1.1} achieves a slightly weaker control on the conformal energy which is sufficient
for our purpose.  The core difficulty is  then to control the conformal energy
in the very rough background. The proof of the boundedness theorem  will dominate this paper.

\subsection{Normalization and hidden structures}

One standard way for deriving the bound of conformal energy is to use the Morawetz vector field $K=\frac{1}{2} (u^2 \Lb +\ub^2 L)$,
where $\ub=2t-u$, as a multiplier in the energy momentum tensor
\begin{equation}\label{qmom}
Q[\psi]_{\mu\nu}=\p_\mu \psi \p_\nu \psi-\f12 \bg_{\mu\nu}(\bg^{\a\b}\p_\a \psi \p_\b \psi)
\end{equation}
associated with scalar functions $\psi$. To be more precise, we consider the modified weighted energy
$$
\widetilde Q[\psi](t)=\int_{\Sigma_t} \widetilde \P^{(K)}_\mu[\psi]\bT^\mu dx,
$$
where $\bT = \p_t$ and $\widetilde \P^{(K)}_\mu[\psi]$ is the modified energy current
\begin{equation}\label{gneng1}
\widetilde \P_\mu^{(K)}[\psi] = Q_{\mu\nu}[\psi] K^\mu+\f12\Theta  \p_\mu (\psi^2)-\frac{1}{2} \psi^2\p_\mu \Theta
\end{equation}
with scalar function $\Theta$. Running the typical energy approach  with $\Theta=2t$ gives the identity
\begin{equation}\label{eniden}
\widetilde Q[\psi](t_0) -\widetilde Q[\psi](t)= \f12\int_{\D_0^+[t_0, t]} {}^{(K)}\bar\pi_{\a\b} Q[\psi]^{\a\b}
+\int_{[t_0, t] \times {\mathbb R}^3} \Box_{\bg} \psi( K \psi+2t \psi)+\cdots,
\end{equation}
where $\bar\pi^{(K)}:=\pi^{(K)}-4t\bg$ and $\pi^{(K)}:={\mathscr L}_K \bg$
is the deformation tensor of $K$. The term involving $\Box_\bg$ in (\ref{eniden}) vanishes when $\psi$ satisfies $\Box_\bg \psi=0$.

To bound the energy $\widetilde Q[\psi](t)$, one needs to control the components of $\pi^{(K)}$ and some part of $\bd \,\pi^{(K)}$
due to the integration by part involved for adjusting the weights in energy. Since the Morawetz vector field $K$ is written in terms of
the null frame $\{L, \Lb\}$, the deformation tensor $\pi^{(K)}$ can be expressed by the connection coefficients
of the null frame (\cite{CK}) such as
\begin{align*}
\chi_{AB}=\bg(\bd_A L, e_B), &\quad \chib_{AB}=\bg(\bd_A \Lb, e_B), \quad
\zeta_A=\f12 \bg(\bd_\Lb L, e_A). 
\end{align*}
Involved in $\bd \,\pi^{(K)}$, there are $\sn \tr\chi$ and the mass aspect function  $\mu=\Lb \tr\chi+\f12 \tr\chi \tr\chib$, where $\tr\chi$ and $\tr\chib$ denote respectively
the trace parts of $\chi$ and $\chib$ whose traceless parts are denoted by $\chih$ and $\chibh$.
Besides $L_x^\infty$-type control on $\tr\chi$, one needs $\|\chih, \zeta\|_{L_t^1 L_x^\infty}$ and suitable control
on $\sn \tr\chi, \mu$ to perform the energy argument. Various estimates on Ricci coefficients were established
in \cite{KRduke, KR2, KRd, Wangricci} by using the null structure equations
\begin{align}
&L\tr\chi+\f12 (\tr\chi)^2=-|\chih|^2-{ k}_{NN} \tr\chi-\bR_{LL},\label{s0} \displaybreak[0]\\
&L\sn \tr\chi+\frac{3}{2}\tr\chi =-\sn \bR_{L L}+\cdots\label{ans0}, \displaybreak[0]\\
&L\mu+\tr\chi \mu=-\Lb \bR_{L L}-\tr\chi \bR_{L \Lb} -\f12 \tr \chib \bR_{LL} +\cdots. \label{lmuu}
\end{align}
With the metric $\bg$ from the general quasi-linear wave equation (\ref{wave1}),
the Ricci terms in (\ref{s0})--(\ref{lmuu}) contain high order derivative of $\phi$, which fails the method devised for Einstein equations in \cite{KR2,KRd,Wangricci}.
Moreover, we can only expect to have $L_t^\infty H_x^s$ control on metric, thus,  without normalizing the metric,
we would have no control on $\bd \bf{Ric}$.

\subsubsection{Normalization}
Recall that, for a Lorentzian metric $\bg$, the Ricci tensor admits the standard decomposition
\begin{equation}\label{7.2.1}
\bR_{\a\b}=-\f12 \Box_\bg (\bg_{\a\b})+\f12 (\bd_\a V_\b+\bd_\b V_\a)+S_{\a\b},
\end{equation}
where $V$ is a 1-form whose components take the form $\bg \c \bp \bg$, and the term $S_{\a\b}$ is
quadratic  in $\bp \bg$.  
Based on (\ref{7.2.1}), in \cite{KRduke}, \cite{SmTT} and \cite{shock},  the term $\bR_{LL}$ in (\ref{s0})--(\ref{lmuu}) is decomposed as
\begin{equation}\label{dc_322}
\bR_{LL}=L (V_L)-\f12L^\a L^\b\Box_\bg (\bg_{\a\b})+\cdots.
\end{equation}
  With the help of (\ref{BA1}) this can give the $L_t^2 L_x^\infty $ control on $\tr\chi+V_L$.  The treatment of pairing $\tr\chi$ with $V_L$  implies $\tr\chi$ is as smooth as $V_L$. This treatment works through if $V_L$ is smooth enough which is the case in \cite{KRduke} when the initial data of (\ref{wave1}) is assumed to be smoother. With data in $H^{2+\ep}$, $V_L$ has limited regularity.
 Meanwhile, to  obtain  boundedness of conformal energy by using (\ref{eniden}), the bound on $ \sn \tr\chi$ needs to be stronger than the regularity of $\sn (V_L)$. This indicates that we need an energy argument which relies on the term $\sn (\tr\chi+V_L)$, instead of $\sn \tr\chi$.

  Moreover, if estimating $\sn (\tr\chi+V_L)$ in view of (\ref{ans0}) and (\ref{7.2.1}),  the  term  $\I=L^\a L^\b(\sn \Box_\bg \bg_{\a\b}-\sn(S_{\a\b}))$  splitted from $\sn\bR_{LL}$ remains on the right side of the equation of $L\sn(\tr\chi+V_L)$.  By using (\ref{wave1}), we can write
 \begin{equation}\label{w8.4.1}
\I=f(\phi) \sn \bp \phi\c \bp \phi+\cdots
 \end{equation}
 where $f$ is a smooth function of $\phi$. If the angular derivative is replaced by a generic differentiation, we  need  $1/2$ more derivative on data to achieve the desired estimate for $\sn (\tr\chi+V_L)$. When regularizing the metric $\bg$ to the frequency-truncated metric $H=P_{\le \la}\bg(P_{\le \la}\phi)$,
$\sn \Box_H H$ approximates $\sn \Box_\bg \bg$ in terms of asymptotic estimates in frequency $\la$. On null hypersurface, such estimate reaches the regularity of $f(\phi)  \bp^2 \phi\c \bp \phi $.  But the regularity offered by the physical angular differentiation in (\ref{w8.4.1}) can not be recovered. Due to the loss of $1/2$ derivative on null hypersurface  caused by  the regularization  in phase space, we normalize the metric in physical space only.

Hence our idea is to construct a new metric $\ti \bg$ on $\D^+_0$ so that $\widetilde{\tr\chi}:=\tr\chi+V_L$ is exactly the null expansion
under the metric $\ti \bg$. With an energy argument carried out under $\ti\bg$,  the new null expansion $\widetilde{\tr\chi}$ can be directly involved.

In Section \ref{sec_6} we employ a conformal change to introduce a new metric $ \ti\bg=\Omega^{-2} \bg$
in $\D^+$, the causal future (containing $\D_0^+$) of the null cone initiating from $t=0$ on $\Ga^+$, with conformal factor $\Omega = e^{-\sigma}$,
where $\sigma$ is a scalar function defined in $\D^+$ by
\begin{equation}\label{3.23.1}
L\sigma=\f12 V_L,\quad  \sigma(\Ga^+)=0.
\end{equation}
 To prove Theorem \ref{thm8.1.1},  instead of considering the homogeneous wave equation  $\Box_\bg \psi=0$, the equation we consider is changed to (\ref{idc1}), i.e. with $\ti \psi=\Omega \psi$,
   \begin{equation}\label{7.3.1}
\Box_{\ti \bg}\ti\psi=-\Omega^2(\Box_\bg \sigma+\bd^\mu \sigma\bd_\mu \sigma) \ti \psi.
\end{equation}
(See the derivation in Section \ref{sec_6}).

Benefiting from the construction of the new metric,  in contrast to (\ref{s0}) and (\ref{ans0}), the higher order terms in $\bR_{LL}$ and  $\sn \bR_{LL}$ are canceled out in the structure
equations for $\widetilde{\tr\chi}$ and $\sn \widetilde{\tr\chi}$, see (\ref{lz}) and (\ref{ldz_2}).
Nevertheless, the inhomogeneous term in (\ref{7.3.1})  poses an issue in energy estimate. The difficulty is shifted to deriving enough control on $\sigma$.  $\sigma$ is introduced to carry the rough part away from the original metric. It is expected to have issues on regularity. The estimates on  $\sigma$  provided by (\ref{3.23.1}) is too weak to close the energy argument,  see Lemma \ref{pre_sig}. We do not have a good control on $\Lb\sigma$, hence the estimate on full derivative $\bp\ti\bg$ is actually much weaker than $\bp \bg$. The new metric $\ti \bg$  is only smoother along null hypersurface.

\subsubsection{Energy argument}
Since $\bp \ti\bg$ is rougher, we  employ the original equation $\Box_\bg \psi=0$
for the standard energy estimates obtained by using multipliers  $\bT$ and  $N$.
For conformal energy,  we consider (\ref{7.3.1})  and  adapt the energy argument using multiplier of type $r^m L$ in \cite{DaRod1} from Minkowski space to our domain of influence $(\D^+_0, \ti\bg)$.
 Due to the deformation terms caused by the rough, curved background,  our energy argument is a very delicate procedure, which is presented in Section \ref{sec_7}. Since the advantage of  the normalized metric $\ti \bg$ can be only seen along null hypersurface, the controllable geometric quantities are very limited. Therefore we carefully choose the form of the multiplier  to avoid creating extra error terms related to $\Lb \sigma$.

After careful pretreatment,  we can see immediately the main geometric quantities needed for proving Theorem \ref{thm8.1.1}.
Besides the same control on $\chih$ and $\zeta$ as required when using (\ref{eniden}), schematically, we need to control terms
taking the form
\begin{align}\label{7.3.2}
\int \tir^{m-1} \psi \left[(\mu+2\sD \sigma, |\sn\sigma|^2) L(\tir \psi)+ \tir\sn \widetilde{\tr\chi}\sn (\Omega\psi)\right]
\end{align}
with $m=1,2$, see (\ref{B_1}) in Section \ref{sec_7}.

 The quantity $\mu+2\sD \sigma$ emerges in the multiplier argument (\ref{B_1}) in Section \ref{sec_7}  when combining $\mu(\ti\bg)$ with the inhomogeneous term of (\ref{7.3.1}). We encounter major difficulties in controlling $\mu+2\sD\sigma$,  $\sn \sigma$  and $\zeta$. To solve these difficulties, we have to rely on the special structures in $\mu$ and its transport equation, which will be elaborated very shortly. Even after using our new observations,  the estimate on  $\mu+2\sD \sigma$ is  still weaker than the one for $\sn\widetilde{\tr\chi}$ due to the rough metric. Fortunately, this weaker control can be compensated by the control of the  conformal flux $\|\tir^{\frac{m-2}{2}} L(\tir \psi)\|_{L^2}$ on null cones, see Proposition \ref{e1}.

  We do not employ the argument of (\ref{eniden})  due to the potential technical  issues.  The control on lower order terms is also a difficulty if using this approach.


\subsubsection{Hidden structures on the mass aspect function}
As explained in the above energy argument, it is crucial to obtain enough control on $\sn\sigma$, $\mu+2\sD \sigma$, $\zeta$ and
$\chih$. By using their transport equations, $\sn \sigma$ is only $L^{2+}$ on $S_{t,u}$ and no estimate on $\sD\sigma$
can be obtained due to the presence of the third-order derivatives of the metric. We manage to solve these difficulties
by exploring hidden structures in $\mu$ and its transport equation.

{\bf $\bullet$} {\it The estimate on $\zeta$.}  Now we discuss how to obtain $\|\zeta \|_{L_t^1 L_x^\infty}$.
Since the work of \cite{CK}, the control for $\zeta$ has been tied to the control of $\mu=\Lb\tr\chi+\f12 \tr\chib\tr\chi$.
The idea for deriving the $L_t^1 L_x^\infty$ control on  $\zeta$ is to find a pair of terms $(\pi_1, \pi_2)$, of the type  $\bg\c\bp \bg$, such that
\begin{equation}\label{dmz}
(\div, \curl)\zeta=\sn (\pi_1, \pi_2)+\cdots.
\end{equation}
This roughly implies $\zeta=\D_0 (\pi_1, \pi_2)+\cdots$, where $\D_0$ is a Calderon-Zygmund operator on $S_{t,u}$
for which boundedness holds in $L^\infty(S_{t,u})$ with a $\log$-loss, in view of the Calderon-Zymund theory.
To achieve (\ref{dmz}), we use the decomposition of curvature (see Lemma \ref{decom_lem}) to rewrite the Codazzi equations
for $\zeta$, i.e. (\ref{dze}) and (\ref{dcurl}), as
\begin{equation}\label{cod_1}
\div \zeta=\f12\mu+\sn(\bg\c \bp \bg)+\cdots,\quad \curl \zeta=\sn(\bg\c \bp \bg)+\cdots
\end{equation}
In Einstein spacetime, due to  vanishing Ricci, $\mu$ can be well under control via (\ref{lmuu}).
In this case, conceptually $\mu$ can be regarded as a negligible term in (\ref{cod_1}).
This treatment on $\mu$ by using (\ref{lmuu}) can be seen in many works \cite{CK, KRduke, KR2, KRC1, Wangrough}.
This strategy does not work under the rough metric $\bg$,  due to the presence of  $\bf {Ric}$ terms
on the right of (\ref{lmuu}).  In view of (\ref{cod_1}), we observe that (\ref{dmz}) can be achieved  if $\mu$ has the structure
\begin{equation}\label{llmu}
\mu =\sn(\bg\c \bp \bg)+\cdots.
\end{equation}
Instead of relying on (\ref{lmuu}),  we prove directly that $\mu$ has the structure of (\ref{llmu}) in Section \ref{ss1}.
By using an auxiliary function $\varphi$, we explicitly compute  $\mu-\underline{\mu}$ in (\ref{comm1}) and prove that it is of lower order, where
$\underline{\mu}=L \tr\chib+\f12 \tr\chi\tr\chib$. We then show that $\underline{\mu}=\sn (\bg\c \bp \bg)+l.o.t$.
This  implies (\ref{llmu}) and solves the issue of controlling $\zeta$.

{\bf$\bullet$} {\it Control of the conformal factor $\sigma$.}
Even after the structure (\ref{llmu}) is identified, $\mu$ is only as good as $\sn \bp \bg$, which is
weak in $t$-variable for the purpose of deriving control for (\ref{7.3.2}). Moreover, no estimate on $\sD \sigma$ can
be obtained by using its transport equation.
Fortunately,  we find a  delicate structure in the transport equation of $\mu$, which offers control on $\mu+2\sD \sigma$.

Indeed, the most dangerous term contained in $L\Lb \tr\chi$ can be seen by differentiating the term $-(\bR_{LL}+k_{NN}\tr\chi)$
in (\ref{s0}) in the direction $\Lb$.  Using the identity (\ref{lbkn}) for $\Lb k_{NN}$ and the curvature decomposition (\ref{7.2.1}), we have
\begin{equation}\label{s3}
-\Lb(\bR_{LL}+k_{NN}\tr\chi)=-\Lb L (V_L)-\f12 \tr\chi \Lb (V_L)+\cdots
\end{equation}
which exhibits the full $\Lb$ derivative terms in
\begin{equation*}
\Box_\bg (V_L)=-\Lb L (V_L)-\f12 \tr\chi\Lb (V_L)-(\f12 \tr\chib -k_{NN})L (V_L)+2\zeta\c \sn (V_L)+\sD( V_L).
\end{equation*}
Inspired by the this observation, we treat the collection of bad terms in (\ref{s3}) by adding $\sD (V_L)$ which is an
equally dangerous term appearing in the transport equation of $2\sD \sigma$. This yields
\begin{equation}\label{cancel}
L(\mu+2\sD \sigma)+\tr\chi(\mu+2\sD \sigma)=\Box_\bg(V_L)+l.o.t.
\end{equation}
Thanks to (\ref{wave1}), the higher order term in $\Box_\bg  (V_L)$  is canceled out.
The precise form of (\ref{cancel}) and its derivation can be seen in Section \ref{sec_6_imp}.

There  contain in $\Box_\bg (V_L)$ terms of the form $\bp^2 \bg\c \bp \bg$. Due to the rough data,
we do not have control of the full derivative $\bp^2\bg$ on null hypersurfaces. This causes the loss of the $L_u^\infty$
control for  $\mu+2\sD\sigma$. Nevertheless, by using (\ref{cancel}),  $L_u^2$ control can be obtained  without loss for  $L_t^\infty L^{2+}(C_u)$ norm of $\mu+2\sD\sigma$.  Combined with the conformal flux, such control is sufficient for estimating (\ref{7.3.2}). Finally, by using the result for $\mu+2\sD\sigma$,
 and the $\div$-$\curl$ system obtained by adding  $\sD \sigma$ to (\ref{cod_1}),  we manage to decompose $\sn \sigma+\zeta$ into two parts. For both parts, we have integral control on their  $L^\infty(S_{t,u})$ bounds. Since  $\zeta$ is well under control, such decomposition provides sufficient regularity for $\sn \sigma$.  This treatment is contained in Proposition \ref{dcmpsig}.

This paper is organized as follows. In Section \ref{sec_2}, we establish energy estimates for $\phi$ with the  help of  the analysis contained in Appendix A. In Section \ref{sec_3} and \ref{sec_4}, we reduce the main theorem to Theorem \ref{BT}, the boundedness  theorem of the conformal energy, with some technicalities enclosed in Appendix B and C. Section \ref{sec_5} - Section \ref{sec_7} form the core part of this paper.  In Section \ref{sec_5}, we establish the complete sets of estimates on the connection coefficients of the null frame. In Section \ref{sec_6}, we introduce the conformal change of the metric, and provide sufficient control on the conformal factor and $\mu+2\sD\sigma$. In Section \ref{sec_7}, we complete the proof of Theorem \ref{BT}, which closes the proof of the main theorem.

\section{\bf Energy estimates and estimates for metrics}\label{sec_2}

Throughout this paper we use $\phi$ to denote the solution of (\ref{wave1}). Under the bootstrap assumption (\ref{BA1}),
in this section we will prove some basic energy estimates that are
contained in Proposition \ref{eng3}.

We recall the standard energy approach. Let
$$
\bg = -d t^2 + g_{ij}(\phi) d x^i d x^j
$$
For any scalar function $f$, let $Q_{\a\b}:=Q[f]_{\a\b}$
be the energy momentum tensor defined by (\ref{qmom}). By applying the divergence theorem to the energy current
$\P_\a^{(\bT)}[f]=Q[f]_{\a\b}\bT^\b$ with $\bT = \p_t$, we can obtain the energy identity
\begin{equation}\label{qs}
\int Q_{\a\b}\bT^\a\bT^\b(t)d\mu_g-\int Q_{\a\b}\bT^\a\bT^\b(0) d\mu_g
=-\int_{[0,t]\times{\mathbb R}^3} \left(\Box_{\bg} f \bd_\bT f+\f12 \pi^{(\bT)}_{\a\b} Q^{\a\b}\right),
\end{equation}
where $\Box_{\bg}$ denotes the Laplace-Beltrami operator induced by $\bg$ and $\pi^{(\bT)}:= {\mathscr L}_{\bT} \bg$ is
the deformation tensor of $\bT$. Since $\pi_{\a\b}^{(\bT)} = \p_t \bg_{\a\b}$, one can use (\ref{BA1}) to derive that
\begin{equation}\label{eng7.10}
\|\bp f(t)\|_{L_x^2}\les \|\bp f(0)\|_{L_x^2}+ \int_0^t \|\Box_\bg f(t')\|_{L_x^2} dt'
\end{equation}
which can be used to show the following preliminary energy estimates.

\begin{lemma}\label{lem4.15}
Under the bootstrap assumption \emph{(\ref{BA1})}, there holds
\begin{equation}\label{eng0}
\|\phi[t]\|_{\dot{H}^1}\les \|\phi[0]\|_{\dot{H}^1}\quad \mbox{and} \quad \|\bp\phi[t]\|_{\dot{H}^1}\les \|\bp\phi[0]\|_{\dot{H}^1}+\|\phi[0]\|_{\dot{H}^1}.
\end{equation}
With  $\|\phi[0]\|_{H^2}=\|(\phi(0), \p_t \phi(0))\|_{H^2\times H^1}$, there holds
\begin{equation}\label{eng00}
\|(\phi(t), \p_t\phi(t))\|_{H^2\times H^1}\les \|\phi[0]\|_{H^2}.
\end{equation}
Moreover, for $\psi$ satisfying $\Box_\bg \psi=0$, there holds
$
\|\psi[t]\|_{\dot{H}^1}\les \|\psi[0]\|_{\dot{H}^1}.
$
\end{lemma}
 With  the help of  (\ref{eng00}),  the Ricci tensor of the metric $g$ on $\Sigma_t$ verifies $\|Ric(g)\|_{L^2(\Sigma_t)}\les 1$ for $t\in [0,T]$, with the constant  depending on  $C$ in (\ref{g9}) and  $\|\phi[0]\|_{H^2}.$ Using \cite{Ander} or \cite[Theorem 5.4]{Petersen}, there exists a constant $d_0>0$ depending only on $C$ and  $\|\phi[0]\|_{H^2}$, which is the uniform lower bound of radius of injectivity on $\Sigma_t$ for $t\in [0,T]$. We assume $T$ satisfies  $0<T\le \min(d_0, 1)$ throughout the paper.

In order to derive finer estimates for the solution of (\ref{wave1}) with rough initial data, the Littlewood-Paley (LP)
decomposition is an important tool. Throughout the paper we will use $P_\la$ to denote the LP projection in ${\mathbb R}^3$
by (\ref{LP8.4}) with frequency size $\la$, where $\la$ denotes a dyadic number, i.e. $\la = 2^j$ for some integer
$j\in {\mathbb Z}$. Using LP decomposition, the Sobolev space $H^s({\mathbb R}^3)$ and the homogeneous $\dot{H}^s({\mathbb R}^3)$
for $s>0$ admit the equivalent norms
$$
\|f\|_{H^s}^2 = \|f\|_{L_x^2}^2 + \sum_{\la\ge 1} \la^{2s} \|P_\la f\|_{L_x^2}^2
\quad\mbox{and} \quad  \|f\|_{\dot{H}^s}^2 = \sum_{\la\ge 1} \la^{2s} \|P_\la f\|_{L_x^2}^2
$$
respectively. For a sequence $(a_\la)$ and $p\ge 1$, we use $\|a_\la\|_{\ell_\la^p}^p=\sum_{\la\ge 1} |a_\la|^p$.

To treat the inhomogeneous term in (\ref{wave1}), we first state the following product estimates whose proof will be
given in Appendix A.

\begin{lemma}\label{pres}
For $m=0,1$ and $0<\ep\le s-2$ there hold
\begin{align}
\|\la^{m+\ep}P_\la(\W(\phi)\bp \phi)\|_{l_\la^2 L_x^2}&\les \|\bp\phi\|_{\dot{H}^{m+\ep}}+\|\phi[0]\|_{H^2},\label{nphi}\\
\|\la^\ep P_\la\bp(\N(\phi, \bp \phi)\|_{l_\la^2 L_x^2}
      &\les (\|\bp \phi\|_{H^{1+\ep}}+\|\phi[0]\|_{H^2})(\|\bp \phi\|_{H^1}^2+\|\bp\phi\|_{L_x^\infty})\label{prd}\\
\|\la^\ep P_\la(\N(\phi, \bp \phi))\|_{l_\la^2 L_x^2}&\les \|\phi[0]\|_{H^2},   \label{prd1}\\
\|\la^\ep P_\la (g(\phi)\p^2 \phi )\|_{l_\la^2 L_x^2}&\les \|\p^2\phi\|_{H^\ep}+\|\phi[0]\|_{H^2}, \label{prd_2}
\end{align}
where  $g(\phi)$ stands for $g_{ij}(\phi)$ or $g^{ij}(\phi)$ and $\W$ is a product of factors in $(g, \N, g', \N')$. 
\end{lemma}

\begin{lemma}\label{lem4.15.1}
Let $\phi$ be the solution of (\ref{wave1}) and let $\phi_\la = P_\la \phi$. If $\phi$ satisfies (\ref{BA1}), then
\begin{equation}\label{llbox_1}
-\p_t^2 \phi_\la +g^{ij}(\phi) \p_i \p_j \phi_\la=\ti R_\la,
\end{equation}
where, for $0<\ep\le s-2$ and $t\in [0,T]$, $\ti R_\la$ satisfies the estimate
\begin{equation}\label{rlambda}
\|\la^{1+\ep}\ti R_\la(t)\|_{\ell_\la^2 L_x^2} \les \left(\|\bp \phi\|_{L_x^\infty} +1\right) (\|\bp \phi\|_{H^{1+\ep}}+\|\phi[0]\|_{H^2}).
\end{equation}
\end{lemma}

\begin{proof}
Since $\phi$ satisfies (\ref{wave1}), $\phi_\la$ satisfies (\ref{llbox_1}) with
\begin{equation*}
\ti R_\la= P_\la \N(\phi, \bp \phi) -( E_\la + A_\la),
\end{equation*}
where
\begin{align*}
E_\la &= P_\la(g^{ij}(\phi)\p_i\p_j \phi)-P_{\le \la}( g^{ij}(\phi)) P_\la (\p_i\p_j\phi), \\
A_\la &=\left(P_{\le \la} (g^{ij}(\phi))-g^{ij}(\phi) \right) P_\la (\p_i\p_j \phi).
\end{align*}
By using (\ref{prd}) in Lemma \ref{pres}, the embedding $\|\phi\|_{L_x^\infty} \les \|\phi\|_{H_x^2}$ and (\ref{eng00}) in Lemma \ref{lem4.15},
we obtain
\begin{equation}\label{r3}
\|\la^{1+\ep} P_\la \N(\phi, \bp \phi)\|_{l_\la^2 L_x^2}\les \left(\|\bp \phi\|_{L_x^\infty} +1\right) \|\bp \phi\|_{H^{\ep+1}}.
\end{equation}
We use Lemma \ref{lem8.4.1} below to estimate $E_\la$ and $A_\la$. Note that $E_\la$, $A_\la$ can be written
as $\widetilde E_\la(\p \phi)$ and  $\widetilde A_\la(\p^2\phi)$ respectively. Thus, it follows from Lemma \ref{lem8.4.1} that
\begin{align}
&\|\la^{1+\ep} E_\la\|_{l_\la^2 L_x^2}\les \|\bp \phi\|_{L_x^\infty} (\|\bp \phi\|_{H^{1+\ep}}+\|\phi[0]\|_{H^2}), \label{r1}\\
&\|\la^{1+\ep} A_\la\|_{l_\la^2 L_x^2}\les \|\p \phi\|_{L_x^\infty} \|\p^2\phi\|_{\dot{H}^\ep}. \label{r_7.10.3}
\end{align}
By combining  (\ref{r3})--(\ref{r_7.10.3}), we therefore obtain (\ref{rlambda}).
\end{proof}

\begin{lemma}\label{lem8.4.1}
For any scalar function $f$ let
\begin{align*}
\widetilde E_\la(f)&:=P_\la (g(\phi) \p f)-P_{\le \la} (g(\phi)) P_\la (\p f);\quad
\widetilde A_\la(f):=(P_{\le \la}(g(\phi))-g(\phi))P_\la f.
\end{align*}
Then for $0<\ep\le s-2$ there hold
\begin{align}
&\|\la^{1+\ep} \widetilde E_\la(f)\|_{l_\la^2 L_x^2}+\|\la^\ep \widetilde E_\la(\p f)\|_{l_\la^2 L_x^2}
   \les \|\p \phi\|_{L_x^\infty} \|f\|_{H^{1+\ep}}+\|f\|_{L_x^\infty}(\|\p \phi\|_{H^{1+\ep}}+\|\phi[0]\|_{H^2}),\label{e_7_10.1}\\
&\|\la^{1+\ep} \widetilde A_\la(f)\|_{l_\la^2 L_x^2}+\|\la^{\ep}\widetilde A_\la(\p f)\|_{l_\la^2 L_x^2}\les \|\p \phi\|_{L_x^\infty} \|f\|_{\dot{H}^\ep}.\label{a_7.10.2}
\end{align}
\end{lemma}

\begin{proof}
We will prove the first inequalities in (\ref{e_7_10.1}) and (\ref{a_7.10.2}). The others can be derived similarly.
We first consider $\widetilde A_\la(f)$. By the finite band property we have
\begin{align*}
\|\widetilde A_\la(f)\|_{L_x^2} &\les \|f_\la\|_{L_x^2} \|g(\phi)-g(\phi)_{\le \la}\|_{L_x^\infty}
\les \|f_\la\|_{L_x^2} \sum_{\mu>\la}\mu^{-1} \|\p g(\phi)_\mu\|_{L_x^\infty}\\
&\les \la^{-1} \|f_\la\|_{L_x^2} \|\p \phi\|_{L_x^\infty}.
\end{align*}
This implies (\ref{a_7.10.2}).

Next we consider $\widetilde E_\la(f)$.  Applying the trichotomy of LP decomposition to $P_\la(g(\phi) \p f)$ we obtain
\begin{align*}
\widetilde E_\la(f) &= \Big(P_\la(g(\phi)_{\le \la}\c (\p f)_\la )-g(\phi)_{\le\la}\c (\p f)_\la\Big)+ P_\la\Big(g(\phi)_\la \c(\p f)_{\le\la}\Big)\\
& \quad \, +P_\la\Big(\sum_{\mu>\la}g(\phi)_{\mu} \c(\p f)_\mu\Big)=:a_\la+b_\la+c_\la.
\end{align*}
By using \cite[Corollary 1]{Wangrough} and the finite band property of LP projections, we have
\begin{align*}
\|a_\la\|_{L_x^2} &=\|[P_\la, g(\phi)_{\le\la}](\p f)_\la\|_{L_x^2}\les \la^{-1} \|\p g(\phi)_{\le \la}\|_{L_x^\infty} \|(\p f)_\la\|_{L_x^2}
\les \|\p \phi\|_{L_x^\infty} \| f_\la\|_{L_x^2}.
\end{align*}
Consequently
\begin{equation*}
\|\la^{\ep+1} a_\la\|_{l_\la^2 L_x^2}\les \|\p \phi\|_{L^\infty} \|f\|_{H^{1+\ep}}.
\end{equation*}
By using again the finite band property, we have
\begin{align*}
\|b_\la\|_{L_x^2} &\les \|g(\phi)_\la\|_{L_x^2}\|(\p f)_{\le \la}\|_{L_x^\infty} \les\|f\|_{L_x^\infty}\|\p g(\phi)_\la \|_{L_x^2},
\end{align*}
Thus it follows from (\ref{nphi}) with $m=1$ that
\begin{equation*}
\|\la^{1+\ep} b_\la\|_{l_\la^2 L_x^2}\les \|f\|_{L^\infty} (\|\p\phi\|_{H^{1+\ep}}+\|\phi[0]\|_{H^2}).
\end{equation*}
For $c_\la$ we have
\begin{align*}
\|\la^{1+\ep} c_\la\|_{L_x^2} &\les \sum_{\mu\ge \la} \left(\frac{\la}{\mu}\right)^{1+\ep}
\|\p g(\phi)_{\mu}\|_{L_x^\infty}\|\mu^{\ep}(\p f)_\mu\|_{L_x^2},
\end{align*}
which implies
\begin{equation*}
\|\la^{1+\ep} c_\la\|_{l_\la^2 L_x^2}\les \|\p g(\phi)\|_{L_x^\infty} \|\p f\|_{H^\ep}
\les \|\p \phi\|_{L_x^\infty} \|f\|_{H^{1+\ep}}.
\end{equation*}
Combining the above estimates we obtain (\ref{e_7_10.1}). The proof is complete.
\end{proof}

Now we are ready to give the basic energy estimate.

\begin{proposition}[Energy estimates]\label{eng3}
Let $\phi$ be the solution of (\ref{wave1}). If $\phi$ satisfies the bootstrap assumption (\ref{BA1}), there holds
\begin{equation}\label{w7.11.2}
\|\bp\phi(t)\|_{H^{s-1}}+\|\p_t^2 \phi(t)\|_{H^{s-2}}\les\|(\phi(0), \p_t\phi(0))\|_{H^s\times H^{s-1}}.
\end{equation}
\end{proposition}

\begin{proof}
We apply (\ref{eng7.10}) to $\phi_\la$  which, in view of Lemma \ref{lem4.15.1}, satisfies
\begin{equation}\label{bog}
\Box_\bg \phi_\la =R_\la, \qquad \mbox{where } R_\la :=\ti R_\la+\bp g\c \p \phi_\la.
\end{equation}
Since $\bp g = g'(\phi) \bp \phi$, we obtain
\begin{align*}
\|\bp \phi_\la(t)\|_{L_x^2} \les \|\bp \phi_\la(0)\|_{L_x^2}
+ \int_0^t \left(\|\ti R_\la(t')\|_{L_x^2}+\|\bp \phi(t')\|_{L_x^\infty} \|\p\phi_\la(t')\|_{L_x^2} \right) d\tt.
\end{align*}
Multiplying the both sides by $\la^{s-1}$, summing over $\la$ and using (\ref{rlambda}) with $\ep = s-2$, we have
\begin{align*}
\|\bp \phi(t)\|_{\dot{H}^{s-1}}&\les \|\bp \phi(0)\|_{\dot{H}^{s-1}}
+\int_0^t \left(\|\bp \phi(t')\|_{L_x^\infty}+1\right) \left(\|\bp \phi(t')\|_{\dot{H}^{s-1}}+\|\phi[0]\|_{H^2}\right) dt'
\end{align*}
which, in view of Gronwall's inequality and (\ref{BA1}), gives the desired estimate.

The second estimate in (\ref{w7.11.2}) follows from the equation in (\ref{wave1}), and (\ref{prd1}), (\ref{prd_2}) in Lemma \ref{pres} with $\ep=s-2$.
\end{proof}

We can similarly conclude the following result.

\begin{lemma}\label{pweng}
Consider the equation (\ref{llbox_1}), under the bootstrap assumption (\ref{BA1}), there holds
\begin{equation*}
\|\phi_\la [t]\|_{\dot{H}^1}\les \|\phi_\la[0]\|_{\dot{H}^1}+ \|\ti R_\la\|_{L_t^1 L_x^2}
\end{equation*}
\end{lemma}

\begin{lemma}\label{rmdt}
Let $0<\ep \le s-2$.
There holds with $\phi_t=\p_t\phi$,  for
$
R_\la[\phi_t]:=-\p_t^2P_\la(\phi_t)+g^{ij}(\phi)\p_i \p_j P_\la(\phi_t),$
\begin{equation*}
\|\la^\ep R_\la [\phi_t]\|_{l_\la^2 L_x^2}\les (\|\bp \phi\|_{L_x^\infty}+1)\|(\phi(0), \p_t\phi(0))\|_{H^s\times H^{s-1}}
\end{equation*}
\end{lemma}
\begin{proof}
Differentiating (\ref{wave1}) yields
\begin{equation}\label{w8.9.3}
R_\la[\phi_t]=P_\la (\p_t \N(\phi, \bp \phi))-P_\la(\p_t (g^{ij}(\phi)) \p_i\p_j \phi)-\left(\widetilde E_\la(\p \phi_t)+\widetilde A(\p^2 \phi_t)\right).
\end{equation}
where the definition of the last two terms can be seen in Lemma \ref{lem8.4.1}.  In view of  (\ref{w8.9.1}) in Appendix A with $G=g'(\phi)\bp \phi$,  we have
\begin{equation}\label{w8.9.2}
\|\la^\ep P_\la(\p_t (g^{ij}(\phi)) \p_i\p_j \phi)\|_{l_\la^2 L_x^2}\les\|\bp \phi\|_{L_x^\infty}\left(\|\p \phi\|_{H^{1+\ep}}+\|\phi[0]\|_{H^2}\right).
\end{equation}
In view of (\ref{w8.9.3}), Lemma \ref{rmdt} can be obtained by  using  (\ref{w8.9.2}), (\ref{prd}), (\ref{e_7_10.1})-(\ref{w7.11.2}).
\end{proof}

Let $C_u[t_1, t_2]$ be a (truncated) null cone\begin{footnote}{In Section \ref{sec_4}, by introducing optical functions $u$
we will further set up the null cones $C_u$.}\end{footnote}  with $t$-variable varying in $[t_1, t_2]\subset [0,T]$.
For any scalar function $f$ we define on $C_u[t_1,t_2]$ the following flux
\begin{equation*}
\F[f]=\int_{C_u} \left(|\sn f|_{\ga}^2+|L f|^2\right) d \mu_\ga dt,
\end{equation*}
where $C_u$ represents $C_u[t_1, t_2]$ for shorthand if no confusion occurs and  $\ga$ denotes the induced metric on $C_u\cap \Sigma_t$.

We can apply the divergence theorem to $\P_\a^{(\bT)}[f]$ in the spacetime region $\tilde{\mathfrak{D}}$,
enclosed by $C_u[t_1,t_2],$  $\Sigma_{t_1}$ and $\Sigma_{t_2}$. Similar to (\ref{qs}),
\begin{align}\label{ff_2}
\F[f] & \les \left|\int_{{\mathbb R}^3} Q_{\a\b}\bT^\a\bT^\b(t_2)d\mu_g\right|+\left|\int_{{\mathbb R}^3} Q_{\a\b}\bT^\a\bT^\b(t_1) d\mu_g\right| \nn\\
& \quad \, +\left|\int_{\tilde{\mathfrak{D}}} \left(\Box_{\bg} f \bd_\bT f+\f12 \pi^{(\bT)}_{\a\b} Q^{\a\b}\right)\right|.
\end{align}
\begin{lemma}\label{lem:flux}
\begin{equation}\label{flux}
\F[\bp \phi]+\sum_{\mu>1}\mu^{2(s-2)} \F[(\bp \phi)_\mu]\les \|(\phi(0), \p_t\phi(0))\|^2_{H^s\times H^{s-1}}.
\end{equation}
\end{lemma}

\begin{proof}
Applying (\ref{ff_2}) to $(\bp \phi)_\mu$, and using  $\|\bp g\|_{L_t^1 L_x^\infty}\les 1$ from (\ref{BA1}),  we  obtain
\begin{align*}
&\Big(\sum_{\mu>1}\mu^{2(s-2)}\F[(\bp \phi)_\mu]\Big)^\f12\\
&\les \sup_{t\in[t_1,t_2]} \Big(\sum_{\mu>1}  \|\mu^{s-2}\bp(\bp\phi)_\mu\|_{L_x^2}^2\Big)^\f12
+\int_{t_1}^{t_2} \Big(\sum_{\mu>1}\|\mu^{s-2}\Box_\bg (\bp \phi)_\mu\|_{L_x^2}^2\Big)^\f12 dt.
\end{align*}
The first term on the right hand side can be treated by Proposition \ref{eng3}.
For  the other term, we use (\ref{bog}) for $\Box_\bg (\p \phi)_\mu$, the estimate (\ref{rlambda}),
and Lemma \ref{rmdt} for $\Box_\bg (\phi_t)_\mu$.  Hence
\begin{align*}
& \Big(\sum_{\mu>1}\mu^{2(s-2)}\F[(\bp \phi)_\mu]\Big)^\f12\\
&\les \|(\phi(0), \p_t\phi(0))\|_{H^s\times H^{s-1}}
+\int_{t_1}^{t_2} \Big(\sum_{\mu>1}\mu^{2(s-2)}\|R_\mu[\phi_t], \mu \ti R_\mu\|_{L_x^2}^2  \Big)^\f12 dt\\
&+\int_{t_1}^{t_2} \|\bp g(\phi)\|_{L_x^\infty}\Big(\sum_{\mu>1}  \|\mu^{s-2}\bp(\bp\phi)_\mu\|_{L_x^2}^2\Big)^\f12 dt
 \end{align*}
 Using Proposition \ref{eng3} again,
 the second inequality in (\ref{flux}) can be obtained. The other one can be proved in the same way.
\end{proof}


\section{\bf Reduction of Strichartz estimates}\label{sec_3}

As mentioned in the introduction, the main goal of this paper is to show that, under the bootstrap assumption (\ref{BA1}),
we actually have the improved estimate (\ref{p2}). We restate this in the following result.

\begin{theorem}[Main estimate]\label{main}
Let $s>2$ and let $\phi$ be the solution of (\ref{wave1}). If $\phi$ satisfies (\ref{BA1}), then there holds with a number $8\delta_0<\delta_1<s-2$
\begin{equation}\label{7.11.300}
\|\bp \phi\|_{L^2_{[0,T]}L_x^\infty}^2+\sum_{\la\ge 2}\la^{2\delta_1}\|\bar P_\la\bp \phi\|_{L^2_{[0,T]} L_x^\infty}^2
\les T^{2\delta} (\|\phi(0)\|_{H^s}^2+\|\p_t\phi(0)\|_{H^{s-1}}^2)
\end{equation}
where $\bar P_\la$ denote the Littlewood-Paley projections with $\sum_{\la} \bar P_\la=Id$ in $L^2({\mathbb R}^3)$ and $0<\d\le \ep_0$.
\end{theorem}

In this section, we first reduce the proof of Theorem \ref{main} to the proof of Strichartz estimates on small time intervals.
Let $\la$ be a fixed large dyadic number and let $0<\ep_0<\frac{s-2}{5}$ be a fixed number as mentioned in the introduction.
By using the bootstrap assumption (\ref{BA1}), we can partition $[0, T ]$ into disjoint union of sub-intervals
$I_k := [t_{k-1}, t_k]$ of total number $\les \la^{8\ep_0}$ with the properties that $|I_k|\le \la^{-8\ep_0} T$ and
\begin{equation}\label{BA2}
\|\bp \phi\|_{L_{I_k}^2 L_x^\infty}^2+\sum_{\mu\ge 2}\mu^{2\delta_0}\|\bar P_\mu\bp \phi\|_{L^2_{I_k} L_x^\infty}^2\les \la^{-8\ep_0}.
\end{equation}
On each time interval $I_k$ we will show the following Strichartz estimate.

\begin{theorem}[Strichartz estimates]\label{dyastric}
 Fix $\la\ge\La$ and let $\psi$ be a solution of
 \begin{equation}\label{eq2}
 \Box_\bg\psi=0
 \end{equation}
on the time interval $I_k$. Then for any $q>2$ sufficiently close to $2$ there holds
 \begin{equation*}
\|P_\la \bp \psi\|_{L^q_{I_k}L_x^\infty}\les \la^{\frac{3}{2}-\frac{1}{q}} \|\psi[t_k]\|_{\dot{H}^1}.
 \end{equation*}
\end{theorem}

\begin{proof}[Proof of Theorem \ref{main} assuming Theorem  \ref{dyastric}]
We use $W(t, s)$ to denote the operator that sends $(f_0, f_1)$ to the solution of $\Box_{\bg} \psi=0$ satisfying the initial conditions
$\psi(s) = f_0$ and $\p_t \psi(s) = f_1$ at the time $s$.  By the reproducing property of LP projections, we can write $\bar P_\la=P_\la P_\la$,
where the LP projections $P_\la$ are associated to a different symbol. Let $\phi_\la:= P_\la\phi$. Recall that
$\Box_{\bg} \phi_\la = R_\la$ by (\ref{bog}), we have from the Duhamel principle that
$$
\phi_\la(t) = W(t, t_{k-1}) \phi_\la[t_{k-1}] +\int_{t_{k-1}}^t W(t, s) (0, R_\la(s)) ds.
$$
Thus
\begin{equation*}
\bar P_\la  \phi(t)=P_\la \left(W(t,t_{k-1})\phi_\la[t_{k-1}]\right)+\int_{t_{k-1}}^t P_\la \left(W(t,s)(0, R_\la(s))\right) ds.
\end{equation*}
By differentiation and the fact $W(t, t)(0, R_\la(t))=0$, we have for $t\in I_k$ that
 \begin{equation*}
 \bar P_\la \bp \phi(t)=P_\la \bp W(t, t_{k-1}) \phi_\la[t_{k-1}]+\int_{t_{k-1}}^t P_\la \bp W(t,s)(0, R_\la(s)) ds.
 \end{equation*}
In view of H\"{o}lder inequality and Theorem \ref{dyastric}, it follows that
\begin{align*}
\|P_\la \bp \left(W(\cdot, t_{k-1}) \phi_\la[t_{k-1}]\right)\|_{L^2_{I_k} L_x^\infty}
&\les |I_k|^{\frac{1}{2}-\frac{1}{q}} \|P_\la \bp \left(W(\cdot, t_{k-1}) \phi_\la[t_{k-1}]\right)\|_{L_{I_k}^q L_x^\infty}\\
&\les \la^{\frac{3}{2}-\frac{1}{q}}|I_k|^{\frac{1}{2}-\frac{1}{q}}\|\phi_\la [t_{k-1}]\|_{\dot{H}^1}.
\end{align*}
Let $\delta=\f12-\frac{1}{q}$. We may use $|I_k|\les T \la^{-8\ep_0}$ and Lemma \ref{pweng} to derive that
\begin{align*}
\|P_\la \bp \left(W(\cdot, t_{k-1}) \phi_\la[t_{k-1}]\right)\|_{L^2_{I_k} L_x^\infty}
&\les \la^{1+(1-8\ep_0)\delta} T^\delta \left(\|\phi_\la[0]\|_{\dot{H}^1}+\|\ti R_\la\|_{L_t^1 L_x^2}\right). 
\end{align*}
Similarly, by using the Minkowski inequality, Theorem \ref{dyastric} and Lemma \ref{lem4.15}, we have
\begin{align*}
\left\| \int_{t_{k-1}}^t P_\la \bp W(t,s)(0, R_\la(s)) ds\right\|_{L^2_{I_k}L_x^\infty}
& \les \int_{t_{k-1}}^{t_k} \left\|P_\la \bp W(t,s)(0, R_\la(s))\right\|_{L_{[s, t_k]}^2L_x^\infty} ds\\
&\les\la^{1+(1-8\ep_0)\delta} T^\delta \|R_\la\|_{L_t^1 L^2_x}.
\end{align*}
Recall that $R_\la = \ti R_\la + \bp g(\phi)\cdot \bp \phi_\la$, we may use (\ref{BA1}) and Lemma \ref{pweng} to derive that
\begin{align*}
\|R_\la\|_{L_t^1 L_x^2} & \les \|\ti R_\la\|_{L_t^1 L_x^2} + \|\bp \phi\|_{L_t^2 L_x^\infty} \|\bp \phi_\la\|_{L_t^2 L_x^2}
 \les \|\ti R_\la\|_{L_t^1 L_x^2} + \|\phi_\la[0]\|_{\dot{H}^1}.
\end{align*}
Therefore
\begin{align*}
\left\| \int_{t_{k-1}}^t P_\la \bp W(t,s)(0, R_\la(s)) ds\right\|_{L^2_{I_k}L_x^\infty}
&\les\la^{1+(1-8\ep_0)\delta} T^\delta \left(\|\phi_\la[0]\|_{\dot{H}^1} + \|\ti R_\la\|_{L_t^1 L^2_x}\right).
\end{align*}
Consequently
\begin{align*}
\|\bar P_\la \bp \phi\|_{L^2_{I_k} L_x^\infty}
\les \la^{1+(1-8\ep_0)\delta} T^\delta \left(\|\phi_\la[0]\|_{\dot{H}^1} + \|\ti R_\la\|_{L_t^1 L^2_x}\right).
\end{align*}
Now we sum over all subintervals $I_k$, remembering that the total number of such intervals is $\les \la^{8\ep_0}$, we obtain
\begin{align*}
\|\bar P_\la \bp \phi\|^2_{L^2_{[0,T]} L_x^\infty}
&\les \la^{8\ep_0} \la^{2+2(1-8\ep_0)\delta} T^{2\delta} (\|\phi_\la[0]\|^2_{\dot{H}^1}+\|\ti R_\la\|_{L_t^1 L^2_x}^2)\\
&\les \la^{-2\d_1} T^{2\delta}(\|\la^{s-1} \phi_\la[0]\|^2_{\dot{H}^1}+\|\la^{s-1}\ti R_\la\|_{L_t^1 L_x^2}^2)
\end{align*}
where  $\d_1=s-2-4\ep_0 -(1-8\ep_0)\d>0$. We can choose $q>2$ such that $0<\delta\le \ep_0$. Then
$s-2>\delta_1> 8\ep_0^2 =8\delta_0$ since we assumed $\delta_0=\ep_0^2$ in (\ref{7.10.6}).
Now we multiply the both sides by $\la^{2\d_1}$ and sum over $\la\ge \La$. Notice that
\begin{align*}
\sum_{\la\ge \La} \|\la^{s-1} \ti R_\la\|_{L_t^1L_x^2}^2
& \les \sum_{\la\ge \La} \|\la^{s-1} \ti R_\la\|_{L_t^2L_x^2}^2 T\les \|\la^{s-1} \ti R_\la\|_{L_t^2 l_\la^2 L_x^2}^2 T\\
& \les (\|\bp \phi(0)\|_{H^{s-1}}^2+\|\phi[0]\|_{H^2}^2) \left(\|\p \phi\|_{L_t^2 L_x^\infty}^2 +1\right)T
\end{align*}
which follows from (\ref{rlambda}) and Proposition \ref{eng3}. Using (\ref{BA1}), we then  complete the proof.
\end{proof}

\subsection{Prove Theorem \ref{dyastric} using dispersive estimate}

In order to prove Theorem \ref{dyastric} on each spacetime slab $I_k\times {\Bbb R}^3$, we consider the coordinate change
$(t,x)\rightarrow(\la ( t-t_k), \la x)$. The interval $I_k$ becomes $I_*=[0, \tau_*]$ with $\tau_*\le \la^{1-8\ep_0}T$.
Under the rescaled coordinates the function $\phi$ and the metric $g$ become
\begin{align}\label{7.26.1}
\phi(\la^{-1}t+t_k, \la^{-1}x) \quad \mbox{and} \quad g(\phi(\la^{-1}t+t_k, \la^{-1}x))
\end{align}
which are still denoted as $\phi$ and $g$. We write
\begin{align}\label{7.26.2}
\bg=-d t^2+g_{ij}dx^j\otimes dx^j
\end{align}
for the corresponding Lorentzian metric under the rescaled coordinates. In view of (\ref{BA2}) and $|I_k|\le \la^{-8\ep_0} T$,
we have $\|\bp g\|_{L_{I_*}^1 L_x^\infty}\les \la^{-8\ep_0}$. Therefore, to prove Theorem \ref{dyastric}, it is equivalent to showing
the following Strichartz estimate on $I_*$ with respect to LP projection $P$ on the frequency domain $\{1/2\le |\xi|\le 2\}$.

\begin{theorem}\label{str2}
There is a large universal constant $C_0$ such that if, on the time interval $I_*=[0,\tau_*]$, there holds
\begin{equation}\label{smallas}
C_0 \|\bp g\|_{L_{I_*}^1 L_x^\infty}\le 1
\end{equation}
then for any solution $\psi$ of $\Box_{\bg} \psi=0$ on the time interval $I_*$ and $q>2$ sufficiently close to $2$, there holds
\begin{equation}\label{str1}
\|P \bp \psi\|_{L_{I_*}^q L_x^\infty}\les \|\psi[0]\|_{\dot{H}^1}.
\end{equation}
\end{theorem}

The proof of Theorem \ref{str2} crucially relies on the following decay estimate.

\begin{theorem}[Decay estimate]\label{decayth}
Let $0<\ep_0<\frac{s-2}{5}$ be a given number. There exists a large number $\La$ such that for
any $\la\ge \La$ and any solution $\psi$ of the equation
\begin{equation}\label{wave.4}
\Box_\bg \psi=0
\end{equation}
on the time interval $I_*=[0, \tau_*]$ with $\tau_*\le \la^{1-8\ep_0} T$ there is a function $d(t)$
satisfying
\begin{equation}\label{correccondi}
 \|d\|_{L^{\frac{q}{2}}}\les 1, \mbox{ for } q>2 \mbox{ sufficiently close to }2
 \end{equation}
such that for any $0\le t\le \tau_*$ there holds
\begin{equation}\label{decay}
 \|P \p_t \psi(t) \|_{L_x^\infty}\le \left(\frac{1}{{(1+t)}^{\frac{2}{q}}}+d(t)\right)
\left(\sum_{m=0}^3\|\p^m \psi(0)\|_{L_x^1}+\sum_{m=0}^2 \|\p^m \p_t \psi(0)\|_{L_x^1}\right).
\end{equation}
\end{theorem}

The proof of Theorem \ref{decayth} will be given in next section using the boundedness of conformal energy.
Assuming Theorem \ref{decayth}, we can prove Theorem \ref{str2} by running the $\T \T^*$ argument, see the appendix in
Section \ref{apd_tt}.

\section{\bf Reduce dispersive estimate to bounded conformal energy}\label{sec_4}

In this section we give the proof of Theorem \ref{decayth} on the interval $I_*=[0, \tau_*]$ under the
rescaled coordinates, where $\tau_*\le\la^{1-8\ep_0} T$. For each $t$ we set $\Sigma_t =\{t\}\times {\mathbb R}^3$
on which we use $g=g(t)$ to denote the Riemannian metric defined by (\ref{7.26.1}). Given $\rho>0$ and ${\bf p}\in \Sigma_t$ we use
$B_\rho({\bf p})$ and $B_\rho({\bf p}, g)$ to denote the Euclidean ball and the geodesic ball on $\Sigma_t$ with respect to $g$.
We can find $R>0$ such that
\begin{equation}\label{8.7.1}
B_R({\bf p}) \subset B_{\f12}({\bf p}, g(t)), \quad \forall {\bf p}\in \Sigma_t \mbox{ and } 0\le t\le \tau_*
\end{equation}
This is possible due to the ellipticity condition from (\ref{g9}). Now we fix $t_0=1$. We may take a sequence of Euclidean
balls $\{B_J\}$ of radius $R$ such that their union covers ${\mathbb R}^3$ and any ball in this collection intersect
at most $20$ other balls. Let $\{\varrho_J\}$ be a partition of unity subordinate to the cover $\{B_J\}$. We may assume that
$\sum_{m=1}^3|\p^m\varrho_{J}|_{L_x^\infty}\le C_1$ uniformly in $J$. By using this partition of unity and a standard argument
we can reduce the proof of Theorem \ref{decayth} to establishing the following dispersive estimate for the solution of
$\Box_{\bg} \psi =0$ with $\psi[t_0]:=(\psi(t_0), \p_t \psi(t_0))$ supported on a Euclidean ball of radius $R$.

\begin{proposition}\label{lcestimate}
There is a large constant $\La$ such that for any $\la\ge \La$ and
any solution $\psi$ of
\begin{equation*}
\Box_\bg \psi=0
\end{equation*}
on the time interval $[0, \tau_*]$ with $\tau_*\le \la^{1-8\ep_0} T$ and with $\psi[t_0]$ supported in the Euclidean
ball $B_{R}$  of radius $R$, there exists a function $d(t)$ satisfying
\begin{equation}\label{correc}
 \|d\|_{L^\frac{q}{2}[0,\tau_*]}\les 1 \quad  \mbox{for  } q>2 \mbox{ sufficiently close to }2
 \end{equation}
 such that for all $t_0\le t\le \tau_*$,
 \begin{equation}\label{decaycp}
 \|P \p_t \psi(t)\|_{L_x^\infty}\les \left(\frac{1}{{(1+|t-t_0|)}^{\frac{2}{q}}}+d(t)\right)(\|\psi[t_0]\|_{H^1}+\|\psi(t_0)\|_{L^2}).
 \end{equation}
\end{proposition}

In order to show  Theorem \ref{decayth} using Proposition \ref{lcestimate}, we need the following energy estimates.

\begin{lemma}\label{geoeg}
Under the bootstrap assumption (\ref{BA1}), there holds, for any solution $\psi$ of (\ref{wave.4}), the standard energy estimate
\begin{equation}\label{geoge}
\|\psi[t]\|_{\dot{H}^1}\les \| \psi[0]\|_{\dot{H}^1},
\end{equation}
and for $0<t\le t_0$,
\begin{equation*}
\|\psi(t)\|_{L_x^2}\les  \|\psi[0]\|_{\dot{H}^1}+\|\psi(0)\|_{L_x^2}.
\end{equation*}
\end{lemma}

\begin{proof}
(\ref{geoge}) follows from (\ref{eng7.10}). The other inequality then can be obtained by using the fundamental theorem of calculus.
\end{proof}

\begin{proof}[Proof of Theorem \ref{decayth} assuming Proposition \ref{lcestimate}]
For $0<t<t_0$, we may use the Bernstein inequality of LP projections to obtain
\begin{equation*} 
\|P \p_t \psi(t)\|_{L_x^\infty}\les \|\p_t \psi(t)\|_{L_x^2}.
\end{equation*}
which, together with Lemma \ref{geoeg} and the Sobolev embedding $W^{2,1} \hookrightarrow L^2$,
implies the desired inequality. So we may assume $t_0\le t\le \tau_*$. By using the partition of unity $\{\varrho_J\}$ we can write
$
\psi =\sum_J \psi_J,
$
where $\psi_J$ is the solution of $\Box_{\bg} \psi_J=0$ satisfying the initial conditions
\begin{equation*}
\psi_J(t_0)=\varrho_J \psi(t_0),\qquad  \p_t \psi_J(t_0)=\varrho_J  \p_t \psi(t_0).
\end{equation*}
By using (\ref{decaycp}) in Proposition \ref{lcestimate}, we have
\begin{equation*} 
 \|P \p_t \psi_J(t)\|_{L_x^\infty}\les \left(\frac{1}{{(1+|t-t_0|)}^{\frac{2}{q}}}+d(t)\right)(\|\psi_J[t_0]\|_{H^1}+\|\psi_J(t_0)\|_{L^2}).
\end{equation*}
In view of Lemma \ref{geoeg} and the Sobolev embedding $W^{2,1}\hookrightarrow L^2$, we obtain
\begin{equation*} 
\|P \p_t \psi_J(t)\|_{L_x^\infty}\les \left(\frac{1}{{(1+|t-t_0|)}^{\frac{2}{q}}}+d(t)\right)
\left(\sum_{m=0}^3 \|\p^m \psi_J(0)\|_{L_x^1} + \sum_{m=0}^2 \|\p^m \p_t \psi_J(0)\|_{L_x^1}\right).
\end{equation*}
Summing over $J$ and using the fact that any ball $B_J$ intersects with at most $20$ other balls, we can conclude
the desired estimate.
\end{proof}

Next we will reduce the proof of Proposition \ref{lcestimate} to controlling certain weighted energy. To define such energy, we first
introduce the necessary geometric setup.

Recall that $\Sigma_t = \{t\}\times {\mathbb R}^3$ and $\bT := \p_t$ is the future directed unit normal to $\Sigma_t$
in $\M:=({\mathbb R}^{1+3}, \bg)$. The second fundamental form of $\Sigma_t$ is defined by
$$
k(X, Y) = -\bg (\bd_X \bT, Y)
$$
for any vector fields $X, Y$ tangent to $\Sigma_t$. The trace of $k$ is denoted by $\Tr \, k= g^{ij} k_{ij}$.
Let $\Ga^+$ be the portion of the time axis in $[0,\tau_*]$. We define the optical function $u$ to be the solution of eikonal equation
\begin{equation}\label{optial}
\bg^{\a\b}\p_\a u\p_\b u=0
\end{equation}
with $u=t$ on $\Ga^+$. Geometrically the optical function $u$ can be constructed as follows.
Let ${\bf p}$ be any point on $\Ga^+$ with time coordinate $t$ and let $L_\omega, \omega\in {\Bbb S}^2$,
be the family of null vectors in $T_{\bf p} \M$. For each $\omega\in {\mathbb S}^2$, there exists a unique null
geodesic $\Upsilon_\omega(t)$ initiating from ${\bf p}$ with
\begin{equation}\label{iniv}
\left.\frac{d}{dt}\Upsilon_\omega\right|_{\bf p} = L_\omega.
\end{equation}
We define $u =t$ on the ruled surface formed by this family of null geodesics. This function $u$ is then a solution of
(\ref{optial}) with $u= t$ on $\Ga_+$, see \cite{CK}.

For each $0\le u\le \tau_*$, the level set $C_u$ of $u$ is the outgoing null cones
with vertex on $\Ga^+$ at $t=u$. Let $S_{t,u}=C_u\cap \Sigma_t$ which is a smooth surface. We introduce the two solid cones
$$
\D^+_0=\bigcup_{\{t\in [t_0,\tau_*],0\le u\le t\}} S_{t,u} \quad \mbox{and}\quad
\D^+=\bigcup_{\{t\in [0,\tau_*], 0\le u\le t\}} S_{t,u}.
$$

\begin{figure}[ht]
\centering
\includegraphics[width = 0.49\textwidth, height= 1.6in]{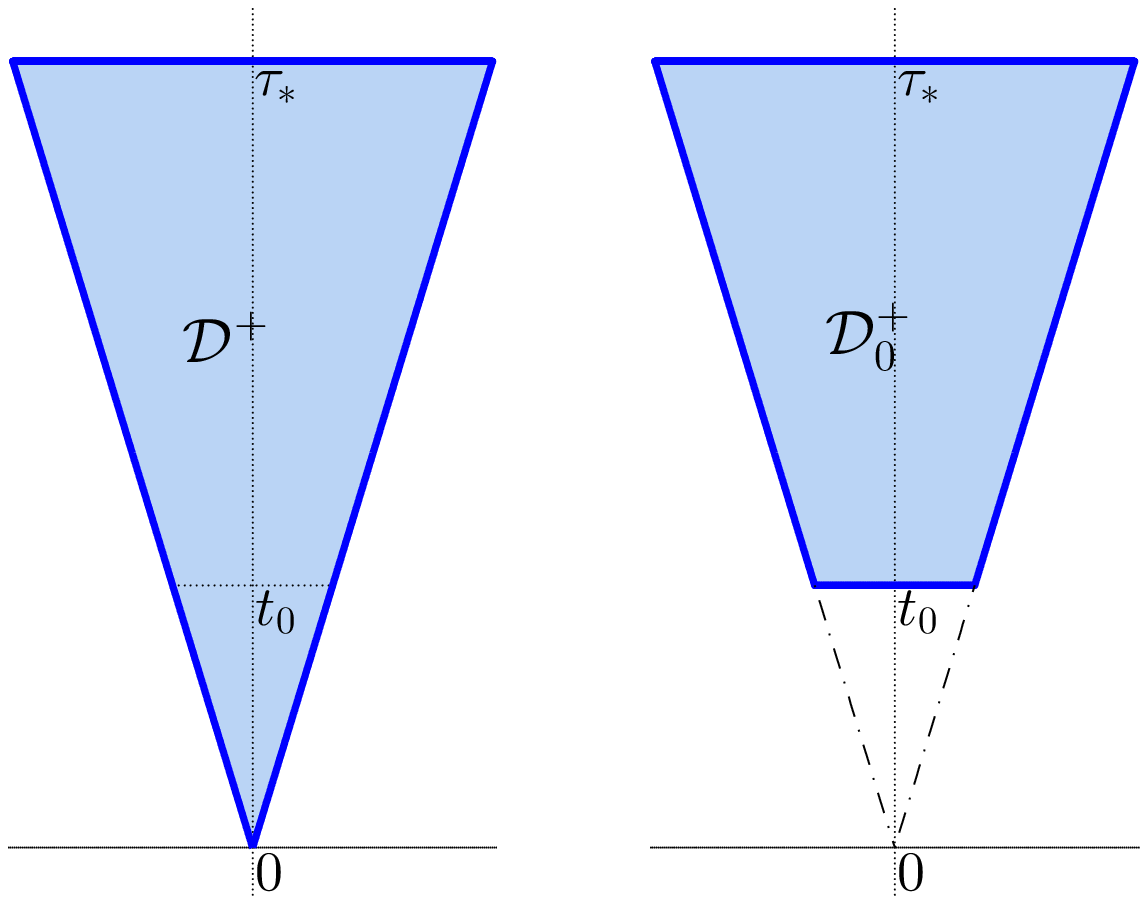}
  \includegraphics[width = 0.46\textwidth, height= 1.6in]{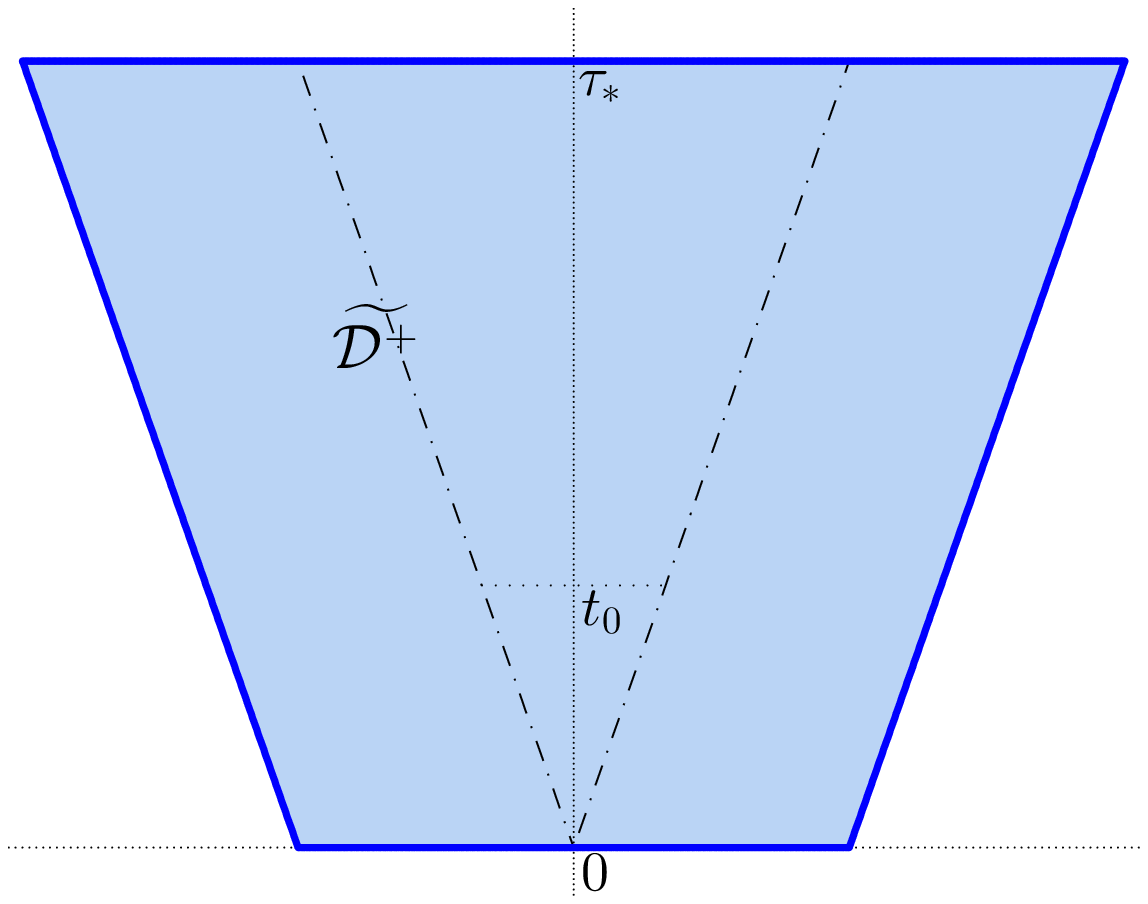}
  \caption{Illustration of $\D^+$, $\D_0^+$ and $\widetilde{\D_+}$}\label{fig0}
\end{figure}

Next we extend the foliation of spacetime by null hypersurfaces to a neighborhood of  $\D^+$ in $[0, \tau_*]\times {\mathbb R}^3$.
We denote by ${\bf o}$ the point on $\Ga^+$ where $t=u=0$. Recall that $\tau_*\le \la^{1-8\ep_0} T$ and $0<T\le d_0$, where $d_0>0$ is defined in Section 2.
Let $v_*=\frac{4}{5} \tau_*$. We can guarantee that there is a neighborhood ${\mathscr O}$ of ${\bf o}$ contained in the geodesic ball
$B_{\tau_*}({\bf o}, g)$ of radius $\tau_*$ on $\{t=0\}$ such that it can be foliated by the level sets $S_v$ of a function $v$
taking all values in $[0, v_*]$ with $v({\bf o})=0$ and with each $S_v$, for $v>0$, diffeomorphic to ${\mathbb S}^2$;
see Proposition \ref{exten} below for various nice properties.
Let $a^{-1}=|\nab v|_g$ be the lapse function on $\{t=0\}$. Then, in ${\mathscr O}$ the metric $g$ can be written as
\begin{equation}\label{radial}
ds^2=a^{2} dv^2+\ga_{AB} d\omega^A d\omega^B,
\end{equation}
where $\omega^A$, $A=1,2$, denote the angular variables on ${\mathbb S}^2$. For each $0<v<v_*$, at any point ${\bf p}\in S_v$,
parametrized by $\omega\in {\mathbb S}^2$, there is a unique outward unit normal $N_\omega$.
We set $L_\omega:=N_\omega+\bT$, which is a null vector in $T_{\bf p}\M$. Initiating from ${\bf p}$,
there is a unique null geodesic $\Upsilon_\omega(t)$ satisfying (\ref{iniv}). We then define $u=-v$ on such ruled surface
formed by the family of null geodesics $\Upsilon_\omega, \omega\in {\Bbb S}^2$ for each $0<v\le v_*$. This gives an extension
of $u$ satisfying (\ref{optial}) in the extended region. Similar as above, we can introduce null cones $C_u$ and $S_{t, u}$
for $-v_*\le u<0$ and $t\in [0, \tau_*]$.
The causal future of $\cup_{0\le v\le v_*}S_v$ can be foliated by the null hypersurfaces $C_u$, initiating
from $S_{v}$ with  $u=-v$ at $t=0$, and those initiating from the time axis $\Ga^+$ with $u=t$. We set
$$
\widetilde{\D^+}=\bigcup_{\{t\in [0,\tau_*], -v_*\le u\le t\}} S_{t,u}.
$$
The local coordinates $(\omega^1, \omega^2)$, together with functions $t,u$ define a complete system of coordinates on $\widetilde{\D^+}$.\begin{footnote}{This  can be justified  by running the same continuity  argument for proving \cite[Theorem 1.2]{Wang10}, based on  (\ref{comp2}), Proposition \ref{ricpr}, (\ref{bb_4}) and (\ref{BA1}).}\end{footnote}  We set $\tir = t - u$.
In $\widetilde{\D^+}$, we define the function of null lapse $\bb$ by
\begin{equation}\label{bb2}
\bb^{-1}=|\nab u|_g.
\end{equation}
It is clear that
\begin{equation}\label{db}
\bb=a \quad \mbox{ when } t=0.
\end{equation}
Moreover, in view of (\ref{optial}), we have $\bb^{-1} =\bT(u)$. Since $u=t$ on $\Ga^+$, we have $\bb=1$ on $\Ga^+$.

In Lemma \ref{6.17.1} we will show that $|\bb-1|<\frac{1}{4}$ on $\widetilde{\D^+}$. Using this fact
we can prove that $B_{\f12}({\bf p}, g(t_0))\subset\D^+_0\cap\Sigma_{t_0}$, where ${\bf p}$ denotes the
point on $\Ga_+$ with time coordinate $t_0$. To see this, we introduce geodesic distance $\ti s$ to ${\bf p}$
on $\D^+_0\cap\Sigma_{t_0}$. Let $\sl{g}$ be the induced metric of $g$ on the level sets of $\ti s$
which are diffeomorphic to ${\Bbb S}^2$. Then on $\D^+_0\cap\Sigma_{t_0}$, we can write $g(t_0)$ as
$d {\ti s}^2 +\sl{g}_{AB} d\omega^A d\omega^B$. In view of (\ref{bb2}), this gives
\begin{equation}\label{7.12.1}
-\frac{\p u}{\p \ti s}\le \bb^{-1}.
\end{equation}
Thus for any point in $B_{\f12}({\bf p}, g(t_0))$ with coordinates $(t, u, \omega)$, by using (\ref{7.12.1}) we have
$$
t_0-u=\int_0^\f12 \bb^{-1} d{\ti s}< \frac{2}{3}
$$
This implies that $B_{\f12}({\bf p}, g(t_0)) \subset \cup_{t_0-\frac{2}{3}< u\le t_0} S_{t_0, u}$ which is in
the interior of $\D_0^+\cap \Sigma_{t_0}$. Consequently, in view of (\ref{8.7.1}), we have
$$
B_R({\bf p}) \subset \D_0^+\cap \Sigma_{t_0}.
$$


We denote by $N$ the outward unit normal of $S_{t,u}$ in $\Sigma_t$.
Let $\ga$ be the metric on $S_{t,u}$ induced from $g$ and let $\sn$ be the corresponding covariant differentiation on $S_{t,u}$.
The second fundamental form of $S_{t,u}$ in $\Sigma_t$ for $0\le t\le\tau_*$ is given by
\begin{equation*}
\theta(X, Y)=\l \nab_X N, Y\r
\end{equation*}
for any vector fields $X, Y$ tangent to $S_{t,u}$. The trace of $\theta$ is defined by $\tr\theta=\ga^{AB} \theta_{AB}$, and the traceless part of $\theta$ is denoted by $\hat \theta$.

Let $(e_A)_{A=1,2}$ be an orthonormal frame on $S_{t, u}$. Then (see \cite{CK,KRduke})
\begin{equation}\label{thetan}
\sn_N N=-\sn\log \bb, \qquad  \sn_A N_B=\theta_{AB}.
\end{equation}
As in \cite{CK,KRduke}, we introduce $L=\p_t +N$ and $\Lb=\p_t-N$. Then $\{L, \Lb\}$ form a null pair with $\bg(L,L)=\bg(\Lb,\Lb)=0$
and $\bg(L,\Lb)=-2$. Under the null frame $(e_A)_{A=1,2}$, $e_3=\Lb$ and $e_4=L$, the null second fundamental forms $\chi$ and $\chib$,
the torsion $\zeta$, and the Ricci coefficient $\zb$ of the foliation $S_{t,u}$ are defined by
\begin{equation}\label{ricc_def}
\begin{split}
\chi_{AB}=\bg(\bd_A e_4, e_B), &\qquad \chib_{AB}=\bg (\bd_A e_3, e_B)\\
\zeta_A=\f12 \bg(\bd_3 e_4, e_A), &\qquad \zb_A=\f12 \bg (\bd_4 e_3, e_A).
\end{split}
\end{equation}
Then there hold
\begin{align}\label{6.7con}
\begin{split}
\bd_A e_4=\chi_{AB} e_B-k_{AN} e_4, &\qquad  \bd_A e_3=\chib_{AB}e_B+k_{AN} e_3, \\
\bd_4 e_4=-k_{NN} e_4,   &\qquad  \bd_4 e_3=2\zb_A e_A+  k_{NN} e_3, \\
\bd_3 e_4=2\zeta_A e_A+k_{NN} e_4, & \qquad \bd_4 e_A=\sn_L e_A+\zb_A e_4, \\
\bd_B e_A=\sn_B e_A+\f12 \chi_{AB} e_3+\f12 \chib_{AB} e_4& \qquad \bd_3 e_3 = -2 (\sn_A \log \bb) e_A -k_{NN} e_3
\end{split}
\end{align}
and
\begin{equation}\label{3.19.1}
\chi_{AB}=\theta_{AB}-k_{AB},\quad\zb_A=-k_{AN}, \quad  \zeta=\sn \log \bb+k_{AN}.
\end{equation}

Relative to any coordinate, the Ricci curvature of $\bg$ can be decomposed as
\begin{equation}\label{ricc6.7.1}
\bR_{\a\b}=-\f12 \Box_\bg (\bg_{\a\b})+\f12 (\bd_\a V_\b+\bd_\b V_\a)+S_{\a\b}
\end{equation}
where $V$ is a 1-form with
\begin{equation}\label{ricc6.7.2}
V_\ga=(\Ga_{\a\b}^\eta-{\hat \Ga}_{\a\b}^\eta)\bg^{\a\b}\bg_{\ga\eta},
\end{equation}
with $\hat\Ga $ being the Christoffel symbol of a smooth reference metric $\hat \bg$.
For convenience, $\hat\bg$ is chosen to be the Minkowski metric $\bm$. Under this choice, the term
$S_{\a\b}$ is quadratic in $\bp \bg$. We set $V_4=V(e_4)$.

The existence of a $v$-foliation with the desired properties in a neighborhood of ${\bf o}$ on $\{t=0\}$ is guaranteed by the following result.

\begin{proposition}\label{exten}
On $\{t=0\}$ there exists a function $v$ with $0\le v\le v_*=\frac{4}{5}\tau_*$ such that each level set $S_v$ is diffeomorphic to ${\Bbb S}^2$ and
\begin{equation}\label{amc}
{\emph\tr} \, \theta + k_{NN}=\frac{2}{av} + {\emph\Tr} \, k-V_4, \qquad a({\bf o})=1.
\end{equation}
Let $\ga^{(0)}$ be the canonical metric on ${\Bbb S}^2$ and $\ga$ the induced metric of $g$ on $S_v$. Let $\cga=v^{-2} \ga$
and $\ckk\ga=v^2\ga^{(0)}$. Then on $\cup_{0\le v\le v_*}S_v$ there hold
\begin{align}
&|a-1|\les \la^{-4 \ep_0} <\frac{1}{4},\quad \|v^{\f12}(\hat\theta, \sn \log a) \|_{L_\omega^{q_*}}\les \la^{-\f12},\label{a_3} \\
&\|\sn \log a\|_{L_v^2 L^\infty_{S_v}} +\| \chih\|_{L_v^2 L^\infty_{S_v}}\les\la^{-\f12}, \label{a_4}\\
&|\cga-\ga^{(0)}|+\|\p(\cga-\ga^{(0)})\|_{L_\omega^{q_*}(S_v)} \les \la^{-4\ep_0}, \label{a_5}\\
&\|v^\f12\sn (\log \sqrt{|\ga|}-\log \sqrt{|\ckk\ga|})\|_{L_\omega^{q*}}\les \la^{-\f12}, \label{4a_6}
\end{align}
where $0<1-\frac{2}{q_*}< s-2$. Moreover
\begin{equation}\label{w8.1.1}
|a-1|\les \la^{-4\ep_0},\qquad \|v^{-\f12}(a-1)\|_{L^\infty}\les \la^{-\f12}, \qquad \sqrt{|\ga|}\approx v^2
\end{equation}
and there holds the inclusion
\begin{equation}\label{w7.13.1}
\cup_{0\le v\le v_*} S_v\subset B_{\tau_*}({\bf o}).
\end{equation}
\end{proposition}

The existence of the $v$- foliation  for $v \in [0,\ep)$ with a small $\ep>0$  can be proved by Nash-Moser implicit function theorem (see \cite{JS1,JS}).
In Appendix C, we will prove the above estimates. During the proof, we will simultaneously show that the foliation can be extended to $v_*$
by an argument of continuity.

In order to give the definition of our conformal energy, we take two smooth cut-off functions $\underline{\varpi}$ and $\varpi$
depending only on two variables $t,u$; for $t>0$ they are defined in a manner such that
\begin{equation*}
\underline{\varpi}=\left\{\begin{array}{lll}
1  & \mbox{ on }\,  0\le u\le t\\
0  & \mbox { on } u \le -\frac{t}{4} ,
\end{array}\right.
\quad \mbox{and} \quad
 \varpi=\left\{\begin{array} {lll}
 1  & \mbox { on } 0\le\frac{u}{t}\le \f12 \\
 0  & \mbox{ if } \frac{u}{t}\ge \frac{3}{4} \mbox { or } u\le  -\frac{t}{4}.
\end{array}\right.
\end{equation*}
We may define $\varpi$ and $\underline{\varpi}$ such that they coincide in the region $\cup_{\{t\in [t_0, \tau_*],-\frac{t}{4}<u\le 0 \}}S_{t,u}$.

\begin{definition}\label{cfen}
For any scalar function $\psi$ vanishing outside $\D^+$,  we define the conformal energy $\C[\psi]$ of $\psi$ by
 \begin{equation*}
\C[\psi](t)=\C[\psi]^\bi(t)+\C[\psi]^\be(t),
 \end{equation*}
 where
 \begin{align}\label{confen}
 &\C[\psi]^\bi(t)=\int_{\Sigma_t} (\underline{\varpi}-\varpi)t^2\left(|\bd \psi|^2+|\tir^{-1} \psi|^2\right) d\mu_g,\\
 &\C[\psi]^\be(t)=\int_{\Sigma_t}\varpi \left(\tir^2|\bd_L \psi|^2+\tir^2|\sn \psi|^2+|\psi|^2\right) d\mu_g.
 \end{align}
\end{definition}

For the conformal energy we have the following boundedness result whose proof forms the core part of this paper and is given through Sections 5--7.

\begin{theorem}[Boundedness theorem]\label{BT}
Let (\ref{BA1}) and (\ref{BA2}) hold. Let $\psi$ be any solution of $\Box_\bg \psi=0$ on $I_*=[0, \tau_*]$ with $\psi[t_0]$ supported in
$B_R\subset \D^+\cap \Sigma_{t_0}$.  Then, for $t\in [t_0, \tau_*]$, the conformal energy of $\psi$ satisfies the estimate
$$
\C[\psi](t)\les(1+t)^{2\ep}\left(\|\psi[t_0]\|_{\dot{H}^1(\Sigma)}^2+\|\psi(t_0)\|_{L^2(\Sigma)}^2\right),
$$
where $\ep>0$ is an arbitrarily small number.
\end{theorem}
In what follows,  the small growth indicated in Theorem \ref{BT} will be proved to be harmless.
We adapt \cite[Proposition 11]{Wangrough} to prove Proposition \ref{lcestimate} by assuming Theorem \ref{BT}.

\subsection{Reduction to bounded conformal energy}

In this section we give the proof of Proposition \ref{lcestimate} using Theorem \ref{BT} concerning the boundedness of conformal energy.

\begin{lemma}\label{comh1}
Let $b\ge 2$. Then for any scalar functions $f$ and $\psi$ there holds
\begin{equation*}
\|[P,f]\p \psi\|_{W^{1,b}}+\|[P, f]\p \psi\|_{L_x^\infty}\les \|\p f\|_{L_x^\infty}\|\p \psi\|_{L_x^2},
\end{equation*}
where  $P$ denote the Littlewood-Paley projection with frequency $1$.
\end{lemma}

\begin{proof}
This is \cite[Lemma 14]{Wangrough}.
\end{proof}

\begin{lemma}\label{err1111}
Let (\ref{BA2}) hold. Then for $q>2$ close to $2$ there holds
\begin{equation}\label{aricc}
\left\|\varpi\left(\chih, \sn \log \bb, \tr\chi-2/\tir\right)\right\|_{L^\frac{q}{2}[0,\tau_*] L_x^\infty}
\les\la^{ \frac{2}{q}-1-4\ep_0(\frac{4}{q}-1)}.
\end{equation}
\end{lemma}

\begin{proof}
The result will be verified in Section \ref{sec_5} where (\ref{aricc}) is restated as (\ref{ric1.1}) in Proposition \ref{ricco}.
\end{proof}

\begin{lemma}\label{err11}
Let (\ref{BA2}) hold. Then, for any solution $\psi$ of $\Box_{\bg} \psi =0$ in $I_*=[0, \tau_*]$ with $\tau_*\le \la^{1-8\ep_0} T$,
there holds
\begin{equation}\label{4.18.1}
\left\|[P, \varpi N^i] \p_i\psi\right\|_{L^\infty(\Sigma_t)}+ \left\|[P,\varpi\sn]\psi \right\|_{H^1(\Sigma_t)}
\les \ti d(t)\|\psi[t_0]\|_{\dot{H}^1},
\end{equation}
where
\begin{equation*}
(1+t)^{\delta}\ti d(t)\les (1+t)^{-\frac{2}{q}}+d(t) \quad \mbox{and} \quad  \|d(t)\|_{L^{\frac{q}{2}}[0, \tau_*]}\les 1
\end{equation*}
with $q>2$ close to $2$ and $0<\delta< 1-\frac{2}{q}$.
\end{lemma}

\begin{proof}
Notice that $\sn_j \psi = \p_j \psi -N_j N^i \p_i \psi$ with $N_j = g_{ij} N^i$. We have
$$
[P, \varpi \sn_j]\psi = [P, \varpi] \p_j \psi - [P, \varpi N_j N^i] \p_i \psi.
$$
 In view of Lemma \ref{comh1}, we can obtain
\begin{align*}
& \|[P, \varpi N^i]\p_i \psi\|_{L_x^\infty} +\|[P, \varpi \sn] \psi\|_{H^1} \les \left(\|\p (\varpi N^i)\|_{L_x^\infty} +\|\p(\varpi N_j N^i)\|_{L_x^\infty}+\|\p \varpi\|_{L_x^\infty}\right) \|\p \psi\|_{L_x^2}
\end{align*}
Since the support of $\varpi$ on $\Sigma_t$ with $t\ge t_0$ is contained in $\cup_{-t/4\le u\le 3t/4} S_{t, u}$, we have
$\|\p \varpi\|_{L_x^\infty} \les (1+t)^{-1}$ and $t-u\approx t$. By using (\ref{thetan}), (\ref{3.19.1}) and $k=-\f12\p_t g$ we have
\begin{align}\label{h1theta}
& \|\p (\varpi N^i)\|_{L_x^\infty}+ \|\p (\varpi N_j N^i)\|_{L_x^\infty} \nonumber\\
& \les \|\varpi \left(\chih, \sn \log \bb, \tr\chi-2/\tir\right), k,\p g\|_{L^\infty}+(1+t)^{-1}.
\end{align}
Let
$$
\ti d(t) =  \|\varpi \left(\chih, \sn \log \bb, \tr\chi-2/\tir\right), k,\p g\|_{L_x^\infty}+(1+t)^{-1}.
$$
Then we have
\begin{align*}
& \|[P, \varpi N^i] \p_i\psi\|_{L_x^\infty} +\|[P, \varpi \sn] \psi\|_{H^1} \les \ti d(t) \|\p \psi\|_{L_x^2},
\end{align*}
which, together with Lemma \ref{geoeg}, implies (\ref{4.18.1}).

We can write $\ti d(t) = (1+t)^{-1} + (1+t)^{-\d} d(t)$, where
$$
d(t) = (1+t)^\d \|\varpi \left(\chih, \sn \log \bb, \tr\chi-2/\tir\right), k,\p g\|_{L_x^\infty}
$$
Note that $|k|, |\p g|\les |\bp \phi|$. We may use (\ref{BA2}), (\ref{aricc}) and $t\le \tau_*\les \la^{1-8\ep_0}$ to conclude that
$\|d\|_{L^{\frac{q}{2}}[0, \tau_*]}\les 1$.
\end{proof}

\begin{lemma}\label{basic1}
(i) For $2\le q<\infty$ and any $S_{t, u}$-tangent tensor $F$, there holds
\begin{equation}\label{sob.12}
\|r^{1-2/q}F\|_{L^q(S_{t,u})}\les\|r\sn F\|_{L^2(S_{t,u})}^{1-2/q}\|F\|_{L^2(S_{t,u})}^{2/q}+\|F\|_{L^2(S_{t,u})}.
\end{equation}
(ii) For any $\delta\in (0,1)$, any $q\in (2,\infty)$ and any scalar function $f$
there hold
\begin{align*}
\sup_{S_{t,u}}|f|\les  \tir^{\frac{2\delta(q-2)}{2q+\delta(q-2)}}
&\left(\int_{S_{t,u}} \left(|\sn f|^2+\tir^{-2} |f|^2\right)\right)^{\f12-\frac{\delta q}{2q+\delta(q-2)}}
\left(\int_{S_{t,u}} \left(|\sn f|^q+\tir^{-q}|f|^q\right) \right)^{\frac{2\delta}{2q+\delta(q-2)}}.
\end{align*}
\end{lemma}

\begin{proof}
This is \cite[Theorem 5.2]{KRduke}
 \end{proof}

\begin{lemma}\label{lem113}
For any $\Sigma$-tangent tensor field $F$ there hold
$$
\|F\|_{L^2(S_{t, u})}^2 \les \|F\|_{H^1(\Sigma_t)} \|F\|_{L^2(\Sigma_t)}, \qquad
\|F\|_{L^4(S_{t,u})} \les \|F\|_{H^1(\Sigma_t)}.
$$
\end{lemma}

\begin{proof}
This is \cite[Proposition 7.5]{Wang10}.
\end{proof}

\begin{proof}[Proof of Proposition \ref{lcestimate}]
Since $\Box_{\bg}\psi=0$ and $\psi[t_0]$ is supported on $B_R\subset \D^+\cap \Sigma_{t_0}$,
by the finite speed of propagation, see \cite[Proposition 3.3]{KRduke},
$\psi(t)$, for each $t\ge t_0$, is supported on $\{0\le u\le t\}$ on which $\underline{\varpi}=1$. Thus
$$
\|P \left(\p_t \psi\right)\|_{L_x^\infty}= \|P(\underline{\varpi} \p_t \psi)\|_{L_x^\infty} \le \|P (\varpi \p_t \psi)\|_{L_x^\infty}
+\|P\left((\underline{\varpi}-\varpi) \p_t \psi\right)\|_{L_x^\infty}.
$$
By the Bernstein inequality for LP projections and Theorem \ref{BT}, we have
\begin{align*}
\|P\left( (\underline{\varpi}-\varpi) \p_t \psi\right)\|_{L_x^\infty}
& \les\|(\underline{\varpi}-\varpi) \p_t \psi\|_{L_x^2}
\les (1+t)^{-1} \C[\psi]^\f12(t)\\
& \les(1+t)^{-1+\ep} (\|\psi[t_0]\|_{\dot{H}^1(\Sigma)}+\|\psi(t_0)\|_{L^2(\Sigma)})
\end{align*}
and
\begin{align*}
\|P(\varpi \p_t\psi)\|_{L_x^\infty}&\les \|P(\varpi L\psi)\|_{L_x^\infty}
+\|P(\varpi N\psi)\|_{L_x^\infty}
\les \|\varpi L\psi\|_{L_x^2}+\|P(\varpi N\psi)\|_{L_x^\infty},
\end{align*}
where we used $\p_t = L-N$. From Theorem \ref{BT} we have
$$\|\varpi L\psi\|_{L_x^2}\les (1+t)^{-1} \C[\psi]^{\f12}(t)
\les (1+t)^{-1+\ep}(\|\psi[t_0]\|_{\dot{H}^1(\Sigma)}+\|\psi(t_0)\|_{L^2_\Sigma}).$$
Moreover
\begin{align}
\|P(\varpi N\psi)\|_{L_x^\infty}&\le \|[P, \varpi N^l]\p_l \psi\|_{L_x^\infty}
+\|\varpi N^l P \p_l \psi\|_{L_x^\infty}. \label{eqn19}
\end{align}
The first term in (\ref{eqn19}) can be estimated by Lemma \ref{err11}.
For the second term  in (\ref{eqn19}) we have
\begin{equation*}
\|\varpi N^l P \p_l \psi(t)\|_{L_x^\infty}\le \|\varpi \ti P\psi(t)\|_{L_x^\infty},
\end{equation*}
where $\ti P$ denotes an LP projection with frequency $1$ associated to
a different symbol. Therefore, it suffices to show that
\begin{equation}\label{eqn16}
\|\varpi \ti P \psi\|_{L^\infty(\Sigma_t)}\les \left( (1+|t-t_0|)^{-\frac{2}{q}}+d(t)\right)(\|\psi[t_0]\|_{\dot{H}^1}+\|\psi(t_0)\|_{L^2}).
\end{equation}
Since $\varpi$ vanishes outside the region $\{-t/4\le u< 3t/4\}$, we only need to consider $\varpi \ti P \psi$ in
the region $\{-t/4\le u< 3t/4\}$ on which we have $\tir\approx t$.
Recall that $\varpi$ is constant on each $S_{t,u}$,  we have from Lemma \ref{basic1} (ii) that
\begin{align*}
\sup_{S_{t,u}}|\varpi P\psi|^2
& \les \tir^\delta \left( \int_{S_{t,u}} \left(|\varpi \sn P \psi|^2
+\tir^{-2}|\varpi P\psi|^2\right)\right)^{1-\delta}
\left(\int_{S_{t,u}} \left(|\varpi \sn P\psi|^4
+\tir^{-4}|\varpi P\psi|^4\right) \right)^{\f12\delta}.
\end{align*}
Using $r\approx t$ and Lemma \ref{lem113} we then obtain
\begin{align*}
\sup_{S_{t,u}}|\varpi P\psi|^2 & \les \tir^\delta \left(\int_{S_{t,u}} \left(|P(\varpi\sn
\psi)|^2+\tir^{-2}|\varpi P\psi|^2+|[P,\varpi\sn]\psi|^2\right) \right)^{1-\delta}\\
&\quad \, \times \left(\int_{S_{t,u}} \left(|P(\varpi \sn \psi)|^4+\tir^{-4}|\varpi P\psi|^4
+| [P,\varpi\sn]\psi|^4\right)\right)^{\f12\delta}\\
&\les \tir^\delta \left( \|P(\varpi \sn \psi)\|_{H^1}^2 + t^{-2}\|\varpi P \psi\|_{H^1}^2 + \|[P,\varpi\sn]\psi\|^2_{H^1}\right).
\end{align*}
By the finite band property, we have
\begin{align*}
&\|P(\varpi \sn \psi)\|_{H^1}^2 \les \|\varpi \sn \psi\|_{L^2}^2 \les t^{-2} \C[\psi](t),\\
&\|\varpi P\psi\|_{H^1}^2 \les \|\psi\|_{L^2(\Sigma_t)}^2 = \|\underline{\varpi} \psi\|_{L^2}^2\les \C[\psi](t).
\end{align*}
Therefore
\begin{align*}
\sup_{-t/4\le u\le 3t/4}|\varpi P\psi|^2 & \les t^\d \left(t^{-2} \C[\psi](t) + \|[P, \varpi \sn]\psi\|_{H^1(\Sigma_t)}^2\right)
\end{align*}
By letting $0<\delta< 2( 1-\frac{2}{q})$,  we then obtain (\ref{eqn16}) from Lemma \ref{err11} and  Theorem \ref{BT} since  $\ep$ can verify $\ep<1-\frac{2}{q}-\frac{\delta}{2}$.
\end{proof}

The proof of Theorem \ref{BT} is the core part of this paper which will be given through Sections 5--7. It crucially relies on the control of Ricci coefficients
defined in (\ref{ricc_def}) on $\D^+$  along null hypersurfaces. In particular, we need to tackle the difficulty arising from the
non-vanishing Ricci tensor which contains the second derivative of the unknown solution.  In Section \ref{sec_5},
we will give the control of Ricci coefficients on $\widetilde{\D^+}$, which shows
Lemma \ref{err1111}. In  the wave zone $\D^+$, which is a subset of $\widetilde{\D^+}$, we will provide a
set of stronger estimates.

\section{\bf Control of Ricci coefficients}\label{sec_5}

Let $\phi$ be obtained from the solution of (\ref{wave1}) by the coordinate change $(t, x) \to (\la(t-t_k), \la x)$ as
was done in (\ref{7.26.1}) and let $g = g(\phi)$. They are defined on $I\times {\mathbb R}^3$ with $I= [0,\tau_*]$ and $\tau_*\le \la^{1-8\ep_0}T$.
According to (\ref{BA2}), the rescaled function $\phi$ and the metric $g$ satisfy the estimates
\begin{equation}
\|\bp\phi, \bp g\|_{L_t^2L_x^\infty(I\times {\mathbb R}^3)}+\la^{\delta_0}\left(\sum_{\mu\ge 2 }\mu^{2\delta_0}
\| P_\mu\ti\pi\|^2_{L_t^2 L_x^\infty(I\times {\mathbb R}^3)}\right)^{\f12}\les
\la^{-1/2-4\ep_0}. \label{pi.2}
\end{equation}
 Here $\ti \pi$ denotes the collection of terms taking the form of $f(\phi) \bp \phi$, with $f$ being a smooth function\begin{footnote}{ To be more precise, the functions $f$ are products of factors among $\{g^{(k)}, k=0,1,2\}$, constants, and the function $\N^{\a\b}$ in (\ref{wave1}).}\end{footnote}  of $\phi$. To derive the last inequality in (\ref{pi.2}) we also used (\ref{prd7}) in Appendix A. In the following sections we will work under the condition (\ref{pi.2}).

We will let $\bg$ denote the Lorentzian metric defined by (\ref{7.26.2}) and let $\bd$ denote the corresponding covariant differentiation.
Using the optical function $u$ introduced in Section 4, we have defined the null hypersurfaces $C_u$ and the regions $\D^+$ and $\widetilde{\D^+}$.
On each $S_{t, u}= C_u \cap \Sigma_t$, let $\ga$ be the induced metric from $\bg$ and let $\sn$ be the associated covariant differentiation.


For null cones $C_u$ with $u=-v<0$ initiating from $S_v$ at $t=0$, using local coordinates $(\omega^1, \omega^2)\in {\mathbb S}^2$
on $S_v$ as in Proposition \ref{exten} and following the trajectories $\Upsilon_\omega(t)$ of the null generator vector field $L$, i.e.
\begin{equation}\label{trscoord}
\frac{d \Upsilon_\omega}{dt}=L
\end{equation}
with $\Upsilon_\omega(0)\in S_v$, we can parametrize $C_u$ by the transport coordinates
$t, \omega^1, \omega^2$. For a null cone $C_u$ with $u\ge 0$ initiating from a point $\wp$ on the time axis $\Ga^+$, by parametrizing null vectors
at $\wp$ by $(\omega^1, \omega^2)\in {\mathbb S}^2$, we can use (\ref{trscoord}) with $\Upsilon_\omega(0)=\wp$ to introduce
on $C_u$ the transport coordinates $t, \omega^1, \omega^2$.  Relative to the transport coordinates $t, \omega^1, \omega^2$, we can verify that
\begin{equation}\label{trscoord2}
\frac{d\ga_{ab}}{dt}=2\chi_{ab}
\end{equation}
along null geodesics $\Upsilon_\omega(t)$. Consequently, for the associated area element $v_t=\sqrt{|\ga|}$ of $S_{t, u}$, there holds
\begin{equation}\label{lv}
L(v_t) = v_t \tr \chi.
\end{equation}
Using the transport coordinates, we can also parametrize $\widetilde{\D^+}$ by $t, u, \omega^1, \omega^2$. Thus any function $f$
on $\widetilde{\D^+}$ can be represented by $f(t, u, \omega^1, \omega^2)$. By virtue of such representation we can introduce
various norms which will be used in the following. For any vector field $F$ on $S_{t, u}$ we will use the two norms
$$
\|F\|_{L_x^q(S_{t, u})}^q= \int_{S_{t, u}} |F|^q d \mu_\ga\quad \mbox{and} \quad
\|F\|_{L_\omega^q(S_{t, u})}^q = \int_{{\mathbb S}^2} |F|^q(\omega) d \mu_{{\mathbb S}^2}.
$$
For $S_{t, u}$-tangent tensor field $F$ on $C_u$, we introduce the mixed norms
\begin{align*}
\|F\|_{L_\omega^q L_t^\infty(C_u)}^q&= \int_{{\mathbb S}^2} \sup_{\Upsilon_\omega}|F|^q   d\mu_{{\Bbb S}^2}\quad
\mbox{and} \quad \|F\|_{L_x^q L_t^\infty(C_u)}^q=\int_{{\Bbb S}^2} \sup_{t\in \Upsilon_\omega}(v_{t} |F|^q) d\mu_{{\Bbb S}^2}
\end{align*}
For tensor fields $F$ defined on $\Sigma_t\cap \widetilde{\D^+}$ we use the norm
\begin{align*}
&\|F\|_{L_x^q L_u^\infty}=\left(\int_{{\Bbb S}^2}(\sup_{u}(v_t |F|^q ))(\omega) d\omega\right)^\frac{1}{q}.
 \end{align*}
There are many other mixed norms used in the paper which should be clear from the context.

In this section we will derive various estimates on the Ricci coefficient defined in (\ref{ricc_def}) which are crucial for proving
Theorem \ref{BT} in Section 7. To obtain these estimates, we first recall the initial conditions satisfied by these Ricci coefficients.
We fix the convention that
$$
\tir = t-u, \quad \widetilde{\tr\chi}=\tr\chi+V_4 \quad \mbox{ and } \quad z=\widetilde{\tr\chi}-\frac{2}{t-u},
$$
where $V_4=V(L)$ and $V$ is the 1-form introduced in (\ref{ricc6.7.2}).

\begin{lemma}\label{inii}

 \begin{enumerate}
\item[(i)]On any null cone $C_u$ initiating from a point on the time axis at $t=u\ge 0$, there hold
\begin{align*}
&\tir z,\bb-1, \sn \bb,  \tir\sn z, \tir^2\mu\rightarrow 0 \mbox{ as }
t\rightarrow u,\quad \lim_{t\rightarrow u}\|\chih,\zeta,\zb,k\|_{L^\infty(S_{t,u})}<\infty.
\end{align*}

\item[(ii)] Let $\gac:=(t-u)^{-2}\ga$ be the rescaled metric on $S_{t,u}$ and let ${\gamma}^{(0)}$ denote the
canonical metric on ${\Bbb S}^2$. Then, relative to the transport coordinates,  there hold
\begin{equation}
\lim_{t\rightarrow u}\stackrel{\circ}
\gamma_{ab}={\gamma}_{ab}^{(0)},\qquad  \lim_{t\rightarrow
u}\p_c\!\!\stackrel{\circ}\ga_{ab}=\p_c{\ga}_{ab}^{(0)}\label{8.1.1}
\end{equation}
where $a, b, c=1,2$.

\item[(iii)]
On $\bigcup_{v\in (0,v_*]}S_v$ there hold $|v z|\les \la^{-4\ep_0}$
and $\|v^\frac{3}{2} \sn z\|_{L_v^\infty L_\omega^p}+\|v^\f12 z\|_{L^\infty}\les \la^{-\f12}$,
where $\|F\|_{L_v^\infty L_\omega^p} = \sup_{v\in (0, v_*]} \left(\int_{S_v} |F|^p d\omega\right)^{1/p}$
for any tensor field $F$.
\end{enumerate}
\end{lemma}

The items (i) and (ii) in Lemma \ref{inii} can be found from \cite{Wangthesis} and  \cite[Section 2]{Wangricci}.
 The item (iii) follows from Proposition \ref{exten}.

The main results of this section are Proposition \ref{ricco} and Proposition \ref{ricpr} on connection coefficients which will
be crucially used in Section \ref{sec_6} and Section \ref{sec_7}.

\begin{proposition}\label{ricco}
Let $p$ be a fixed number satisfying $0<1-\frac{2}{p}<s-2$. Let $\D_*=(\sn, \sn_L)$. Under the condition (\ref{pi.2}), there hold on $\widetilde{\D^+}\subset[0,\tau_*]\times \Sigma$ the estimates
\begin{align}
&\tir\widetilde{{\emph \tr} \chi}\approx 1,\quad \|\tir^\f12 z\|_{L^\infty(\widetilde{\D^+})}\les \la^{-\f12}, \label{comp2} \displaybreak[0]\\
&\|\tir^{3/2}\sn z\|_{L_t^\infty L_u^\infty L_\omega^p(\widetilde{\D^+})}\les \la^{-\f12}\label{ricp}, \displaybreak[0]\\
&\|\tir \sn z, \tir \sn \chih\|_{L_t^2 L_\omega^p(C_u\cap \widetilde{\D^+})}\les \la^{-\f12}\label{sna}, \displaybreak[0]\\
&\left\|z, \chih, {\emph\tr}\chi-\frac{2}{t-u}, \zeta \right\|_{L_t^{\frac{q}{2}} L_x^\infty}\les \la^{\frac{2}{q}-1-4\ep_0(\frac{4}{q}-1)},
\quad q>2 \mbox{ and close to } 2, \label{ric1.1} \displaybreak[0]\\
&\left\|\frac{\bb^{-1}-1}{\tir}\right\|_{L_t^2 L_x^\infty} + \left\|\frac{\bb^{-1}-1}{\tir^{\f12}}\right\|_{L^{2p}_\omega}
+\|\tir\D_*(\frac{\bb^{-1}-1}{\tir})\|_{L_t^2 L_\omega^p} \les \la^{-\f12}. \label{ric4}
\end{align}
Moreover, in $\D^+$ there holds
\begin{align}
&\left\|z, \chih, {\emph \tr}\chi-\frac{2}{t-u},\zeta\right\|_{L_t^2 L_x^\infty(\D^+)}\les \la^{-\f12-4\ep_0}. \label{ric1}
\end{align}
\end{proposition}

\begin{proposition}\label{ricpr}
On null cone $C_u$ contained in $\widetilde{\D^+}$, there hold
\begin{align}
&\|z\|_{L_t^2 L_\omega^\infty(C_u\cap \widetilde{\D^+})}+\|\chih\|_{L_t^2 L_\omega^\infty(C_u\cap \widetilde{\D^+})}\les\la^{-\f12}, \label{ric3.18.1}\\
 &\|\p ( \stackrel{\circ}{\ga}-\ga^{(0)})\|_{L_\omega^p L_t^\infty(C_u\cap \widetilde{\D^+})}
 \le \la^{-4\ep_0}, \quad \|\stackrel{\circ}{\ga}-\ga^{(0)}\|_{L^\infty}\les \la^{-4\ep_0}, \label{8.0.3}
\end{align}
where $\stackrel{\circ}{\ga}=(t-u)^{-2}\ga$.
\end{proposition}

The proofs of Proposition \ref{ricco} and Proposition \ref{ricpr} rely on a bootstrap argument.  We make the bootstrap assumptions that
on any $C_u$ contained in $\widetilde{\D^+}$ there hold
\begin{align}
&\|\chih\|_{L_t^2 L_\omega^\infty(C_u)}+\|z\|_{L_t^2 L_\omega^\infty(C_u)}\le \la^{-\f12 +2\ep_0},\label{ba3.18.1}\\
 &\|\p (\stackrel{\circ}{\ga}-\ga^{(0)})\|_{L_t^\infty L_\omega^p}\le \la^{-\ep_0}, \,\|\stackrel{\circ}{\ga}-\ga^{(0)}\|_{L^\infty}\le \la^{-\ep_0}\label{ba3},\\
&\|\chih\|_{L_t^2 L_\omega^\infty(C_u)}+\|z\|_{L_t^2 L_\omega^\infty(C_u)}\le \la^{-\f12} \quad \mbox{in } \D^+,\label{ric2a}\\
& \|\zeta\|_{L_t^2L_\omega^\infty(C_u\cap \D^+)} \le \la^{-\f12} \quad \mbox{in } \D^+. \label{ric3a}
\end{align}
We also assume that on any $S_{t,u}\subset\widetilde{\D^+}$, there hold
\begin{align}
& \|\tir (\chih, z)\|_{L_\omega^p}\le 1 \label{aux_1}, \displaybreak[0]\\
& |\bb-1|\le \f12. \label{bb_3}
\end{align}
We will improve (\ref{ba3.18.1}) to (\ref{ric3.18.1}), (\ref{ric2a}) and (\ref{ric3a}) to (\ref{ric1}),  and (\ref{ba3}) to (\ref{8.0.3}).
(\ref{aux_1}) will be improved in (\ref{pric2}). (\ref{bb_3}) will be improved in (\ref{bb_4}). At the same time, we will
derive all the other conclusions in Proposition \ref{ricco} and Proposition \ref{ricpr}.
The proofs of Proposition \ref{ricco} and Proposition \ref{ricpr} then follow by a classical continuity
argument.

To complete the bootstrap argument, we rely heavily on the following structure equations for the connection coefficients
on null hypersurfaces $C_u$ in $\widetilde{\D^+}$ (see \cite[Chapter 7]{CK} and \cite{KRduke}):
\begin{align}
&L \bb=-\bb { k}_{NN}, \label{lb} \displaybreak[0]\\
&L\tr\chi+\f12 (\tr\chi)^2=-|\chih|^2-{ k}_{NN} \tr\chi-\bR_{44}, \label{s1} \displaybreak[0]\\
&\sn_L \chih_{AB}+\f12 \tr\chi \chih_{AB}=-{k}_{NN} \chih_{AB}-(\bR_{4A4B}-\f12 \bR_{44} \delta_{AB}), \label{s2} \displaybreak[0]\\
&\sn_L \zeta+\f12\tr\chi \zeta=-(k_{BN}+\zeta_B) \chih_{AB}-\f12 \tr\chi k_{AN}-\f12 \bR_{A4 43}, \label{tran2} \displaybreak[0]\\
&(\div \chih)_A+\chih_{AB}\c k_{BN}=\f12(\sn \tr\chi+k_{AN} \tr\chi)+\bR_{B4BA}, \label{dchi} \displaybreak[0]\\
&\div \zeta=\f12(\mu-k_{NN} \tr\chi-2|\zeta|^2-|\chih|^2-2k_{AB}\chih_{AB})-\f12\delta^{AB}\bR_{A43B}, \label{dze} \displaybreak[0]\\
&\curl \zeta=-\f12 \chih\wedge \chibh+\f12 \ep^{AB}\bR_{B43A}, \label{dcurl} \displaybreak[0]\\
&L \tr\chib+\f12 \tr\chi \tr\chib=2\div \zb+k_{NN} \tr\chib-\chih\c \chibh+2|\zb|^2+\delta^{AB}\bR_{A34B}, \label{mub} \displaybreak[0]\\
&\sn_\Lb \chih_{AB}+\f12 \tr\chib \chih_{AB}=-\f12 \tr\chi\chibh_{AB}+2\sn_A\zeta_B-\div \ze \delta_{AB}+k_{NN} \chih_{AB} \label{3chi} \\
&\quad\quad\quad\quad\quad \quad\quad+(2\zeta_A\ze_B-|\ze|^2\delta_{AB})+\bR_{A43B}-\f12 \delta^{CD}\bR_{C43D}\delta_{AB}.\nn
\end{align}

We first draw some simple important conclusions from the bootstrap assumptions (\ref{aux_1}), (\ref{bb_3})
and the second assumption in (\ref{ba3.18.1}).

\begin{lemma}\label{6.17.1}
On $\widetilde{\D^+}$ there holds
\begin{align}
v_t &\approx (t-u)^2, \label{comp_3_27}\\
|\bb-1| &\les \la^{-4\ep_0}<\frac{1}{4}. \label{bb_4}
\end{align}
\end{lemma}

\begin{proof}
Let $C_u$ be a null cone contained in $\widetilde{\D^+}$.  Using (\ref{trscoord2}) we can derive that
\begin{eqnarray}\label{larea_1}
L(v_t\tir^{-2})=\left(\tr\chi-\frac{2}{\tir}\right) v_t \tir^{-2}. 
\end{eqnarray}
By virtue of (\ref{pi.2}) and the estimate for $z$ in (\ref{ba3.18.1}), we have
$
\|\tr\chi-\frac{2}{\tir}\|_{L_t^1 L_x^\infty(C_u)}\le \la^{-2\ep_0}.
$
Therefore, (\ref{comp_3_27}) follows by integrating (\ref{larea_1}) along  null geodesics and using the initial data conditions
given in Lemma \ref{inii} and Proposition \ref{exten} for the cases $u\ge 0$ and $u<0$.

Next we derive (\ref{bb_4}) using (\ref{lb}). For null cones $C_u$ with $u\ge 0$, by integrating (\ref{lb}) along null geodesics
and using Lemma \ref{inii} (i), we have
\begin{equation}\label{w8.1.2}
\bb-1=-\int_u^t \bb k_{NN} dt'
\end{equation}
which together with (\ref{bb_3}) and (\ref{pi.2}) implies that $|\bb-1|\les \la^{-8\ep_0}<\frac{1}{4}$.
For null cones $C_u$ with $u=-v<0$, by integrating (\ref{lb}) along null geodesics and using (\ref{db}) we have
\begin{equation}\label{w8.1.3}
\bb-a=-\int_0^t \bb k_{NN} dt'
\end{equation}
This together with (\ref{a_3}), (\ref{bb_3}) and (\ref{pi.2}) implies
$
|\bb-1|\les \la^{-8\ep_0}+\la^{-4\ep_0}\les \la^{-4\ep_0}<\frac{1}{4}.
$
\end{proof}

Lemma \ref{6.17.1} shows that $v_t$ and $\tir^2$ are comparable and $\bb$ can be considered as a positive constant away from zero.
Thus for any tensor $F$ on $S_{t, u}$ and $1\le q<\infty$ we have
$$
\|F\|_{L^q(S_{t, u})}\approx \|\tir^{\frac{2}{q}} F\|_{L_\omega^q(S_{t, u})}.
$$
In what follows, we will frequently use Lemma \ref{6.17.1} without explicit mention.

Using Lemma \ref{6.17.1} and  the second assumption of (\ref{ba3}), we can show that on $\widetilde{\D^+}$ there hold the following Sobolev inequalities
and trace inequalities:

\begin{enumerate}
\item[$\bullet$] For any scalar functions or $S_{t,u}$-tangent tensor fields $F$ there hold (see \cite{Wang10})
\begin{align}
&\|\tir^{-\f12} F\|_{L^2(S_{t,u})}+\|F\|_{L^4(S_{t,u})}\les \|F\|_{H^1(\Sigma_t)}, \label{trc_1} \\
&\|F\|_{L_u^2 L_\omega^2}\les \|\tir \nab_N F\|_{L^2_u L_\omega^2}+\|\tir^\f12 F\|_{L_u^\infty L_\omega^2}. \label{trc_2}
\end{align}

\item[$\bullet$] For any scalar functions or $S_{t,u}$-tangent tensor fields $F$ and $2<q<\infty$ there hold (see \cite{CK,KRsurf,KRduke}
\begin{align}
&\|F\|_{L_\omega^q(S_{t, u})} \les \|\tir \sn F\|_{L_\omega^2(S_{t, u})}^{1-\frac{2}{q}}\|F\|_{L_\omega^2(S_{t, u})}^{\frac{2}{q}}
+\|F\|_{L_\omega^2(S_{t, u})}, \label{sob}\\
&\|F\|_{L_\omega^\infty(S_{t,u})} \les \|r \sn F\|_{L_\omega^q(S_{t,u})}+\|F\|_{L_\omega^2(S_{t,u})}. \label{sobinf}
\end{align}

\item[$\bullet$]
For $q\ge 2$ and any scalar functions or tensor fields $F$ defined on $\widetilde{\D^+}\cap \Sigma_t$ there holds
\begin{equation}\label{tran_sob}
\|\tir^{\f12-\frac{1}{q}}F\|^2_{L_x^{2q} L_u^\infty}\les \|F\|_{L_\omega^\infty L_u^2} \left(\|\tir \nab_N F\|_{L_\omega^q L_u^2}+\|F\|_{L_\omega^q L_u^2}\right);
\end{equation}
see \cite[Lemma 2.13]{Wangricci} and its proof.
\item[$\bullet$]
For $0< 1-\frac{2}{q}<s-2$, there holds for scalar function $f$ on $\widetilde{\D}^+$
\begin{equation}\label{w7.28.1}
\|\tir f\|_{L_u^2 L_\omega^q(\Sigma_t\cap \widetilde{\D}^+)}\les \|f\|_{H^{s-2}(\Sigma_t\cap \widetilde{\D}^+)}
\end{equation}
\end{enumerate}
The fractional Sobolev inequality (\ref{w7.28.1}) can be proved by
decomposing the scalar function $f=\sum_{\mu>1} P_\mu f+P_{\le 1}f$. By (\ref{sob}) and (\ref{sobinf})
\begin{equation*}
\tir \|P_\mu f\|_{L_\omega^q}\les \|\tir |\sn P_\mu f|_\ga\|_{L_x^2}^{1-\frac{2}{q}}\|P_\mu f\|_{L_x^2}^\frac{2}{q}+\|P_\mu f\|_{L_x^2}
\end{equation*}
Using finite band property and the second assumption in (\ref{ba3}),
\begin{equation*}
\|\tir P_\mu f\|_{L_u^2 L_\omega^q(\Sigma_t\cap \widetilde{\D}^+)}\les (\mu^{1-\frac{2}{q}}+1)\|P_\mu f\|_{L_u^2 L_x^2(\Sigma_t\cap \widetilde{\D}^+)}.
\end{equation*}
Summing over $\mu>1$, with $0<1-\frac{2}{q}<s-2$
\begin{equation}\label{5.3.3}
\sum_{\mu>1}\|\tir P_\mu f\|_{L_u^2 L_\omega^q(\Sigma_t\cap \widetilde{\D}^+)}\les \|f\|_{H^{s-2}(\Sigma_t\cap \widetilde{\D}^+)}.
\end{equation}
By Bernstein inequality
\begin{equation}\label{5.3.4}
\|\tir P_{\le 1} f\|_{L_u^2 L_\omega^q(\Sigma_t\cap \widetilde{\D}^+)}\les \|f\|_{L^2(\Sigma_t\cap \widetilde{\D}^+)}.
\end{equation}
Hence (\ref{w7.28.1}) is proved.

By using Proposition \ref{eng3}, Lemma \ref{lem:flux}, (\ref{comp_3_27}) together with the above Sobolev and trace inequalities,
we can derive the following result.

\begin{lemma}\label{flux00}
Let $\D_*=(\sn, \sn_L)$ and let $2\le q\le p$. Then on $\widetilde{\D^+}$ there hold
\begin{align}
&\|\tir^{1-\frac{2}{q}}\emph{\bd} \bp \phi\|_{L_u^2 L_x^q(\Sigma_t\cap \widetilde{\D^+})}\les \la^{-\f12}, \label{mtr_1}\\
&\|\bp \phi\|_{L_u^2 L_\omega^p(\Sigma_t\cap \widetilde{\D^+})}
+\|\tir^{\f12}\bp \phi\|_{L^\infty L_\omega^{2p}(\Sigma_t\cap \widetilde{\D^+})}\les \la^{-\f12},\label{mtr_2}\\
&\|\D_* \bp \phi\|_{L^2(C_u\cap \widetilde{\D^+})}+\|\tir^{1-\frac{2}{p}}\D_* \bp \phi\|_{L_t^2 L_x^p(C_u \cap \widetilde{\D^+})}\les \la^{-\f12}.\label{mtr_3}
\end{align}
\end{lemma}
\begin{proof}
Recall that $\phi$ is obtained from the solution of (\ref{wave1}) by the coordinate change
$(t, x) \to (\la(t-t_k), \la x)$. We first consider $\phi$ under the original coordinates.
Because of Proposition \ref{eng3} and Lemma \ref{lem:flux}, we may use the similar argument in the proof of
\cite[Propositions 2.6]{Wangricci} to show that  (\ref{mtr_3}) hold with $\la^{-\f12}$ on the
right hand sides replaced by a universal constant.
(\ref{mtr_1}) can be derived by using (\ref{w7.28.1}) and Proposition \ref{eng3}.

Consider (\ref{mtr_2}). By (\ref{trc_1}) and (\ref{trc_2}),
\begin{equation*}
\|\bp\phi\|_{L_u^2 L_\omega^2(\Sigma_t\cap \widetilde{\D^+})}\les \|\bp\phi\|_{H^1}.
\end{equation*}
 Also using (\ref{sob}), we obtain
  \begin{equation}\label{w7.28.2}
  \| \bp\phi\|_{L_u^2 L_\omega^p(\Sigma_t\cap \widetilde{\D^+})}\les \|\tir\sn\bp \phi\|_{L_u^2 L_\omega^2}^{1-\frac{2}{p}}\|\bp \phi\|_{L_u^2 L_\omega^2}^{\frac{2}{p}}+\|\bp \phi\|_{L_u^2 L_\omega^2}\les \|\bp\phi\|_{H^1}
  \end{equation}
 By using (\ref{tran_sob}), (\ref{sobinf}), (\ref{w7.28.2}) and (\ref{w7.28.1})
 \begin{align}
 \|\tir^{\f12}\bp \phi\|_{L^\infty L_\omega^{2p}(\Sigma_t\cap \widetilde{\D^+})}^2&\les \|\bp \phi\|_{L_\omega^\infty L_u^2}\|\tir |\bd \bp \phi|+ |\bp \phi|\|_{L_\omega^q L_u^2}\les \|\bp \phi\|_{H^{s-1}}^2
 \end{align}
  The desired estimates on $\phi$ then follow by using Proposition \ref{eng3} and rescaling the coordinates back.
%
\end{proof}

As consequences of Lemma \ref{flux00}, we have

\begin{proposition}\label{flux_2}
Let $\D_*=(\sn, \sn_L)$. Then for any $2\le q\le p$ and any scalar component of metric $\bg$ there hold
\begin{align}
&\|\tir \D_*\ti \pi\|_{L_t^2 L_\omega^q(C_u\cap \widetilde{\D^+})}\les \la^{-\f12},\label{flux2}\\
&\|\tir \emph{\bd}\ti\pi\|_{L_u^2 L_\omega^q(\Sigma_t\cap \widetilde{\D^+})}\les \la^{-\f12}, \label{eng4}\\
&\|\ti\pi\|_{L_u^2 L_\omega^q(\Sigma_t\cap \widetilde{\D^+})}
+\|\tir^{\f12}\ti\pi\|_{L^\infty L_\omega^{2q}(\widetilde{\D^+})}\les \la^{-\f12}. \label{transtrace}
\end{align}
\end{proposition}

\begin{proof}
Since $\ti\pi=f(\phi)\bp \phi$,
$
\D_*\ti\pi=f'(\phi) \D_*\phi\c \bp \phi +f(\phi) \D_*\bp\phi.$
Then
$$|\ti\pi|\les |\bp \phi|, \quad\, |\D_* \ti\pi|\les |\bp \phi|^2+|\D_* \bp \phi|.$$

Thus, by using (\ref{mtr_3}) and the last inequality in (\ref{mtr_2}) we obtain
\begin{align*}
\|\tir \D_* \ti\pi\|_{L_t^2 L_\omega^q(C_u\cap \widetilde{\D^+})}
& \les \|\tir (\bp \phi)^2 \|_{L_t^2 L_\omega^q(C_u \cap \widetilde{\D^+})}
+ \|\tir \D_* \bp \phi\|_{L_t^2 L_\omega^q(C_u\cap \widetilde{\D^+})}\\
& \les \tau_*^{\f12}\|\tir^{\f12}\bp \phi\|_{L_t^\infty L_\omega^{2q}(C_u\cap \widetilde{\D^+})}^2 + \la^{-\f12}\les \la^{-\f12}
\end{align*}
which shows (\ref{flux2}). Similarly (\ref{eng4}) and (\ref{transtrace}) can be proved by virtue of Lemma \ref{flux00}.
\end{proof}

\subsection*{The symbols $\pi$}

 Let $\pi$ represent scalar functions or tensor fields
derived  by contracting  ( multiplying ) terms of type $\ti \pi$ with $L$, $\Lb$, or the projection
\begin{equation}\label{proj}
\sl\Pi_{\a\b}=\bg_{\a\b}+\f12 L_\a \Lb_\b+\f12 L_\b \Lb_\a.
\end{equation}
Clearly $|\pi|\les |\bp \phi|$. By virtue of Proposition \ref{flux_2} we have the following result.

\begin{lemma}
On null cones $C_u$ contained in $\widetilde{\D^+}$,
\begin{equation}\label{flux_3}
\|\tir (\sn \pi,\sn_L  \pi), \pi\|_{L_t^2 L_\omega^p(C_u\cap \widetilde{\D^+})}\les \la^{-\f12}.
\end{equation}
\end{lemma}

\begin{proof}
The last estimate follows immediately from (\ref{pi.2}). Now we prove the other two estimates.
We consider only the case $\pi=\sl\Pi^\mu_{\nu}\ti\pi_{\mu\cdots}$ since other cases can be proved similarly.
By using (\ref{6.7con}), we have
\begin{align*}
\sn_L \sl\Pi^{\a\b}=2\zb_A (L^\a e_A^\b+e_A^\a L^\b), \quad
\sn_A \sl\Pi^{\a\b}=\chi_{AB} e_B^\a \Lb^\b+\chib_{AB} e_B^\a L^\b.
\end{align*}
Hence, by using (\ref{3.19.1}) we have
\begin{equation}\label{snpi_1}
|\sn_L \pi|\le |\bd_L \ti\pi|+|\ti\pi|\c|k_{AN}|, \quad
|\sn \pi|\les |\sn\ti\pi|+ (|\chi|+|\chib|)\c| \ti\pi|.
\end{equation}
In view of (\ref{aux_1}) and (\ref{mtr_2}), we have
$
\|\tir(\chi, \chib, k_{AN})\|_{L_\omega^p}\les 1.
$
Therefore
\begin{equation*}
\|\tir (|\chi|+|\chib|+|k_{AN}|)\c \ti\pi\|_{L_\omega^p}(t)\les\|\ti\pi(t)\|_{L_\omega^\infty}
\end{equation*}
Applying (\ref{pi.2}) to $\ti\pi$,  we derive that
$
\|\tir(|\chi|+|\chib|+|k_{AN}|)\c \ti\pi\|_{L_t^2 L_\omega^p}\les \la^{-\f12-4\ep_0}.
$
Thus  from (\ref{snpi_1}) and (\ref{flux2}) we obtain (\ref{flux_3}).
\end{proof}

\subsection*{Hardy-Littlewood maximal function} For a scalar function $f(t)$ defined on $[0, \tau_*]$, its Hardy-Littlewood maximal function is defined by
\begin{equation*}
\M(f)(t)=\sup_{0\le t'\le \tau_*} \frac{1}{|t-t'|}\int_{t'}^t |f(\tau)| d\tau.
\end{equation*}
It is well-known that for any $1<q<\infty$ there holds
\begin{equation}\label{hlm}
\|\M(f)\|_{L_t^q} \les \|f\|_{L_t^q}.
\end{equation}

\subsection*{Hodge systems}
We denote by $\D_1$ the Hodge operator which sends $S_{t,u}$-tangent 1-forms $F$ to $(\div F, \curl F)$
and by $\D_2$ the Hodge operator which maps covariant $S_{t,u}$-tangent symmetric traceless 2-tensors $F$ to $\div F$.
We will rely on the following Calderon-Zygmund theorems, which are consequences of (\ref{ba3}).

\begin{lemma}\label{hdgm1}
Let $\D$ denote either $\D_1$ or $\D_2$ and let $p>2$ be the number in (\ref{ba3}). Then for
$2\le q\le p$ there holds
\begin{equation*}
\|\sn  F\|_{L^q(S_{t,u})}+\|\tir^{-1}F\|_{L^q(S_{t,u})}\les \| \D F\|_{L^q(S_{t,u})}
\end{equation*}
for any $S_{t,u}$-tangent tensor $F$ in the domain  of $\D$.
\end{lemma}

\begin{proposition}\label{cz}
Let $F$ be a covariant symmetric traceless $2$-tensor satisfying the Hodge system
\begin{equation}\label{fe1}
{\emph\div} F=\sn G+e \qquad \mbox{on } S_{t,u}
\end{equation}
for some scalar function $G$ and $1$-form $e$. Then for $2<q<\infty$ and $\frac{1}{q'}=\frac{1}{2}+\frac{1}{q}$ there hold
\begin{equation}\label{lpp2}
\|F\|_{L^q(S_{t,u})}\les\|G\|_{L^q(S_{t,u})}+\|e\|_{L^{q'}(S_{t,u})}
\end{equation}
and
\begin{equation}\label{cz0}
\|F\|_{L^\infty(S_{t,u})}\les\|G\|_{L^\infty(S_{t,u})}\log { \left(2+\|\tir^{\frac{3}{2}-\frac{2}{q}}\sn
G\|_{L^q(S_{t,u})}\right)}+\tir^{1-\frac{2}{q}}\|e\|_{L^q(S_{t,u})}.
\end{equation}
Similarly, for the Hodge system
\begin{equation}\label{divcurl}
\left\{
\begin{array}{lll}
{\emph\div} F= \sn\c G_1+e_1,\\
{\emph\curl} F=\sn\c G_2+e_2,
\end{array}
\right.
\end{equation}
where $G=(G_1, G_2)$ are $1$-forms and $e=(e_1, e_2)$ are scalar functions, there hold (\ref{lpp2}) and (\ref{cz0}) for any $q>2$.
\end{proposition}

\begin{proposition}\label{cz.2}
Let $F$ and $G$ be $S_{t,u}$-tangent tensor fields of suitable type satisfying (\ref{fe1}) or
(\ref{divcurl}) with certain term $e$. Suppose $G$ is  a projection of a
tensor field $\ti G$ to tangent space of $S_{t,u}$ by $ \sl\Pi_{\mu}^{\mu'}{\ti G}_{\mu'\cdots}$ or
takes the form of $N^\mu{\ti G}_{\mu\cdots}, L^\mu{\ti G}_{\mu\cdots}$. Then for $q>2$, $1\le c<\infty$ and
$\delta>0$ sufficiently close to $0$, there holds
\begin{align}
\|F\|_{L^\infty(S_{t,u})}\les \|\mu^{\delta}P_\mu \ti G\|_{l_\mu^c L^\infty(S_{t,u})}
+\|\ti G\|_{L^\infty(S_{t,u})}+\tir^{1-\frac{2}{q}}\|e\|_{L^q(S_{t,u})}.\label{cz2}
\end{align}
Here $\ti G$ is regarded as its components under the coordinate frame $\p_t, \p_1, \p_2, \p_3$.
\end{proposition}

Lemma \ref{hdgm1} can be found in \cite{KR2} and \cite[Lemma 2.18]{Wangricci}, Proposition \ref{cz} is \cite[Proposition 6.20]{KR2}
and Proposition \ref{cz.2} is \cite[Proposition 3.5]{Wangricci}.

\subsection*{Commutation formulas}

We will frequently use the following commutation formulas which can be found in \cite{CK,KRduke, KR1}.
\begin{enumerate}
\item[$\bullet$] Let $U_A$ be an $m$-covariant tensor tangent to $S_{t,u}$. There holds
\begin{align}\label{cmu1}
\begin{split}
&\sn_L\sn_B U_A-\sn_B \sn_L U_A  \\
&=-\chi_{BC}\c \sn_C U_A+\sum_{i}(\chi_{A_i B} \zb_C-\chi_{BC} \zb_{A_i}
+\bR_{{A_i}C4 B})U_{A_1\cdots\ckk C\cdots A_m}
\end{split}
\end{align}
and for any scalar function $f$ there holds
\begin{equation}\label{cmu2}
[L,\sn_A] f=- \chi_{AB} \sn_B f.
\end{equation}
Consequently, for any scalar function $f$ there holds
\begin{align}
L\sD f+  \tr\chi \sD f= \sD L f-2\chih\c \sn^2 f-\sn_A\chi_{AC}\sn_C f
+(\tr\chi \zb_C-\chi_{AC}\zb_{A}-\delta^{AB}\bR_{CA4B}) \sn_C f.\label{tran1}
\end{align}

\item[$\bullet$] For the vector fields $L$ and $\Lb$ we have
\begin{equation}\label{cmu3}
[L, \Lb]=2(\zb_A-\zeta_A) e_A+k_{NN}\Lb-k_{NN}L.
\end{equation}
\end{enumerate}

\subsection*{The transport lemma} We will use transport equations to control various Ricci coefficients.
The following result can be derived in view of (\ref{lv}) and (\ref{comp_3_27}).

\begin{lemma}[The transport lemma]\label{tsp2}
For $C_u$ contained in $\widetilde{\D^+}$ let $t_\tmin =\max\{u, 0\}$. For any $S_{t,u}$-tangent tensor field $F$ satisfying
\begin{equation*}
\sn_L F+m {\emph\tr}\chi F= W
\end{equation*}
with a constant $m$, there holds
\begin{equation*}
v_t^m F(t)=\lim_{\tau\rightarrow t_{\tmin}}v_\tau^m F(\tau)+\int_{t_{\tmin}}^t v_{t'}^m Wdt'.
\end{equation*}
Similarly, for the transport equation
\begin{equation*}
\sn_L F+\frac{2m}{t-u} F=G\c F+W
\end{equation*}
with a constant $m$, if $\|G\|_{L_\omega^\infty L_t^1}\le C$, then there holds
\begin{equation*}
\tir^{2m} |F(t)|\les\lim_{\tau\rightarrow t_{\tmin}}(\tau-u)^{2m}|F(\tau)|+\int_{t_\tmin}^t {(t'-u)}^{2m} |W| d\tt.
\end{equation*}
The same result holds when $\frac{2}{t-u}$ in the transport equation is replaced by ${\emph\tr} \chi$.
The above integrals are taken along null geodesics on $C_u$.
\end{lemma}

\subsection*{Curvature decomposition}\label{decom}

To employ the structural equations, it is necessary to provide useful decompositions of the
curvature tensor. In (\ref{ricc6.7.2}) we have introduced the 1-form $V$ by which the Ricci tensor $\bR_{\a\b}$ has
the decomposition (\ref{ricc6.7.1}). In particular
\begin{equation*}
\bR_{44}=L^\a L^\b \bR_{\a\b} = \bd_4 V_4+S_{44}-\f12 L^\a L^\b \Box_\bg (\bg_{\a\b}),
\end{equation*}
where $S_{44}$ takes the form $L^\a L^\b \bg\c (\bp \bg)^2$. Since by  (\ref{ricc_def}), $\bd_4 V_4 = L(V_4) +k_{NN} V_4$. We have
\begin{equation}\label{r44.1}
\bR_{44}=L(V_4)+k_{NN} V_4+\ti \bE,
\end{equation}
where $\ti \bE=L^\a L^\b (\bg(\bp \bg)^2+\Box_{\bg} (\bg_{\a\b}))$ which takes the form
$\ti \bE = \pi \c \pi$ by virtue of (\ref{wave1}).

For the Riemannian curvature tensor there holds the following decomposition result.

\begin{lemma}\label{decom_lem}
 \begin{enumerate}
\item[(i)] Let $\D_*=(\sn, \sn_L)$. There hold
$$
{\emph \bR}_{4A4B}, {\emph\bR}_{A443}, {\emph \bR}_{44}, {\emph \bR}_{4A}=\D_* \pi+\sl \bE.
$$

\item[(ii)]  There exist scalar $\pi$, 1-form $\sl \bE$ and $S_{t,u}$ tangent 2-vector $\pi_{AB}$  such that
\begin{equation*}
\delta^{AB} {\emph \bR}_{CA4B}=\sn_C \pi+\sn^B\pi_{CB}+\sl \bE_C \quad \mbox{ and } \quad
{\emph \bR}_{CA4B}=\sn \pi+\sl \bE.
\end{equation*}

\item[(iii)]There exists 1-form $\pi$ and scalar $\sl \bE$ that ${\emph \bR}_{ABAB}={\emph \div} \pi+\sl \bE$.

\item[(iv)] There exist 1-forms $\pi$ and scalar $\sl \bE$ such that
\begin{equation*}
\delta^{AB} {\emph\bR}_{B43A}={\emph \div} \pi+\sl \bE, \quad \ep^{AB} {\emph\bR}_{A43B}={\emph \curl} \pi+\sl \bE.
\end{equation*}
\end{enumerate}
\noindent
Here $\sl \bE = {\emph \bA}\c \pi + \tr \chi \c \pi$, where ${\emph \bA}$ denotes the collection of $\chih, z$ and $\pi$.
\begin{footnote}{
The collection of terms in $\bA$ will be expanded further in Section \ref{ss1}.}
\end{footnote}
\end{lemma}

This result follows from the similar argument in \cite[Section 4]{KR2}.  In particular, for deriving the result in (iv),
we write
\begin{equation}\label{cdc1}
\delta^{AB}\bR_{B43A}=\delta^{AB}(\bR_{AB}-\delta^{CD}\bR_{ACBD}), \quad \ep^{AB} \bR_{AB43} =-2 \ep^{AB} \bR_{A43B};
\end{equation}
for $\bR_{ACBD}$ and $\bR_{AB43}$ we use \cite[Proposition 4.1]{KR2} and for $\bR_{AB}$ we use
$
\bR_{AB}=\f12(\sn_A V_B+\sn_B V_A)+\sl \bE
$
which follows from (\ref{ricc6.7.1}).

\subsection{Estimates on $\chih$, $\zeta$ and $z$}

We derive some useful estimates on $\chih$, $\zeta$ and $z$ from (\ref{ba3.18.1}) and (\ref{ba3}).

\begin{proposition}\label{p1}
There hold on $\widetilde{\D^+}$ that
\begin{align}
& \|\chih, \zeta,z,  \ti r\sn_L \chih, \ti r \sn_L \zeta, \tir \sn_L z\|_{L_t^2 L_\omega^p(C_u)}\les \la^{-\f12} , \label{pric1}\\
& \|{\ti r}^{\f12} (\chih, \zeta, z)\|_{L_t^\infty L_\omega^p(C_u)}\les \la^{-\f12}, \label{pric2}\\
& \|{\ti r}^{\f12}(\chih, \zeta, z)\|_{L_\omega^{2p} L_t^\infty(C_u)}\les \la^{-\f12} \quad \mbox{ on } \D^+. \label{pric3}
\end{align}
\end{proposition}

\begin{proof}
We first derive the estimates for $\chih$ by using (\ref{s2}). In view of Lemma \ref{decom_lem} (i), we can write (\ref{s2}) symbolically as
\begin{equation} \label{lchi}
\sn_L \chih+\f12 \tr\chi\chih=\pi\c \bA+\tir^{-1}\pi+(\sn, \sn_L) \pi.
\end{equation}
By using Lemma \ref{tsp2} we obtain
\begin{align}\label{tchih}
|\tir\chih(t, \omega)| \les \lim_{\tau\rightarrow t_\tmin}|\tir\chih(\tau, \omega)|
+\int_{t_\tmin}^t \left(|\pi| + \tir (|\pi\c \bA| + (\sn, \sn_L) \pi|) \right) d\tt,
\end{align}
where the integral is taken over null geodesics.

On null cones $C_u$ with $u\ge 0$, by Lemma \ref{inii} the first term on the right of (\ref{tchih}) vanishes.
Thus, dividing the both sides of (\ref{tchih}) by $\tir$ and using (\ref{hlm}), it follows
\begin{align*}  
\|\chih\|_{L_t^2 L_\omega^p}\les\|\pi, \tir(\sn_L, \sn)\pi\|_{L_t^2 L_\omega^p}+\tau_*^\f12\|\pi\|_{L_t^2 L^\infty_\omega}\|\bA\|_{L^2_t L_\omega^p}.
\end{align*}
In view of (\ref{pi.2}) and (\ref{ba3.18.1}), we have
\begin{equation}\label{a6.19}
\|\bA\|_{L^2_t L_\omega^\infty(C_u\cap \widetilde{\D^+})}\les \la^{-\f12+2\ep_0}.
\end{equation}
This together with (\ref{flux_3}), (\ref{pi.2}) and $\tau_*\les \la^{1-8 \ep_0}$ then shows that
\begin{align}\label{chih1}
\|\chih\|_{L_t^2 L_\omega^p}\les \la^{-\f12}.
\end{align}
Similarly, we can obtain in $\D^+$ that
$
\|{\ti r}^{\f12} \chih\|_{L_\omega^p L_t^\infty} \les \la^{-\f12}.
$

On the null cones $C_u$ with $u=-v$ for $0<v<v_*$, we may use (\ref{tchih}) and the similar argument for deriving (\ref{chih1}) to obtain
\begin{align*}
\|\chih\|_{L_t^2 L_\omega^p}&\les \|\tir^{-1} v\chih(0)\|_{L_t^2 L_\omega^p}+\la^{-\f12}.
\end{align*}
In view of (\ref{3.19.1}), (\ref{a_3}) and (\ref{transtrace}), we have
\begin{equation}\label{chi0}
\|v^\f12 \chih(0)\|_{L_\omega^p}\les \|v^\f12 \hat\theta(0)\|_{L_\omega^p}+\|v^\f12 \hat k_{AB}(0)\|_{L_\omega^p}\les \la^{-\f12}.
\end{equation}
Consequently, by using $\tir = t +v$ we have $ \|\tir^{-1} v\chih(0)\|_{L_t^2 L_\omega^p(C_u)}\les \la^{-\f12}$.
Thus on $\widetilde{\D^+}$ there holds
\begin{equation}\label{6.7chih}
\|\chih\|_{L_t^2 L_\omega^p(C_u)}\les \la^{-\f12}.
\end{equation}
By a similar argument, we can derive on $\widetilde{\D^+}$ that $\|\tir^\f12 \chih\|_{L^\infty L_\omega^p}\les \la^{-\f12}$.
Thus we obtain the estimates for $\chih$ in (\ref{pric1}) and (\ref{pric2}).

To show the estimate for $\sn_L \chih$ in (\ref{pric1}), we use (\ref{lchi}) and write $\tr\chi = z-V_4 + \frac{2}{\tir}$ to derive that
\begin{equation*}
\|\tir\sn_L \chih\|_{L_t^2 L_\omega^p}\le \|\chih\|_{L_t^2 L_\omega^p}+\|\tir (z, V_4)\chih\|_{L_t^2 L_\omega^p}
+\|\tir(\bA\c \pi + \tir^{-1}\pi + (\sn, \sn_L) \pi)\|_{L_t^2 L_\omega^p}.
\end{equation*}
The last term on the right can be estimated by the similar argument for deriving (\ref{chih1}). The first term can be
controlled by (\ref{6.7chih}). For the second term, we can employ (\ref{a6.19}) and the estimate on $\chih$ in (\ref{pric2}) to obtain
\begin{align*}
\|\tir (z, V_4)\chih\|_{L_t^2 L_\omega^p} \les \tau_*^{\f12}\|\tir^{\f12} \chih\|_{L_t^\infty L_\omega^p}\|z, V_4\|_{L_t^2 L_\omega^\infty(C_u)}
\les \la^{-\f12-2\ep_0}.
\end{align*}
Hence we conclude on $\widetilde{\D^+}$ that
$$
\|\tir \sn_L \chih\|_{L_t^2 L_\omega^p(C_u)}\les \la^{-\f12}.
$$

With the help of (\ref{tran2}), Lemma \ref{decom_lem}, Lemma \ref{inii} and Proposition \ref{exten}, we can obtain the estimates for $\zeta$ and $\sn_L\zeta$ in
(\ref{pric1}) and (\ref{pric2}) similarly. Finally the estimates for $\chih$ and $\zeta$ in (\ref{pric3}) follow from (\ref{pric1}),
(\ref{ric2a}), (\ref{ric3a})  and the inequality
\begin{equation}\label{7.22.30}
\|\tir^{\f12} F \|_{L_\omega^{2p} L_t^\infty}^2 \les \left(\|\tir \sn_L F\|_{L_\omega^p L_t^2} + \|F\|_{L_\omega^p L_t^2}\right) \|F\|_{L_\omega^\infty L_t^2}
\end{equation}
for any $S_{t, u}$-tangent tensor $F$; see \cite[Lemma 2.13]{Wangricci}. We will prove the estimates for $z$ in  Section \ref{sec_6.19}
by using the following transport equations satisfied by $z$ and $\sn z$.
\end{proof}

\begin{lemma}
Let ${\emph\bA}$ denote the collection of the terms $\chih, z, \pi$.
\begin{align}
& Lz+\frac{2z}{t-u}=-\f12 V_4^2+\left(V_4-k_{NN}\right)\widetilde{\tr\chi}-|\chih|^2-\f12 z^2+\pi\c \pi, \label{lz}\\
& \sn_L \sn z+\frac{3}{t-u} \sn z= {\emph\bA}\c\sn z+\left(\frac{1}{(t-u)},  z, \pi \right)\c \sn  \pi+\sn \chih\c \chih, \label{ldz_2}\\
& {\emph\div} \chih=\f12 \left(\sn z-\sn V_4 \right)+\sn \pi+\frac{\pi}{t-u}+{\emph\bA}\c \pi.\label{dchi1}
\end{align}
\end{lemma}

\begin{proof}
To obtain (\ref{lz}), we first use (\ref{s1}), $\widetilde{\tr\chi}=\tr\chi+V_4$ and (\ref{r44.1}) to derive that
\begin{equation*}
L\widetilde{\tr\chi}+\f12 (\widetilde{\tr\chi})^2=-\f12 V_4^2+V_4\widetilde{\tr\chi}-|\chih|^2-k_{NN} \widetilde{\tr\chi}+\pi\c \pi.
\end{equation*}
By using $z=\widetilde{\tr\chi}-\frac{2}{t-u}$, we then obtain (\ref{lz}).

Next we take the covariant derivative $\sn$ on the both sides of (\ref{lz}) and use the commutation formula (\ref{cmu2})
\begin{equation*}
\sn_L  \sn z+\frac{3}{t-u} \sn z=- \chih\c \sn z+\f12 (V_4-z) \sn z+\sn G,
\end{equation*}
where $G$ denotes the right hand side of (\ref{lz}). In view of the structure of $G$ we then obtain (\ref{ldz_2}).
Finally, (\ref{dchi1}) follows from (\ref{dchi}), Lemma \ref{decom_lem} (ii) and $\tr\chi = z - V_4 + \frac{2}{t-u}$.
\end{proof}

\subsection{Proof of Proposition \ref{ricpr} and \ref{ricco}}\label{sec_6.19}

We first consider $z$. By integrating the transport equation (\ref{lz}) along null geodesics,
we have from (\ref{ba3.18.1}) and Lemma \ref{tsp2} that
\begin{align}\label{tirz}
\tir^2|z(t)|&\les \left|\lim_{\tau\rightarrow t_{\tmin}} (\tau-u)^2z(\tau)\right|
+\int_{t_\tmin}^t \tir^2 \left( |\pi\c\pi| + |(V_4- { k}_{NN}) \widetilde{\tr\chi}|+  |\chih|^2 \right).
\end{align}
For null cones $C_u$ with $u\ge 0$, the limit term is zero by Lemma \ref{inii} (i). Thus, in view of (\ref{pi.2}),
(\ref{ba3.18.1}) and $\widetilde{\tr\chi}=z+2\tir^{-1}$, we obtain
\begin{equation*}
|\ti r z|\les \tau_*\|\pi\|_{L_t^2 L_x^\infty}^2+\|\pi\|_{L_t^1 L_x^\infty}
+\tau_*\|z\|_{L_t^2 L_\omega^\infty}\|\pi\|_{L_t^2 L_\omega^\infty}+\tau_* \|\chih\|_{L_\omega^\infty L_t^2}^2\les \la^{-4\ep_0}.
\end{equation*}
For null cones $C_u$ with $u=-v$ for $0<v<v_*$, the limit term is $|v^2 z(0)|$. Thus, by using Lemma \ref{inii} (iii) and (\ref{tirz})
we can again conclude that
$
|\tir z |\les \la^{-4\ep_0}.
$
Consequently $\tir \widetilde{\tr\chi}\approx 1$. We can then use (\ref{tirz}) with $\widetilde{\tr\chi}$ replaced by $\tir^{-1}$ to obtain
$
\|\tir^\f12 z\|_{L^\infty}\les\la^{-\f12}.
$
We thus obtain (\ref{comp2}). This directly gives the estimates for $z$ in (\ref{pric2}) and (\ref{pric3}).

Now we improve the estimate of $z$ in (\ref{ba3.18.1}) by showing the estimate of $z$ in (\ref{ric3.18.1}). For null cones $C_u$ with $u=-v<0$,
it follows from (\ref{tirz}) and Lemma \ref{inii} (iii) that
\begin{align}\label{6.9.1}
|z(t)| 
\les \tir^{-2} v^\frac{3}{2} \la^{-\f12}+\tir^{-2}\int_{t_\tmin}^t \left(\tir^2|\pi|^2 +\tir|\pi| +\tir^2 |\chih|^2\right).
\end{align}
Consequently, in view of (\ref{ba3.18.1}) and (\ref{pi.2}), we obtain
$
\|z\|_{L_t^2 L_\omega^\infty(C_u)}\les \la^{-\f12}.
$
For null cones $C_u$ with $u\ge 0$, the first term on the right of (\ref{tirz}) vanishes. Thus, by virtue of (\ref{pi.2}) and (\ref{ric2a}),
we can derive that
$
\|z\|_{L_t^2 L_\omega^\infty(C_u)}\les \la^{-\f12-4\ep_0}.
$
Therefore we obtained the estimate for $z$ in (\ref{ric3.18.1}).
This estimate also implies the estimate for $z$ in (\ref{pric1}).

In view of (\ref{lz}), schematically $\sn_L z= \bA\c \bA + \tir^{-1} \bA$. Thus
\begin{equation}\label{ee_1}
\|\tir \sn_L z\|_{L_t^2 L_\omega^p(C_u)}\les \|\bA\|_{L_t^2 L_\omega^\infty}\|\tir \bA\|_{L_t^\infty L_\omega^p}+\|\bA\|_{L_t^2 L_\omega^p}.
\end{equation}
Note that (\ref{transtrace}) and (\ref{pric2}) imply  $\|\tir^\f12 \bA\|_{L^\infty L_\omega^p}\les \la^{-\f12}$,
and (\ref{pric1}) and (\ref{flux_3}) imply $\|\bA\|_{L_t^2 L_\omega^p}\les \la^{-\f12}$. These estimates together with (\ref{a6.19})
imply that $\|\tir \sn_L z\|_{L_t^2 L_\omega^p(C_u)} \les \la^{-\f12}$.
We thus complete the proof of the estimates for $z$ in Proposition \ref{p1}.

To prove the estimate for $\chih$ in (\ref{ric3.18.1}), we need to derive (\ref{ricp}) and (\ref{sna}) first.
We will use (\ref{ldz_2}). According to (\ref{a6.19}), we may employ Lemma \ref{tsp2} to derive that
\begin{align*}
\ti r^{3} |\sn z|&\les \left|\lim_{\tau\rightarrow t_{\tmin}} (\tau-u)^3\sn z(\tau)\right|
+\int_{t_\tmin}^t \left(\tir^3 |\sn \chih\c \chih|+ \tir^2|\sn \pi|+ \tir^3 |(z, \pi)\c \sn  \pi|\right).
\end{align*}
For null cones $C_u$ with $u=-v<0$, we have
\begin{align*}
\|\tir\sn z\|_{ L_\omega^p}&\les\|\tir^{-2} v^3\sn z(0)\|_{L_\omega^p}+ \M(\|\tir\sn\pi \|_{L_\omega^p})
+\|\tir \sn \chih, \tir \sn \pi\|_{L_\omega^p L_t^2}\|\bA\|_{L_\omega^\infty L_t^2}.
\end{align*}
In view of Lemma \ref{inii} (iii), (\ref{hlm}), (\ref{a6.19}) and (\ref{flux_3}), we then obtain
\begin{align}\label{2psnz}
\|\ti r \sn z\|_{L_t^2 L_\omega^p} &\les \la^{-\f12}+\|\tir \sn \pi\|_{L_t^2 L_\omega^p}+\la^{-2\ep_0}\|\tir \sn \chih\|_{L_t^2 L_\omega^p} \nn \\
& \les \la^{-\f12}+\la^{-2\ep_0}\|\tir \sn \chih\|_{L_t^2 L_\omega^p}.
\end{align}
For null cone $C_u$ with $u\ge 0$, using Lemma \ref{inii} (i), we can obtain (\ref{2psnz}) in the same way.

From (\ref{dchi1}) and  Lemma \ref{hdgm1} it follows
\begin{equation*}
\|\tir\sn\chih\|_{L_\omega^p}\les \|\tir \sn z\|_{ L_\omega^p}+ \|\tir \sn \pi\|_{L_\omega^p}+\|\tir \bA\c\bA + \bA\|_{L_\omega^p}.
\end{equation*}
The last term can be estimated by the above same argument in deriving the estimate for $\|\tir \sn_L z\|_{L_t^2 L_\omega^p}$.
This together with (\ref{flux_3}) gives
\begin{equation}\label{rdchi}
\|\tir\sn \chih\|_{L_t^2 L_\omega^p}\les \|\tir \sn z\|_{L_t^2 L_\omega^p}+\la^{-\f12}.
\end{equation}
The combination of (\ref{rdchi}) and (\ref{2psnz}) gives
\begin{equation}\label{sna_1}
\|\ti r \sn z\|_{L_t^2 L_\omega^p}\les \la^{-\f12} \quad \mbox{and} \quad \|\tir\sn \chih\|_{L_t^2 L_\omega^p}\les\la^{-\f12}.
\end{equation}
Similarly, we can obtain $\|{\ti r}^{\frac{3}{2}}\sn z\|_{L_\omega^p L_t^\infty}\les \la^{-\f12}$.
Thus we proved (\ref{ricp}) and (\ref{sna}).
In view of (\ref{sobinf}), (\ref{pric1}) and (\ref{sna}), we can derive that
$
\|\chih\|_{L_t^2 L_\omega^\infty}\les \la^{-\f12}.
$
This is the estimate for $\chih$ in (\ref{ric3.18.1}) which improves the estimate of $\chih$ in (\ref{ba3.18.1}).

\subsubsection{Proof of (\ref{ric1.1}) and (\ref{ric1})}

We postpone the estimates on $\zeta$ to the end of Section 5.4.
In what follows, we give the estimates for $\chih, z, \tr\chi-\frac{2}{\tir}$ in (\ref{ric1.1}) and (\ref{ric1}).
We first use (\ref{tirz}), Lemma \ref{inii} (i), (\ref{pi.2}) and (\ref{ric3.18.1}) to deduce on $\D^+$ that
\begin{equation}\label{z1}
|z|\les \M(\|\pi\|_{L_x^\infty})+\la^{-1}.
\end{equation}
By (\ref{hlm}), (\ref{pi.2}) and $\tau_*\les \la^{1-8\ep_0}$, we can get $\|z\|_{L_t^2 L_x^\infty(\D^+)}\les \la^{-\f12-4\ep_0}$.
This proves the estimate for $z$ in (\ref{ric1}).

By virtue of (\ref{6.9.1}), (\ref{ric3.18.1}) and (\ref{pi.2}),  we have on $\widetilde{\D^+}$ that
\begin{equation}\label{z_3.19.1}
\|z(t)\|_{L_x^\infty}\les t^{-\f12}\la^{-\f12} +\la^{-1}+\M(\|\pi\|_{L_x^\infty}).
\end{equation}
After integration in $0\le t\le \tau_*$, it follows
\begin{equation}\label{zz3}
\|z\|_{L_t^\frac{q}{2} L_x^\infty(\widetilde{\D^+})}\les \la^{\frac{2}{q}-1-4\ep_0(\frac{4}{q}-1)},\quad q>2 \mbox{ and close to } 2,
\end{equation}
which gives the estimate on $z$ in (\ref{ric1.1}).
Using  (\ref{pi.2}), we get the  estimate on $\tr\chi-\frac{2}{\tir} = z-V_4$ in  (\ref{ric1.1}) and (\ref{ric1}).

To derive the estimates on $\chih$ in (\ref{ric1}) and  (\ref{ric1.1}), we use (\ref{dchi1}) to write
$$
\chih = \D_2^{-1} \left(\f12 \sn z + \sn \pi + \tir^{-1} \pi + \bA\c \pi\right).
$$
Thus, by using Proposition \ref{cz} and Proposition \ref{cz.2}, we have
\begin{align*}
\|\chih\|_{L^\infty(S_{t,u})}
&\les \|\D_2^{-1} (\sn z)\|_{L^\infty(S_{t, u})} + \|\D_2^{-1} \left(\sn \pi + \tir^{-1}\pi + \bA\c \pi\right)\|_{L^\infty(S_{t, u})}\\
& \les  \|z\|_{L^\infty(S_{t,u})}\log\left(2+\|\tir^{\frac{3}{2}-\frac{2}{p}}\sn z\|_{L^{p}(S_{t,u})}\right)\\
&\quad \, + \|\mu^{0+}P_\mu \ti\pi\|_{l_\mu^2 L^\infty (S_{t,u})}+\|\ti\pi\|_{L^\infty(S_{t,u})}+\|\tir^{1-\frac{2}{p}} \bA\c \pi\|_{L^{p}(S_{t,u})}\\
& \les  \|z\|_{L^\infty(S_{t,u})}+ \|\mu^{0+}P_\mu \ti\pi\|_{l_\mu^2 L^\infty (S_{t,u})}+\|\ti\pi\|_{L^\infty(S_{t,u})},
\end{align*}
where for the last inequality we used (\ref{ricp}) and (\ref{pric2}). By virtue of (\ref{pi.2}) and the estimate for $z$ in (\ref{ric1}),
we have
$$
\|\chih\|_{L_t^2 L_x^\infty(\D^+)} \les \|z\|_{L_t^2 L_x^\infty(\D^+)} + \la^{-\f12-4\ep_0} \les \la^{-\f12-4\ep_0};
$$
while by using (\ref{pi.2}) and (\ref{zz3}) with $q>2$ close to $2$, we have
$$
\|\chih\|_{L_t^\frac{q}{2} L_x^\infty(\widetilde{\D^+})}
\les \|z\|_{L_t^\frac{q}{2} L_x^\infty(\widetilde{\D^+})} +\la^{\frac{2}{q}-1-4\ep_0(\frac{4}{q}-1)}
\les \la^{\frac{2}{q}-1-4\ep_0(\frac{4}{q}-1)}.
$$
We therefore obtain the estimates for $\chih$ in (\ref{ric1}) and  (\ref{ric1.1}).

\subsubsection{ Proof of (\ref{8.0.3})}

Using (\ref{trscoord2}), the metric $\gac$ relative to the transport coordinate on $C_u$ satisfies
\begin{equation}\label{dtga}
\frac{d}{dt}\!\gac_{ab}=\left(\tr\chi-\frac{2}{\tir}\right) \gac_{ab}+2 \tir^{-2}\chih_{ab}.
\end{equation}
Thus
\begin{equation*}
\frac{d}{dt}(\gac_{ab}-\ga^{(0)}_{ab})=\left(\tr\chi-\frac{2}{\tir}\right)(\gac_{ab}-\ga_{ab}^{(0)})
+\left(\tr\chi-\frac{2}{\tir}\right) \c \ga^{(0)}_{ab}+2\tir^{-2} \chih_{ab}.
\end{equation*}
We now consider (\ref{8.0.3}) in $\D^+$. Using (\ref{8.1.1}), (\ref{ric1}) and Lemma \ref{tsp2} we can obtain
\begin{equation*}
\|\!\gac_{ab}-\ga^{(0)}_{ab}\|_{L^\infty(\D^+)}\les \la^{-8\ep_0}.
\end{equation*}
This gives the second estimate in (\ref{8.0.3}) on $\D^+$, improving the second assumption in (\ref{ba3}) on $\D^+$.

Next, by differentiating (\ref{dtga}) we have
\begin{align*}
\frac{d}{dt}\p_c \!\!\gac_{ab}
&=\left(\tr\chi-\frac{2}{\tir}\right)\p_c \!\!\gac_{ab}+ \p_c \tr\chi \!\gac_{ab}+2 \tir^{-2} \p_c \chih_{ab}
\end{align*}
with the initial condition given by (\ref{8.1.1}), where $a, b, c=1,2$. By using $\|\tr\chi-\frac{2}{t-u}\|_{L_t^2 L_\omega^\infty}\les \la^{-\f12-4\ep_0}$
from (\ref{ric1}) and  Lemma \ref{tsp2}, we can obtain
\begin{align*}
\left|\p_c \!\!\stackrel{\circ}\gamma_{ab}(t)-\p_c{\gamma}^{(0)}_{ab}\right|
&\les\int_u^t\left|\p_c \tr\chi\cdot
{\stackrel{\circ}\ga}_{ab}+\tir^{-2}({\sn}_c\chih_{ab}-\Gamma\cdot
\chih)+\Big(\tr\chi-\frac{2}{\tir}\Big) \p_c \ga^{(0)}_{ab}\right|  d \tt,
\end{align*}
where $\Gamma$ denotes the Christoffel symbols with respect to $\ga$, and $\Gamma \cdot
\chih$ stands for the term $\sum_{l=1}^2\Ga_{ki}^l\chih_{lj}$. Consequently
\begin{align}
\left\|\sup_{u\le t\le \tau_*}\left|\p_c \!\!\stackrel{\circ}\gamma_{ab}(t)-\p_c{\gamma}^{(0)}_{ab}\right|\right\|_{L_\omega^{p}}
&\lesssim \|\Gamma\|_{L_t^\infty L_\omega^{p}} \|\chih\|_{L_t^1 L_\omega^\infty} +\|\tir \sn \tr \chi \|_{L_\omega^{p} L_t^1} \nn\\
& +\|\tir\sn \chih\|_{L_\omega^{p} L_t^1} +\left\|\tr\chi-\frac{2}{\tir}\right\|_{L_\omega^p L_t^1}.\nn
\end{align}
Using (\ref{sna}), (\ref{ric3.18.1}) and (\ref{pi.2}) we then obtain
\begin{equation*}
\left\|\sup_{u \le t\le \tau_*}\left|\p_c \!\!\stackrel{\circ}\gamma_{ab}(t)-\p_c{\gamma}^{(0)}_{ab}\right|\right\|_{L_\omega^p}
\les\la^{-4\ep_0}+ \la^{-4\ep_0} \|\Gamma\|_{L_t^\infty L_\omega^p}.
\end{equation*}
Using the local expression of $\Ga_{ab}^c$ and the second estimate in (\ref{8.0.3}), it follows
\begin{equation*}
\|\Ga\|_{L_\omega^{p}}\les\sum_{a,b,c=1,2} \|\p_c (\gac_{ab}-{\gamma}^{(0)}_{ab})\|_{L_\omega^p}+C,
\end{equation*}
where $C$ is the constant such that  the Christoffel symbol of $\ga^{(0)}$ satisfies $|\p \ga^{(0)}|\le C$.
The combination of the above two inequality implies the first inequality in (\ref{8.0.3}) and the bound
\begin{equation}\label{gaa}
\|\Ga\|_{L_\omega^p} \les 1.
\end{equation}

By the similar argument, with the help of (\ref{a_5}) in Proposition \ref{exten} we can obtain (\ref{8.0.3})
for null cones on the region of $\widetilde{\D^+}$.
\subsubsection*{Proof of (\ref{ric4})}
We prove (\ref{ric4}) with the help of (\ref{lb}).  If $u>0$, in view of (\ref{w8.1.2}), we can derive
$
|\frac{\bb^{-1}-1}{t-u}|\les \M(\|k_{NN}\|_{L_\omega^\infty}).$
By using (\ref{hlm}) and (\ref{pi.2}), we obtain
\begin{equation}\label{bb_8.1}
\|\frac{\bb^{-1}-1}{t-u}\|_{L_t^2 L_x^\infty}\les \la^{-\f12-4\ep_0}
\end{equation}
If $u\le 0$, in view of (\ref{w8.1.3})
\begin{align}\label{w8.1.5}
|\frac{\bb^{-1}-1}{t-u}|&=|\frac{\bb^{-1}-a^{-1}}{t-u}|+|\frac{a^{-1}-1}{t-u}|\les\M(\|k_{NN}\|_{L_\omega^\infty})+|\frac{a^{-1}-1}{t-u}|.
\end{align}
Using (\ref{w8.1.1}), we have
\begin{align*}
\|\frac{a^{-1}-1}{t-u}\|_{L_t^2 L_x^\infty}\les\|\frac{a^{-1}-1}{v^\f12}\|_{L_x^\infty}(\int_0^{\tau_*} v(t+v)^{-2} dt)^\f12\les \la^{-\f12}.
\end{align*}
Combined with (\ref{w8.1.5}), we have (\ref{bb_8.1}) also holds for $u\le 0$ by using (\ref{hlm}) and (\ref{pi.2}). Hence  the first estimate in (\ref{ric4}) is proved.
The second estimate can be proved by using (\ref{w8.1.2}), (\ref{w8.1.3}) and (\ref{transtrace}).

By using (\ref{lb}), (\ref{pi.2}) and the first estimate in (\ref{ric4}),
$$
\|\tir\bd_L (\frac{\bb^{-1}-1}{\tir})\|_{L_t^2 L_\omega^p}\le \|\frac{\bb^{-1}-1}{\tir}\|_{L_t^2 L_\omega^p}+\|k_{NN}\|_{L_t^2 L_\omega^p}\les \la^{-\f12}.$$
In view of $\zeta=\sn \log \bb+k_{AN}$, we obtain $\|\sn\log \bb\|_{L_t^2 L_\omega^p}\les \la^{-\f12}$ by using (\ref{pric3}) and (\ref{pi.2}). Hence the last inequality of (\ref{ric4}) is proved.

Thus we complete the proof of Proposition \ref{ricpr} and Proposition \ref{ricco} except for the estimates of $\zeta$
in (\ref{ric1.1}) and (\ref{ric1}) which will be proved in Section \ref{ss1} based on new observations.

\subsection{Estimates on $\mu$ and $\zeta$ }\label{ss1}

We will use (\ref{dze}) and (\ref{dcurl}) to derive the estimates of $\zeta$ in (\ref{ric1.1}) and (\ref{ric1}).
Because the space-time metric $\bg$ is not necessarily Einstein,  the term $\mu:=\Lb \tr\chi + \f12 \tr\chi \tr \chib$
in the equation (\ref{dze}) presents technical challenges. Fortunately, we observe that $\mu$ has a hidden structure
which enables us to solve these difficulties. Our observation is based on the function $\varphi$ introduced
in the following result.

\begin{lemma}\label{larea}
Let $\ckk\ga=(t-u)^2 \ga^{(0)}$ and $\varphi=\log\sqrt{|\ga|}-\log\sqrt{|\ckk\ga|}$ on $S_{t, u}$. Then on $\widetilde{\D^+}$ there hold
\begin{equation*}
\|\tir^\f12\sn\varphi\|_{L_t^\infty L_\omega^p(C_u)}+\|\sn \varphi\|_{L_t^2 L_\omega^p(C_u)}
+\|\tir\sn_L \varphi\|_{L_t^2 L_\omega^p(C_u)}\les \la^{-\f12}.
\end{equation*}
\end{lemma}
\begin{proof}
By using (\ref{cmu2}) and $L\varphi=\tr\chi-\frac{2}{t-u}$ we derive that
\begin{equation*}
\sn_L \sn\varphi+\frac{1}{2}\tr\chi\sn\varphi=-\chih \c\sn\varphi+\sn (\tr\chi-\frac{2}{t-u}).
\end{equation*}
Since $\tr \chi - \frac{2}{t-u}= z -V_4$, we may use (\ref{sna}) and (\ref{flux_3}) to obtain
$
\left\|\tir\sn (\tr\chi-\frac{2}{\tir})\right\|_{L_t^2 L_\omega^p(C_u)}\les \la^{-\f12}.
$
Thus we may use the same argument in the proof of Proposition \ref{p1} to derive the desired estimates;
for null cones $C_u$ with $u\ge 0$ we use (\ref{8.1.1}) and for null cones $C_u$ with $u<0$ we use (\ref{4a_6})
to deal with the initial data.
\end{proof}

\begin{proposition}\label{mude}
For $\mu =\Lb {\emph\tr} \chi +\f12 {\emph\tr} \chi {\emph\tr} \chib$ there holds the decomposition
\begin{equation}\label{msb}
 \mu={\emph\div} \pi+{\emph\bA}\c {\emph\bA}+ \chi\c\pi+\sn\varphi(\zeta+\pi).
\end{equation}
Moreover, for $p$ satisfying $0<1-\frac{2}{p}<s-2$ there hold the estimates
\begin{align}
&\|\tir\mu, \tir\sn \zeta\|_{L_t^2L_\omega^{p}(C_u)}\les \la^{-\f12}.\label{n1cz}
\end{align}
\end{proposition}

\begin{proof}
Let $\varphi:=\log \sqrt{|\ga|}-\log \sqrt{|\ckk\ga|}$ be as in Lemma \ref{larea}. By (\ref{cmu3}) we have
\begin{equation}\label{comm0}
L \Lb \varphi-\Lb L \varphi=2(\zb-\zeta) \c\sn \varphi-2k_{NN} N\varphi.
\end{equation}
Using $L(t-u)=1$ and $\Lb(t-u)=1-2\bb^{-1}$, we can derive that
$$
\Lb \varphi=\tr\chib+(2\bb^{-1}-1)\frac{2}{t-u},\quad L\varphi=\tr\chi-\frac{2}{t-u},
\quad N(\varphi)=\tr\theta-\frac{2\bb^{-1}}{t-u}.
$$
Thus, in view of (\ref{comm0}), we obtain
\begin{equation*}
L\tr\chib-\Lb \tr\chi+L((2\bb^{-1}-1)\frac{2}{t-u})+\Lb(\frac{2}{t-u})=-2k_{NN} (\tr\theta-\frac{2\bb^{-1}}{t-u})+2(\zb-\zeta) \c \sn \varphi
\end{equation*}
which gives
\begin{equation*}
L\tr\chib-\Lb\tr\chi+\frac{4L(\bb^{-1})}{t-u}=-2k_{NN} \left(\tr\theta-\frac{2\bb^{-1}}{t-u}\right)+2(\zb-\zeta)\c\sn \varphi.
\end{equation*}
By virtue of (\ref{lb}), it then follows that
\begin{equation}\label{comm1}
L\tr\chib-\Lb\tr\chi=-2k_{NN} \tr\theta+2(\zb-\zeta)\c \sn \varphi.
\end{equation}
Recall the transport equation (\ref{mub}) and the fact $2\theta = \chi -\chib$. We then obtain
\begin{equation}\label{mu.2}
\Lb \tr\chi+\f12 \tr\chi\tr\chib=2 \div \zb+k_{NN}\tr\chi-\chih\c \chibh+2|\zb|^2+ \delta^{AB}\bR_{A34B}-2(\zb-\zeta) \sn \varphi.
\end{equation}
Thus, by using the curvature decomposition formula Lemma \ref{decom_lem} (iv) and $\zb_A = -k_{AN}$, we obtain (\ref{msb}).

We next use (\ref{msb}) to prove (\ref{n1cz}). By using (\ref{msb}) and Lemma \ref{decom_lem} (iv) we can derive from
(\ref{dze}) and (\ref{dcurl}) that
\begin{align}\label{zz1}
\begin{split}
\div \zeta &=\sn\pi+\bA\c \bA+\pi \c \chi+(\zeta+\sn \varphi)\c \zeta+\pi \c \sn\varphi,\\
\curl \zeta &=\sn \pi+\bA\c \bA+\pi \c \chi.
\end{split}
\end{align}
Thus it follows from Lemma \ref{hdgm1} that
\begin{align*}
\|\tir \sn \zeta\|_{L_t^2 L_\omega^p}
&\les \|\tir \sn \pi\|_{L_t^2 L_\omega^p}+\|\pi\|_{L_t^2 L_\omega^p}+\|\zeta,\bA\|_{L_t^2 L_\omega^\infty}
\|\tir( \zeta, \bA, \sn \varphi)\|_{L_t^\infty L_\omega^p}.
\end{align*}
Note that  (\ref{transtrace}), (\ref{pric2}) and Lemma \ref{larea}  imply
 \begin{equation}\label{aux_2}
 \|\tir^\f12( \zeta, \bA, \sn \varphi)\|_{L_t^\infty L_\omega^p}\les \la^{-\f12}.
\end{equation}
Note also that (\ref{ric3.18.1}) and (\ref{pi.2}) give
\begin{equation}\label{ainfty}
\|\bA\|_{L_t^2 L_\omega^\infty}\les\la^{-\f12}.
\end{equation}
These estimates together with (\ref{flux_3}) then imply that
\begin{align}\label{7.22.1}
\|\tir \sn \zeta\|_{L_t^2 L_\omega^p} \les \la^{-\f12} +\la^{-4\ep_0} \|\zeta\|_{L_t^2 L_\omega^\infty}.
\end{align}
By employing (\ref{sobinf}) and (\ref{pric1}) we have
$$
\|\zeta\|_{L_t^2 L_\omega^\infty} \les \|\tir \sn \zeta\|_{L_t^2 L_\omega^p} + \|\zeta\|_{L_t^2 L_\omega^2}
\le \|\tir \sn \zeta\|_{L_t^2 L_\omega^p} +\la^{-\f12}.
$$
Combining this with (\ref{7.22.1}) gives the estimate for $\sn \ze$ in (\ref{n1cz}) and consequently
\begin{equation}\label{zza}
\|\zeta\|_{L_t^2 L_\omega^\infty(C_u)}\les \la^{-\f12}.
\end{equation}
Therefore, in view of (\ref{msb}), (\ref{flux_3}), (\ref{aux_2}), (\ref{ainfty}) and (\ref{zza}),  we can obtain
\begin{align*}
\|\tir\mu\|_{L_t^2 L_\omega^p} & \les \|\tir \sn \pi\|_{L_t^2 L_\omega^p}+\|\pi\|_{L_t^2 L_\omega^p(C_u)}
+\tau_*^\f12 \|\bA\|_{L_t^2 L_\omega^\infty} \|\tir^\f12 \bA\|_{L_t^\infty L_\omega^p} \\
& \quad + \tau_*^{\f12} \|\zeta, \bA\|_{L_t^2 L_\omega^\infty} \|\tir^{\f12} \sn \varphi\|_{L_t^\infty L_\omega^p}
\les \la^{-\f12}
\end{align*}
which gives the estimate on $\mu$ in (\ref{n1cz}).
\end{proof}

\begin{proof}[Proof of the estimates for $\zeta$ in (\ref{ric1}) and (\ref{ric1.1})]
Now we prove the estimate of $\ze$ in (\ref{ric1}) using (\ref{zz1}).
In view of (\ref{zz1}), Proposition \ref{cz.2}, (\ref{pi.2}) and (\ref{aux_2}), we have on $\D^+$ that
\begin{align*}
& \left(\int_{0}^{\tau_*} \sup_u \|\zeta\|_{L^\infty(S_{t,u})}^2 dt \right)^{\f12}
\les\|\mu^{0+}P_\mu \ti\pi\|_{L_t^2 l_\mu^2 L_x^\infty}+\|\ti \pi\|_{L_t^2 L_x^\infty}\\
& \quad\qquad +\left(\int_0^{\tau_*}\left(\sup_{u} \|\tir^{1-\frac{2}{p}}(\bA, \ze)\c(\bA,\ze,\sn\varphi)\|_{L^p(S_{t,u})}\right)^2 dt\right)^{\f12}\\
&\quad \quad\les \la^{-\f12-4\ep_0}+\|\bA, \zeta\|_{L_t^2 L_x^\infty} \|\tir (\bA,\zeta, \sn \varphi)\|_{L^\infty L_\omega^p}\\
&\quad \quad \les \la^{-\f12-4\ep_0}+\la^{-4\ep_0}\|\bA, \zeta\|_{L_t^2 L_x^\infty}.
\end{align*}
By using the estimates for $z$, $\chih$ in (\ref{ric1}) and the estimate for $\pi$ in (\ref{pi.2}), we have
$\|\bA\|_{L_t^2 L_x^\infty}\les \la^{-\f12-4\ep_0}$. Therefore
\begin{align*}
\left(\int_{0}^{\tau_*} \sup_u \|\zeta\|_{L^\infty(S_{t,u})}^2 dt \right)^{\f12}
\les \la^{-\f12-4\ep_0}+\la^{-4\ep_0}\|\zeta\|_{L_t^2 L_x^\infty}
\end{align*}
which implies the estimate for $\zeta$ in (\ref{ric1}).

Similarly, we have on $\widetilde{\D^+}$ that
\begin{align*}
\|\zeta\|_{L_t^\frac{q}{2}L^\infty_x} &\les \|\mu^{0+}P_\mu \ti\pi\|_{L_t^\frac{q}{2} l_\mu^2 L_x^\infty}
+\|\pi\|_{L_t^\frac{q}{2} L_x^\infty}+\|\bA,\zeta\|_{L_t^\frac{q}{2} L_x^\infty} \|\tir (\bA, \zeta, \sn \varphi)\|_{L^\infty L_\omega^p}
\end{align*}
We then employ (\ref{pi.2}), (\ref{aux_2}), the estimates on $\bA$ in (\ref{ric1.1})  to obtain
\begin{equation*}
 \|\zeta\|_{L_t^\frac{q}{2}L^\infty}\les \la^{\frac{2}{q}-1-4\ep_0(\frac{4}{q}-1)}
\end{equation*}
which shows the estimate for $\zeta$ in (\ref{ric1.1}).
\end{proof}

From now on we will regard $\zeta$ as an element of the collection of $\bA$. We will use the convention that $\bA$
represents the terms of $\zeta, \chih, z, \pi, \frac{\bb^{-1}-1}{\tir}$ and their combinations. We use ${\bf \Ab}$ to denote all the elements in $\bA$ and $\sn \varphi$.
We have proved that

\begin{proposition}\label{awave}
On $\D^+$, with $0<1-\frac{2}{p}<s-2$, there hold $\|{\emph\bA}\|_{L_t^2 L_x^\infty}\les \la^{-\f12-4\ep_0}$ and
\begin{equation*}
\|\tir(\sn {\bf\Ab}, \sn_L {\bf \Ab})\|_{L_t^2 L_\omega^p}+\|{\bf \Ab}\|_{L_t^2 L_\omega^p}
+\|\tir^{\f12}{\bf \Ab}\|_{L_t^\infty L_\omega^p}+\|\tir^\f12 \bA\|_{L_t^\infty L_\omega^{2p}}\les \la^{-\f12}.
\end{equation*}
\end{proposition}

\section{\bf Conformal method}\label{sec_6}

In this section we will work on the wave zone ${\mathscr D}^+$ under the same setup in Section 5.
We define a scalar function $\sigma$ by
\begin{equation}\label{c1}
L\sigma=\f12 V_4,\quad \sigma(\Ga^+)=0,
\end{equation}
where $\Ga^+$ is the time axis. We then introduce the metric $\ti \bg=\Omega^{-2} \bg$ with the conformal factor
$\Omega=e^{-\sigma}$. Let $\ti \ga$ be the induced metric of $\ti \bg$ on each wave front $S_{t,u}$. By taking
the Lie derivatives of $\ti \ga$ in the directions of $L$ and $\Lb$ respectively, we introduce
$$
\ti\chi = \f12 {\mathscr L}_L\ti{\ga}, \qquad \ti\chib =\f12 {\mathscr L}_\Lb\ti{\ga}.
$$
Recall that $\chi =\f12 {\mathscr L}_L \ga$ and $\chib = \f12 {\mathscr L}_\Lb \ga$. It is straightforward to show that
$$
\ti \chi = \Omega^{-2} \left(\chi +(L\sigma) \ga\right), \qquad \ti \chib = \Omega^{-2} \left(\chib + (\Lb \sigma) \ga\right).
$$
We set $\tr\ti \chi:= \ti \ga^{AB} \ti\chi_{AB}$ and $\tr\ti\chib: = \ti \ga^{AB} \ti \chib_{AB}$.
It is easily seen that
\begin{align}\label{c_4_20}
\tr \ti \chi=\tr\chi+2L \sigma, \qquad \tr \ti\chib = \tr \chib + 2\Lb \sigma.
\end{align}
Considering (\ref{c1}) we have $\tr\ti\chi = \widetilde{\tr\chi}$ for the $\widetilde{\tr\chi}$ introduced in
Proposition \ref{ricco}. This explains the choice of $\sigma$.

Let $\tiR$ and $\Box_{\ti\bg}$ denote the Ricci curvature tensor and the Laplace-Beltrami operator associated with $\ti \bg$.
For a given scalar function $w$, we set $\tiw=\Omega w$. Then there hold (see \cite[p. 275]{hmd})
\begin{equation*}
\Box_\bg w-\frac{1}{6}\bR w=\Omega^{-3} (\Box_{\ti \bg} \tiw-\frac{1}{6} {\tiR} \tiw )
\end{equation*}
and
\begin{equation*}
\tiR =\Omega^2(\bR-6\Box_\bg \sigma-6 \bd^\mu \sigma\bd_\mu \sigma).
\end{equation*}
Hence
\begin{equation*}
\Box_\bg w-\Omega^{-3} \Box_{\ti \bg} \tiw =(\Box_\bg \sigma+ \bd^\mu \sigma\bd_\mu \sigma)w.
\end{equation*}
In Section \ref{sec_7}, we will use this equation to derive estimates on the conformal energy for the solutions
of $\Box_\bg w=0$ in ${\mathscr D}^+$; in such situation the above equation becomes
\begin{equation}\label{idc1}
\Box_{\ti \bg} \tiw=-\Omega^2(\Box_\bg \sigma+\bd^\mu \sigma\bd_\mu \sigma) \tiw.
\end{equation}

In this section we will provide various estimates on the conformal factor $\sigma$ which are crucial in Section \ref{sec_7}.
We start with some preliminary estimates.

\begin{lemma}\label{pre_sig}
Let $0<1-\frac{2}{p}<s-2$.  Within ${\mathscr D}^+$ there hold
\begin{equation}\label{l4.1}
\|\tir^\f12 L\sigma\|_{L^{2p}_\omega(C_u)}+ \|r^{\f12-\frac{2}{p}}\sn \sigma\|_{L_x^p L^\infty(C_u)}+\|\sn \sigma \|_{L_t^2 L_\omega^p(C_u)}\les \la^{-\f12}
\end{equation}
and
\begin{equation}\label{l4.19}
\|\sigma\|_{L^\infty}\les \la^{-8\ep_0}, \quad \|\tir^{-\f12} \sigma\|_{L^\infty}\les \la^{-\f12 -4\ep_0},\quad  \Omega\approx 1.
\end{equation}
\end{lemma}

\begin{proof}
We first show (\ref{l4.19}). Let $(t,u,\omega)\in {\mathscr D}^+$ be any point. We integrate along the
null geodesic $\Upsilon_\omega(t')$ initiating from the time axis at $t'=u$. By (\ref{c1}) we have
\begin{equation*}
\sigma(t,u, \omega)=\f12\int_u^t V_4.
\end{equation*}
This together with (\ref{pi.2}) implies that
\begin{equation*}
|(t-u)^{-\f12}\sigma(t,u, \omega)|\les (t-u)^{-\f12} \int_u^t| V_4|
\les \la^{-\f12-4\ep_0}.
\end{equation*}
Since $t-u\le \tau_*\le \la^{1-8\ep_0}$, we obtain
\begin{equation*}
|\sigma(t,u, \omega)|\les (\tau_*)^\f12 \la^{-\f12-4\ep_0}\les \la^{-8\ep_0}.
\end{equation*}
By definition, this also gives $\Omega \approx 1$ . Thus we obtain (\ref{l4.19}).

We next prove (\ref{l4.1}). Since $L\sigma=\f12 V_4$ and $V_4=\bg \c {\bp \bg}$, the estimate for $L\sigma$ follows
from Proposition \ref{awave}. To derive the estimates for $\sn \sigma$, we rely on the transport equation
\begin{equation*}
L\sn\sigma+\f12 \tr\chi \sn \sigma=-\chih\c \sn \sigma+\f12\sn V_4.
\end{equation*}
which can be derived from (\ref{c1}) and (\ref{cmu2}).  Since $\|\chih\|_{L_t^2 L_x^\infty}\les \la^{-\f12-4\ep_0}$
from Proposition \ref{awave}, we may use Lemma \ref{tsp2} to obtain
\begin{equation*}
|\sn \sigma|\les\frac{1}{t-u}\int_u^t (t'-u) |\sn V_4| dt'.
\end{equation*}
In view of (\ref{flux_3}), we can obtain $\|\tir^{\f12-\frac{2}{p}}\sn \sigma\|_{L_x^p L^\infty(C_u)}
+\|\sn \sigma\|_{L_t^2 L_\omega^p(C_u)}\les \la^{-\f12}$. The proof is thus complete.
\end{proof}

\subsection{Improved estimates on $\sn \sigma$ and $\mu+2\sD \sigma$}\label{sec_6_imp}

The preliminary estimates contained in (\ref{l4.1}) are not strong enough for our purpose.
Due to the construction of $\sigma$, we observe some remarkable cancellation which leads us to
discover an important hidden structure in the transport equation of $\mu+2\sD\sigma+l.o.t$,
see Lemma \ref{mu4.4}. This nice structure enables us to derive a set of improved estimates
in Proposition \ref{improve} for $\sn\sigma$ and $\mu+2\sD\sigma$ both of which are crucial
for controlling the conformal energy.

To be more precise, we start with $\sD \sigma$. By using (\ref{tran1}), (\ref{dchi}) and (\ref{c1}), we can see that $\sD \sigma$
verifies the transport equation
\begin{align}
L \sD \sigma+ \tr\chi \sD \sigma&=\f12 \sD (V_4)-2\chih_{AC} \sn_A \sn_C \sigma- \sn_A\tr\chi \sn_A \sigma
-2 \delta^{AB}\bR_{CA4B}\c \sn_C\sigma\nn\\
&-\left(\f12 k_{AN} \tr\chi-\chih_{AB} k_{BN}\right)\sn_A\sigma-\chi_{AB} \zb_A \sn_B\sigma\label{tran4}.
\end{align}

Next we consider $\mu:=\Lb \tr\chi+\f12 \tr\chi\tr\chib$. By using (\ref{s1}) and (\ref{cmu3}) we can derive that
\begin{align}
L\mu+\tr\chi\mu&=\f12 L(\tr\chi\tr\chib)+\f12 (\tr\chi)^2 \tr\chib-2\sn_\Lb \chih\c \chih\label{lmu1}\\
&+2(\zb_A-\ze_A) \sn_A\tr\chi-\Lb \bR_{44}-\tr\chi \Lb k_{NN}-k_{NN} L\tr\chi.\nn
\end{align}
In view of (\ref{mub}) and (\ref{s1}), we have
\begin{align*}
& L (\tr\chi \tr \chib) = \tr\chib L\tr\chi+\tr\chi L\tr\chib\\
&=\tr\chi \left(2\div\zb-\chih\c \chibh+2|\zb|^2+\delta^{AB}\bR_{A34B}\right)-\bR_{44}\tr\chib-\tr\chib|\chih|^2-\tr\chib (\tr\chi)^2.
\end{align*}
Plugging this equation into (\ref{lmu1}) and using (\ref{3chi}) for $\sn_\Lb \chih$, we can conclude
 \begin{align}
 &L\mu+\tr\chi\mu=-\Lb \bR_{44}-\f12\bR_{44}\tr\chib-\tr\chi\Lb (k_{NN})-k_{NN}L\tr\chi\label{6.15.1}\\
 &+2(\zb_A-\ze_A)\sn_A\tr\chi+\tr\chi\left(\div\zb+|\zb|^2+\f12\delta^{AB} \bR_{A34B}\right)+\f12 \left( \tr\chi\chih\c\chibh+\tr\chib|\chih|^2\right)\nn\\
 &-2\chih_{AB} \left(2\sn_A\zeta_B+k_{NN}\chih_{AB}+2\zeta_A\zeta_B+\bR_{A43B} \right).\nn
 \end{align}
 By direct calculation (see \cite[eq. (1.0.3a)]{CK}, we have
 \begin{equation}\label{7.16.11}
\Lb (k_{NN})=-L(k_{NN})+\f12 \bR_{4343}+2 k_{Nm}k^{m}_N.
\end{equation}
Noting that by decomposing Riemann curvature,
\begin{equation*}
\bR_{3434}=2(\bR_{34}- \delta^{AB} \bR_{A3B4}).
\end{equation*}
Consequently, in view of Lemma \ref{decom_lem} (iv) we can write
\begin{equation}\label{r34}
\bR_{3434}=2\left(\bR_{34}+\div \pi+\sl \bE\right),
\end{equation}
where $\sl \bE=\bA\c \pi +\tr\chi \c \pi$. Hence, by combining (\ref{r34}) with (\ref{7.16.11}) we have
\begin{equation}\label{lbkn}
\Lb (k_{NN})=\bR_{34}-L(k_{NN})+2 k_{Nm}k^{m}_N+ \div \pi+\sl \bE.
\end{equation}
By substituting (\ref{lbkn}) into (\ref{6.15.1}), it yields
\begin{align}\label{6.15.100}
L\mu&+\tr\chi\mu=R(\mu)-k_{NN}L\tr\chi +2(\zb_A-\ze_A)\sn_A\tr\chi+\f12\left( \tr\chi\chih\c\chibh+\tr\chib|\chih|^2\right)\nn\\
& +\tr\chi\left(\div\zb+|\zb|^2+\f12\delta^{AB} \bR_{A34B}+L(k_{NN})-2k_{Nm} k_N^m -(\div \pi + \sl \bE)\right) \nn\\
& -2\chih_{AB} \left(2\sn_A\zeta_B+k_{NN}\chih_{AB}+2\zeta_A\zeta_B+\bR_{A43B}\right),
\end{align}
where
$$
R(\mu):=-\Lb \bR_{44}-\tr\chi \bR_{34}-\f12\tr\chib \bR_{44}
$$
which determines the behavior of $L\mu+\tr\chi \mu$. By virtue of  (\ref{r44.1}), (\ref{ricc6.7.1}) and (\ref{6.7con}), we can derive that
\begin{align*}
\bR_{44} & = L(V_4)+k_{NN} V_4 + \pi\c\pi,\\
\bR_{34}&=\f12( \Lb(V_4)+L(V_3))+V\c (\zeta+\zb)+k_{NN}\c V+\pi\c\pi,\\
-\Lb \bR_{44}&=-\Lb L V_4-\Lb (k_{NN}) V_4+k_{NN}\c (\zeta+k)\c V+k \c\bd V+\Lb (\pi\c\pi).
\end{align*}
Therefore
\begin{align}
R(\mu)=-L \Lb (V_4)-\f12 \tr\chi\Lb (V_4)-\f12 \tr\chib L (V_4)-\f12 \tr\chi L (V_3)+\tr\chi \pi\c\pi+\bA \c \bd \ti \pi+\bA^3, \label{4.2.14}
\end{align}
where the meaning of the symbols $\pi$, $\ti \pi$ and $\bA$ can be found in Section \ref{sec_5}.  It is remarkable to observe that
the first three terms on the right hand side of (\ref{4.2.14}), plus $\sD (V_4)$, become exactly $\Box_\bg(V_4)$ modulo some lower order terms.
In fact, this can be seen from the decomposition formula of $\Box_\bg (V_4)$ under the null frame which can be stated as
\begin{equation}\label{psi2}
\Box_\bg (V_4) =- L \Lb (V_4)+\sD V_4-\f12 \tr\chi \Lb V_4-\f12 \tr\chib L V_4+2 \zb^A \sn_A (V_4)+k_{NN} \Lb V_4.
\end{equation}
Thus, it is natural to add to (\ref{6.15.100}) a twice multiple of the transport equation (\ref{tran4}). After further dealing with the terms $\tr\chi L(k_{NN})$
and $\tr\chi L(V_3)$, it motivates us to introduce
$$
\ckk\mu=2\sD \sigma +\mu- \tr\chi k_{NN}+\f12\tr\chi V_3.
$$
Therefore, in view of (\ref{tran4}), (\ref{6.15.100}), (\ref{4.2.14}), (\ref{psi2}) and (\ref{s1}) we obtain
\begin{align}
L\ckk \mu+\tr\chi \ckk\mu&=\Box_\bg( V_4)-2\left(2 \delta^{AB} \bR_{CA4B}+\f12 k_{AN} \tr\chi-\chih\c k+\chi\c \zb\right)\sn\sigma\nn\\
& +2(\zb-\ti\zeta) \sn \tr\chi -4 \chih\c\sn\ti\ze-2\chih_{AB} \big(k_{NN}\chih_{AB}+2\zeta_A\zeta_B+\bR_{A43B}\big)\nn\\
&+\tr\chi \left(\div\zb+|\zb|^2+\f12\delta^{AB} \bR_{A34B}\right)+\f12 \left( \tr\chi\chih\c\chibh+\tr\chib|\chih|^2\right)\nn\\
&+2k_{NN} \left(|\chih|^2+k_{NN} \tr\chi+\bR_{44}\right)+ \bA \c(\bd \bpi+\bE)+\tr\chi \left(\div \pi+\sl \bE\right)\label{mu2}
\end{align}
after incorporating more lower order terms into the last two terms of (\ref{mu2}), where $\ti\zeta=\sn\sigma+\zeta$ and
$\bE= \bA\c \bA +\tr\chi \c \bA$.

We now consider $\Box_{\bg} (V_4)$. By writing $V_4=V(L) = V_\mu L^\mu$, we can obtain
\begin{equation}\label{bv4}
\Box_\bg(V_4)=V_\mu \Box_\bg L^\mu+2\bd_\a V_\mu \bd^\a L^\mu+L^\mu \Box_\bg V_\mu.
\end{equation}
Decomposing $\Box_g L^\mu$ under the null frame, with the help of (\ref{6.7con}) we have
\begin{align}\label{4.4}
\Box_\bg L^\mu &= -\bd_\Lb (\bd_L L^\mu)+\sD L^\mu+2\zeta_B \sn_B L^\mu-\left(\f12 \tr\chib-k_{NN}\right)\bd_L L^\mu \nn \\
& \quad\, -\f12 \tr\chi\bd_\Lb L^\mu-\f12 \bR^\mu_{\dum LL\Lb}.
\end{align}
By using (\ref{6.7con}) again, we can deduce that
\begin{align*}
\sD L^\mu&=(\div \chi)_B e_B^\mu-\sn_B k_{BN} L^\mu  - k_{BN} \left(\chi_{BC} e_C^\mu-k_{BN}L^\mu\right) +\chi_{BC}\sn_B e_C^\mu\\
& \quad\, + \f12 \chi_{BC} \left(\chi_{BC} \Lb^\mu +\chib_{BC} L^\mu\right)
- \sn_B e_B^C \left(\chi_{CD} e_D^\mu -k_{CN} L^\mu\right).
\end{align*}
Note that
$$
\chi_{BC} \, \sn_B e_C^\mu - \sn_B e_B^C \, \chi_{CD} \,e_D^\mu = \chih_{BC} \sn_B e_C^\mu - \chih_{CD} e_D^\mu\sn_B e_B^C.
$$
We thus obtain
\begin{align*}
V_\mu \sD L^\mu&=(\div \chi)_B V_B-V_4 \sn_B k_{BN} - k_{BN} \left(\chi_{BC} V_C-k_{BN} V_4\right) +\chih_{BC} V_\mu \sn_B e_C^\mu \\
& \quad \, - \chih_{CD} V_D \sn_B e_B^C + \f12 \chi_{BC} \left(\chi_{BC} V_3 +\chib_{BC} V_4\right)
- V_4 k_{CN} \sn_B e_B^C.
\end{align*}
Let $\widetilde\Ga$ denote the collection of connection coefficients $\sn_A e_B^\mu$. In view of $\chi +\chib = -2k$,  symbolically we have
\begin{align*}
V_\mu \sD L^\mu = V\c \div \chi+V\c \bd k + \chi\c \chi \c V +V\c k\c (\chi, k)+V\c \widetilde\Ga \c (\chih, k).
\end{align*}
By straightforward calculation we also have
\begin{align*}
-V_\mu \bd_\Lb (\bd_L L^\mu) &= V_4 \Lb(k_{NN}) + k_{NN}\left(2\zeta_B V_B+k_{NN}V_4\right)
  = V\c \bd k + V\c k\c (\zeta, k),\\
V_\mu \zeta_B \sn_B L^\mu & = \zeta_B \chi_{BC} V_C- \zeta_B k_{BN} V_4 = V\c \zeta\c (k, \chi).
\end{align*}
For the remaining terms in (\ref{4.4}) and (\ref{bv4}), we have from (\ref{6.7con}) that
\begin{align*}
V^\mu& \left((k_{NN}-\f12 \tr\chib)\bd_L L^\mu-\f12 \tr\chi\bd_\Lb L^\mu\right)\\
& = k_{NN}(\f12 \tr\chib-k_{NN})V_4-\f12\tr\chi(2 \zeta_A V_A+k_{NN} V_4) =V\c k\c (\tr\chi, k)+ V\c \zeta\c \tr\chi,\\
\bd_\a& V_\mu \bd^\a L^\mu=-\f12(\bd_L V_\mu \bd_3 L^\mu+\bd_3 V_\mu \bd_4 L^\mu)+\sn V_\mu \sn L^\mu\\
&=-\f12(\bd_L V_\mu(2\zeta_B e_B^\mu+k_{NN} e_4^\mu)-\bd_3 V_\mu k_{NN} L^\mu) +\sn_B V_\mu(\chi_{BC}e_C^\mu-k_{BN} L^\mu) \\
&=\bd V\c (\chih, k, \zeta)+\f12\sn_B V_B \c  \tr\chi
\end{align*}
For the last term, it is important to discriminate $\sn_A V_A$ from the full derivative $\bd V$ in order to have the control
on $\sn\sigma$, see Proposition \ref{dcmpsig}.

Thus, by combining the above equations with (\ref{bv4}), we obtain
\begin{align}
\Box_\bg (V_4)&= \tr\chi \sn_B V_B+L^\mu \Box_\bg V_\mu+V\c (\div \chi+\bd k)+\bd V\c (\chih, k, \zeta)\nn\\
& + V\c k\c (k, \zeta) + V\c \chi \c (k, \zeta) +V\c \chi\c \chi + V\c \widetilde\Ga\c (\chih, k)+\bR_{\mu L L\Lb} V^\mu. \label{boxv4}
\end{align}

Now we plug (\ref{boxv4}) into (\ref{mu2}). We then write $\chi = \chih + \f12 \tr\chi \ga$ and use
$\tr\chi = \widetilde{\tr\chi} -V_4 = z-V_4 + \frac{2}{\tir}$. By making use of the curvature decomposition
in Lemma \ref{decom_lem} and the transport equation (\ref{dchi}) for $\div \chih$, we finally obtain
the following result.

\begin{lemma}\label{mu4.4}
Let  $\ckk\mu=2\sD \sigma+\mu - \tr\chi k_{NN}+\f12\tr\chi V_3$. Then $\ckk\mu$ verifies
\begin{align}
(L\ckk\mu +\tr\chi\ckk\mu)& -(\tir^{-1} \div \pi +\tir^{-2}\pi)=L^\mu\Box_\bg V_\mu+\chih \c\sn \ti \zeta+\sn \sigma\c \bE\label{mu3}\\
&+\sn\sigma\c (\sn \pi+\sn \widetilde{\tr\chi})+\bA\c(\sn \widetilde{\tr\chi}, \bE, \bd \ti\pi, V\c \widetilde\Ga)\nn
\end{align}
where $\bA= (\chih, k, \zeta, \pi, z)$, $\bE= \bA\c \bA + \tr\chi \c \bA$
and $z= \widetilde{tr \chi} - \frac{2}{\tir}$.  For $\pi$ and $\ti \pi$ we refer to Section \ref{sec_5} for their meaning.
\end{lemma}

To treat the term $\sn \ti \zeta$ with $\ti \zeta = \zeta +\sn \sigma$, we recall the Hodge operator $\D_1$
which sends $S_{t,u}$-tangent tensor $F$ to $(\div F, \curl F)$. Thus we can write
$
\sn \tilde \zeta = \nabla \D_1^{-1} (\div\ti \zeta, \curl \ti \zeta)
$
and use the Hodge system
\begin{align}
\div\ti\zeta&=\f12(\ckk \mu-2|\zeta|^2-|\chih|^2-2 k_{AB} \chih_{AB})+\div \pi_1+\bE =\frac{1}{2} \ckk \mu +\div \pi_1 +\bE, \label{hodge1} \\
\curl \ti\zeta&=\f12 \ep^{AB} k_{AC} \chih_{CB}+\f12 \ep^{AB}\bR_{B43A}=\curl \pi_2+\bE\label{hodge2}
\end{align}
which are directly derived from (\ref{dze}) and  (\ref{dcurl}) together with the curvature decomposition Lemma \ref{decom_lem} (iv)
contained in Lemma \ref{decom_lem}, where $\pi_1$ and $\pi_2$ are 1-forms of type $\pi$.

We now employ (\ref{mu3}), (\ref{hodge1}) and (\ref{hodge2}) to obtain

\begin{proposition}\label{improve}
For any $p$ satisfying $0<1-\frac{2}{p}<s-2$ there hold
\begin{align}
&\|\sn \sigma\|_{L_u^2 L_t^2 L_\omega^\infty({\mathscr D}^+)}+\|r\ckk \mu,r\sn \ti \zeta\|_{L_u^2 L_t^2 L_\omega^{p}({\mathscr D}^+)}\les \la^{-4\ep_0},\label{ckmu2}\\
&\|r^{\frac{3}{2}} \ckk \mu\|_{L_u^2 L_t^\infty L_\omega^{p}({\mathscr D}^+)}\les \la^{-4\ep_0}. \label{ckmu3}
\end{align}
\end{proposition}

\begin{proof}
We prove the estimate on $\|\sn \sigma\|_{L_u^2 L_t^2 L_\omega^\infty}$ by making the bootstrap assumption
\begin{equation}\label{ba.1}
\|\sn \sigma\|_{L_t^2 L_u^2 L_\omega^\infty}\le 1
\end{equation}
and improving it to
\begin{equation}\label{con.1}
\|\sn \sigma\|_{L_t^2 L_u^2 L_\omega^\infty}\les \la^{-4\ep_0}.
\end{equation}
The other estimates will be established during the course of the derivation.

We first consider the estimates for $\ckk \mu$. By Lemma \ref{tsp2} and (\ref{comp_3_27}), we have
\begin{equation}\label{gronmu}
|\tir^2\ckk\mu|\les\left|\int_u^t  {\ti r}^2 |L\ckk\mu+\tr\chi\ckk \mu| d t'\right|.
\end{equation}
We will give estimates on $\ckk\mu$ in (\ref{ckmu2}) and (\ref{ckmu3}), by taking the corresponding norm in view of (\ref{gronmu}) and (\ref{mu3}).

For the first term on the right of (\ref{mu3}), we note that $|L^\mu \Box_\bg V_\mu|\le |\Box_\bg V_\mu|$
and $\Box_\bg V_\mu=\bp\ti\pi\c \ti\pi+(\ti\pi)^3$ by virtue of (\ref{wave1}).
By using (\ref{pi.2}) and Proposition \ref{flux_2}, we then obtain
\begin{align*}
\|\tir L^\mu \Box_\bg V_\mu\|_{L_t^1 L_u^2 L_\omega^{p}}
&\les \|\tir\bp\ti\pi, \tir(\ti\pi)^2\|_{L_t^\infty L_u^2 L_\omega^{p}}\|\ti\pi\|_{L_t^1 L_x^\infty}\les \la^{-\f12-8\ep_0}.
\end{align*}
Therefore
\begin{equation}\label{1st}
\left\|\frac{1}{\tir} \int_u^t \tir^2 |L^\mu \Box_\bg V_\mu| dt'\right\|_{L_u^2 L_t^\infty L_\omega^{p}}
\les \|\tir L^\mu \Box_\bg V_\mu\|_{L_t^1 L_u^2 L_\omega^{p}}\les \la^{-\f12-8\ep_0}.
\end{equation}

We now treat the remaining terms on the right hand side of (\ref{mu3}). By virtue of (\ref{ba.1}) and Proposition \ref{awave}
we have
\begin{align*}
\left\|\frac{1}{\tir}\int_u^t \tir^2 \sn \sigma(\sn \pi +\sn \widetilde{\tr\chi})\right\|_{L_u^2 L_t^\infty L_\omega^{p}}
\les \|\sn \sigma\|_{L_u^2 L_t^2 L_x^\infty} \|\tir \sn \pi, \tir\sn \widetilde{\tr\chi}\|_{L_u^\infty L_t^2 L_\omega^{p}} \les \la^{-\f12}
\end{align*}
Recall that $\bE = \bA\c \bA + \tr\chi \c \bA$ and $\tr\chi = z -V_4 +\frac{2}{\tir}$, we can use (\ref{l4.1}) and Proposition \ref{awave}
to obtain
\begin{align*}
\left\|\frac{1}{\tir}\int_u^t \tir^2  \sn \sigma \c \bE \right\|_{L_u^2 L_t^\infty L_\omega^{p}}
&\les \tau_*^{\f12} \|\tir^\f12 \sn \sigma \|_{L^\infty L_\omega^{p}}\|\tir^\f12 \bA\c \bA\|_{L_u^\infty L_t^1 L_x^\infty}\\
&+ \tau_*^{\f12} \|\sn \sigma\|_{L_u^\infty L_t^2 L_\omega^{p}}\|\bA\|_{L_t^2 L_x^\infty} \les \la^{-\f12-4\ep_0}.
\end{align*}
Moreover, it is easily seen that
\begin{equation}\label{err_4_4}
\|\tir \bE\|_{L_t^2 L_u^\infty L_\omega^p}\les \la^{-\f12-4\ep_0}, \quad\|\tir^\frac{3}{2} \bE\|_{L_u^2 L_t^\infty L_\omega^p}\les \la^{-4\ep_0}.
\end{equation}
Thus, it follows from (\ref{err_4_4}), (\ref{eng4}) in Proposition \ref{flux_2} and Proposition \ref{awave} that
\begin{align*}
\left\|\frac{1}{\tir}\int_u^t \tir^2 \bA\c(\sn \widetilde{\tr\chi}, \bd \ti\pi, \bE)\right\|_{L_u^2 L_t^\infty L_\omega^{p}}
&\les\|\bA\|_{L_u^\infty L_t^2 L_x^\infty} \|\tir (\bd \ti\pi, \bE, \sn \widetilde{\tr\chi})\|_{L_u^2 L_t^2 L_\omega^{p}}\les \la^{-\f12-8\ep_0}.
\end{align*}
For the term $\chih\c \sn\ti\zeta$, we may use (\ref{hodge1}), (\ref{hodge2}), Lemma \ref{hdgm1}, (\ref{err_4_4}) and Proposition \ref{awave}
to derive that
\begin{align*}
\left\|\frac{1}{\tir} \int_u^t \tir^2 \chih\c \sn \ti\zeta dt'\right\|_{L_u^2 L_t^2 L_\omega^p}
&= \left\|\frac{1}{\tir}\int_u^t \tir'^2 \chih\c \sn \D_1^{-1} (\div \ti \zeta, \curl \ti \zeta)d\tt\right\|_{L_u^2 L_t^\infty L_\omega^{p}}\\
&\les \|\chih\|_{L_t^2 L_x^\infty} \|\tir (\ckk\mu,\sn \pi, \bE)\|_{L_u^2 L_t^2 L_\omega^p}\\
&\les (\la^{-4\ep_0}+\|\tir\ckk\mu\|_{L_u^2 L_t^2 L_\omega^p})\la^{-\f12-4\ep_0}.
\end{align*}
For the term involving $\widetilde\Ga$ which comes from $\sn_{e_A} e_B$, we can employ (\ref{gaa}) to get $\|\tir \widetilde \Ga\|_{L^\infty L_\omega^p}\les 1$.
Thus, by using Proposition \ref{awave} we have
\begin{align*}
\left\|\frac{1}{\tir} \int_u^t \tir^2 \chih\c \widetilde\Ga\c V\right\|_{L_u^2 L_t^\infty L_\omega^p}
&\les \tau_*^{\f12} \|\tir \widetilde\Ga\|_{L_t^\infty L_\omega^p} \|\chih\|_{L_t^2 L_x^\infty}
\|V\|_{L_t^2 L_x^\infty}\les \la^{-\f12-12\ep_0}.
\end{align*}
Finally, for the second part on the left hand side of (\ref{mu3}), we use (\ref{hlm}) and Proposition \ref{awave} to derive that
\begin{equation*}
\left\|\tir^{-1} \int_u^t (\tir'\div \pi, \pi)\right\|_{L_u^\infty L_t^2 L_\omega^p} \les \la^{-\f12}
\end{equation*}
Now we may divide (\ref{gronmu}) by $\tir$ and use the above estimates to derive that
\begin{equation*}
\|\tir\ckk\mu\|_{L_u^2 L_t^2 L_\omega^p}\les \la^{-8\ep_0}\|\tir\mu\|_{L_u^2 L_t^2 L_\omega^p}+\la^{-4\ep_0}
\end{equation*}
which gives
\begin{equation}\label{mu_44}
\|\tir \ckk\mu\|_{L_u^2 L_t^2 L_\omega^p}\les \la^{-4\ep_0}.
\end{equation}
We may divide (\ref{gronmu}) by $\tir^{\f12}$ and employ the similar argument as above to derive (\ref{ckmu3}).

By making use of (\ref{hodge1}), (\ref{hodge2}), Lemma \ref{hdgm1}, (\ref{mu_44}), (\ref{flux_3})  and (\ref{err_4_4}), we can obtain that
\begin{align*}
\|\tir\sn \ti \zeta\|_{L_u^2 L_t^2 L_\omega^p}&\les \|\tir \ckk\mu\|_{L_u^2 L_t^2 L_\omega^p}
+\|\tir (\bE+\sn\pi)\|_{L_u^2 L_t^2 L_\omega^p}\les \la^{-4\ep_0}.
\end{align*}
In view of $\sn \sigma=\ti\zeta-\zeta$, (\ref{con.1}) can be obtained by using (\ref{sobinf}) and Proposition \ref{awave}.
The proof is thus complete.
 \end{proof}

\begin{proposition}\label{dcmpsig}
$\sn \sigma$ can be decomposed as
$$
\sn \sigma=\bA+\bA^\dag+\mu^\dag,
$$
where $\bA^\dag$ and $\mu^\dag$ are 1-forms satisfying the estimates
$$
\|\bA^\dag\|_{L_t^2 L_x^\infty}\les \la^{-\f12-4\ep_0} \quad \mbox{and} \quad \|\mu^\dag\|_{L_u^2 L^\infty}\les \la^{-\f12-4\ep_0}.
$$
\end{proposition}

\begin{proof}
For any scalar function $f$, we use $\overline f$ to denote the function that takes on each $S_{t,u}$ the average value of $f$,
i.e. $\overline f= \frac{1}{|S_{t,u}|} \int_{S_{t,u}} f d\mu_\ga$ on $S_{t,u}$. Let $v_t$ be the area element of $S_{t, u}$. Recall that
$L v_t = v_t \tr\chi$, we can use it to derive that
\begin{equation}\label{7.18.1}
L\overline{f}+\tr\chi \overline{f}=(\tr\chi-\overline{\tr\chi})\c\overline{f}+\overline{L f+\tr\chi f}.
\end{equation}

Now, using $\ckk \mu$ we introduce a function $\sl \mu$ such that on each $S_{t,u}$ it is defined by
\begin{equation}\label{hdg_1}
\div \slashed{\mu}=\f12(\ckk \mu-\overline{\ckk\mu}),  \qquad \curl \slashed{\mu}=0.
\end{equation}
By using (\ref{hdg_1}) and (\ref{cmu1}) we can derive that
\begin{align}\label{7.19.1}
\curl (L\sl \mu  + \frac{1}{2} \tr \chi \sl \mu) 
& = \ep^{AB} \left(\chi_{AC} \zb_B - \bR_{BC4A}\right) \sl\mu_C + \ep^{AB} \chi_{AC} \sn_C \sl\mu_B\nn\\
&\quad \, + \frac{1}{2} \ep^{AB} \sn_A \tr \chi \c \sl \mu_B
\end{align}
and
\begin{align*}
& \div (L \sl \mu + \f12 \tr\chi \sl \mu) \\
& = L \div \sl \mu + \chi_{AB} \sn_A \sl\mu_B -\left(\tr \chi \, \zb_C - \chi_{AC} \zb_A +\bR_{BC4A}\right) \sl \mu_C
+ \frac{1}{2} \tr \chi \c \div \sl \mu +\f12 \sn \tr\chi \c \sn \sl \mu \displaybreak[0]\\
& = \f12 \left(L(\ckk \mu-\overline{\ckk \mu})+\tr \chi (\ckk \mu -\overline{\ckk \mu})\right)
+ \chih \c \sn\sl\mu - \left(\tr \chi \, \zb_C - \chi_{AC} \zb_A +\bR_{BC4B}\right) \sl \mu_C
+\f12 \sn \tr\chi \c \sn \sl \mu.
\end{align*}
In view of (\ref{7.18.1}) with $f=\ckk \mu$, we have
$$
L(\ckk \mu-\overline{\ckk \mu})+\tr \chi (\ckk \mu -\overline{\ckk \mu})
= -(\tr\chi-\overline{\tr\chi}) \overline{\ckk \mu} + \left(L \ckk \mu+\tr \chi \, \ckk \mu-\overline{L \ckk \mu +\tr\chi \ckk \mu}\right).
$$
Let $G$ denote the right hand side of the identity in Lemma \ref{mu4.4} and note that $\overline{\tir^{-1} \sn_A \sl\pi_A}=0$. We have
$$
L(\ckk \mu-\overline{\ckk \mu})+\tr \chi (\ckk \mu -\overline{\ckk \mu})
= -(\tr\chi-\overline{\tr\chi}) \overline{\ckk \mu} + G-\overline{G} + \tir^{-1} \div \pi
+ \tir^{-2} (\pi -\overline{\pi}).
$$
Therefore
\begin{align}\label{7.19.2}
\div (L \sl \mu  + \f12 \tr\chi \sl \mu) & =-(\tr\chi-\overline{\tr\chi}) \overline{\ckk \mu} + \f12 \left(G-\overline{G} + \tir^{-1} \div\pi
+ \tir^{-2} (\pi -\overline{\pi})\right) + \chih \c \sn\sl\mu \nn\\
& \quad\, - \left(\tr \chi \, \zb_C - \chi_{AC} \zb_A +\bR_{BC4B}\right) \sl \mu_C
+\f12 \sn \tr\chi \c \sn \sl \mu.
\end{align}
Thus, in view of (\ref{7.19.1}), (\ref{7.19.2}) and the curvature decomposition
in Lemma \ref{decom_lem}, we obtain symbolically that
\begin{align}\label{g1_1}
\begin{split}
\div (L\smu+\f12 \tr\chi\smu)&=G_1+\f12\left(\tir^{-1} \div \pi+\tir^{-2}(\pi-\overline\pi)\right),\\
\curl (L\smu+\f12 \tr\chi \smu)&=G_2,
\end{split}
\end{align}
where
\begin{align*}
G_1&=\f12 \sn\tr\chi \c\smu+\chih \c \sn \smu+\f12( G-\overline{G}) +\smu\c( \sn \pi + \bE) - (\tr\chi-\overline{\tr\chi}) \overline{\ckk \mu},\\
G_2&=\hat\chi \c \sn \smu+\f12\sn \tr\chi\c\smu+\smu\c (\sn \pi + \bE).
\end{align*}
Consequently
\begin{equation}\label{3_21_1}
L\smu+\f12 \tr\chi \smu =\f12\D_1^{-1} (\tir^{-1}\div \pi+\tir^{-2} (\pi -\overline\pi),0)+\D_1^{-1}(G_1, G_2)
\end{equation}
which together with Lemma \ref{tsp2} implies that
\begin{align*}
\smu=v_t^{-\f12}\int_u^t v_\tt^\f12 \D_1^{-1}(G_1, G_2) d \tt
+ v_t^{-\f12}\int_u^t v_\tt^\f12 \D_1^{-1} (\tir^{-1}\div \pi+\tir^{-2}(\pi-\overline{\pi}),0) d\tt.
\end{align*}
By virtue of the fact that $\ti \zeta = \D_1^{-1} (\div \ti \zeta, \curl \ti \zeta)$ and $\sl \mu = \D_1^{-1} (\f12(\ckk \mu - \overline{\ckk \mu}),0)$,
we can obtain
$$
\zeta +\sn \sigma -\sl \mu = \D_1^{-1} \left(\div \ti \zeta -\f12 (\ckk \mu-\overline{\ckk \mu}), \curl \ti\zeta\right).
$$
Therefore $\nabla \sigma = \bA + \mu^\dag + \bA^\dag$ with $\bA=-\zeta$ and
\begin{align*}
\mu^\dag & = v_t^{-\f12}\int_u^t v_\tt^\f12 \D_1^{-1}(G_1, G_2) d \tt, \displaybreak[0]\\
\bA^\dag & =  v_t^{-\f12}\int_u^t v_\tt^\f12 \D_1^{-1} \left(\tir^{-1}\div \pi+\tir^{-2}(\pi-\overline{\pi}),0\right) d\tt
 + \D_1^{-1} \left(\div \ti \zeta -\f12 (\ckk \mu-\overline{\ckk \mu}), \curl \ti\zeta\right).
\end{align*}
It remains only to show that
\begin{equation}\label{smu1}
\|\smu^\dag\|_{L_u^2 L^\infty}\les \la^{-\f12-4\ep_0}, \qquad \|\bA^\dag\|_{L_t^2 L_x^\infty}\les \la^{-\f12-4\ep_0}.
\end{equation}

We first consider $\mu^\dag$. It follows from (\ref{sobinf}) and Lemma \ref{hdgm1} that
\begin{align*}
\|\mu^\dag\|_{L_u^2 L_t^\infty L_x^\infty}&\les\|\D_1^{-1}(G_1, G_2)\|_{L_u^2 L_t^1 L_\omega^\infty}
\les\|\tir (G_1, G_2)\|_{L_u^2 L_t^1 L_\omega^p}.
\end{align*}
for any $p$ satisfying $0<1-\frac{2}{p}<s-2$. Thus, to obtain the estimate on $\mu^\dag$, it suffices to show that
\begin{equation}\label{g12}
\|\tir (G_1, G_2)\|_{L_u^2 L_t^1 L_\omega^p}\les \la^{-\f12-4\ep_0}.
\end{equation}
By the definition of $\ckk \mu$ we have $(\tr\chi-\overline{\tr \chi}) \overline{\ckk \mu}=\bA\c \bA\c \bA+ \tir^{-1} \bA\c \bA$.
Therefore, symbolically we have
\begin{align*}
G_1, G_2&=(\smu, \sn\sigma)\c (\sn\widetilde{\tr\chi}, \sn \pi, \bE)+\chih\c (\sn \ti\zeta, \sn \smu)
+\bA\c(\sn \widetilde{\tr\chi}, \bE, \bd \ti\pi, \tir^{-1} \bA, V\c \widetilde\Ga)
\end{align*}
From Lemma \ref{hdgm1}, (\ref{ckmu2}) and (\ref{sobinf}) it follows
\begin{equation}\label{smu2}
\|\tir\sn\smu, \smu\|_{L_t^2 L_u^2 L_\omega^p} + \|\smu\|_{L_t^2 L_u^2 L_x^\infty}\les \la^{-4\ep_0}.
\end{equation}
Hence, by using Proposition \ref{awave}, Proposition \ref{flux_2} , (\ref{ckmu2}), (\ref{smu2}), (\ref{err_4_4}) and the fact that
$\|\tir \widetilde \Ga\|_{L^\infty L_\omega^p} \les 1$, we can deduce that
\begin{align*}
\|\tir (G_1, G_2)\|_{L_u^2 L_t^1 L_\omega^p}
&\les \|\smu, \sn \sigma\|_{L_u^2 L_t^2 L_x^\infty}\|\tir(\sn \widetilde{\tr\chi}, \sn \pi, \bE)\|_{L_u^\infty L_t^2 L_\omega^p}\\
& \quad \, + \|\chih\|_{L_t^2L_u^\infty L_\omega^\infty} \|\tir (\sn \zeta, \sl \mu)\|_{L_t^2 L_u^2 L_\omega^p} \\
& \quad \, + \|\bA\|_{L_t^2 L_x^\infty} \|r(\sn \widetilde{\tr\chi}, \bE, \bd \ti\pi, \tir^{-1} \bA, V\c\widetilde\Ga)\|_{L_u^2L_t^2 L_\omega^p}\\
& \les \la^{-\f12-4\ep_0}.
\end{align*}

Next we consider $\bA^\dag$. We use $v_t^{\f12} \approx t-u$. By (\ref{hlm}), Proposition \ref{cz.2} and (\ref{pi.2}),
we have with $0<c<\delta_0$ that
\begin{equation*}
\left\|v_t^{-\f12}\int_u^t v_{t'}^\f12 |\D_1^{-1} (\tir^{-1}\div \pi, 0)| d\tt \right\|_{L_t^2 L_x^\infty}
\les \|\ell^c P_\ell \ti\pi\|_{L_t^2 l_\ell^2 L_x^\infty}+\|\ti\pi\|_{L_t^2 L_x^\infty}\les \la^{-\f12-4\ep_0}.
\end{equation*}
By using (\ref{hlm}), (\ref{sobinf}),  Lemma \ref{hdgm1} and Proposition \ref{awave}, we have
\begin{align*}
&\left\|v_t^{-\f12}\int_u^t v_\tt^\f12 |\D_1^{-1} (\tir^{-2} (\pi -\overline{\pi}),0)| d\tt \right\|_{L_t^2 L_x^\infty}
\les \left\| \tir^{-1} \D_1^{-1} (\pi -\overline{\pi},0)\right\|_{L_t^2L_u^\infty L_\omega^\infty} \\
& \qquad\qquad\les \left\|\sn \D_1^{-1} (\pi -\overline{\pi},0)\right\|_{L_t^2L_u^\infty L_\omega^p}
+ \left\| \tir^{-1} \D_1^{-1} (\pi -\overline{\pi},0)\right\|_{L_t^2L_u^\infty L_\omega^p} \\
& \qquad \qquad\les \|\pi-\overline{\pi}\|_{L_t^2 L_u^\infty L_\omega^p} \les \|\pi\|_{L_t^2 L_u^\infty L_\omega^\infty} \les \la^{-\f12-4\ep_0}.
\end{align*}
Finally, we consider $\D_1^{-1} (\div\ti\zeta-\f12 (\ckk \mu -\overline{\ckk \mu}), \curl \ti\zeta)$. By virtue of
(\ref{hodge1}) and (\ref{hodge2}), we have
\begin{align*}
 \|\D_1^{-1} (\div\ti\zeta-\f12 (\ckk \mu -\overline{\ckk \mu}), \curl \ti\zeta)\|_{L_t^2 L_u^\infty L_\omega^\infty}
& \les \|\D_1^{-1} (\bE-\overline{\bE}, \bE -\overline{\bE})\|_{L_t^2 L_u^\infty L_\omega^\infty}\\
& \quad \, + \|\D_1^{-1} (\div \pi_1, \curl \pi_2)\|_{L_t^2 L_u^\infty L_\omega^\infty}.
\end{align*}
For the first term we use (\ref{sobinf}) and Lemma \ref{hdgm1} with $p$ satisfying $0<1-\frac{2}{p}<s-2$,
and for the second term we use Proposition \ref{cz.2}. We then obtain
\begin{align*}
 \|\D_1^{-1} (\div\ti\zeta-\f12 (\ckk \mu -\overline{\ckk \mu}), \curl \ti\zeta)\|_{L_t^2 L_u^\infty L_\omega^\infty}
& \les \|\tir (\bE-\overline{\bE})\|_{L_t^2 L_u^\infty L_\omega^p}+\|\ell^c P_{\ell} \ti \pi\|_{L_t^2 l_\ell^2 L_x^\infty}
+\|\ti \pi \|_{L_t^2 L_x^\infty}\\
& \les \|\tir \bE\|_{L_t^2 L_u^\infty L_\omega^p}+\|\ell^c P_{\ell} \ti \pi\|_{L_t^2 l_\ell^2 L_x^\infty}
+\|\ti \pi \|_{L_t^2 L_x^\infty}\\
&\les \la^{-\f12 -4\ep_0},
\end{align*}
where we used (\ref{err_4_4}) and (\ref{pi.2}) to derive the last inequality.
The proof is therefore complete.
\end{proof}

\section{\bf Bounded conformal energy in curved spacetime}\label{sec_7}

The goal of this section is to prove Theorem \ref{BT} concerning the boundedness of conformal energy for any function
$\psi$ satisfying
 \begin{equation}\label{wh}
\Box_{\bg} \psi=0
\end{equation}
on $I_*=[0, \tau_*]$ with $\tau_*\les \la^{1-8\ep_0}$ and with $\psi[t_0]$ supported in $B_R\subset {\mathcal D}_0^+\cap \Sigma_{t_0}$  with $R< t_0:=1$.
We will carry out the proof in three steps. We first adapt the argument of Morawetz to obtain the integrated energy
estimates by using the vector fields $N$ (see  \cite{Morawetz, Sterb, DaRod2,yang}). We then make conformal change of metric
and  use the multiplier approach with the  vector field of the type $r^p L$ in $\D_0^+$ but away from the time axis to control
the conformal energy in the exterior region and to provide a sequence of  null slices with  energy decay. Then, based on
the argument in \cite{DaRod1}, we obtain energy decay in each spatial-null slice. Finally, we control the conformal energy in the
interior region.

To this end, we consider the momentum tensor $Q_{\a\b}[\psi]$, see (\ref{qmom}). For any vector field $X$ we
introduce the energy current
$$
\P_\a^{(X)}[\psi] = Q_{\a\b}[\psi] X^\b.
$$
By direct calculation and (\ref{wh}) it is easily seen that
\begin{equation}\label{7.8.1}
\bd^\a \P_\a^{(X)}[\psi] = (\Box_{\bg} \psi) X \psi + \f12 Q^{\a\b}[\psi] \pi_{\a\b}^{(X)}= \f12 Q^{\a\b}[\psi] \pi_{\a\b}^{(X)},
\end{equation}
where $\pi_{\a\b}^{(X)}= \bd_\a X_\b+\bd_\b X_\a$ denotes the deformation tensor of $X$. Moreover, for
$\bT=\p_t$ there holds
\begin{equation}\label{7.8.2}
\P_\a^{(\bT)}[\psi] \bT^\a = \frac{1}{2} \left((\bT \psi)^2+ |\nabla \psi|^2\right).
\end{equation}

Recall that $C_{u}$ denotes the level set of the optical function $u$ and $\ti r=\tir(t) := t-u$. Given $t_0\le \tau_1 <\tau \le \tau_*$
and $0\le R'\le t_0$, let
$$
\widetilde \Sigma_{\tau_1, R'}^{\tau}:=\{\ti r\le R', t=\tau_1\}\cup (C_{\tau_1-R'}\cap \{\tau_1\le t\le \tau\}) \quad \mbox{and} \quad
\widetilde \Sigma_{\tau_1, R'} :=\widetilde \Sigma_{\tau_1, R'}^{\tau_*}
$$
and let $\R_{\tau_1, R'}^{\tau}$ denote the region enclosed by $\widetilde\Sigma_{\tau_1, R'}^{\tau}$ and $\{t=\tau\}$.
For any submanifold of codimension 1, we will use $\bn$ to denote its normal vector field in $({\mathbb R}^{3+1}, \bg)$.
It is easily seen that $\bn = \bT$ on each $\{t=\tau\}$ and $\bn = L$ on each $C_u$.

\begin{figure}[ht]
\centering
\includegraphics[width = 0.45\textwidth]{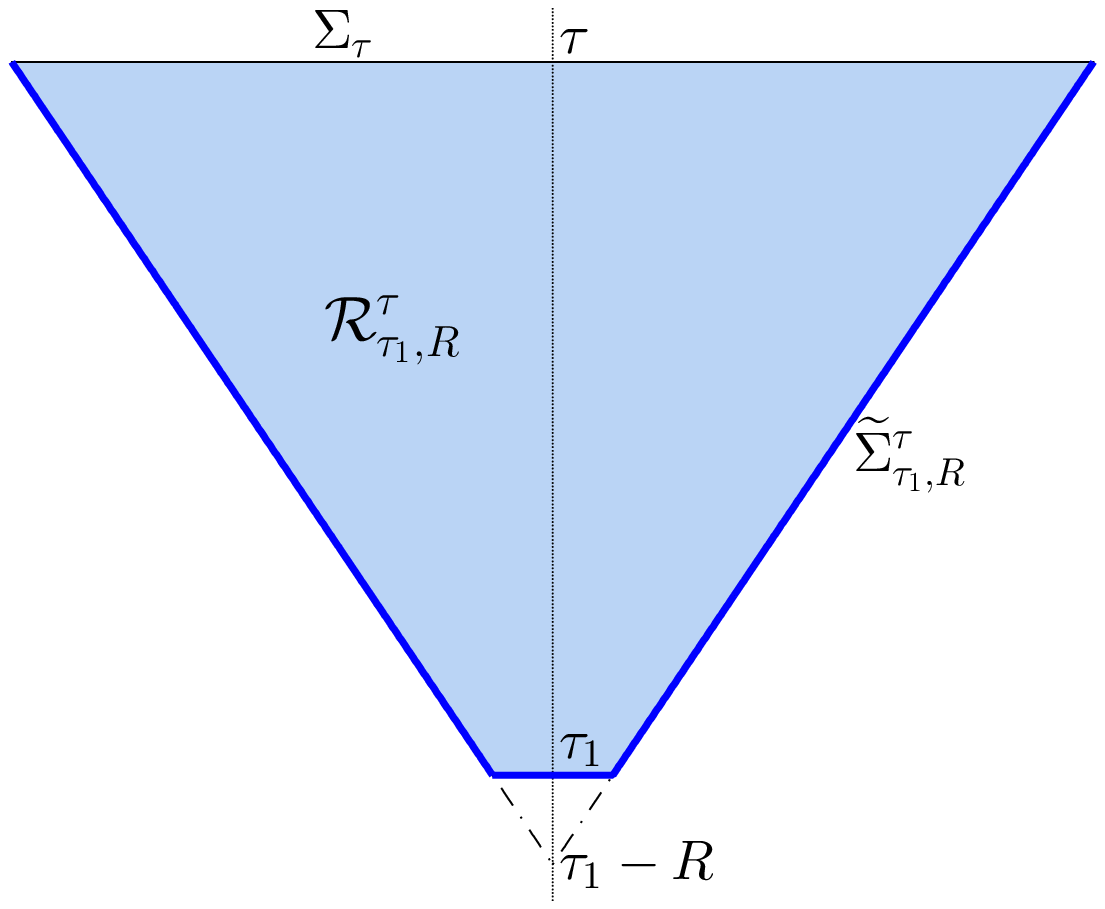}
  \includegraphics[width = 0.45\textwidth]{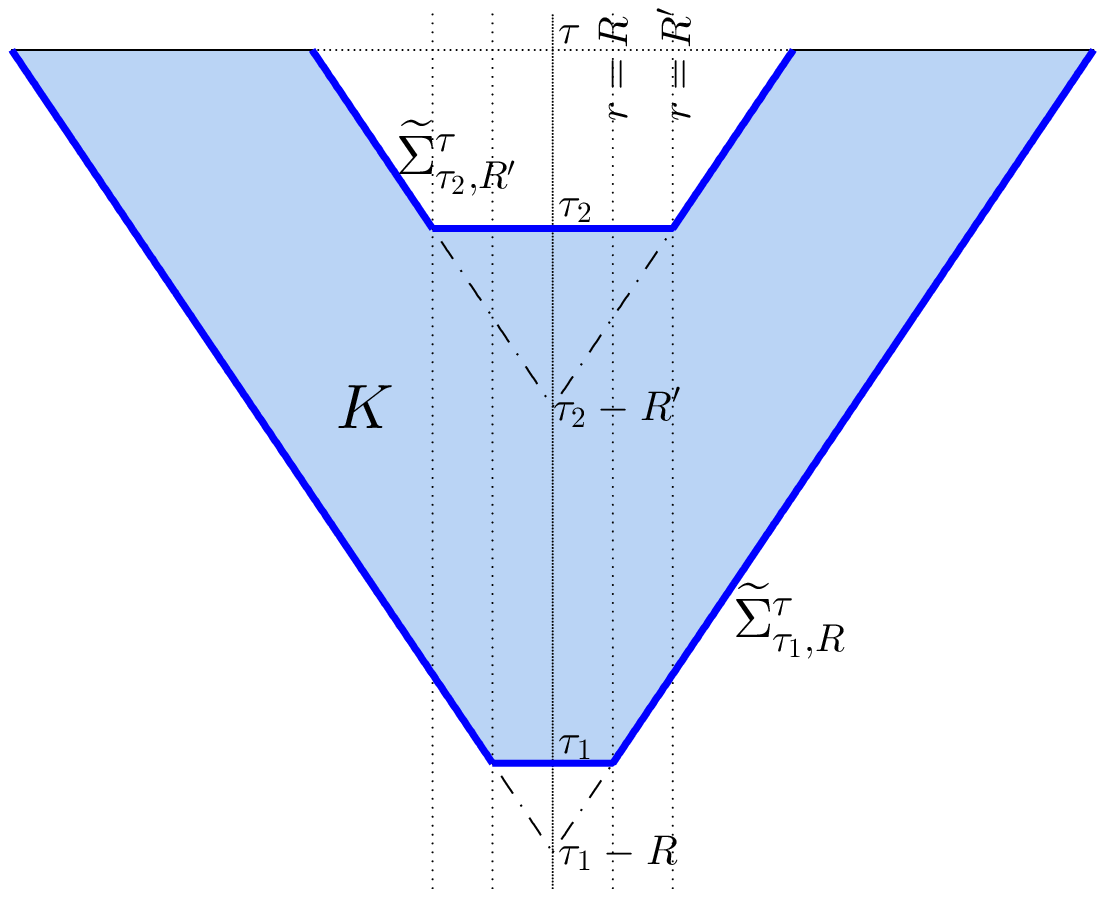}
  \vskip -0.4cm
  \caption{Illustration of $\widetilde \Sigma_{\tau_1, R}^\tau$, $\R_{\tau_1, R}^\tau$, and $K$ used in the proof of Lemma \ref{monoeng}}\label{fig1}
\end{figure}

We first give the following standard energy estimates.

\begin{lemma}\label{monoeng}
Let $\psi$ be any solution of $\Box_\bg \psi=0$ in $I_*=[t_0, \tau_*]$. Given $t_0\le \tau_1<\tau_2<\tau \le \tau_*$ and
$0\le R'\le \tau_2-\tau_1+R$. Let $\Sigma_\tau$ be the part of spacelike slice $\{t=\tau\}$ contained in the causal future of ${\widetilde\Sigma}_{\tau_1, R}$.
Then there hold
\begin{equation}\label{std}
\int_{{\Sigma}_\tau}\P_\a^{(\bT)}[\psi] \bn^\a\les \int_{{\widetilde\Sigma}_{\tau_1,R}^\tau}\P_\a^{(\bT)}[\psi] \bn^\a
\end{equation}
and
\begin{equation}\label{dec_1}
\int_{\{t=\tau, \tau_1-R\le u\le \tau_2-R'\}} \P_\a^{(\bT)}[\psi] \bn^\a+\int_{\widetilde{\Sigma}_{\tau_2, R'}^{\tau} }\P_\a^{(\bT)}[\psi] \bn^\a
\le 2\int_{\widetilde{\Sigma}_{\tau_1, R}^{\tau}} \P_\a^{(\bT)}[\psi] \bn^\a.
\end{equation}
\end{lemma}

\begin{proof}
By integrating (\ref{7.8.1}) with $X=\bT$ over the region $\R_{\tau_1,R}^\tau$ enclosed by $\widetilde\Sigma_{\tau_1,R}^\tau$ and $\{t=\tau\}$ and using
the divergence theorem it follows that
\begin{equation}\label{eng1}
\int_{{\Sigma}_\tau}\P_\a^{(\bT)}[\psi] \bn^\a=\int_{{\widetilde\Sigma}_{\tau_1, R}^\tau} \P_\a^{(\bT)}[\psi] \bn^\a
-\f12 \int_{\R_{\tau_1, R}^\tau} \pi^{(\bT)}_{\a\b}\c Q^{\a\b},
\quad \forall \tau>\tau_1.
\end{equation}
In view of (\ref{7.8.2}) and the fact $\|\pi^{(\bT)}\|_{L_t^1 L_x^\infty}\les \la^{-8\ep_0}$ we have
\begin{equation}\label{eng2}
\int_{\R_{\tau_1, R}^\tau} |\pi_{\a\b}^{(\bT)} \c Q^{\a\b}|
\les \la^{-8\ep_0}  \sup_{\tau_1 \le t\le \tau} \int_{\Sigma_t}\P_\a^{(\bT)}[\psi] \bn^\a.
\end{equation}
The combination of (\ref{eng1}) and (\ref{eng2}) then implies that
\begin{equation*}
\sup_{\tau_1\le t\le \tau}\int_{{\Sigma}_t}\P_\a^{(\bT)}[\psi] \bn^\a\les \int_{{\widetilde\Sigma}_{\tau_1,R}^\tau}\P_\a^{(\bT)}[\psi] \bn^\a
\end{equation*}
which shows (\ref{std}) and the estimate
\begin{equation}\label{deferror}
\int_{\R_{\tau_1,R}^\tau} |\pi^{(\bT)}_{\a\b}  \c Q^{\a\b}|\les  \la^{-8\ep_0} \int_{{\widetilde\Sigma}_{\tau_1,R}^\tau} \P_\a^{(\bT)}[\psi] \bn^\a.
\end{equation}

To prove (\ref{dec_1}), let $K$ be the region enclosed by ${\widetilde{\Sigma}_{\tau_2, R'}}^{\tau},
\widetilde{\Sigma}_{\tau_1,R}^{\tau}$ and $\{t=\tau, \tau_1-R\le u\le \tau_2-R'\}$, see Figure \ref{fig1}.
Note that $K\subset \R_{\tau_1,R}^{\tau}$. Integrating (\ref{7.8.1}) over $K$ gives
\begin{equation*}
\int_{\{t=\tau, \tau_1-R\le u\le \tau_2-R'\}} \P_\a^{(\bT)}[\psi] \bn^\a+\int_{\widetilde{\Sigma}_{\tau_2, R'}^{\tau}} \P_\a^{(\bT)}[\psi] \bn^\a
=\int_{\widetilde{\Sigma}_{\tau_1,R}^{\tau}} \P_\a^{(\bT)}[\psi] \bn^\a-\f12 \int_{K} \pi^{(\bT)}_{\a\b}\c Q^{\a\b}.
\end{equation*}
This together with (\ref{deferror}) then implies (\ref{dec_1}).
\end{proof}

We will employ the following result.

\begin{lemma}\label{lem7.15.1}
For $\hat u \ge 0$ let $\itt_{t,\hat u}=\Sigma_t\cap \{u\ge \hat u\}$. Then for any scalar function $\psi$ there holds
\begin{equation}\label{lot4}
\|\psi\|_{L_u^2 L_\omega^2(\itt_{t,\hat u})}\les \|\tir N(\psi)\|_{L_u^2 L_\omega^2(\itt_{t,\hat u})}+\|\tir^{-\f12}\psi\|_{L^2(S_{t,\hat u})}
\end{equation}
\end{lemma}

\begin{proof}
Let $m \in {\mathbb N}$. Using $N u = -\bb^{-1}$ we have
\begin{equation*}
\bb N(\tir^m\psi^2 )=2 \bb\tir^m\psi N \psi  + m\tir^{m-1}\psi^2
\end{equation*}
Integrating this identity over $\itt_{t, \hat u}$ gives
\begin{equation}\label{6.27.2}
\int_{\itt_{t, \hat u}} m\tir^{m-1}\psi^2 d u d\omega + \int_{\itt_{t, \hat u}} 2 \tir^{m} \psi N(\psi) \bb d u d\omega
=\int_{S_{t, \hat u}} \tir^m \psi^2 d\omega.
\end{equation}
Setting $m=1$ and using the Cauchy-Schwarz inequality, we can obtain the inequality.
\end{proof}

\subsection{Step 1: Integrated energy estimates}

Consider the solution of $\Box_\bg \psi=0$ within $\D_0^+$. We will first give, in Proposition \ref{int_est3}, an integrated energy
estimate in the region $[t_0,\tau_*]\times B_R$. Similar to (\ref{radial}), with radial foliation formed by $\cup_{0\le u\le t}S_{t,u}$,
the metric on each $\Sigma_t \cap \D_0^+$ can be written as
\begin{equation*}
ds^2=\bb^{2} du^2+\ga_{AB} d\omega^A d\omega^B
\end{equation*}
with area element $\bb d u d\mu_\ga$, where $\ga$ is the induced metric of $g$ on $S_{t, u}$.

\begin{proposition}\label{int_est3}
Let $\psi$ be any solution of $\Box_\bg \psi=0$ on $[t_0, \tau_*]$ with $\tau_*\les \la^{1-8\ep_0}$ and with $\psi[t_0]$
supported within $B_R$. Let $\tau_2>\tau_1\ge t_0$ and $R\le R'<2R$. Then there holds
\begin{equation*}
\int_{\tau_1}^{\tau_2}\int_{\ti r\le R'} \left[({\emph \bT} \psi)^2+|N \psi|^2+|\sn \psi|^2 +\tir^{-1} \psi^2 \right] dx d\tau
\les \int_{\Sigma_{t_0}}\P^{(\emph{\bT})}_\a [\psi] \bn^\a.
\end{equation*}
\end{proposition}

In view of (\ref{7.8.2}) and the fact that $\psi =0$ in $\{t_0\le t\le \tau_*\}\setminus \D_0^+$, Proposition \ref{int_est3}
follows immediately from the following result.

\begin{lemma}\label{lem:mint_1}
Under the same conditions on Proposition \ref{int_est3}, with $\a = 2\ep_0$ there holds
\begin{align*}
& \int_{\R_{\tau_1,R}^{\tau_2}} \left[\frac{(N \psi)^2+(\bT \psi)^2+|\sn \psi|^2}{(1+\tir)^{\a+1}}+\frac{\psi^2}{\tir(1+\tir)^{\a+2}}\right]\nn\\
& \quad \les \int_{\widetilde{\Sigma}_{\tau_1, R}^{\tau_2}}\P^{(\emph{\bT})}_\a[\psi] \bn^\a
+\|\tir^\f12 \psi\|^2_{L_\omega^2(S_{\tau_2, \tau_1-R})}+\|\tir^\f12 \psi\|^2_{L_\omega^2(S_{\tau_1, \tau_1-R})}.
\end{align*}
\end{lemma}

\begin{proof}
We consider the modified energy current
\begin{equation}\label{j1}
\widetilde \P_\mu^{(X)}[\psi]=\P_\mu^{(X)}[\psi]-\f12 \p_\mu \Theta\c \psi^2+\f12 \Theta\c \p_\mu (\psi^2),
\end{equation}
where the vector field $X$ and the scalar function $\Theta$ will be chosen later. Direct calculation shows
\begin{align*}
\bd^\mu {\widetilde \P}_\mu^{(X)}[\psi]
=\Box_\bg \psi (X \psi+\Theta \psi)+\Theta \bd^\mu \psi \bd_\mu \psi-\f12 \Box_\bg \Theta \c \psi^2 + \f12 \pi^{(X)}_{\mu\nu} Q^{\mu\nu}[\psi].
\end{align*}
Integrating this identity over $\R_{\tau_1,R}^{\tau_2}$ and using $\Box_\bg \psi =0$, we obtain
\begin{equation}\label{eng5}
\int_{\widetilde{\Sigma}_{\tau_1,R}^{\tau_2}} \widetilde \P_\mu^{(X)}[\psi] \bn^\mu-\int_{\Sigma_{\tau_2}} \widetilde \P_\mu^{(X)}[\psi] \bT^\mu
=\int_{\R_{\tau_1,R}^{\tau_2}} \left(\Theta \bd^\mu \psi \bd_\mu \psi-\f12 \Box_\bg \Theta\c \psi^2 + \f12 \pi^{(X)}_{\mu\nu} Q^{\mu\nu}\right).
\end{equation}
We take $X=f N$ with $f$ being a function of $\tir $ to be determined. Notice that $\sn f=0$ and
\begin{equation*}
\pi^{(X)}_{NN}=2N f, \quad \pi^{(X)}_{AN} = f \pi^{(N)}_{AN}, \quad \pi^{(X)}_{AB}= 2 f \theta_{AB},\quad \pi^{(X)}_{00}=0, \quad \pi^{(X)}_{0N}=0,
\end{equation*}
where $\pi^{(N)}_{AN}=k_{AN}- \zeta_A$. We have
\begin{align*}
\f12 \pi_{\mu\nu}^{(X)} Q^{\mu\nu} &= \frac{1}{2} f \tr \theta \left((\bT \psi)^2 -(N\psi)^2\right) + \frac{1}{2} Nf \left((N\psi)^2 +(\bT \psi)^2-|\sn \psi|^2\right)\\
& \quad \, +f \left(\hat \theta_{AB} Q^{AB} + \pi_{AN}^{(N)} Q^{AN}\right).
\end{align*}
We also have $\bd^\mu \psi \bd_\mu \psi = -(\bT \psi)^2 +(N\psi)^2 +|\sn \psi|^2$. Plugging these two identities into the right hand side
of (\ref{eng5}) gives
\begin{align}\label{i2}
\begin{split}
&\int_{\widetilde{\Sigma}_{\tau_1,R}^{\tau_2}} \widetilde \P_\mu^{(X)}[\psi] \bn^\mu-\int_{\Sigma_{\tau_2}} \widetilde \P_\mu^{(X)}[\psi] \bT^\mu\\
&=\int_{\R_{\tau_1,R}^{\tau_2}} \left[\f12 (N \psi)^2(N f+2\Theta-\tr\theta f)+\f12(\bT \psi)^2(N f-2\Theta+\tr\theta f)\right.\\
&  \quad \, \left. +|\sn \psi|^2(\Theta-\f12 Nf) -\f12 \Box_\bg \Theta\c \psi^2 +f \left(\hat \theta_{AB} Q^{AB} + \pi_{AN}^{(N)} Q^{AN}\right)\right].
\end{split}
\end{align}
Now we choose $f$ and $\Theta$ as
 \begin{equation}
 f=\b-\frac{\b}{(1+\ti r)^\a}, \quad \Theta=\ti r^{-1} f, \quad \b\a=2, \label{j2}
 \end{equation}
where $0<\a<1$ is a constant to be chosen later. Straightforward calculation shows that
\begin{equation}\label{7.8.5}
-\left(\Theta''+\frac{2}{\tir} \Theta'\right)=\frac{2(\a+1)}{\ti r(1+\ti r)^{\a+2}}, \qquad 0\le -\Theta'\le \frac{\b}{(1+\ti r)^2},
\end{equation}
where prime denotes the derivative with respect to $\tir$. By using $N=\bb^{-1} \p_{\tir}$ on $\Sigma_t$ and $|\bb-1|<\frac{1}{4}$,
see (\ref{bb_4}), we also have
\begin{equation}\label{theta_8}
\bb \Theta-\frac{1}{2} \bb N f=\frac{\bb \beta[(1+\tir)^\a-1]}{\tir(1+\tir)^\a}-\frac{1}{(1+\tir)^{\a+1}}
\ge \frac{2\bb-1}{(1+\tir)^{\a+1}} \ge \frac{\bb}{2(1+\tir)^{\a+1}}.
\end{equation}

To deal with the term $\Box_{\bg} \Theta$, recall the decomposition formula of $\Box_\bg \Theta$ under the null frame
\begin{equation*}
\Box_\bg \Theta=-L \Lb \Theta +2 \zb^A \sn_A \Theta+k_{NN} \Lb \Theta+\sD \Theta-\f12 \tr\chi \Lb \Theta-\f12 \tr\chib L\Theta.
\end{equation*}
Notice that $\sn \Theta=0$ and
\begin{equation}\label{lr}
\Lb (u)=2\bb^{-1},\quad \Lb (t)=1,\quad L(t)=1,\quad L(u)=0,
\end{equation}
we can obtain
\begin{align*}
\Box_\bg \Theta&=-L \Lb \Theta-\f12 \tr\chi \Lb \Theta-\f12 \tr\chib L \Theta+k_{NN} \Lb\Theta\\
&=\left(\Theta''+\frac{2}{\tir} \Theta'\right)(2\bb^{-1}-1)+\f12 \left(\tr\chi-\frac{2}{\tir}\right)(2\bb^{-1}-1)\Theta'\\
&\quad \, -\f12 \left(\tr\chib+\frac{2}{\tir}\right)\Theta'+\frac{\Theta'}{\tir}(2-2\bb^{-1})-k_{NN}\Theta'.
\end{align*}
We plug this identity into (\ref{i2}). We may use (\ref{7.8.5}), the fact $2 \bb^{-1}-1 \ge \frac{3}{5}$ which follows from
$|\bb-1|<\frac{1}{4}$, (\ref{theta_8}), and
$$
N f = \frac{2 \bb^{-1}}{(1+ \tir)^{\a+1}}, \quad 2 \Theta - f \tr \theta = \left(\frac{2}{\tir}-\tr \theta\right) f
$$
to obtain
\begin{align}\label{l2}
\begin{split}
\J&:=\int_{\R_{\tau_1,R}^{\tau_2}} \left[\frac{(N \psi)^2+(\bT \psi)^2+|\sn \psi|^2}{(1+\tir)^{\a+1}}+\frac{\psi^2}{\tir(1+\tir)^{\a+2}}\right]\\
&\les \left|\int_{\R_{\tau_1,R}^{\tau_2}} f \left(\hat \theta_{AB} Q^{AB} + \pi_{AN}^{(N)} Q^{AN}\right)\right|
+\left|\int_{\R_{\tau_1,R}^{\tau_2}} \left(\tr\theta-\frac{2}{\tir}\right) f ((N\psi)^2-(\bT \psi)^2)\right|\\
& + \int_{\R_{\tau_1,R}^{\tau_2}} \psi^2 |\bA| |\Theta'|
+\left|\int_{\Sigma_{\tau_2}} {\ti \P}^{(X)}_\mu [\psi] \bn^\mu -\int_{\widetilde{\Sigma}_{\tau_1,R}^{\tau_2}} {\ti \P}^{(X)}_\mu [\psi] \bn^\mu\right|,
\end{split}
\end{align}
where $\bA$ denotes any term among $\frac{\bb^{-1}-1}{\ti r}, \tr\chi-\frac{2}{\tir}, \tr\chib+\frac{2}{\tir}, k_{NN}$.
For the first term on the right side of (\ref{l2}), we may use the estimates
$\|\hat \theta, \zeta, k\|_{L_t^2 L_x^\infty}\les \la^{-\f12-4\ep_0}$ from Proposition \ref{awave}
and the standard energy estimate (\ref{std}) to conclude
\begin{align*}
\left|\int_{\R_{\tau_1,R}^{\tau_2}} f \left(\hat \theta_{AB} Q^{AB} + \pi_{AN}^{(N)} Q^{AN}\right)\right|
\les \tau_*^{1/2} \|\hat \theta, \zeta, k\|_{L_t^2 L_x^\infty} \|\bd \psi\|_{L_t^\infty L_x^2}^2
\les \la^{-8\ep_0}  \int_{{\widetilde\Sigma}_{\tau_1,R}^{\tau_2}}\P_\mu^{(\bT)}[\psi] \bn^\mu.
\end{align*}
For the second term in (\ref{l2}), we may use $\theta =\chi +k$ and $\|\tr\chi-\frac{2}{\tir}, k\|_{L_t^2 L_x^\infty}\les \la^{-\f12-4\ep_0}$
from Proposition \ref{awave} to obtain the same estimate. To estimate the third term in (\ref{l2}),  we use
\begin{equation*}
\|\psi\|_{L^4(S_{t, u})}^2\les \|\sn \psi\|_{L^2(S_{t,u})}\|\psi\|_{L^2(S_{t,u})}+\|\tir^{-\f12}\psi\|_{L^2(S_{t,u})}^2,
\end{equation*}
which is (\ref{sob}) with $q=4$, to obtain the estimate on $\itt_{t, u}$
\begin{align*}
\left\|\frac{|\psi|^2 \tir} {(1+\tir)^{\a+\frac{3}{2}} }\right\|_{L_u^1 L_\omega^2}
&\les (1+ t)^\f12 \left(\left\|\frac{\sn \psi}{(1+\tir)^{\frac{\a+1}{2}}}\right\|_{L_u^2 L_x^2}
\left\|\frac{\psi}{(1+\tir)^{\frac{\a}{2}+1} \tir^{\f12}}\right\|_{L_u^2 L_x^2}\right.\\
&\left.+\left\|\frac{\psi}{(1+\tir)^{\frac{\a}{2}+1}\tir^{\f12}}\right\|_{L_u^2 L_x^2}^2\right).
\end{align*}
Consequently, we may use (\ref{7.8.5}) and $\|\tir^{\f12}  \bA\|_{L^\infty L_\omega^2}\les \la^{-\f12}$ in Proposition \ref{awave}
to  obtain
\begin{align*}
\int_{\R_{\tau_1,R}^{\tau_2}} \psi^2 |\bA| |\Theta'|  & \les \int_{\Sigma\times I} \psi^2 |\bA| (1+\tir)^{-2}
\les \|\tir (1+\tir)^{\a-\f12} \bA\|_{L^\infty L_\omega^2} \left\|\frac{\psi^2 \tir}{(1+\tir)^{\a+\frac{3}{2}}} \right\|_{L_t^1 L_u^1 L_\omega^2}\nn\\
&\les (1+\tau_*)^{\f12+\a} \|\tir^{\f12} \bA\|_{L^\infty L_\omega^2}\J\les\la^{-2\ep_0} \J,
\end{align*}
where we choose $\a=2\ep_0$ and use $\tau_*\les \la^{1-8\ep_0}$.

Combining the above estimates with (\ref{l2}) we obtain
\begin{align*}
\J &\les \la^{-2\ep_0} \J +\la^{-8\ep_0}  \int_{{\widetilde\Sigma}_{\tau_1,R}^{\tau_2}}\P_\mu^{(\bT)}[\psi] n^\mu
+\left|\int_{\Sigma_{\tau_2}} {\widetilde \P}^{(X)}_\mu [\psi] \bn^\mu -\int_{\widetilde{\Sigma}_{\tau_1,R}^{\tau_2}} {\widetilde \P}^{(X)}_\mu [\psi] \bn^\mu\right|.
\end{align*}
Therefore, we can conclude the desired estimate by using the following result.
\end{proof}

\begin{lemma}\label{bdry1}
For $\tau_2> \tau_1\ge t_0$  there hold
\begin{align}
\int_{\Sigma_{\tau_2}}{\widetilde \P}^{(X)}_{\mu}[\psi] \bn^\mu
&\les \int_{\widetilde{\Sigma}_{\tau_1,R}^{\tau_2}} \P^{(\emph{\bT})}_\mu[\psi] \bn^\mu
+\|\tir^{\f12}\psi(t)\|_{L_\omega^2(S_{\tau_2, \tau_1-R})}^2,\label{7.7.1}\\
\int_{\widetilde{\Sigma}_{\tau_1,R}^{\tau_2}}{\widetilde \P}^{(X)}_\mu[\psi] \bn^\mu
&\les \int_{\widetilde{\Sigma}_{\tau_1,R}^{\tau_2}}\P_\mu^{(\emph{\bT})}[\psi] \bn^\mu
+\|\tir^\f12 \psi\|^2_{L_\omega^2(S_{\tau_2, \tau_1-R})}+\|\tir^\f12 \psi\|^2_{L_\omega^2(S_{\tau_1, \tau_1-R})}.\label{7.7.2}
\end{align}
\end{lemma}

\begin{proof}
Recall that $X= fN$ and $\bn =\bT$ on $\Sigma_{\tau_2}$, we may use the definition of $\widetilde\P_\mu^{(X)}[\psi]$ in (\ref{j1}),
the expression of $f$ and $\Theta$ given in (\ref{j2}), the Cauchy-Schwarz inequality and (\ref{7.8.2}) to obtain
\begin{equation}\label{bdr1}
\int_{\Sigma_{\tau_2}}{\widetilde \P}^{(X)}_\mu[\psi] \bn^\mu\les  \int_{\Sigma_{\tau_2}}\P^{(\bT)}_\mu[\psi] \bn^\mu+\|\psi\|_{L_u^2 L_\omega^2(\itt_{t, \tau_1-R})}^2.
\end{equation}
In view of Lemma \ref{lem7.15.1} and (\ref{std}), we obtain (\ref{7.7.1}).

We next prove (\ref{7.7.2}). On the spatial part of $\widetilde{\Sigma}_{\tau_1,R}^{\tau_2}$, we repeat the proof for (\ref{7.7.1}) to obtain
\begin{align}
\int_{\Sigma_{\tau_1}\cap \{\tau_1-R\le  u\le \tau_1\}}{\widetilde \P}^{(X)}_{\mu}[\psi] \bn^\mu
\les \int_{\widetilde{\Sigma}_{\tau_1,R}^{\tau_2}} \P^{(\bT)}_\mu[\psi] \bn^\mu +\|\tir^{\f12}\psi\|_{L_\omega^2(S_{\tau_1, \tau_1-R})}^2\label{7.7.3}
\end{align}
On the null part of $\widetilde{\Sigma}_{\tau_1,R}^{\tau_2}$ on which $\bn= L$, we have
\begin{align*}
\ti\P_\mu^{(X)}[\psi] L^\mu &=\frac{1}{2} f\left(|L\psi|^2-|\sn \psi|^2\right)-\frac{1}{2} \psi^2 L \Theta +\Theta \psi L\psi,\\
\P_\mu^{(\bT)}[\psi] L^\mu & = \frac{1}{2} \left(|L\psi|^2 + |\sn \psi|^2\right).
\end{align*}
In view of the definition of $f$ and $\Theta$, we derive that
$
|\ti\P_\mu^{(X)}[\psi] L^\mu| \les \P_\mu^{(T)}[\psi] L^\mu + |\tir^{-1}\psi|^2.
$
Hence, by combining this with (\ref{7.7.3}) we obtain
\begin{align}
\int_{\widetilde{\Sigma}_{\tau_1,R}^{\tau_2}} \widetilde \P_\mu^{(X)}[\psi] \bn^\mu
& \les \int_{\widetilde{\Sigma}_{\tau_1,R}^{\tau_2}} \P_\mu^{(\bT)}[\psi] \bn^\mu + \|\tir^{\f12}\psi(t)\|_{L_\omega^2(S_{\tau_1, \tau_1-R})}^2 \nn\\
&\quad \,  +\int_{{\mathbb S}^2} \int_{\tau_1}^{\tau_2}\psi^2(t, \tau_1-R, \omega) dt d\omega.   \label{7.7.6}
\end{align}
To estimate the last term on the right hand side, we may integrate the identity
$
\psi^2 = L(\tir \psi^2) - 2 \tir \psi L\psi
$
over the null cone $C_{\tau_1-R}$ to obtain
\begin{align*}
\int_{{\mathbb S}^2} \int_{\tau_1}^{\tau_2}\psi^2(t, \tau_1-R, \omega) dt d\omega
& = \int_{S_{\tau_2, \tau_1-R}} \tir \psi^2  d\omega -\int_{S_{\tau_1, \tau_1-R}} \tir \psi^2  d\omega
+\int_{{\mathbb S}^2} \int_{\tau_1}^{\tau_2}2\tir \psi L\psi dt d\omega
\end{align*}
which, by the Cauchy-Schwarz inequality, implies
\begin{align*}
\int_{{\mathbb S}^2} \int_{\tau_1}^{\tau_2}\psi^2(t, \tau_1-R, \omega) dt d\omega
& \les \|\tir^{\f12} \psi\|_{L_\omega^2(S_{\tau_2, \tau_1-R})}^2 + \|\tir^{\f12} \psi\|_{L_\omega^2(S_{\tau_1, \tau_1-R})}^2
+ \int_{\widetilde{\Sigma}_{\tau_1,R}^{\tau_2}} |L\psi|^2.
\end{align*}
Combining this with (\ref{7.7.6}), we obtain (\ref{7.7.2}).
\end{proof}


\subsection{Control of lower order terms }

Let $\psi$ be any solution of (\ref{wh}) with $\psi[t_0]$ supported on $B_R\subset \D_0^+\cap \Sigma_{t_0}$.
We introduce $\ti\psi = \Omega \psi$ with $\Omega = e^{-\sigma}$, where $\sigma$ is defined by (\ref{c1}).
On each wave front $S_{t,u}$, we use $v_t=\sqrt{|\ga|}$ to denote its area element. 
Let $m\in {\mathbb N}$ and $R\le R'\le 2 R$. We introduce the conformal flux
\begin{align}
CF_m[\psi]_{R'}(u,\tau)=\int_{u+R'}^\tau\int_{S_{t, u}}\tir^m |L(v_{t}^{\f12}\psi)|^2 d\omega dt.\label{3.22.2}
\end{align}
For a region $U\subset \D_0^+$ with the property that there exist two positive functions $u_0(t)$
and $u_1(t)$ on $[t_0, \tau_*]$ such that
\begin{equation}\label{8.8.1}
U_t:=U\cap \Sigma_t = \bigcup_{u_0(t) \le u\le u_1(t)} S_{t, u},
\end{equation}
we can introduce the energy
\begin{align}\label{3.22.1}
\E_m[\psi]_{U, R'}(t)&=\int_{U_t\cap \{\tir\ge R'\}} \tir^m \left( |L (v_t^{\f12} \psi)|^2+|\sn \ti\psi|^2 \Omega^{-2} v_t\right) d\omega du.
\end{align}
We will derive the bound on the conformal energy by a multiplier approach in which the conformal flux (\ref{3.22.2})
and the energy (\ref{3.22.1}) arise naturally, see Section \ref{mul_1}.

By incorporating the lower order term into (\ref{3.22.1}), we may introduce the modified energy
\begin{equation*}
\widetilde \E_m[\psi]_{U,R}(t)=\|r^{\frac{m}{2}}\psi(t)\|_{L_u^2 L_\omega^2(U_t\cap \{\tir \ge R\})}^2+\E_m[\psi]_{U,R}(t).
\end{equation*}
We will prove a comparison result in Lemma \ref{com5} which implies
\begin{equation*}
\C[\psi]^\be(t)\les \widetilde \E_2[\psi]_{\D_0^+,R}(t), \qquad \forall t\in [t_0, \tau_*].
\end{equation*}
We will give the control on $\widetilde \E_2[\psi]_{\D_0^+,R}(t)$ by a multiplier approach.
While to control $\C[\psi]^\bi(t)$, we adapt the method devised in \cite{DaRod1}. However, this does not directly
give the decay estimates for lower order terms, such as $\|\psi(t)\|_{L_u^2 L_\omega^2(U_t\cap \{\tir \ge R\})}$.
In curved spacetime, due to various error terms related to Ricci coefficients (\ref{ricc_def}),  the derivation of the decay
energy estimates on $\bp \psi$ has to be coupled with controlling weighted $L^2$ estimates on $\psi$ itself.

Given $\tau>\tau_1\ge t_0> R$, we have defined ${\widetilde\Sigma}^{\tau}_{\tau_1,R}$ and $\R_{\tau_1, R}^{\tau}$.
We set
\begin{equation}\label{7.14.5}
T_{\tau_1,R}^{\tau} = \R_{\tau_1, R}^{\tau} \cap \{\tir \ge R\}.
\end{equation}
We will first give the following result on the lower order terms.

\begin{figure}[ht]
\centering
\includegraphics[width = 0.46\textwidth]{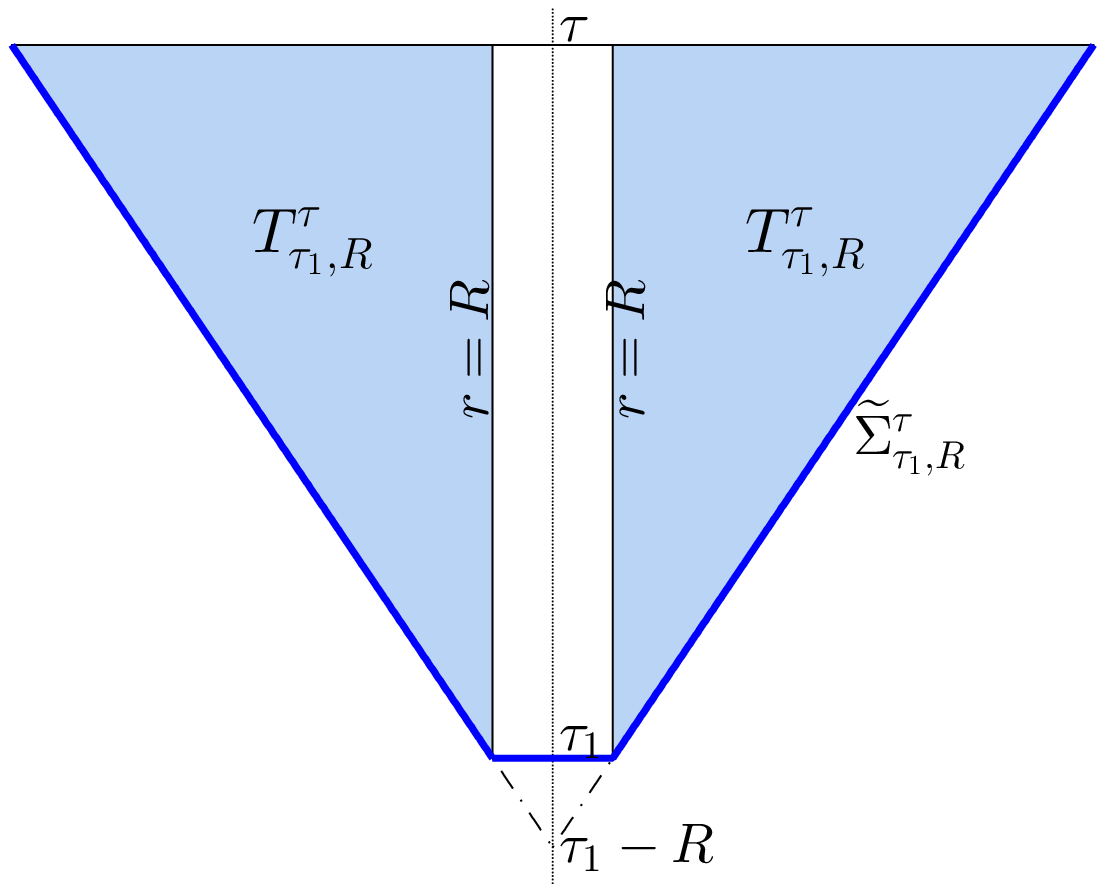}
  \includegraphics[width = 0.46\textwidth]{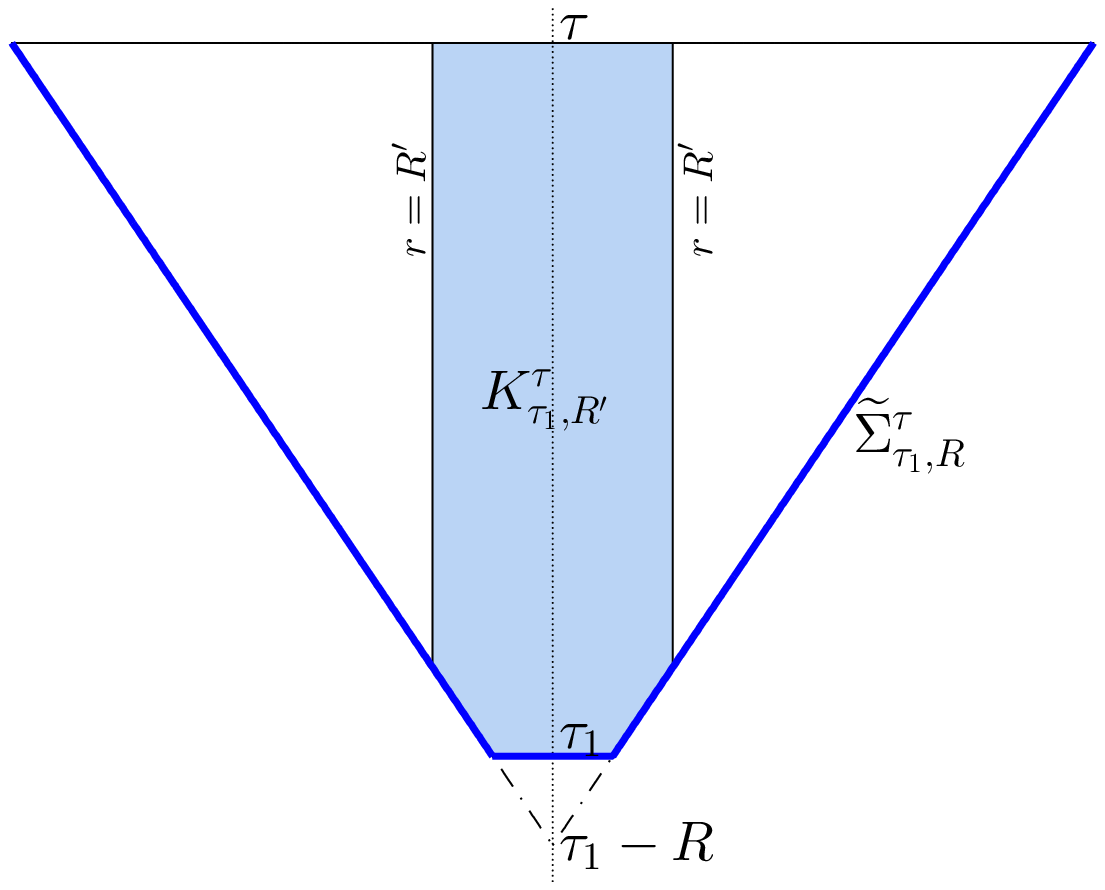}
  \vskip -0.4cm
  \caption{Illustration of $T_{\tau_1, R}^\tau$ and $K_{\tau_1, R'}^\tau$ used in Lemma \ref{lotlm} and Proposition \ref{6.27.3}}
\end{figure}

\begin{lemma}\label{lotlm}
Given $\tau >\tau_1\ge t_0$. There hold
\begin{align}\label{lot1}
\begin{split}
\sup_{\tau_1\le t\le \tau}\|\tir^\frac{m}{2}\psi\|_{L_u^2 L_\omega^2(T_{\tau_1, R}^{\tau}\cap \Sigma_t)}
& \les \Phi_m(\tau)\|\tir^{\frac{m-1}{2}}L(v_t^{\f12}\psi)\|_{L_u^2 L_t^2 L_\omega^2(T_{\tau_1, R}^{\tau})}\\
& +R^{m/2} \left(\int_{\tau_1}^{\tau} \int_{S_{t, t-R}} \psi^2 d\omega dt\right)^{\f12},
\end{split}
\end{align}
where $m= 1, 2$ with $\Phi_1(\tau) \equiv 1$ and $\Phi_2(\tau) = (\ln (2+\tau))^\f12$;
\begin{align}
& \sup_{\tau_1\le t\le \tau}\int_{S_{t,u}} \tir^2 \psi^2 d\omega\les\int_{S_{\tau_1, u}} \tir^2 \psi^2 d\omega
  +\int_{\tau_1}^{\tau}\int_{S_{t,u}} \tir \left(|L(v_t^{\f12}\psi)|^2+\psi^2\right) d\omega d t, \label{lot6} \displaybreak[0]\\
&\sup_{\tau_1\le t\le \tau}\int_{S_{t,u}} \tir \psi^2 d\omega \les\int_{S_{\tau_1, u}} \tir \psi^2 d\omega
  + \int_{\tau_1}^{\tau}\int_{S_{t,u}} |L(v_t^{\f12} \psi)|^2 d\omega dt, \label{lot10} \displaybreak[0]\\
&\left(\int_{\tau_1}^{\tau}\int_{S_{t,\tau_1-R}} \tir \psi^2 d\omega dt\right)^{\f12} \nn\\
&\les \ln\left(\frac{\tau-\tau_1+R}{R}\right)\left( R\|\psi\|_{L_\omega^2(S_{\tau_1, \tau_1-R})}
+\left(\int_{\tau_1}^{\tau}\int_{S_{t,\tau_1-R}}\tir |L(v_t^{\f12}\psi)|^2 d\omega dt\right)^{\f12}\right).\label{lot7}
\end{align}
\end{lemma}

\begin{proof}
We first claim that
\begin{align}\label{3.20.1}
&\left(\int_{T_{\tau_1, R}^\tau} \tir^{m+1} \left(L \psi+\frac{\psi}{\tir}\right)^2 d\omega du dt\right)^\f12 \nn\\
& \les \|\tir^{\frac{m-1}{2}}L(v_t^{\f12}\psi)\|_{L_t^2 L_u^2 L_\omega^2(T_{\tau_1, R}^\tau)}
 +\la^{-8\ep_0} \sup_{\tau_1<t\le \tau}\|\tir^\frac{m}{2} \psi\|_{L_u^2 L_\omega^2(T_{\tau_1, R}^{\tau}\cap \Sigma_t)}.
\end{align}
To see this, we use $L v_t = v_t \tr \chi$ to derive that
\begin{equation}\label{6.27.1}
L(v_t^{\f12}\psi)=v_t^{\f12}\left(L\psi+\frac{\psi}{\tir}+\frac{\psi}{2}\left(\tr\chi-\frac{2}{\tir}\right)\right).
\end{equation}
Since  $v_t^\f12$ is comparable with $\tir$,  we have
\begin{equation}\label{l_3_20_1}
{\ti r}^\frac{m+1}{2}\left|L\psi+\frac{\psi}{\tir}\right|
\les \tir^{\frac{m-1}{2}}|L(v_t^{\f12}\psi)| +{\ti r}^\frac{m+1}{2}\left|\left(\tr\chi-\frac{2}{\tir}\right) \psi\right|.
\end{equation}
By using (\ref{ric1}) in Proposition \ref{ricco}, we deduce that
\begin{align*}
\left\|{\ti r}^\frac{m+1}{2}\left(\tr\chi-\frac{2}{\tir}\right) \psi\right\|_{L_t^2 L_u^2 L_\omega^2(T_{\tau_1, R}^\tau)}
&\les \left\|\tir^\f12 \left(\tr\chi-\frac{2}{\tir}\right)\right\|_{L_t^2 L_x^\infty}
\sup_{\tau_1\le t\le \tau}\|\tir^\frac{m}{2} \psi\|_{L_u^2 L_\omega^2(T_{\tau_1,R}^\tau\cap \Sigma_t)}\\
&\les \la^{-8\ep_0}\sup_{\tau_1\le t\le \tau}\|\tir^\frac{m}{2}\psi\|_{L_u^2 L_\omega^2(T_{\tau_1, R}^\tau \cap \Sigma_t)}.
\end{align*}
Combining this with (\ref{l_3_20_1}), we obtain (\ref{3.20.1}).

Now we prove (\ref{lot1}) for $m=1$. By integrating the identity
\begin{equation}\label{7.9.1}
\tir^2 (L\psi + \frac{\psi}{\tir})^2 = \tir^2 (L\psi)^2 + L(\tir \psi^2)
\end{equation}
over $T_{\tau_1, R}^{\tau}$, we obtain
\begin{align}
\int_{T_{\tau_1, R}^{\tau}} \tir^2 (L \psi +\frac{\psi}{\tir})^2 d\omega du dt
& =\int_{T_{\tau_1, R}^{\tau}} \tir^2 (L\psi)^2+\int_{\tau_1-R}^{\tau-R} \int_{{\mathbb S}^2} \tir \psi^2(\tau, u, \omega) d\omega d u \nn\\
& -R\int_{\tau_1}^{\tau} \int_{{\mathbb S}^2} \psi^2(t, t-R, \omega) d\omega d t. \label{p1_1}
\end{align}
In view of (\ref{3.20.1}), we have
\begin{align*}
\|\tir^{\f12} \psi\|_{L_u^2L_\omega^2(T_{\tau_1, R}^\tau\cap \Sigma_\tau)}
& \les \|L(v_t^{\f12} \psi)\|_{L_t^2 L_u^2 L_\omega^2(T_{\tau_1,R}^\tau)}
+ \left(R\int_{\tau_1}^{\tau} \int_{{\mathbb S}^2} \psi^2(t, t-R, \omega) d\omega d t\right)^{\f12}\\
& +\la^{-8\ep_0} \sup_{\tau_1<t\le \tau}\|\tir^\f12 \psi\|_{L_u^2 L_\omega^2(T_{\tau_1, R}^{\tau}\cap \Sigma_t)}
\end{align*}
from which we can obtain (\ref{lot1}).

To prove (\ref{lot1}) for $m=2$,
We may integrate $L(v_t^{\f12}\psi)$ along null geodesics over the null hypersurface $C_u$ to obtain
\begin{equation}\label{psi1}
\|(v_t^{\f12} \psi)(\tau,u, \cdot)\|_{L_\omega^2}\le \left\|\int_{u+R}^{\tau} L(v_t^{\f12}\psi) dt\right\|_{L_\omega^2}
+\|(v_t^{\f12} \psi)(u+R,u,\cdot )\|_{L_\omega^2}
\end{equation}
By squaring the both sides, integrating for $\tau_1-R \le u\le \tau-R$ and using $v_t^{\f12}\approx \tir$, we derive that
\begin{equation*}
\|\tir \psi(\tau)\|_{L_u^2 L_\omega^2}^2\les  \int_{\tau_1-R}^{\tau-R} \|L(v_t^{\f12}\psi)\|_{L_t^2 L_\omega^2}^2 du
+R^2 \int_{\tau_1}^{\tau} \int_{{\mathbb S}^2} \psi^2(t, t-R, \omega) d\omega dt.
\end{equation*}
This gives (\ref{lot1}) for $m=2$ if we deal with the first term on the right by writing $L(v_t^{\f12} \psi) = \tir^{-\f12} \tir^{\f12} L(v_t^{\f12} \psi)$
and applying the Cauchy-Schwarz inequality for the integral in $t$.

To prove (\ref{lot7}), we divide (\ref{psi1}) by $\tir$ and use $v_t^{\f12}\approx \tir$ to obtain
$$
\int_{S_{t, \tau_1-R}} \tir \psi^2 d\omega \les \frac{1}{\tir} \int_{{\mathbb S}^2} \left(\int_{\tau_1}^t L(v_{\tt}^{\f12} \psi) d\tt\right)^2 d\omega
+ \frac{1}{\tir} \|R\psi(\tau_1, \tau_1-R, \c)\|_{L_\omega^2}^2.
$$
By integrating $t$ over $[\tau_1, \tau]$ and using the same argument in the above to deal with the first term on the right, we obtain (\ref{lot7}).

Similar to the proof for (\ref{3.20.1}), we can obtain
\begin{equation}\label{3.20.2}
\int_{C_u[\tau_1, \tau]} \tir^{m-1} (\tir L \psi+\psi)^2 d\omega dt
\les \left\|\tir^\frac{{m-1}}{2}L(v_t^{\f12}\psi) \right\|_{L_t^2 L_\omega^2}^2
+\la^{-16\ep_0} \sup_{\tau_1\le t\le \tau} \left\|\tir^\frac{m}{2} \psi(t) \right\|_{L_\omega^2}^2
\end{equation}
for $m\in {\mathbb N}$, where $C_u[\tau_1, \tau]:=C_u\cap \{\tau_1\le t\le \tau\}$. In view of (\ref{3.20.2}) with $m=2$, we
can obtain (\ref{lot6}) by integrating the identity
$$
\tir (\tir L\psi +\psi)^2 = \tir (\tir L\psi)^2 - \tir \psi^2 + L(\tir^2 \psi^2)
$$
over $C_u[\tau_1, \tau]$ and using the fact
\begin{equation*}
\int_{C_u[\tau_1, \tau]} L(\tir^2 \psi^2) dt d\omega =\int_{S_{\tau, u}} \tir^2 \psi^2 d\omega - \int_{S_{\tau_1, u}} \tir^2 \psi^2 d\omega.
\end{equation*}
In view of (\ref{3.20.2}) with $m=1$, we can obtain (\ref{lot10}) by integrating (\ref{7.9.1}) over $C_u[\tau_1,\tau]$.
\end{proof}

Given $\tau>\tau_1\ge t_0$, we set
\begin{align}\label{btau}
\begin{split}
\B_{\tau_1}^{\tau}:&=(\tau-\tau_1+R)^{-1} \int_{\tau_1}^{\tau}\int_{S_{t,\tau_1-R}} \tir \left(|L(v_t^{\f12}\psi)|^2+\psi^2\right) d\omega dt\\
&\quad \,  +\int_{\widetilde{\Sigma}^{\tau}_{\tau_1, R}}\P_\a^{(\bT)}[\psi] \bn^\a+\int_{S_{\tau_1,\tau_1-R}} \tir \psi^2 d\omega.
\end{split}
\end{align}

\begin{proposition}\label{6.27.3}
Let  $ R\le R'\le  2R$ and let $K_{\tau_1, R'}^{\tau} = \R_{\tau_1, R}^{\tau} \cap \{\tir \le R'\}$. There holds
\begin{equation}\label{intes_1}
\int_{K_{\tau_1, R'}^{\tau}} \left(\P_\a^{({\emph\bT})}[\psi] {\emph\bT}^\a+\frac{\psi^2}{\tir}\right) \les \B_{\tau_1}^{\tau}.
\end{equation}
\end{proposition}

\begin{proof}
In view of (\ref{7.8.2}), we can obtain (\ref{intes_1}) by using (\ref{lot6}) and Lemma \ref{lem:mint_1}.
\end{proof}

\begin{lemma}
\begin{equation}\label{ieq.2}
R^2 \int_{\tau_1}^{\tau}\int_{S_{t, t-R}} \psi^2  d\omega d t
\les \int_{K_{\tau_1, R}^{\tau}} \left(\P_\a^{({\emph\bT})}[\psi] {\emph\bT}^\a+\frac{\psi^2}{\tir}\right)
\les \B_{\tau_1}^{\tau}
\end{equation}
\end{lemma}

\begin{proof}
Consider the identity $\p_{\tir} (\tir^2 \psi^2) = 2 \tir \psi^2 + 2\tir^2 \psi \p_{\tir} \psi$, where $\psi = \psi(t, t-\tir, \omega)$.
Integrating in $\tir$ over $[0, R]$ gives
\begin{align*}
R^2 \psi^2(t, t-R, \omega)= 2\int_0^R \left(\tir \psi^2 + \tir^2 \psi \p_{\tir} \psi\right) d\tir
\le \int_0^R \left(2 \tir \psi^2 + \tir^2 \psi^2 + \tir^2 (\p_{\tir} \psi)^2\right) d\tir.
\end{align*}
Integrating over $[\tau_1, \tau]\times {\Bbb S}^2$, we obtain
\begin{align*}
R^2 \int_{\tau_1}^{\tau}\int_{S_{t, t-R}}  \psi^2 d\omega d t
&\le \int_{\tau_1}^{\tau} \int_{\tir\le R} \left(2 \frac{\psi^2}{\tir}+\psi^2+(\p_{\tir} \psi)^2\right) \tir^2 d\omega du dt.
\end{align*}
By using $\p_{\tir}=\bb N$, $|\bb^{-1}-1|< \f12$ and $\tir^2 \approx v_t$, we can obtain the first inequality in (\ref{ieq.2});
the second inequality is (\ref{intes_1}).
\end{proof}

We can adapt from (\ref{lot1}), and use (\ref{intes_1}) and (\ref{ieq.2}) to get

\begin{proposition}\label{lot8}
Given $\tau> \tau_1\ge t_0$. We have
\begin{align*}
\|\tir^\frac{m}{2}\psi\|_{L_u^2 L_\omega^2(T_{\tau_1, R}^\tau \cap \Sigma_\tau)}^2
&\les \Phi_m(\tau)^2 \|\tir^{\frac{m-1}{2}}L(v_t^{\f12}\psi)\|_{L_u^2 L_t^2 L_\omega^2(T_{\tau_1, R}^{\tau})}^2+\B_{\tau_1}^{\tau}.
\end{align*}
In the region with $\{R\le \tir\le 2R\}\cap \{\tau_1-R\le u\le \tau-R\}$, we have
$
\|\tir^\frac{m}{2}\psi (t)\|_{L_u^2 L_\omega^2}^2\les \B_{\tau_1}^{\tau}.
$
\end{proposition}

\subsection{Comparison results}\label{cmpsec}

We will prove some comparison results which will be used in following subsections.

\begin{proposition}\label{comp3}
Let $U\subset \D_0^+\cap \{R\le \tir \le 2 R\}$ be a  region satisfying (\ref{8.8.1}). There holds
\begin{equation*}
\int_U \left(|L(v_t^{\f12}\psi)|^2+ |\sn \ti\psi|^2 v_t\right)\bb du dt d\omega
\les\int_U \left(\P^{(\emph{\bT})}_\a[\psi] \bT^\a+\frac{\psi^2}{\tir}\right).
\end{equation*}
\end{proposition}

\begin{proof}
Recall (\ref{6.27.1}) for $L(v_t^{\f12} \psi)$. Since $\ti\psi=\Omega \psi$ with $\Omega=e^{-\sigma}$, we also have
\begin{equation}\label{6.30.1}
\sn\ti\psi=\Omega \left(\sn \psi-\psi\sn \sigma\right).
\end{equation}
In view of (\ref{l4.19}),  $v_t^\f12\approx \tir$ and $\tir\approx R$, it suffices to show that
\begin{equation}\label{comp4}
\int_U \psi^2 |M|^2 v_t \bb d\omega dt du\les \int_U \left(\P^{(\bT)}_\a[\psi] \bT^\a+\frac{\psi^2}{\tir}\right)
\end{equation}
where $M=\sn \sigma$ or $\tr\chi-\frac{2}{\tir}$. To see this, we take $q>2$ with $0<1-\frac{2}{q}<s-2$ and let $q_*$ be such that
$2/q+1/q_*=1$. Then
\begin{align}
\int_U \psi^2 |M|^2 v_t \bb d\omega dt du
&\le \|r^{\f12} M\|_{L^\infty L_\omega^q(U)}^2 \|\tir^{\f12}\psi\|_{L_t^2L_u^2 L_\omega^{2q_*}(U)}^2\label{p2_2}
\end{align}
By the Sobolev embedding (\ref{sob}) and $\tir \approx R$, we have
 \begin{align*}
\|\tir^\f12 \psi\|_{L_\omega^{2q_*}} &\les\|\tir\sn \psi\|_{L_\omega^2} + \|\tir^{\f12} \psi\|_{L_\omega^2}.
\end{align*}
Therefore
 \begin{align}
\| \tir^{\f12} \psi\|_{L_t^2L_u^2 L_\omega^{2q_*}(U)}^2 \les \int_U \left(|\sn \psi|^2 + \frac{\psi^2}{\tir}\right)
\les \int_U \left(\P^{(\bT)}_\a[\psi] \bT^\a+\frac{\psi^2}{\tir}\right).\label{3_21_2}
 \end{align}
Recall that $\tr\chi-\frac{2}{\tir}=z-V_4$. We may use (\ref{pric2}), (\ref{transtrace}) and (\ref{l4.1}) to obtain
$\|\tir^{\frac{1}{2}}M\|_{L^\infty L_\omega^q}\les \la^{-\f12} \les 1$.
Combining this with  (\ref{p2_2}) and (\ref{3_21_2}), we thus obtain (\ref{comp4}).
\end{proof}

\begin{lemma}\label{com5}
Given $R\le R'\le 2R$ and $m\in {\mathbb N}$. Let $U:=\D_0^+\cap \{\tir \ge R'\}$ and $U_t:= U\cap \Sigma_t$. Then there hold
\begin{align*}
&\|\tir^{\frac{m}{2}} \sn \ti\psi\|_{L_u^2 L_x^2(U_t)}+\|\tir^{\frac{m}{2}}\psi\|_{L_u^2 L_\omega^2(U_t)}
\approx \|\tir^{\frac{m}{2}}\sn \psi\|_{L_u^2 L_x^2(U_t)}+\|\tir^{\frac{m}{2}}\psi\|_{L_u^2 L_\omega^2(U_t)}
\end{align*}
and
\begin{align*}
&\|\tir^{\frac{m}{2}} L(v_t^{\f12} \psi)\|_{L_u^2L_\omega^2(U_t)} +\|\tir^{\frac{m}{2}} \sn \ti\psi\|_{L_u^2 L_x^2(U_t)}
+\|\tir^{\frac{m}{2}}\psi\|_{L_u^2 L_\omega^2(U_t)}\\
&\approx \|\tir^{\frac{m}{2}} L\psi\|_{L_u^2 L_x^2(U_t)} +\|\tir^{\frac{m}{2}}\sn \psi\|_{L_u^2 L_x^2(U_t)}
+\|\tir^{\frac{m}{2}}\psi\|_{L_u^2 L_\omega^2(U_t)}.
\end{align*}
\end{lemma}

\begin{proof}
By virtue of (\ref{6.27.1}) and (\ref{6.30.1}), it suffices to show that
$$
\|\tir^{\frac{m}{2}} \psi (\tr\chi -\frac{2}{\tir})\|_{L_u^2 L_x^2} + \|\tir^{m\over 2} \psi \sn \sigma\|_{L_u^2L_x^2}
\les \la^{-4\ep_0} \left(\|\tir^{\frac{m}{2}}\sn \psi\|_{L_u^2 L_x^2}+\|\tir^{\frac{m}{2}}\psi\|_{L_u^2 L_\omega^2}\right).
$$
To see this, we take $q>2$ with $0<1-\frac{2}{q}<s-2$ and let $q_*$ be such that $2/q+1/q_*=1$. By using $v_t \approx \tir^2$
and $\tir \le \tau_*$, we then obtain
\begin{align*}
& \|\tir^{\frac{m}{2}} \psi (\tr\chi -\frac{2}{\tir})\|_{L_u^2 L_x^2}^2 + \|\tir^{m\over 2} \psi \sn \sigma\|_{L_u^2L_x^2}^2 \\
&\les \tau_* \left(\|\tir^{\f12} (\tr \chi - \frac{2}{r})\|_{L^\infty L_\omega^q}^2 +\|\tir^{\f12} \sn \sigma\|_{L^\infty L_\omega^q}^2\right) \|\tir^{m\over 2}\psi\|_{l_u^2L_\omega^{2q_*}}^2
\end{align*}
By using (\ref{sob}) with $p=2q_*$, we have
\begin{equation*}
\|\tir^\frac{m}{2}\psi\|_{L_u^2 L_\omega^{2q_*}}\les \|\tir^\frac{m}{2}\sn \psi\|_{L_u^2 L_x^2}+\|\tir^\frac{m}{2}\psi\|_{L_u^2 L_\omega^2}.
\end{equation*}
From Proposition \ref{awave} and (\ref{l4.1}), we also have
$\tau_*^\f12 (\|\tir^{\f12} (\tr \chi - \frac{2}{\tir})\|_{L^\infty L_\omega^q} +\|\tir^\f12\sn\sigma\|_{L^\infty L_\omega^q})
\les \la^{-4\ep_0}$. The proof is complete.
\end{proof}

\begin{lemma}\label{correc_2}
Given $\tau>\tau_1\ge t_0$, let $U:=T_{\tau_1, R}^\tau$ be the region defined by (\ref{7.14.5}).
Then for $\tau_1\le t\le \tau$ and $\tau_1-R\le u\le \tau-R$ there hold
\begin{align}\label{compp1}
\begin{split}
\E_m[\psi]_{U, R}(t) &\les \E_m[\psi]_{U, 2R}(t) + \B_{\tau_1}^\tau,\\
\widetilde \E_m[\psi]_{U,R}(t) &\les \widetilde \E_m[\psi]_{U, 2R}(t)+\B_{\tau_1}^\tau,\\
CF_m[\psi]_{R}(u, \tau) &\les CF_m[\psi]_{2R}(u, \tau)+\B_{\tau_1}^\tau.
\end{split}
\end{align}
\end{lemma}

\begin{proof}
Let $D:= U\cap \{R\le \tir \le 2R\}$ and $D_t= D\cap \Sigma_t$. By definition it suffices to show that
\begin{align}\label{7.14.7}
\begin{split}
&\int_{D_t} \tir^m \left(|L(v_t^{\f12} \psi)|^2 + |\sn \ti\psi|^2 v_t \Omega^{-2} + \psi^2\right) d\omega du \les \B_{\tau_1}^\tau,\\
& \int_{u+R}^{u+2R} \int_{S_{t,u}} \tir^m |L(v_t^{\f12}\psi)|^2 d\omega dt\les \B_{\tau_1}^\tau.
\end{split}
\end{align}
We first claim that (\ref{7.14.7}) can be obtained by showing that
 \begin{align}
&\left\|L\psi,\sn\psi, \psi\right\|_{L_u^2 L_\omega^2(D_t)}^2\les \B_{\tau_1}^\tau, \label{3.22.3}\\
&\left\|L \psi, \sn\psi\right\|_{L^2(C_u\cap D)}^2+ \left\|\psi \right\|_{L_t^2 L_\omega^2(C_u\cap D)}^2\les \B_{\tau_1}^\tau.\label{phi2}
\end{align}
To see this, by using $v_t\approx \tir^2$, $\Omega\approx 1$, $L v_t = v_t \tr \chi$ and $\sn \Omega =-\Omega \sn \sigma$, we only need to show
\begin{align*}
& \left\|\tir^{\frac{m+1}{2}} \psi (\tr \chi-\frac{2}{\tir}), \, \tir^{\frac{m+1}{2}}\psi \sn \sigma\right\|_{L_u^2 L_\omega^2(D_t)}\les \|\sn \psi\|_{L_u^2 L_\omega^2(D_t)}
+ \|\psi\|_{L_u^2 L_t^2(D_t)}, \\
& \left\|\tir^{\frac{m+1}{2}}\psi (\tr\chi-\frac{2}{\tir}) \right\|_{L_t^2 L_\omega^2(C_u\cap D)} \les \|\sn \psi\|_{L_t^2 L_\omega^2(C_u\cap D)}
+ \|\psi\|_{L_t^2 L_\omega^2(C_u\cap D)}.
\end{align*}
By using Proposition \ref{awave}, (\ref{l4.1}) and $\tir \approx R$ in $D$, we can employ the same argument in the proof of Lemma
\ref{com5} to derive these estimates.

Now we prove (\ref{3.22.3}) and (\ref{phi2}). By using (\ref{dec_1}) in Lemma \ref{monoeng}, we have
\begin{align*}
& \|L\psi,\sn\psi\|_{L^2_u L_\omega^2(D_t)}^2+\|L \psi, \sn \psi\|_{L^2(C_u\cap D)}^2
\les \int_{\widetilde\Sigma_{\tau_1, R}^\tau} \P_\a^{(\bT)}[\psi] \bn^\a\le \B_{\tau_1}^\tau.
\end{align*}
By (\ref{intes_1}), we also have $\|\psi\|_{L_u^2 L_\omega^2(D_t)}^2 \les \B_{\tau_1}^\tau$. Hence (\ref{3.22.3}) is proved.
Next by using (\ref{dec_1}) we derive that
\begin{equation*}
\|L \psi, \sn \psi\|_{L^2(C_u\cap D)}^2 \le \int_{\widetilde{\Sigma}_{u+R, R}^\tau} \P_\a^{(\bT)}[\psi] \bn^\a
\les \int_{\widetilde{\Sigma}_{\tau_1, R}^\tau}  \P^{(\bT)}_\a[\psi] \bn^\a \les \B_{\tau_1}^\tau.
\end{equation*}
which shows the first part of (\ref{phi2}). To see the second part of (\ref{phi2}), we use (\ref{6.27.2}) to obtain
\begin{equation*}
\int_{S_{t, u}} \tir^2 \psi^2 d\omega=\int_{\itt_{t,u}} 2\tir \psi^2 du d\omega +\int_{\itt_{t,u}} 2\tir^2 \psi (N \psi) \bb du d\omega.
\end{equation*}
By integrating this equation in $t$ for $u+R\le t\le u+2R$ and using $\tir\approx R$ in $D$, we have
\begin{equation*}
\|\psi\|_{L_t^2 L_\omega^2(C_u\cap D)}^2\les \int_{\widetilde K}(\P^{(\bT)}_\a[\psi] \bn^\a+\frac{\psi^2}{\tir})
\end{equation*}
where $\widetilde K$ is the region enclosed by $\widetilde{\Sigma}_{u+R, R}^{u+2R}$ and $\Sigma_{u+2R}$. Since
$\widetilde K\subset K_{\tau_1, 2R}^\tau$, we may use (\ref{intes_1}) in Proposition \ref{6.27.3} to derive that
$ \|\psi\|_{L_t^2 L_\omega^2(C_u\cap D)}^2 \les \B_{\tau_1}^\tau$. We thus complete the proof of (\ref{phi2}).
\end{proof}

\subsection{Step 2: Multiplier approach}\label{mul_1}

We use $\Omega= e^{-\sigma}$ to introduce the metric $\ti \bg= \Omega^{-2} \bg$ and $\ti\psi=\Omega \psi$.
We first derive a null decomposition for $\Box_{\ti\bg} \ti \psi$. For simplicity of exposition, we set
\begin{equation*}
h=\f12 \tr\chi,\quad \vb=\f12 \tr\chib, \quad \ti h=\f12 \tr \ti \chi,\quad \ti \vb=\f12 \tr \ti \chib.
\end{equation*}
By using (\ref{c_4_20}) we have
\begin{equation}\label{defv}
\ti h = h +  L\sigma, \qquad \ti \vb = \vb + \Lb \sigma.
\end{equation}
By direct calculation one has (see \cite[page 72]{shock})
\begin{equation*}
\Omega^{-2} \Box_{\ti\bg} \psi=\Box_{\bg} \psi+\bd^\mu \log(\Omega^{-2}) \bd_\mu \psi.
\end{equation*}
Because $\bd \log \Omega=-\bd\sigma$, we thus have $\Omega^{-2} \Box_{\ti\bg}\psi=\Box_\bg \psi+2\bd^\mu \sigma \bd_\mu \psi$.
Recall that
\begin{equation}\label{wd}
\Box_\bg \psi=\sD\psi-\Lb L \psi-(\vb-k_{NN}) L\psi- h \Lb \psi+2\zeta\c\sn \psi
\end{equation}
We therefore obtain
\begin{equation}\label{cwd}
\Omega^{-2}\Box_{\ti\bg} \psi= \sD\psi-\Lb L \psi-(\ti\vb-k_{NN}) L\psi-\ti h\Lb \psi+2(\zeta+\sn \sigma)\c\sn \psi.
\end{equation}
Since (\ref{cwd}) is true for any scalar function $\psi$, we may replace $\psi$ by $\ti \psi$ and regroup terms to obtain
\begin{equation} \label{wave2}
\begin{split}
\Omega^{-2}\Box_{\ti \bg} \ti\psi& =\sD \ti\psi+2(\zeta+\sn \sigma)\c \sn \ti\psi-\Lb(L\ti\psi+\ti h\ti\psi)
-(\ti\vb-k_{NN})(L\ti\psi+\ti h\ti\psi)\\
& \quad \, +(\Lb \ti h+\ti h\ti \vb-\ti h k_{NN})\ti\psi.
\end{split}
\end{equation}
In the following we assume that $\psi$ is a solution of $\Box_\bg \psi=0$. By virtue of (\ref{idc1}) we have
$
\Omega^{-2}\Box_{\ti \bg}\ti \psi=-(\Box_\bg \sigma+\bd^\mu \sigma\bd_\mu \sigma) \ti \psi.
$
Combining this equation with (\ref{wave2}) we obtain
\begin{align*}
0 & = \sD \ti\psi+2(\zeta+\sn \sigma)\c \sn \ti\psi-\Lb(L\ti\psi+\ti h\ti\psi)
-(\ti\vb-k_{NN})(L\ti\psi+\ti h\ti\psi) \nn \displaybreak[0]\\
& \quad \, +(\Lb \ti h+\ti h\ti \vb + \Box_\bg \sigma +\bd^\mu \sigma \bd_\mu \sigma -\ti h k_{NN})\ti\psi.
\end{align*}
By using (\ref{defv}) and (\ref{wd}) for $\Box_\bg \sigma$, we can derive that
\begin{align*}
 \Lb \ti h+\ti h\ti \vb + \Box_\bg \sigma +\bd^\mu \sigma \bd_\mu \sigma -\ti h k_{NN} \displaybreak[0]
& = \Lb h + h \vb +\sD \sigma +|\sn \sigma|^2 +2\zeta\c \sn \sigma - h k_{NN}.
\end{align*}
Note that $2(\Lb h + h\vb) = \mu := \Lb \tr \chi +\f12 \tr \chi\tr \chib$. Therefore
\begin{align}\label{7.11.1}
0 & = \sD \ti\psi+2(\zeta+\sn \sigma)\c \sn \ti\psi-\Lb(L\ti\psi+\ti h\ti\psi)
-(\ti\vb-k_{NN})(L\ti\psi+\ti h\ti\psi) \nn\\
& \quad \, +\left(\f12 \mu +\sD \sigma +|\sn \sigma|^2 +2\zeta\c \sn \sigma - h k_{NN}\right)\ti\psi.
\end{align}
We will use this equation to derive a useful formula by multiplying it a suitably chosen function
and integrating the result over a suitable domain contained in $\D_0^+$.

\begin{figure}[ht]
\centering
\includegraphics[width = 0.46\textwidth]{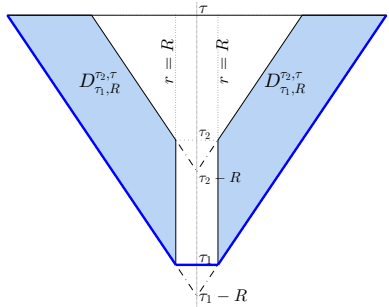}
  \vskip -0.4cm
  \caption{Illustration of $D_{\tau_1, R}^{\tau_2,\tau}$}\label{fig4}
\end{figure}

To this end, for $\tau>\tau_2>\tau_1 \ge t_0$, we use $D_{\tau_1, R}^{\tau_2, \tau}$ to denote the region
enclosed by $\{u=\tau_1-R,  \tau_1\le t<\tau\}$, $\{u=\tau_2-R, \tau_2\le t<\tau\}$,  $\{\ti r=R\}$ and $\{t=\tau, \tau_1-R\le u\le
\tau_2-R\}$. When $\tau=\tau_*$ we set $D_{\tau_1, R}^{\tau_2} = D_{\tau_1, R}^{\tau_2, \tau_*}$.
 We multiply (\ref{7.11.1}) by $(L\ti\psi+\ti h \ti \psi) \tir^m$ and integrate over $D_{\tau_1, R}^{\tau_2, \tau}$ to obtain
\begin{align}\label{7.11.2}
0 & = \int_{D_{\tau_1,R}^{\tau_2, \tau}} \left(\sD \ti\psi+2(\zeta+\sn \sigma)\c \sn \ti\psi\right)
(L\ti \psi + \ti h \ti \psi) \tir^m \bb \Omega^{-2} d\mu_\ga du dt \nn \displaybreak[0]\\
& - \int_{D_{\tau_1,R}^{\tau_2, \tau}} \left(\frac{1}{2} \Lb( (L\ti\psi+\ti h\ti\psi)^2 )
+(\ti\vb-k_{NN}) (L\ti\psi+\ti h\ti\psi)^2 \right) \tir^m \bb d\mu_{\ti\ga}du dt \nn \displaybreak[0]\\
& + \int_{D_{\tau_1,R}^{\tau_2, \tau}} \left(\f12 \mu +\sD \sigma +|\sn \sigma|^2 +2\zeta\c \sn \sigma - h k_{NN}\right)\ti\psi
(L\ti\psi+\ti h \ti \psi) \tir^m \bb d\mu_{\ti\ga} du dt,
\end{align}
where $d\mu_{\ti\ga} =\Omega^{-2} d\mu_\ga$. To simplify the first term in (\ref{7.11.2}), we consider the integral
$$
\I=\int_{S_{t,u}} (\sD \ti \psi+2(\zeta+\sn \sigma)\c \sn \ti \psi)(L\ti\psi+\ti h\ti\psi)\bb\Omega^{-2} d\mu_{\ga}.
$$
By integration by part on $S_{t,u}$, we have
\begin{align*}
\I&=\int \left[-\sn \ti \psi \sn \left(\Omega^{-2} \bb(L\ti \psi+\ti h \ti \psi)\right)
+2 (\zeta+\sn \sigma)\c \sn \ti \psi \c (L\ti \psi+\ti h\ti \psi) \Omega^{-2} \bb\right] d\mu_\ga \displaybreak[0]\\
&=\int \left[-\sn \ti\psi \sn (L \ti\psi+\ti h\ti \psi)+(2\zeta-\sn \log \bb) \c\sn \ti \psi \c (L\ti \psi+\ti h \ti \psi)\right] \bb \Omega^{-2} d\mu_\ga \displaybreak[0]\\
&=\int\left[-\sn \ti \psi \left(\sn_L \sn \ti \psi+[\sn, L]\ti \psi+\sn (\ti h \ti \psi)\right)
+(2\zeta-\sn \log \bb) \sn \ti \psi(L\ti \psi+\ti h \ti \psi)\right] \bb \Omega^{-2} d\mu_\ga \nn \displaybreak[0]\\
& = \int \Big[ -\f12 L(|\sn \ti\psi|^2) - (h+\ti h)  |\sn \ti \psi|^2 -\ti \psi \sn \ti h \sn \ti \psi -\sn \ti \psi \c \chih \c \sn\ti \psi \displaybreak[0]\\
& \quad \, +(2\zeta-\sn \log \bb) \sn \ti \psi(L\ti \psi+\ti h \ti \psi)\Big] \bb \Omega^{-2} d\mu_\ga
\end{align*}
where, in the last step, we dealt with $[\sn, L] \ti \psi$ by using (\ref{cmu2}) and writing $\chi =\chih + h \ga$.
Combining this with (\ref{7.11.2}) and rearranging the terms, we obtain
\begin{align}\label{7.11.3}
& \f12 \int  \bb \tir^m \left[L(|\sn \ti\psi|^2)+2(h+\ti h) |\sn \ti\psi|^2+\Lb( (L\ti\psi+\ti h\ti\psi)^2 )
+2\ti\vb (L\ti\psi+\ti h\ti\psi)^2 \right] d\mu_{\ti\ga} du dt \nn \displaybreak[0]\\
&=\int \tir^m\bb\left[(2\zeta-\sn\log \bb)\c \sn \ti\psi(L\ti\psi+\ti h\ti\psi) +{k}_{NN} (L\ti\psi+\ti h\ti\psi)^2
-\chih\c\sn\ti\psi\c\sn\ti\psi \right]d\mu_{\ti\ga}dudt  \nn \displaybreak[0]\\
&+\int \tir^m \bb\ti\psi\left[\Big(\f12 \mu +\sD \sigma +|\sn \sigma|^2
+ 2\zeta\c \sn \sigma - h k_{NN}\Big) (L\ti\psi+\ti h\ti \psi)-\sn \ti \psi \sn \ti h\right] d\mu_{\ti \ga}dudt,
\end{align}
where the integrals are taken over $D_{\tau_1,R}^{\tau_2, \tau}$.

\begin{proposition}\label{prop:MA}
Let $\tau>\tau_2>\tau_1\ge t_0$. For $m\in {\mathbb N}$ there holds
\begin{align}\label{B_1}\tag{\bf {MA}}
&\int_{\tau_2}^\tau \int_{S_{t, \tau_2-R}} \tir^m |L(v_t^{\f12}\psi)|^2 d\omega dt\nn\\
&\qquad +\int_{D^{\tau_2,\tau}_{\tau_1, R}} \bb \tir^{m-1} \left(m |L(v_t^{\f12}\psi)|^2+(2-m)|\sn\ti\psi|^2v_t \Omega^{-2}\right) d\omega du dt\nn \displaybreak[0]\\
&\qquad +\int_{\tau_1-R}^{\tau_2-R} \int_{S_{\tau, u}} \bb \tir^m \left(|L(v_t^{\f12}\psi)|^2+|\sn\ti\psi|^2v_{t}\Omega^{-2}\right) d\omega du \nn \displaybreak[0]\\
& =\int_{\tau_1}^{\tau_2} \int_{S_{t, t-R}} \tir^m\left(\bb|\sn\ti\psi|^2 v_t\Omega^{-2}-(2-\bb) |L(v_t^{\f12}\psi)|^2\right) d\omega dt \nn \displaybreak[0]\\
& \qquad + \int_{\tau_1}^\tau \int_{S_{t, \tau_1-R}} \tir^m |L(v_t^{\f12}\psi)|^2 d\omega d t +\mbox{error}(m,R) \nn 
\end{align}
with
$$
\mbox{error}(m,R)=2\int_{D_{\tau_1, R}^{\tau_2, \tau}}\er_m(\psi) d\mu_\ga du dt
+\int_{D_{\tau_1,R}^{\tau_2, \tau}} \tir^m \bA\c \J[\psi] \bb\Omega^{-2} d\mu_\ga du dt,
$$
where
$$
\er_m(\psi) =\tir^m \bb \Omega^{-2} \ti \psi \left[ \left(\f12 \mu +\sD \sigma +|\sn \sigma|^2 + 2\zeta\c \sn \sigma - h k_{NN}\right) (L\ti \psi + \ti h \ti \psi)
-\sn \ti \psi \sn \ti h\right],
$$
the symbol $\bA$ denote any terms among $L\log\bb, L\sigma,\zeta, \tr\chi-\frac{2}{t-u},\chih,\frac{\bb-1}{\ti r}, k$, and
the symbol $\J[\psi]$ denotes any quadratic terms with factors being either $\sn \ti\psi$ or $ L\ti\psi+\ti h \ti \psi$.
\end{proposition}

\begin{proof}
Recall that $\sn_A \log \bb = \zeta_A - k_{AN}$. Thus all the terms in (\ref{7.11.3}) can be put into $\mbox{error}(m,R)$ except the terms
\begin{align*}
\I_1 & = \int_{D_{\tau_1,R}^{\tau_2, \tau}} \bb \tir^m \left[L(|\sn \ti\psi|^2)+2(h+\ti h) |\sn \ti\psi|^2\right] \sqrt{|\ti\ga|} d\omega du dt, \displaybreak[0]\\
\I_2 & = \int_{D_{\tau_1,R}^{\tau_2, \tau}} \bb \tir^m \left[\Lb( (L\ti\psi+\ti h\ti\psi)^2 )+2\ti\vb (L\ti\psi+\ti h\ti\psi)^2 \right] \sqrt{|\ti\ga|} d\omega du dt
\end{align*}
The term $\I_1$ can be treated by using $L \sqrt{|\ti \ga|}=2\ti h\sqrt{|\ti \ga|}$ and integrating along the direction of $L$. To deal with the term
$\I_2$, we first use $\Lb \sqrt{|\ti \ga|}=2\ti \vb\sqrt{|\ti \ga|}$ and the identity
\begin{equation}\label{lpsi}
(L\ti \psi+\ti h\ti\psi)^2 \sqrt{|\ti \ga|}= |L(v_t^{\f12} \psi)|^2;
\end{equation}
we then use $\Lb=L-2N$, $\bb N = \p_{\tir}$ and integrate along the directions of $L$ and $N$.
\end{proof}

By applying the above same argument for deriving (\ref{B_1}) on the region $D_{\tau_1, R}^{\tau_2, \tau}\cap \{\tir \ge R'\}$ with $R\le R'\le 2R$,
we can derive that
\begin{align}\label{MA2}
&\int_{\tau_2+R'-R}^\tau \int_{S_{t, \tau_2-R}} \tir^m |L(v_t^{\f12}\psi)|^2 d\omega dt\nn\\
&\qquad +\int_{D^{\tau_2,\tau}_{\tau_1, R}\cap \{\tir \ge R'\}} \bb \tir^{m-1}
  \left(m |L(v_t^{\f12}\psi)|^2+(2-m)|\sn\ti\psi|^2v_t \Omega^{-2}\right) d\omega du dt\nn \displaybreak[0]\\
&\qquad +\int_{\tau_1-R}^{\tau_2-R} \int_{S_{\tau, u}} \bb \tir^m \left(|L(v_t^{\f12}\psi)|^2+|\sn\ti\psi|^2v_{t}\Omega^{-2}\right) d\omega du \nn \displaybreak[0]\\
& =\int_{\tau_1+R'-R}^{\tau_2+R'-R} \int_{S_{t, t-R'}} \tir^m\left(\bb|\sn\ti\psi|^2 v_t\Omega^{-2}-(2-\bb) |L(v_t^{\f12}\psi)|^2\right) d\omega dt \nn \displaybreak[0]\\
& \qquad + \int_{\tau_1+R'-R}^\tau \int_{S_{t, \tau_1-R}} \tir^m |L(v_t^{\f12}\psi)|^2 d\omega d t +\mbox{error}(m,\tir \ge R'),
\end{align}
where
$$
\mbox{error}(m,\tir \ge R')=2\int_{D_{\tau_1, R}^{\tau_2, \tau}\cap \{\tir \ge R'\}}\er_m(\psi) d\mu_\ga du dt
+\int_{D_{\tau_1,R}^{\tau_2, \tau}\cap \{\tir \ge R'\}} \tir^m \bA\c \J[\psi] \bb\Omega^{-2} d\mu_\ga du dt,
$$


\subsubsection{Bootstrap Assumptions for conformal energy}

We will prove Theorem \ref{BT} by a bootstrap argument.  Let $\C_0=\|\psi[t_0]\|_{\dot{H}^1}+\|\psi(t_0)\|_{L^2}$. We  make the following assumptions,
\begin{align}
\C[\psi](t) &\le \la^{2\ep_0}\C_0^2, \label{baen} \displaybreak[0]\\
\|\psi(t)\|_{L_u^2 L_\omega^2} & \le (t+1)^{-1} \la^{\ep_0}\C_0.\label{lowba}
\end{align}
We will improve them to be
\begin{align}
& \int_{\Sigma_t}|(t+1)\sn \psi|^2+|(t+1) L\psi|^2\les (1+t)^{2\ep} \C_0^2, \label{con1} \displaybreak[0]\\
& \int_{\Sigma_t\cap\{u\ge \frac{t}{2}\}} \P_\a^\bT[\psi] \bn^\a\les (1+t)^{-2}\C_0^2, \label{con} \displaybreak[0]\\
&\|\psi(t)\|_{L_u^2 L_\omega^2}\les (t+1)^{-1+\ep} \C_0, \label{con2}
\end{align}
where $\ep>0$ is any number close to $0$.  The combination of (\ref{con1}), (\ref{con}) and (\ref{con2}) gives
\begin{equation}
\C[\psi](t)\les (t+1)^{2\ep}\C_0^2
\end{equation}
which improves (\ref{baen}) because $\ep>0$ can be arbitrarily close to $0$.

\subsubsection{Bounded weighted energy}

In Lemma \ref{lem_629}, we will control error term involved in (\ref{B_1}). We then obtain various weighted
energies in Proposition \ref{e1} and Proposition \ref{intg}. They are crucial for Step 3, which provides the
desired decay results, in particular, for the interior region of $\D^+$ where $u\ge \frac{t}{2}$.  Proposition \ref{e1}
gives conformal energy in the exterior region of $\D^+$ where $u<\frac{t}{2}$. The proof of Lemma \ref{lem_629}
relies crucially on Section \ref{sec_5} and Section \ref{sec_6}.

We first give some consequences of (\ref{baen}) and (\ref{lowba}). We will use the inequality
\begin{equation}\label{sobp}
\|\tir^{\frac{m-1}{2}} \psi(t)\|_{L_u^2 L_\omega^q}
\les \| \tir^\frac{m-1}{2}\sn\psi(t)\|_{L_u^2 L_x^2} +\|\tir^{\frac{m}{2}-\frac{3}{2}}\psi(t)\|_{L_u^2 L_x^2}
\end{equation}
which follows from (\ref{sob}), where $2<q<\infty$.

\begin{lemma}
Let $t_0\le \tau_1<\tau_*$. Within $D_{\tau_1, R}^{\tau_*}$ there hold
\begin{align}
& \|\sn \psi(t)\|_{L_u^2 L_x^2}+\|\psi(t)\|_{L_u^2 L_\omega^2}+\|\sn\ti\psi(t)\|_{L_u^2 L_x^2} \les (t+1)^{-1}\la^{\ep_0}\C_0, \label{bba1} \displaybreak[0]\\
& \|r^{\frac{m-1}{2}}\psi\|_{L_t^2 L_u^2 L_\omega^q(D_{\tau_1, R}^{\tau_*})}\les(\tau_1+1)^{\frac{m-2}{2}} (\ln\la)^\f12 \la^{\ep_0}\C_0,\label{bba2}
\end{align}
where $m=1,2$ and $2<q<\infty$.
\end{lemma}

\begin{proof}
It is easy to obtain (\ref{bba1}) from (\ref{baen}), (\ref{lowba}) and Lemma \ref{com5} with $m=0$.
By using (\ref{sobp}) and (\ref{bba1}), we can obtain for $m\ge 1$ that
\begin{equation*}
\|r^{\frac{m-1}{2}} \psi\|_{L_t^2 L_u^2 L_\omega^q(D_{\tau_1, R}^{\tau_*})}
\les \left(\int_{\tau_1}^{\tau_*} (1+t)^{m-3} dt\right)^{1/2} \la^{\ep_0}\C_0.
\end{equation*}
which together with $\tau_*\le \la^{1-8\ep_0}$ implies (\ref{bba2}) for $m=1,2$.
\end{proof}

Next we consider the error term $\mbox{error}(m,R)$ in (\ref{B_1}). We will use the same notation when all term in the integrand
are taken absolute values, i.e.
$$
\mbox{error}(m,R) = \int_{D_{\tau_1, R}^{\tau_2, \tau}} |\er_m(\psi)| d\mu_\ga du dt +
\int_{D_{\tau_1, R}^{\tau_2, \tau}}  \tir^m  |\bA| |\J[\psi]|\Omega^{-2} \bb d\omega du dt.
$$

\begin{lemma}\label{lem_629}
Let $t_0\le \tau_1<\tau_2\le \tau\le \tau_*$ and $U:=D_{\tau_1,R}^{\tau_2, \tau}$. Then for $m=1,2$ there hold
\begin{align}
\mbox{error}(m, R)&\les \la^{-3\ep_0+} \Big((\tau_1+1)^{(m-2)}\C_0^2+\sup_{\tau_1\le t\le \tau} \|\tir^\frac{m}{2}\psi(t)\|_{L_u^2L_\omega^2(U\cap \Sigma_t)}^2\nn\\
& +\sup_{\tau_1 \le t\le \tau} \E_m[\psi](t)+\sup_{\tau_1-R\le u\le \tau_1-R} CF_m[\psi](u) \Big),\label{e2}
\end{align}
where the domain of integral for $\E_m[\psi](t')$ is $\Sigma_\tt\cap U$ in (\ref{3.22.1}) and the domain of integral for
$CF_m[\psi](u)$ is $C_u\cap U$ in (\ref{3.22.2}).
\end{lemma}

\begin{proof}
Note that $\tir^m |\J[\psi]| \Omega^{-2} v_t \les \J_m[\psi]$. Recall that $\|\bA\|_{L_t^2 L_x^\infty}\les \la^{-\f12-4\ep_0}$
from Proposition \ref{awave}. We have
\begin{equation*}
\int_{U} \tir^m |\bA| |\J[\psi]|\Omega^{-2}\bb d\mu_\ga du dt
\les\la^{-8\ep_0} \sup_{\tau_1\le t\le \tau} \E_m[\psi](t).
\end{equation*}

Next we consider $\int_U |\er_m(\psi)| d\mu_\ga du dt$ by splitting it into $E^{(1)}+E^{(2)}$, where
\begin{align*}
E^{(1)}&=\int_U \tir^m |\psi| |\sn \ti\psi| |\sn \ti h| \bb d\mu_\ga du dt, \displaybreak[0]\\
E^{(2)}&=\int_U \tir^m \left|(\sn \sigma\c\sn \sigma,\zeta\c\sn \sigma, \ckk\mu, \tr\chi\ti V_3)\right|
\left|v_t^{\f12} \psi\right| \left|L(v_t^{\f12}\psi)\right| \bb d\omega du dt.
\end{align*}
where $\ckk\mu = \mu + 2 \sD \sigma -k_{NN} \tr \chi + \f12 \tr\chi V_3$, see Lemma \ref{mu4.4}.

Let $\frac{1}{q}+\frac{1}{q'}=\frac{1}{2}$ and $0<1-\frac{2}{q'}<s-2$.  Then by using the H\"{o}lder inequality and (\ref{bba1})
we have
\begin{align*}
E^{(1)} &\les \int_{\tau_1}^\tau \|\tir^{3\over 2} \sn \ti h(t)\|_{L_u^\infty L_\omega^{q'}} \|\tir^{m-\f12} \psi(t)\|_{L_u^2 L_\omega^q}
\|\sn \ti \psi(t)\|_{L_u^2 L_x^2} dt \\
& \les \la^{\ep_0} \C_0 \int_{\tau_1}^\tau (1+t)^{-1} \|\tir^\frac{3}{2} \sn \ti h(t)\|_{L_u^\infty L_\omega^{q'}}
\|\tir^{m-\f12}\psi(t)\|_{L_u^2 L_\omega^q} dt.
\end{align*}
In view of (\ref{sobp}) and $\tir \le t$, we have for $m\ge 1$ that
\begin{equation}\label{6.28.2}
\|\tir^{m-\f12} \psi(t)\|_{L_u^2 L_\omega^q}
\les t^{\frac{m-1}{2}} \left(\|\tir^\frac{m}{2} \sn \psi(t)\|_{L_u^2 L_x^2} +\|\tir^{\frac{m}{2}}\psi(t)\|_{L_u^2 L_\omega^2}\right).
\end{equation}
Therefore
\begin{align*}
E^{(1)}&\les \la^{\ep_0} \C_0 \int_{\tau_1}^\tau (1+t)^{\frac{m-3}{2}} \|\tir^{\frac{3}{2}}\sn \ti h(t)\|_{L^\infty L_\omega^{q'}}
\left(\|\tir^\frac{m}{2} \sn \psi(t)\|_{L_u^2 L_x^2}+\|\tir^{\frac{m}{2}}\psi(t)\|_{L_u^2 L_\omega^2}\right) dt.
\end{align*}
In view of (\ref{ricp}) and $2\sn \ti h = \sn z$, we obtain for $m=1,2$ that
\begin{align*}
E^{(1)}&\les \la^{-3\ep_0} (\tau_1+1)^{\frac{m-2}{2}} \C_0
\sup_{\tau_1\le t\le \tau} \left(\E_m[\psi](t)^\f12+\|\tir^{\frac{m}{2}} \psi(t)\|_{ L_u^2 L_\omega^2(U\cap \Sigma_t)}\right).
\end{align*}

To complete the proof of (\ref{e2}), it remains to show that
\begin{align}
E^{(2)}\les \la^{-3\ep_0+} \left(\sup_{\tau_1\le t\le \tau} \E_m[\psi](t)
+\sup_{\tau_1-R\le u \le \tau_2-R} CF_m[\psi](u) +(\tau_1+1)^{m-2}\C_0^2 \right). \label{engymu}
\end{align}
In view of Proposition \ref{dcmpsig}, we can write
$\zeta\c\sn \sigma, \sn \sigma\c\sn \sigma=(\bA+\bA^\dag+\mu^\dag)\sn\sigma$. Thus, we can prove (\ref{engymu}) by considering the following terms
\begin{align*}
\A_1&=\int_U \left|\mu^\dag\right| \left|\sn \sigma\right| \left|v_t^{\f12}\psi\right| \left|L(v_t^{\f12}\psi)\right| \tir^m\bb d\omega du dt,\\
\A_2&=\int_U |(\left|\bA\right|+\left|\bA^\dag\right|) \left|\sn \sigma\right| \left|v_t^{\f12}\psi\right| \left|L(v_t^{\f12}\psi)\right| \tir^m \bb d\omega du dt,\\
\A_3&=\int_U \left|(\ckk\mu,\tr\chi V_3)\right| \left|v_t^{\f12}\psi\right| \left|L(v_t^{\f12}\psi)\right |\tir^m\bb d\omega du dt.
\end{align*}
Let $\frac{1}{q}+\frac{1}{q'}=\f12$ with $0<1-\frac{2}{q'}<s-2$. Let $1=\frac{1}{q'}+\frac{1}{q_*}$.
Clearly $\frac{1}{2}+\frac{1}{q}=\frac{1}{q_*}$.  In what follows, we will constantly employ the H\"{o}lder inequality
based on such choices of $q, q_*$ and $q'$.

Let $m=1,2$, by using (\ref{bba2}) and $v_t\approx \tir^2$ we have
\begin{align}
\left\| v_t^{\f12} \psi L(v_t^{\f12}\psi) \tir^{m-\f12} \right\|_{L_u^2 L_t^1 L_\omega^{q_*}}
&\les \|\tir^{\frac{m}{2}+\f12}\psi\|_{L_u^2 L_t^2 L_\omega^q}\|\tir^\frac{m}{2} L(v_t^{\f12} \psi)\|_{L_u^\infty L_t^2 L_\omega^2}\nn \displaybreak[0]\\
&\les \tau_*(\ln \la )^\f12\la^{\ep_0}(\tau_1+1)^\frac{m-2}{2}\|\tir^\frac{m}{2} L(v_t^{\f12}\psi)\|_{L_u^\infty L_t^2 L_\omega^2}\C_0,\label{sob1p1} \displaybreak[0]\\
\left\| v_t^{\f12}\psi L(v_t^{\f12}\psi) \tir^{m-\f12} \right\|_{L_t^2 L_u^1 L_\omega^{q_*}}
&\les \|\tir^{\frac{m}{2}+\f12}\psi\|_{L_t^2 L_u^2 L_\omega^q}\|\tir^\frac{m}{2} L(v_t^{\f12} \psi)\|_{L_t^\infty L_u^2 L_\omega^2}\nn \displaybreak[0]\\
&\les \tau_*(\ln \la )^\f12\la^{\ep_0}(\tau_1+1)^\frac{m-2}{2}\|\tir^\frac{m}{2} L(v_t^{\f12} \psi)\|_{L_t^\infty L_u^2 L_\omega^2}\C_0.\label{sob1p}
\end{align}
With the help of (\ref{l4.1}), Proposition \ref{dcmpsig} and (\ref{sob1p1}), we have
\begin{align*}
\A_1&\les \|\mu^\dag\|_{L_u^2 L^\infty}\| \tir^{\f12}\sn \sigma\|_{L^\infty L_\omega^{q'}}\|v_t^{\f12}\psi L(v_t^{\f12}\psi) \tir^{m-\f12}\|_{L_u^2 L_t^1 L_\omega^{q_*}}\\
&\les\tau_*\la^{-1-4\ep_0}(\tau_1+1)^\frac{m-2}{2}\la^{\ep_0} (\ln \la)^\f12 \|\tir^\frac{m}{2} L(v_t^{\f12} \psi)\|_{L_u^\infty L_t^2 L_\omega^2} \C_0\\
&\les \la^{-10\ep_0} (\tau_1+1)^\frac{m-2}{2}\|\tir^\frac{m}{2} L(v_t^{\f12} \psi)\|_{L_u^\infty L_t^2 L_\omega^2}\C_0
\end{align*}
By using (\ref{l4.1}) and Proposition \ref{awave}, we have
\begin{align*}
\A_2&\les \|\bA, \bA^\dag\|_{L_t^2 L_x^\infty} \|\tir^\f12 \sn \sigma\|_{L_t^\infty L_u^\infty L_\omega^{q'}}
\|v_t^{\f12}\psi L(v_t^{\f12} \psi) \tir^{m-\f12}\|_{L_t^2 L_u^1 L_\omega^{q_*}}\\
&\les \la^{-10\ep_0}(\tau_1+1)^\frac{m-2}{2}\|\tir^\frac{m}{2} L(v_t^{\f12} \psi)\|_{L_t^\infty L_u^2 L_\omega^2}\C_0.
\end{align*}
Using (\ref{ckmu3}), (\ref{err_4_4}) and the similar derivation for (\ref{sob1p1}), we have
\begin{align*}
\A_3&\les \|\tir^\frac{3}{2}\ckk\mu,\tir^{\frac{3}{2}}\tr\chi V_3 \|_{L_u^2 L_t^\infty L_\omega^{q'}}
\|\tir^{m-\frac{1}{2}} \psi L(v_t^{\f12}\psi)\|_{L_u^2 L_t^1 L_\omega^{q_*}}\\
&\les (\tau_1+1)^\frac{m-2}{2} \la^{-3\ep_0} (\ln\la)^\f12 \|\tir^\frac{m}{2} L(v_t^{\f12} \psi)\|_{L_u^\infty L_t^2 L_\omega^2}\C_0.
\end{align*}
Notice that
\begin{align*}
\|\tir^\frac{m}{2} L(v_t^{\f12} \psi)\|_{L_t^\infty L_u^2 L_\omega^2(U)} &\les \sup_{\tau_1\le t \le\tau} \E_m[\psi](t)^{\f12},\displaybreak[0]\\
\|\tir^\frac{m}{2} L(v_t^{\f12} \psi)\|_{L_u^\infty L_t^2 L_\omega^2(U)} &\les \sup_{\tau_1-R\le u\le \tau_2-R} CF_m[\psi](u)^{\f12},
\end{align*}
we thus complete the proof for (\ref{engymu}).
\end{proof}


\begin{proposition}\label{e1}
Given $t_0\le \tau_1\le \tau_*$ and let $U=D_{\tau_1,R}^{\tau_*}$. Then for $m=1,2$ there holds
\begin{align}
&\sup_{t\in  [\tau_1, \tau_*]}\|\Phi_m^{-1} \tir^{\frac{m}{2}} \psi(t)\|_{L_u^2 L_\omega^2(U\cap \Sigma_t)}^2
+\|\tir^{\frac{m-1}{2}}L(v_t^{\f12}\psi)\|_{L_u^2 L_t^2 L_\omega^2(U)}^2
+\sup_{t\in [\tau_1, \tau_*]}\E_m[\psi]_{U, R}(t)\nn\\
&\quad +\sup_{\tau_1-R\le u\le \tau_*-R}CF_m[\psi]_R(u, \tau_*)
\les CF_m[\psi]_R(\tau_1-R, \tau_*)+(\tau_1+1)^{m-2}\C_0^2+\B_{\tau_1},\label{e3}
\end{align}
where $\B_{\tau_1}=\sup_{t\in [\tau_1, \tau_*]} \B_{\tau_1}^t$.
\end{proposition}

\begin{remark}\label{rmk1}
By the definition (\ref{btau}) of $\B_{\tau_1}^t$, we have
\begin{align*}
\B_{\tau_1}\les \int_{\tau_1}^{\tau_*} \!\!\int_{S_{t,\tau_1-R}} \tir \left(|L(v_t^{\f12}\psi)|^2+\psi^2\right) d\omega dt
& + \int_{\widetilde{\Sigma}^{\tau_*}_{\tau_1, R}}\P_\a^{(\bT)}[\psi] \bn^\a+\int_{S_{\tau_1,\tau_1-R}} \tir \phi^2 d\omega.
\end{align*}
Recall that $\psi[t_0]$ is supported over $B_r$. If $\tau_1=t_0$, the terms in the above equation integrated on the null boundary
vanishes, which implies that $\B_{t_0} \les \C_0^2$. Similarly, if $\tau_1=t_0$, the right hand side of (\ref{e3}) is bounded by $\C_0^2$.
\end{remark}

We will use (\ref{MA2}) with $R\le R'\le 2R$ to prove Proposition \ref{e1}. From (\ref{MA2}) and the definition of
$CF_m[\psi]_{R'}$ and $\E_m[\psi]_{D_{\tau_1, R}^{\tau_2, \tau}, R'}(\tau)$ it follows that
\begin{align}\label{7.14.1}
\begin{split}
& CF_m[\psi]_{R'}(\tau_2-R, \tau) + \E_m[\psi]_{D_{\tau_1, R}^{\tau_2, \tau}, R'}(\tau)
+ \left\|\tir^{\frac{m-1}{2}} L(v_t^{\f12}\psi)\right\|^2_{L_t^2L_u^2 L_\omega^2(D_{\tau_1, R}^{\tau_2, \tau}\cap \{\tir\ge R'\})} \\
& \les \int_{\tau_1+R'-R}^{\tau+R'-R} \int_{S_{t, t-R'}} \tir^m \left(|\sn \ti\psi|^2 v_t \Omega^{-2} + |L(v_t^{\f12}\psi)|^2\right) d\omega dt \\
& + CF_m[\psi]_{R'}(\tau_1-R, \tau) +\mbox{error}(m, \tir\ge R'),
\end{split}
\end{align}
where
$$
\mbox{error}(m,\tir \ge R') = \int_{D_{\tau_1, R}^{\tau_2, \tau}\cap \{\tir \ge R'\}} |\er_m(\psi)| d\mu_\ga du dt
+ \int_{D_{\tau_1, R}^{\tau_2, \tau}\cap \{\tir \ge R'\}} \tir^m  |\bA| |\J[\psi]|\Omega^{-2} \bb d\omega du dt.
$$
which comes from the term in (\ref{MA2}) with the same notation but with every integrand taken absolute value.
By integrating (\ref{7.14.1}) with respect to $R'$ for $R\le R'\le 2R$, noting that the left hand side of
(\ref{7.14.1}) is decreasing with respect to $R'$, and using the fact $\mbox{error}(m,\tir\ge R')\le \mbox{error}(m,R)$, we can obtain
\begin{align*}
& CF_m[\psi]_{2R}(\tau_2-R, \tau) + \E_m[\psi]_{D_{\tau_1, R}^{\tau_2, \tau}, 2R}(\tau)
+ \left\|\tir^{\frac{m-1}{2}} L(v_t^{\f12}\psi)\right\|^2_{L_t^2L_u^2 L_\omega^2(D_{\tau_1, R}^{\tau_2, \tau}\cap \{\tir\ge 2R\})} \displaybreak[0]\\
& \les \frac{1}{R} \int_{R}^{2R} \int_{\tau_1+R'-R}^{\tau+R'-R} \int_{S_{t, t-R'}} \tir^m \left(|\sn \ti\psi|^2 v_t \Omega^{-2}
+ |L(v_t^{\f12}\psi)|^2\right) d\omega dt dR' \displaybreak[0]\\
& + CF_m[\psi]_{R}(\tau_1-R, \tau) +\mbox{error}(m, R).
\end{align*}
We may employ Proposition \ref{comp3} and (\ref{intes_1}) to derive that
\begin{align*}
& \frac{1}{R} \int_{R}^{2R} \int_{\tau_1+R'-R}^{\tau+R'-R} \int_{S_{t, t-R'}} \tir^m \left(|\sn \ti\psi|^2 v_t \Omega^{-2}
+ |L(v_t^{\f12}\psi)|^2\right) d\omega dt dR' \nn \displaybreak[0]\\
& \le \int_{\{R\le \tir\le 2R\}\cap D_{\tau_1, R}^{\tau_2, \tau}} \left(|L(v_t^{\f12}\psi)|^2+ |\sn \ti\psi|^2 v_t\right) d\omega du dt
\les \B_{\tau_1}^t. 
\end{align*}
Therefore
\begin{align}\label{7.14.3}
\begin{split}
& CF_m[\psi]_{2R}(\tau_2-R, \tau) + \E_m[\psi]_{D_{\tau_1, R}^{\tau_2, \tau}, 2R}(\tau)
+ \left\|\tir^{\frac{m-1}{2}} L(v_t^{\f12}\psi)\right\|^2_{L_t^2L_u^2 L_\omega^2(D_{\tau_1, R}^{\tau_2, \tau}\cap \{\tir\ge 2R\})} \\
& \les \B_{\tau_1} + CF_m[\psi]_{R}(\tau_1-R, \tau) +\mbox{error}(m, R).
\end{split}
\end{align}

\begin{proof}[Proof of Proposition \ref{e1}]
We first use Lemma \ref{lem_629} to estimate $\mbox{error}(m, R)$, we then use Proposition \ref{lot8}
to control the term $\|\tir^{\frac{m}{2}} \psi\|_{L_u^2 L_\omega^2(D_{\tau_1, R}^{\tau_2, \tau}\cap \Sigma_t)}$.
Combining the result with (\ref{7.14.3}) gives
\begin{align}
& CF_m[\psi]_{2R}(\tau_2-R, \tau) + \E_m[\psi]_{D_{\tau_1, R}^{\tau_2, \tau}, 2R}(\tau)
+ \left\|\tir^{\frac{m-1}{2}} L(v_t^{\f12}\psi)\right\|^2_{L_t^2L_u^2 L_\omega^2(D_{\tau_1, R}^{\tau_2, \tau}\cap \{\tir\ge 2R\})} \nn\\
& \les \B_{\tau_1}+(\tau_1+1)^{m-2}\C_0^2+CF_m[\psi](\tau_1-R, \tau_*) + \la^{-2\ep_0} B, \label{6.29.0}
\end{align}
where
\begin{align*}
B = \sup_{t\in [\tau_1, \tau_*]} \E_m[\psi]_{D_{\tau_1, R}^{\tau_*}, R} (t)
+\sup_{\tau_1-R \le u\le \tau_*-R} CF_m[\psi]_R (u, \tau_*)
+\|\tir^{\frac{m-1}{2}}L(v_t^{\f12}\psi)\|_{L_u^2 L_t^2 L_\omega^2(D_{\tau_1, R}^{\tau_*})}^2.
\end{align*}
In view of Proposition \ref{6.27.3} and Proposition \ref{comp3} we can obtain
\begin{align}\label{7.14.4}
\left\|\tir^{\f12} L(v_t^{\f12} \psi)\right\|^2_{L_t^2 L_u^2 L_\omega^2(D_{\tau_1, R}^{\tau_2, \tau}\cap \{R\le \tir \le 2R\})} \les\B_{\tau_1}.
\end{align}
This together with the comparison result in Lemma \ref{correc_2} implies that (\ref{6.29.0}) holds with the subscript $2R$ replaced by
$R$, i.e.
\begin{align}
& CF_m[\psi]_{R}(\tau_2-R, \tau) + \E_m[\psi]_{D_{\tau_1, R}^{\tau_2, \tau}, R}(\tau)
+ \left\|\tir^{\frac{m-1}{2}} L(v_t^{\f12}\psi)\right\|^2_{L_t^2L_u^2 L_\omega^2(D_{\tau_1, R}^{\tau_2, \tau})} \nn\\
& \les \B_{\tau_1}+(\tau_1+1)^{m-2}\C_0^2+CF_m[\psi](\tau_1-R, \tau_*) +
\la^{-2\ep_0} B. \label{6.29.1}
\end{align}
By using (\ref{6.29.1}) with $\tau_2=u+R$ and $\tau=\tau_*$ to estimate
$\|\tir^{\frac{m-1}{2}} L(v_t^{\f12}\psi)\|_{L_t^2 L_u^2 L_\omega^2(D_{\tau_1, R}^{u+R})}$ and
$CF_m[\psi]_R(u, \tau_*)$ for $\tau_1-R\le u \le \tau_*-R$;
by using (\ref{6.29.1}) with $\tau_2=\tau=t$ to estimate $\E_m[\psi]_{D_{\tau_1, R}^{t,t}}(t)$ for $\tau_1\le t\le \tau_*$.
We can conclude
$$
B\les \B_{\tau_1}+(\tau_1+1)^{m-2}\C_0^2+CF_m[\psi](\tau_1-R, \tau_*) + \la^{-2\ep_0} B
$$
which implies the desired estimate for $B$. This together with Proposition \ref{lot8} gives (\ref{e3}).
\end{proof}

In view of Proposition \ref{e1} with Lemma \ref{lem_629}, we immediately obtain

\begin{corollary}\label{corer_1}
Let $t_0\le \tau_1<\tau_2 \le \tau \le \tau_*$. Then for $m=1,2$ there holds
\begin{align*}
\mbox{error}(m, R)&\les \la^{-2\ep_0} \big((\tau_1+1)^{m-2}\C_0^2+CF_m[\psi](\tau_1-R, \tau_*)+\B_{\tau_1}).
\end{align*}
\end{corollary}


\begin{proposition}\label{intg}
\begin{equation}\label{intes_2}
\int_{D_{t_0,R}^{\tau_*}} \tir \left|L(v_t^{\f12} \psi)\right|^2 d\omega du dt\les \C^2_0.
\end{equation}
\end{proposition}

\begin{proof}
We use (\ref{7.14.3}) with $m=2$, $\tau_1=t_0$, $\tau_2 = \tau =\tau_*$. This together with Corollary \ref{corer_1} then implies that
$$
\left\|\tir^{\f12} L(v_t^{\f12} \psi)\right\|^2_{L_t^2 L_u^2 L_\omega^2(D_{t_0, R}^{\tau_*}\cap \{\tir \ge 2R\})}
\les \C_0^2 +\B_{t_0} + CF_2[\psi]_R (t_0-R, \tau_*).
$$
Since $\psi[t_0]$ is supported within $B_R$, $\psi$ must vanish on $\{u\le t_0-R\}$ and hence $CF_2[\psi]_R(t_0-R, \tau_*)=0$.
According to Remark \ref{rmk1}, $\B_{t_0}\les \C_0^2$. Therefore
\begin{align*}
\left\|\tir^{\f12} L(v_t^{\f12} \psi)\right\|^2_{L_t^2 L_u^2 L_\omega^2(D_{t_0, R}^{\tau_*}\cap \{\tir \ge 2R\})} \les\C_0^2.
\end{align*}
This together with (\ref{7.14.4}) with $\tau_1 = t_0$ and $\tau_2=\tau=\tau_*$ gives (\ref{intes_2}).
\end{proof}

\subsection{Step 3: Decay estimates for energy}

We will provide a decay estimate for $\int_{\widetilde{\Sigma}_{\tau}} \P_\a^{(\bT)}[\phi] \bn^\a$ in Proposition \ref{dis2}.
We first give a consequence of (\ref{intes_2}) and (\ref{ieq.2}).

\begin{proposition}\label{disp}
Given $\eta>1$. There exists a sequence $\{\tau_n\}_{n\ge 0} \subset [t_0, \tau_*]$ satisfying
\begin{equation}\label{ratio}
c_1 \tau_n \le \tau_{n+1}\le c_2 \tau_n
\end{equation}
for some constant $c_2>c_1>1$ such that
\begin{equation}\label{dsq}
\int_{\tau_n}^{\tau_*} \int_{S_{t, \tau_n-R}} \tir \left|L(v_t^{\f12} \psi)\right|^2 d\omega dt+R^2\int_{S_{\tau_n, \tau_n- R}} \psi^2 d\omega  \les (1+\tau_n)^{-1}\C_0^2.
\end{equation}
Consequently, for $\tau_n\le \tau \le \tau_*$ there holds
\begin{equation}\label{dsq2}
\int_{\tau_n}^\tau \int_{S_{t, \tau_n-R}} \tir \psi^2\les (\ln (\tau-\tau_n+R))^2(1+\tau_n)^{-1}\C_0^2.
\end{equation}
\end{proposition}

\begin{proof}
In view of (\ref{intes_2}), (\ref{ieq.2}) and $\B_{t_0}\les \C_0^2$ from Remark \ref{rmk1}, there is a universal constant $B$ such that
\begin{equation}\label{7.13.1}
\int_{D_{t_0, R}^{\tau_*}} \tir \left|L(v_t^{\f12} \psi)\right|^2 d\omega du dt + R^2 \int_{t_0}^{\tau_*} \int_{S_{t,t-R}} \psi^2 d \omega d\tau \le B \C_0^2.
\end{equation}
Fix $1<\rho<\eta$. If, for some $n$ there is no $\tau_n\in[\eta^n, \rho\eta^n]\subset [t_0, \tau_*]$ such that (\ref{dsq}) holds,
then for $B'> 2 B/\ln \rho$ we have
\begin{align*}
\int_{\eta^n}^{\rho\eta^n} \left(\int_\tau^{\tau_*} \int_{S_{t, \tau-R}} \tir \left|L(v_t^{\f12} \psi)\right|^2 d\omega d t
+R^2\int_{S_{\tau, \tau-R}} \psi^2 d\omega \right) d \tau
& \ge B' \int_{\eta^n}^{\rho\eta^n} (1+\tau)^{-1} \C_0^2 d\tau \\
& \ge \frac{1}{2} B' (\ln \rho) \C_0^2>B \C_0^2
\end{align*}
which contradicts to (\ref{7.13.1}).  This shows the existence of $\tau_n$ satisfying (\ref{ratio}) and (\ref{dsq}).
From (\ref{lot7}) and (\ref{dsq}) we immediately obtain (\ref{dsq2}).
\end{proof}

\begin{corollary}\label{decay_1}
Let $\{\tau_n\}_{n\ge 1} \subset [t_0, \tau_*]$ be the sequence obtained in Proposition \ref{disp}. Then
\begin{equation*}
\B_{\tau_n}=\sup_{t\in [\tau_n, \tau_*]}\B_{\tau_n}^t
\les (1+\tau_n)^{-1}\C_0^2 +\int_{\widetilde{\Sigma}_{\tau_n, R}^{\tau_*}} \P_\a^{(\bT)}[\psi] {\bf n}^\a.
\end{equation*}
\end{corollary}

\begin{proof}
In view of (\ref{btau}) and Proposition \ref{disp}, we derive that
\begin{equation*}
\B_{\tau_n}^t \les\left(\frac{(\ln(t-\tau_n+R))^{2}}{t-\tau_n+R}+1\right)(1+\tau_n)^{-1}\C_0^2
+ \int_{\widetilde{\Sigma}_{\tau_n, R}^{\tau_*}} \P_\a^{(\bT)}[\psi] {\bf n}^\a
\end{equation*}
which implies the desired estimate.
\end{proof}

Now we give a  result  which is  crucial to prove (\ref{con1}), (\ref{con}) and (\ref{con2}).

\begin{proposition}\label{dis2}
Let $\psi$ be any solution of $\Box_{\bg} \psi=0$ with $\psi[t_0]$ supported on $B_R\subset \D_0^+\cap \{t=t_0\}$.
There holds the decay estimate
\begin{equation}\label{maincon}
\int_{\widetilde{\Sigma}_{\tau, R}^{\tau_*}} \P_\a^{(\bT)}[\psi] {\bf n}^\a\les (1+\tau)^{-2}\C_0^2, \qquad \forall t_0\le \tau \le \tau_*.
\end{equation}
\end{proposition}

\begin{proof}
Let $\{\tau_n\}\subset [t_0, \tau_*]$ be obtained in Proposition \ref{disp}. We may apply (\ref{MA2}) with $m=1$ over
$D_{\tau_n, R}^{\tau_{n+1}} \cap \{\tir \ge R'\}$ with $R\le R'\le 2R$ to obtain
\begin{align}
&\int_{D^{\tau_{n+1}}_{\tau_n, R}\cap \{\tir\ge R'\}}\bb\left((L\psi)^2+|\sn \ti\psi|^2 v_t\Omega^{-2}\right) d\omega du dt \nn\\
& \le \int_{\tau_n+R'-R}^{\tau_{n+1}+R'-R} \int_{S_{t, t-R'}} \tir \left(\bb |\sn\ti\psi|^2v_t\Omega^{-2}-(2-\bb) |L(v_t^{\f12}\psi)|^2\right) d\omega dt \nn\\
& +\int_{\tau_n+R'-R}^{\tau_*} \int_{S_{t, \tau_n-R}} \tir |L(v_t^{\f12}\psi)|^2 d\omega dt +\mbox{error}(1,R). \label{6.29.4}
\end{align}
According to Corollary \ref{corer_1}, (\ref{dsq}) and Corollary \ref{decay_1}, we have
\begin{equation}\label{6.29.3}
\mbox{error}(1,R)\les \la^{-2\ep_0}((\tau_n+1)^{-1}\C_0^2+\B_{\tau_n})
\les (1+\tau_n)^{-1}\C_0^2 +\int_{\widetilde{\Sigma}_{\tau_n, R}^{\tau_*}} \P_\a^{(\bT)}[\psi] {\bf n}^\a.
\end{equation}
Let $D_n= \{R\le \tir \le 2R\}\cap \{\tau_n-R\le u\le \tau_{n+1}-R\}\cap \D_0^+$. Noting that $D_n \subset K_{\tau_n, 2R}^{\tau_*}$, we may use
Proposition \ref{comp3}, (\ref{intes_1}) and Corollary \ref{decay_1} to deduce that
\begin{align*}
& \int_{R}^{2R} \int_{\tau_n+R'-R}^{\tau_{n+1}+R'-R} \int_{S_{t, t-R'}}
\tir \left(\bb |\sn\ti\psi|^2v_t\Omega^{-2}-(2-\bb) |L(v_t^{\f12}\psi)|^2\right) d\omega dt \\
& \les \int_{D_n} (\P_\a^{(\bT)}[\psi] \bn^\a+\frac{\psi^2}{\tir})
\les \B_{\tau_n}^{\tau_*} \les (1+\tau_n)^{-1}\C_0^2 + \int_{\widetilde{\Sigma}_{\tau_n, R}^{\tau_*}} \P_\a^{(\bT)}[\psi] \bn^\a.
\end{align*}
Therefore, by integrating the both sides of (\ref{6.29.4}) for $R\le R'\le 2R$ and using the monotonicity of the left
hand side with respect to $R'$, we obtain
\begin{align*}
\int_{D_{\tau_n, R}^{\tau_{n+1}}\cap \{\tir \ge 2R\}} \bb(|L(v_t^{\f12}\psi)|^2+|\sn \ti\psi|^2\Omega^{-2} v_t) d\omega du dt
\les (1+\tau_n)^{-1}\C_0^2+ \int_{\widetilde{\Sigma}_{\tau_n, R}^{\tau_*}} \P_\a^{(\bT)}[\psi] \bn^\a.
\end{align*}
In view of Proposition \ref{comp3}, (\ref{intes_1}) and Corollary \ref{decay_1}, we can improve the above estimate to be
\begin{align}\label{ieq.1}
\int_{D_{\tau_n, R}^{\tau_{n+1}}} \bb(|L(v_t^{\f12}\psi)|^2+|\sn \ti\psi|^2\Omega^{-2} v_t) d\omega du dt
\les (1+\tau_n)^{-1}\C_0^2+ \int_{\widetilde{\Sigma}_{\tau_n, R}^{\tau_*}} \P_\a^{(\bT)}[\psi] \bn^\a.
\end{align}

On the other hand, by using Proposition \ref{6.27.3} we have
\begin{align}\label{7.15.3}
 \int_{\tau_n}^{\tau_{n+1}} \!\!\int_{\widetilde\Sigma_{\tau, R}^{\tau_*}} \P_\a^{(\bT)}[\psi] \bn^\a
& =\int_{\tau_n}^{\tau_{n+1}} \!\!\int_{\Sigma_\tau\cap \{\tir \le R\}} \P_\a^{(\bT)}[\psi] \bn^\a
+\int_{\tau_n}^{\tau_{n+1}}\!\! \int_{C_{\tau-R}\cap\{\tir\ge R\}} \P_\a^{(\bT)}[\psi] \bn^\a \nn \\
&\les \B_{\tau_n} + \int_{\tau_n}^{\tau_{n+1}} \!\!\int_{C_{\tau-R}\cap\{\tir\ge R\}} \left(|L\psi|^2+|\sn \psi|^2\right) \bb v_td\omega dt d\tau.
\end{align}
In view of $L(v_t^{\f12} \psi) = (L\psi + \f12 \tr\chi \psi) v_t^{\f12}$, $\tr\ti \chi = \tr \chi + V_4$ and
$\sn \ti \psi = \Omega (\sn \psi - \psi \sn \sigma)$, it is easily seen that
\begin{align}\label{7.15.1}
\left(|L(v_t^{\f12}\psi)|^2 + |\sn \ti \psi|^2 \Omega^{-2} v_t\right) \bb
= \left(|L\psi|^2+|\sn \psi|^2\right) \bb v_t + L(\f12 \tr \ti \chi \bb v_t \psi^2) + \I(\psi),
\end{align}
where
\begin{align*}
\I(\psi) &= \bb \left[ -L(\f12 \tr\ti\chi v_t) +(\f12 \tr \chi)^2 v_t -\f12 L(\log \bb) \tr \ti \chi  v_t + |\sn \sigma|^2 v_t\right] \psi^2 \\
& \quad \, - \left[ V_4 \psi L\psi + 2\psi \sn \psi \sn \sigma\right] \bb v_t
\end{align*}
Integrating (\ref{7.15.1}) over $C_{\tau-R}\cap \{\tir \ge R\}$ gives
\begin{align} 
&\int_{C_{\tau-R}\cap \{\tir\ge R\}} \left(|L \psi|^2+|\sn \psi|^2\right) \bb v_t d\omega dt \nn\\
& = \int_{C_{\tau-R}\cap\{\tir\ge R\}} \bb\left(|L(v_t^{\f12} \psi)|^2+|\sn \ti\psi|^2 \Omega^{-2} v_t\right) d\omega dt
-\f12 \int_{S_{\tau_*, \tau-R}} \tr \ti \chi \psi^2 \bb v_t d\omega \nn\\
& + \f12 \int_{S_{\tau, \tau-R}} \tr \ti \chi \psi^2 \bb v_t d\omega -\int_{C_{\tau-R}\cap\{\tir\ge R\}}\I(\phi) d\omega dt\nn
\end{align}
Combining this with (\ref{7.15.3}) and using (\ref{ieq.1}) and Corollary \ref{decay_1} we obtain
\begin{align}\label{4.21.3}
\int_{\tau_n}^{\tau_{n+1}} \!\!\int_{\widetilde\Sigma_{\tau, R}^{\tau_*}} \P_\a^{(\bT)}[\psi] \bn^\a
& \les \int_{\tau_n}^{\tau_{n+1}} \!\!\int_{S_{\tau, \tau-R}} \tr \ti \chi \psi^2 \bb v_t d\omega d\tau
 +\int_{\tau_n}^{\tau_{n+1}} \!\!\int_{C_{\tau-R}\cap\{\tir\ge R\}}|\I(\psi)| d\omega dt d\tau \nn\\
& + (1+\tau_n)^{-1}\C_0^2+ \int_{\widetilde{\Sigma}_{\tau_n, R}^{\tau_*}} \P_\a^{(\bT)}[\psi] \bn^\a.
\end{align}
By using $\tr \ti \chi v_t\approx \tir$ , $|\bb-1|<\f12$,  (\ref{ieq.2})  and Corollary \ref{decay_1}, we have
\begin{align*}
\int_{\tau_n}^{\tau_{n+1}}&\int_{S_{\tau, \tau-R}} \tr \ti \chi \psi^2 \bb v_t d\omega d\tau
\les \B_{\tau_n}\les (1+\tau_n)^{-1}\C_0^2 +\int_{\widetilde{\Sigma}_{\tau_n, R}^{\tau_*}} \P_\a^{(\bT)}[\psi] \bn^\a.
\end{align*}
We need to estimate the term involving $\I(\psi)$. Let $\bA$ denote any term among $\{k, \hat \chi, V_4\}$. Since
$L(\log \bb)=-k_{NN}$ and
\begin{align*}
 L(\tr \ti \chi v_t)-\f12(\tr\chi)^2 v_t=v_t(-V_4^2 +V_4 \tr \ti \chi-|\hat \chi|^2-k_{NN} \widetilde\tr\chi+\ti \bE)
\end{align*}
with $\ti \bE$ being of the form $\bA\cdot \bA + \tr\ti \chi\cdot \bA$, see Section 5, thus symbolically we have
$$
|\I(\psi)| \les \left(|\bA|^2 + \tr \ti \chi |\bA| + |\sn \sigma|^2 \right) \psi^2 v_t +\left(|V_4 \psi L\psi| + |\psi\sn \psi \sn \sigma|\right) v_t.
$$
Let $q$ and $q'$ be such that $\frac{1}{q}+\frac{1}{q'}=\f12$ and $0<1-\frac{2}{q'}<s-2$.  Recall that $\widetilde{\tr\chi}v_t\approx \tir$
and $|\bb-1|<\f12$, we may use (\ref{bba2}), Proposition \ref{awave}  and (\ref{l4.1}) to derive that
\begin{align*}
& \int_{\tau_n}^{\tau_{n+1}} \int_{C_{\tau-R}\cap\{\tir\ge R\}} \bb(|\sn \sigma|^2, |\bA|^2, \widetilde{\tr\chi} |\bA|) \psi^2 v_t d\omega dt d\tau\\
&\les\tau_*^{\f12}\|\psi\|^2_{L^2_{[\tau_n, \tau_*]} L_u^2 L_\omega^{q}} \|\tir^{\frac{3}{2}} (|\bA|^2+\widetilde{\tr\chi} |\bA|+|\sn \sigma|^2)\|_{L^\infty L_\omega^{\frac{q'}{2}}}
\les \la^{-\ep_0} (\tau_n+1)^{-1}\C_0^2.
\end{align*}
By using (\ref{bba1}) and (\ref{bba2}), we also have
\begin{align*}
& \int_{\tau_n}^{\tau_{n+1}} \int_{C_{\tau-R}\cap \{\tir \ge R\}} \bb|\psi\sn \psi \sn \sigma| v_t d\omega dt d\tau\\
&\les \|\psi\|_{L^2_{[\tau_n, \tau_*]} L_u^2 L_\omega^{q}}
\left(\int_{\tau_n}^{\tau_{n+1}}\|\sn  \psi(t)\|_{L_u^2 L_x^2}^2 dt\right)^\f12 \|\tir \sn \sig\|_{L^\infty L_\omega^{q'}}
\les (\tau_n+1)^{-1} \la^{-\ep_0} \C_0^2.
\end{align*}
Furthermore, by using (\ref{baen}), (\ref{lowba}) and (\ref{pi.2}), we can obtain
\begin{align*}
&\int_{\tau_n}^{\tau_{n+1}}\int_{C_{\tau-R}\cap \{\tir\ge R\}} |\psi L\psi V_4| v_t du dt d\omega\\
&\les\sup_{\tau_n\le t\le \tau_{n+1}}\|\psi\|_{L_u^2 L_\omega^2}\|\tir^2 L\psi\|_{L_t^\infty L_u^2 L_\omega^2}\|V_4\|_{L_t^1 L_x^\infty}
\les (\tau_n+1)^{-1} \la^{-6\ep_0} \C_0^2.
\end{align*}
Therefore
\begin{equation*}
\int_{\tau_n}^{\tau_{n+1}} \int_{C_{\tau-R}\cap\{\tir\ge R\}}|\I(\psi)| d\omega dt d\tau
\les \la^{-\ep_0} (1+\tau_n)^{-1}\C_0^2.
\end{equation*}
Combining the above estimates with (\ref{4.21.3}) we can conclude that
\begin{equation}
\int_{\tau_n}^{\tau_{n+1}} \int_{\widetilde\Sigma_{\tau, R}^{\tau_*}}\P_\a^{(\bT)}[\psi] \bn^\a
\les (1+\tau_n)^{-1}\C_0^2+\int_{\widetilde\Sigma_{\tau_n, R}^{\tau_*}} \P_\a^{(\bT)}[\psi] \bn^\a.\label{ind1}
\end{equation}
Recall that Lemma \ref{monoeng} implies
\begin{equation*}
\int_{\widetilde{\Sigma}_{\tau_{n+1}, R}^{\tau_*}} \P_\a^{(\bT)}[\phi] \bn^\a
\le 2\int_{\widetilde{\Sigma}_{\tau, R}^{\tau_*}} \P_\a^{(\bT)}[\psi] \bn^\a,
\qquad \forall t_0\le \tau \le \tau_{n+1}.
\end{equation*}
We may use (\ref{ind1}) to conclude
\begin{align}
(\tau_{n+1}-\tau_n)\int_{\widetilde{\Sigma}_{\tau_{n+1}, R}^{\tau_*}}& \P_\a^{(\bT)}[\psi] \bn^\a
\les (1+\tau_n)^{-1}\C_0^2+ \int_{\widetilde\Sigma_{\tau_n, R}^{\tau_*}} \P_\a^{(\bT)}[\psi] \bn^\a.\label{bd_4.21}
\end{align}
By using  (\ref{dec_1}) with $\tau'=t_0$, we obtain for all $n\ge 0$ that
\begin{equation*}
\int_{\widetilde\Sigma_{\tau_n, R}^{\tau_*}} \P_\a^{(\bT)}[\psi] \bn^\a \les \C_0^2.
\end{equation*}
Since $\tau_{n+1}-\tau_n \gtrsim \tau_{n+1}$, from (\ref{bd_4.21}) we can deduce for all $n$ that
\begin{equation*}
\int_{\widetilde{\Sigma}_{\tau_n, R}^{\tau_*}}\P_\a^{(\bT)}[\psi] \bn^\a
\les (\tau_n+1)^{-1}\C_0^2.
\end{equation*}
Substituting this estimate into (\ref{bd_4.21}) yields
\begin{equation*}
(\tau_{n+1}-\tau_n)\int_{\widetilde{\Sigma}_{\tau_n, R}^{\tau_*}} \P_\a^{(\bT)}[\psi] \bn^\a
\les (1+\tau_n)^{-1}\C_0^2
\end{equation*}
which implies
\begin{equation*}
\int_{\widetilde{\Sigma}_{\tau_n, R}^{\tau_*}} \P_\a^{(\bT)}[\psi] \bn^\a\les (1+\tau_n)^{-2}\C_0^2, \qquad \forall n\ge 0.
\end{equation*}
For any $\tau \in [t_0, \tau_*]$, by the construction of $\{\tau_n\}$ it is always possible to find $\tau_n\le \tau$ such that
$\tau \les \tau_n$. Thus, by using (\ref{dec_1}) in Lemma \ref{monoeng}, we can conclude
\begin{equation*}
\int_{\widetilde{\Sigma}_{\tau, R}^{\tau_*}} \P_\a^{(\bT)}[\psi] \bn^\a
\les \int_{\widetilde{\Sigma}_{\tau_n, R}^{\tau_*}} \P_\a^{(\bT)}[\psi] \bn^\a
\les (1+\tau_n)^{-2} \C_0^2 \les (1+\tau)^{-2} \C_0^2.
\end{equation*}
The proof is therefore complete.
\end{proof}

\subsection{Proof of  Theorem \ref{BT}}

We first show (\ref{con}). This follows from (\ref{std}) and Proposition \ref{dis2} as we have
\begin{align}\label{f2}
\int_{\Sigma_t\cap\{u\ge \frac{t}{2}\}} \P_\a^{(\bT)}[\psi] \bn^\a
\les \int_{\widetilde{\Sigma}_{\frac{t}{2}+R, R}^t} \P_\a^{(\bT)}[\psi] \bn^\a\les (1+t)^{-2} \C_0^2.
\end{align}
In the interior region $\{u\le \frac{t}{2}\}\cap D_0^+$, (\ref{con1}) is a direct consequence of (\ref{con}).
In the exterior region $\{u\le  \frac{3t}{4}\}\cap \D_0^+$, we use (\ref{e3}) with $m=2$. By taking $\tau_1=t_0$,
using $\B_{t_0}\les \C_0^2$ and noting that $CF_2[\psi](t_0-R, \tau_*)=0$, we can obtain
\begin{align}
&\int_{\Sigma_t\cap\{u\le \frac{3}{4}t\}} \tir^2 \left(|L(v_t^{\f12} \psi)|^2 + |\sn \ti\psi|^2 v_t \Omega^{-2}\right) d\omega du \les \C_0^2, \label{7.16.1}\\
& \|\psi\|_{L_u^2 L_\omega^2(\Sigma_t\cap \{u\le \frac{3}{4}t\})} \les (1+t)^{-1+} \C_0. \label{7.16.2}
\end{align}
In view of Lemma \ref{com5} and (\ref{7.16.1}), we also obtain (\ref{con1}) in the exterior region and thus (\ref{con1}) follows.
To obtain (\ref{con2}), in view of (\ref{7.16.2}), it suffices to show (\ref{con2}) in the interior region, i. e.
\begin{equation*}
\|\psi(t)\|_{L_u^2 L_\omega^2(\Sigma_t\cap \{u\ge \frac{t}{2}\})} \les  (1+t)^{-1+}\C_0
\end{equation*}
To see this, according to the construction of $\{\tau_n\}$ we can find $\tau_n \le \frac{t}{2}$ such that $t\les \tau_n$.
Let $u_n=\tau_n-R$. We can integrate $L(v_t^{\f12} \psi)$ along null geodesics on $C_{u_n}$ to obtain
\begin{equation*}
(v_t^{\f12}\psi)(t, u_n, \omega)-(v_t^{\f12}\psi)(\tau_n, u_n, \omega)=\int_{\tau_n}^t L(v_{t'}^{\f12} \psi)(t', u_n, \omega) dt'
\end{equation*}
which together with $v_t\approx (t-u)^2$ implies
\begin{align*}
\int_{S_{t,u_n}} \psi^2 \tir^2 d\omega &\les R^2 \int_{S_{\tau_n, u_n}} \psi^2 d\omega
+ \log (t-\tau_n+R) \int_{\tau_n}^t \int_{S_{t', u_n}} \tir |L(v_{t'}^{\f12}\psi)|^2 d\omega dt'.
\end{align*}
In view of (\ref{dsq}), we then obtain
 \begin{equation*}
\int_{S_{t,u_n}} \psi^2 \tir^2 d\omega\les (1+\tau_n)^{-1} (\log (t-\tau_n+R)+1)\C_0^2.
 \end{equation*}
Recall that $\tau_n \le \frac{t}{2} \les \tau_n$, we have $t-u_n \ge \frac{t}{2}$. Thus we can conclude that
 \begin{equation}\label{e4}
\int_{S_{t,u_n}} \psi^2 \tir d\omega\les (1+t)^{-2} \log (t+1)\C_0^2.
 \end{equation}
Now we use (\ref{lot4}) from Lemma \ref{lem7.15.1} which gives
\begin{equation*}
\|\psi\|_{L_u^2 L_\omega^2(\Sigma_t\cap \{u\ge u_n\})}^2
\les \|r N(\psi)\|_{L_u^2 L_\omega^2(\Sigma_t\cap\{u\ge u_n\})}^2+\|r^{\f12}\psi\|_{L^2(S_{t,u_n})}^2.
\end{equation*}
By Lemma \ref{monoeng},  (\ref{e4}) and Proposition \ref{dis2},  we obtain
\begin{align*}
\|\psi\|_{L_u^2 L_\omega^2(\Sigma_t\cap \{u\ge u_n\}}^2&\les (1+t)^{-2+} \C_0^2 +\int_{{\widetilde\Sigma}_{\tau_n}}\P_\a^{(\bT)}[\psi] {\bf n}^\a
\les (1+t)^{-2+} \C_0^2.
\end{align*}
This immediately gives the desired estimate because $\{u\ge \frac{t}{2}\} \subset \{u\ge u_n\}$.

\section{\bf Appendix A: Proof of Lemma \ref{pres}}

In this subsection, we give the proof of Lemma \ref{pres}. We will rely on the following product estimates.

\begin{lemma}
Let $0<\ep<\f12$. Then for any scalar functions $f$ and $G$ there hold
\begin{align}
\|f G\|_{\dot{H}^\ep}&\les \|f\|_{L_x^\infty}\|G\|_{\dot{H}^\ep}+\|\p f\|_{L_x^6}\|G\|_{L_x^2},\label{prd3}\\
\|\mu^\ep P_\mu(fG)\|_{L_x^\infty}&\les \mu^{-\f12+\ep}\|\p f\|_{L_x^\infty}\|G\|_{L_x^6}
+\|f\|_{L_x^\infty}\|\mu^{\ep} P_\mu G\|_{L_x^\infty}, \mbox{ for } \mu>1.\label{prd7}
\end{align}
\end{lemma}

\begin{proof}
For $1\le q\le \infty$ write
\begin{align}\label{w7.29.1}
\| P_\mu( f G) \|_{L_x^q}\le \|[P_\mu, f]G\|_{L_x^q}+\|f P_\mu G\|_{L_x^q}.
\end{align}
By using \cite[Eq. (6.195)]{Wangrough} we have
\begin{equation}\label{imu}
\|[P_\mu, f]G\|_{L_x^q} \les \|[P_\mu, f]G_{\le\mu}\|_{L_x^q}+ \Big\|\sum_{\ell>\mu}P_\mu( P_\ell f\c P_\ell G)\Big\|_{L_x^q}.
\end{equation}
To prove (\ref{prd3}), we proceed with $q=2$ in (\ref{w7.29.1}) and (\ref{imu}). For the first term on the right of (\ref{imu}),
by using \cite[Corollary 1]{Wangrough} and  the Bernstein inequality  we derive that
\begin{equation*}
 \|[P_\mu, f] G_{\le\mu}\|_{L_x^2}\les \mu^{-1}\| \p f\|_{L_x^6}\|G_{\le \mu}\|_{L_x^3}\les \mu^{-\f12}\|\p f\|_{L_x^6}\|G\|_{L_x^2} .
\end{equation*}
Hence
\begin{align*}
\|\mu^\ep [P_\mu, f]G_{\le\mu}\|_{l_\mu^2 L_x^2}&\les \|\p f\|_{L_x^6}\|G\|_{L_x^2}.
\end{align*}
For the second term on the right of (\ref{imu}), we have
\begin{align*}
\mu^\ep \Big\|\sum_{\ell>\mu} P_\mu( P_\ell f\c P_\ell G)\Big\|_{L_x^2}
\les \sum_{\ell>\mu}(\mu\ell^{-1})^\ep\|P_\ell f\|_{L_x^\infty} \|\ell^{\ep} P_\ell G\|_{L_x^2},
\end{align*}
by taking $l_\mu^2$,  we derive that
\begin{equation}\label{imu_1}
\|\mu^\ep [P_\mu, f]G\|_{l_\mu^2}\les\|f\|_{L_x^\infty} \|G\|_{\dot{H}^\ep} +\|\p f\|_{L_x^6}\|G\|_{L_x^2}.
\end{equation}
It is straightforward to get that
\begin{equation}\label{jmu}
\|\mu^\ep (f P_\mu G)\|_{l_\mu^2 L_x^2}\le \|f\|_{L_x^\infty}\|G\|_{\dot{H}^{\ep}}.
\end{equation}
Combining (\ref{imu_1}) with (\ref{jmu}), we obtain (\ref{prd3}).

Next we prove (\ref{prd7}) by employing (\ref{w7.29.1}) and (\ref{imu}) with $q=\infty$. By using \cite[Corollary 1]{Wangrough}
and the Bernstein inequality, we obtain
\begin{align}
\mu^\ep\|[P_\mu, f]G_{\le \mu}\|_{L_x^\infty}&\les \mu^{-1+\ep} \|\p f\|_{L_x^\infty}\|G_{\le \mu}\|_{L_x^\infty}
\les \mu^{-\f12+\ep}\|\p f\|_{L_x^\infty}\|G\|_{L_x^6}.\label{w7.29.3}
\end{align}
By the Berntein inequality and the finite band property, we have
\begin{align}
\mu^\ep \Big\|\sum_{\ell>\mu>1} P_\mu(P_\ell f\c P_\ell G) \Big\|_{L_x^\infty}&\les \sum_{\ell>\mu>1} \mu^{\f12+\ep}\|P_\ell f\|_{L_x^\infty} \|P_\ell G\|_{L_x^6}\nn\\
&\les\sum_{\ell>\mu>1}\ell^{-\f12+\ep}(\frac{\mu}{\ell})^{\f12+\ep}\|G\|_{L_x^6}\|\p f\|_{L_x^\infty}\label{w7.29.2}
\end{align}
Combining (\ref{w7.29.3}) with (\ref{w7.29.2}), we can obtain (\ref{prd7}).
\end{proof}

\begin{proof}[Proof of Lemma  \ref{pres}]

We prove (\ref{nphi}) for $m=1$; the case for $m=0$ follows similarly. In view of the finite band property,
it suffices to consider $P_\mu\p(\W(\phi)\bp \phi)$. We write
\begin{equation}\label{n2}
\p(\W(\phi)\bp \phi)=\W'(\phi) (\bp \phi)^2 +\W(\phi) \p\bp\phi.
\end{equation}
Let $f=\W'(\phi)$ and $G= (\bp\phi)^2$, by \cite[Lemma 17]{Wangrough} we have
\begin{equation}\label{g_6_2}
\|G\|_{\dot{H}^\ep}\les \|\bp\phi\|_{H^{\f12+\ep}}\|\bp \phi\|_{H^1}.
\end{equation}
Then, in view of (\ref{prd3}), the Sobolev embedding and (\ref{eng00}), we deduce that
\begin{align}
\|fG\|_{\dot{H}^\ep}&\les \|\bp\phi\|_{H^1}^3+\|\W'(\phi)\|_{L_x^\infty} \|\bp \phi\|_{H^1}\|\bp \phi\|_{H^{\f12+\ep}}
\les \|\phi[0]\|_{H^2}^3.\label{est_1}
\end{align}
With $f=\W(\phi)$ and $G=\p\bp \phi$, also using Sobolev embedding and (\ref{eng00})
\begin{equation*}
\|f G\|_{\dot{H}^\ep}\les \|\p\bp \phi\|_{\dot{H}^\ep}\|\W(\phi)\|_{L_x^\infty}+\|\p\bp\phi\|_{L^2}\|\p\W(\phi)\|_{L_x^6}
\les \|\p\bp \phi\|_{\dot{H}^\ep}+\|\phi[0]\|_{H^2}.
\end{equation*}
In view of (\ref{n2}), we complete the proof of (\ref{nphi}).

To show (\ref{prd1}), we write $\N(\phi, \bp\phi)=f G$ with $f = \N(\phi)$ and $G=\bp \phi\c \bp \phi$.
By using (\ref{g_6_2}) and the similar argument for deriving (\ref{est_1}) we can obtain (\ref{prd1}).

Now we consider (\ref{prd}). We write
\begin{align}\label{6_2.1}
\bp(\N(\phi, \bp \phi))=\N'(\phi) (\bp \phi)^3+\V(\phi) \bp\p \phi\c\bp\phi
\end{align}
where $\V$ are products of factors in $(g, \N)$. The term $\p_t^2 \phi$ does not appear in (\ref{6_2.1}) since it can be replaced  by $g(\phi)\p^2\phi$ and $\N(\phi, \bp\phi)$ in view of (\ref{wave1}).
For the first term on the right of (\ref{6_2.1}), by using \cite[Lemma 18]{Wangrough} and (\ref{prd3}) we derive that
\begin{align}
\|\la^\ep P_\la \left(\N'(\phi) (\bp \phi)^3\right)\|_{l_\la^2 L_x^2}&
\les \|\N'(\phi)\|_{L_x^\infty} \|\la^\ep P_\la ((\bp\phi)^3)\|_{L_x^2}+\|\p \N'(\phi)\|_{L_x^6}\|\bp \phi\|_{L_x^6}^3\nn\\
&\les (\|\bp \p \phi\|_{H^\ep}+\|\bp\phi\|_{L_x^6}^2)\|\bp \phi\|_{L_x^6}^2.\label{tria}
\end{align}
For the second term on the right of (\ref{6_2.1}), let $G= \V(\phi)\bp \phi$ and use the trichotomy law we can write
\begin{align*}
P_\la(G\c \p\bp \phi)&=P_\la  (G_\la (\bp \p \phi)_{\le\la})+\sum_{\mu>\la}P_\la(G_\mu (\bp \p \phi)_\mu)\\
&+P_\la (G_{\le \la}  (\bp\p\phi)_\la)
=a_\la+b_\la+c_\la.
\end{align*}
By the finite band property and (\ref{nphi}), we have
\begin{align}
\|\la^\ep a_\la\|_{l_\la^2 L_x^2}&\les \|\bp \phi\|_{L_x^\infty } \|\la^{\ep+1} G_\la\|_{l_\la^2 L_x^2}
\les \|\bp \phi\|_{L_x^\infty}\left(\|\bp \phi\|_{H^{1+\ep}}+\|\phi[0]\|_{H^2}\right) .\label{a_324}
\end{align}
and
\begin{align*}
\|\la^\ep b_\la\|_{L_x^2}&\les \|\bp \phi\|_{L_x^\infty}\sum_{\mu>\la}( \frac{\la}{\mu})^\ep\|\mu^\ep(\p \bp \phi)_\mu)\|_{L_x^2}
\end{align*}
Hence
\begin{equation}\label{trib}
\|\la^\ep b_\la\|_{l_\la^2 L_x^2}\les \|\bp\phi\|_{L_x^\infty} \|\mu^\ep(\bp \p\phi)_\mu\|_{l_\mu^2 L_x^2}.
\end{equation}
Finally
\begin{align}\label{tric}
\|\la^\ep c_\la\|_{l_\la^2 L_x^2}\les \|\bp \phi\|_{L_x^\infty}\|\la^\ep(\bp\p \phi)_\la \|_{l_\la^2 L_x^2}.
\end{align}
Combining (\ref{a_324}), (\ref{trib}) and (\ref{tric}), we obtain
\begin{align}\label{w8.9.1}
\|\la^{\ep}P_\la (G\c \p\bp \phi)\|_{l_\la^2 L_x^2}
&\les \|\bp \phi\|_{L_x^\infty}\left(\|\bp \phi\|_{H^{1+\ep}}+\|\phi[0]\|_{H^2}\right),
\end{align}
 which, combined with (\ref{tria}), implies (\ref{prd}). (\ref{prd_2}) follows from (\ref{prd3}).
We thus complete the proof of Lemma \ref{pres}.
\end{proof}

\section{\bf Appendix B: Proof of Theorem \ref{str2}}\label{apd_tt}

Using Theorem \ref{decayth},  we will prove Theorem \ref{str2} by  a $\T\T^*$ argument as in \cite{Wangrough}; see also \cite{KRduke,KR1}.



\begin{definition}
Let $\omega:=(\omega_0, \omega_1)\in H^1({\mathbb R}^3)\times L^2({\mathbb R}^3)$.
We denote by $\psi(t;s,\omega)$ the unique solution of the homogeneous geometric wave
equation $\Box_\bg \psi=0$ satisfying the initial condition $\psi(s;s,\omega)=\omega_0$
and $\p_t \psi(s;s,\omega)=\omega_1$. We set $\Psi(t;s,\omega):=(\psi(t;s,\omega),
\p_t\psi(t;s,\omega))$. By uniqueness we have $\Psi\left(t;s,\Psi(s;t_0,\omega)\right)=\Psi(t;t_0,\omega)$.
\end{definition}

We first show that
\begin{equation}\label{pen}
\|P (\p_t\psi)\|_{L^q_{I_*} L_x^\infty} \les \| \bp \psi(0)\|_{L^2({\mathbb R}^3)}.
\end{equation}
To this end, we let $\H:= \dot{H}^1({\mathbb R}^3)\times L^2({\mathbb R}^3)$ endowed with the inner product
\begin{equation*}
\l \omega, v\r=\int_{\Sigma}\big(\omega_1\c v_1+\delta^{ij} \bd_i \omega_0\c \bd_j v_0\big)
\end{equation*}
relative to an orthonormal frame $\{e_i=1,2,3\}$ in ${\mathbb R}^3$. Let $I=[t',t_*]$ with $0\le t'\le t_*$  and let $X=L^q_{I} L_x^\infty$.
 Then the dual of $X$ is $X'=L^{q'}_{I} L_x^1$, where $1/q'+1/q=1$. Let $\T(t'): \H \to X$ be the
 linear operator defined by
\begin{equation*}
\T(t') \omega:=P \p_t \psi(t;t',\omega).
\end{equation*}
By the Bernstein inequality for LP projections and the energy estimate it is easy to see that
\begin{equation*}
\|\T(t')\omega\|_X=\|P \p_t\phi\|_{L^q_{I} L_x^\infty}\le C(t_*)\|\bp\psi(t')\|_{L_x^2} \le C(t_*) \|\omega\|_{\H}
\end{equation*}
for some constant $C(t_*)$ possibly depending on $t_*$, that is, $\T(t'):\H\to X$ is a bounded linear operator.
Let $M(t'):=\|\T(t')\|_{\H\to X}$. Then $M(t')<\infty$, and for the adjoint $\T(t')^*: X'\to \H$ we have
$$
\|\T(t')^*\|_{X'\to \H} =M(t'), \quad \|\T(t')\T(t')^*\|_{X'\to X}=M(t')^2.
$$
Note that $M(\cdot)$ is a continuous function on $I_*$, whose  maximum, denoted by $M$, is achieved at certain $t_0\in [0,t_*)$.
Our goal is to show that $M$ is independent of $\la$. We will confirm this by showing that
\begin{equation}\label{gl1}
M^2\le C+\f12 M^2
\end{equation}
for some universal positive constant $C$ independent of $t_*$. Let us set $I_0=[t_0, t_*]$,  $X=L^q_{I_0} L_x^\infty$ and
$X'=L^{q'}_{I_0} L_x^1$, and consider the operator $\T:=\T(t_0)$ and its adjoint $\T^*$.

We first calculate $\T^*: X'\to \H$. For any $f\in X'$ and $\omega\in \H$ we have
\begin{align*}
\l \T^* f, \omega\r_\H &=\l f, \T\omega\r_{X', X} 
=\int_{I_0\times \Sigma} (Pf) \p_t \psi(t,t_0, \omega).
\end{align*}
Let $\eta$  be the solution of the initial value problem
\begin{equation}\label{psiastz}
\left\{\begin{array}{lll}
\Box_\bg \eta=-Pf, \quad \mbox{in } [t_0, t_*)\times {\mathbb R}^3,\\
\eta(t_*)=\p_t \eta(t_*)=0
\end{array}\right.
\end{equation}
and consider the energy-momentum tensor
\begin{equation*}
Q[\psi, \eta]_{\mu\nu}=\f12 (\bd_\mu \psi\bd_\nu \eta+\bd_\nu \psi \bd_\mu \eta)
-\frac{1}{2}\bg_{\mu \nu} (\bg^{\a\b} \bd_\a\psi \bd_\b\eta).
\end{equation*}
For any vector field $Z$ we set $P_\mu:=Q[\psi, \eta]_{\mu\nu} Z^\nu$.  In view of $\Box_\bg \psi=0$, it is easy to check that
\begin{equation*}
 \bd^\b P_\b=\frac{1}{2}\left((Z \psi) \Box_\bg \eta+Q[\psi, \eta]^{\a\b}\pi^{(Z)}_{\a\b} \right).
\end{equation*}
By the divergence theorem we have
\begin{align*}
\int_{\Sigma_{t_*}} Q[\psi, \eta]_{\mu\nu} Z^\mu \bT^\nu
-\int_{\Sigma_{t_0}} Q[\psi, \eta]_{\mu\nu} Z^\mu \bT^\nu
&=-\int_{I_0\times {\mathbb R}^3} \bd^\b P_\b
\end{align*}
which together with the initial conditions in (\ref{psiastz}) implies that
\begin{equation}\label{enbinlinear}
\int_{I_0\times {\mathbb R}^3} (Z\psi) \Box_\bg \eta=2\int_{\Sigma_{t_0}}\bT^\a P_\a
-\int_{I_0\times{\mathbb R}^3} Q[\psi, \eta]^{\a\b}\pi^{(Z)}_{\a\b}.
\end{equation}
Now we take $Z=\bT:=\p_t$. Then it follows from (\ref{enbinlinear}) that
\begin{align}\label{expr1}
\int_{I_0\times {\mathbb R}^3} \p_t\psi \Box_\bg \eta
&=\int_{\Sigma_{t_0}} \left(\p_t \psi\p_t \eta+\delta^{ij} \bd_i \psi\bd_j \eta\right)
- \int_{I_0\times {\mathbb R}^3} Q[\psi, \eta]^{\a\b}\pi^{(\bT)}_{\a\b}.
\end{align}
Therefore
\begin{equation}\label{10.4.1}
\l \T^*f, \omega\r_\H=\l \eta[t_0],\omega\r_\H+l(\omega),
\end{equation}
where $l(\cdot)$ is a linear functional on $\H$ defined by
\begin{equation*}
l(\omega):= \int_{I_0\times {\mathbb R}^3} Q[\psi, \eta]^{\a\b}\pi^{(\bT)}_{\a\b}.
\end{equation*}
We claim that $l(\cdot)$ is a bounded linear functional on $\H$. To see this, let
$\omega\in \H$ with $\|\omega\|_\H\le 1$. Then by the energy estimate we have
$\|\bd \psi\|_{L_t^\infty L_x^2}\le \|\omega\|_\H\les 1$.  Thus
$$
|l(\omega)|\le \|\pi\|_{L_t^1 L_x^\infty} \|\bd \psi \|_{L_{I_0}^\infty L_x^2}
\|\bd \eta \|_{L_{I_0}^\infty L_x^2}
\les \|\pi\|_{L_t^1 L_x^\infty} \|\bd \eta \|_{L_{I_0}^\infty L_x^2}.
$$
Hence, by the Riesz representation theorem we have $l(\omega)=\l R(f), \omega\r_\H$ for some
$R(f)\in \H$ and there is a universal constant $C_1$ such that
$$
\|R(f)\|_\H\le C_1 \|\pi\|_{L_t^1 L_x^\infty} \|\bd \eta \|_{L^\infty_{I_0} L_x^2}.
$$
Moreover, we have from (\ref{10.4.1}) that $\T^*f=\eta[t_0]+R(f)$ and hence
\begin{equation}\label{10.4.2}
\T \T^* f=\T \eta[t_0]+\T R(f).
\end{equation}
We claim that there is a universal constant $C_2$ such that
\begin{equation}\label{dpsi}
\|\bd\eta\|_{L^\infty_{I_0} L_x^2}\le C_2 M \|f\|_{L^{q'}_{I_0} L_x^1}.
\end{equation}
Assuming this claim for a moment, it follows from the definition of $M$ that
\begin{equation*}
\|\T R(f)\|_{L^{q}_{I_0} L_x^\infty}\le C_1 C_2  M^2 \|\pi\|_{L_t^1 L_x^\infty}
\|f\|_{L^{q'}_{I_0} L_x^1}.
\end{equation*}
Thus, if (\ref{smallas}) holds with $C_0\ge 2C_1 C_2$, then
\begin{equation}\label{10.4.3}
\|\T R(f)\|_{L^{q}_{I_0} L_x^\infty}\le \frac{1}{2} M^2 \|f\|_{L^{q'}_{I_0} L_x^1}.
\end{equation}

Next we estimate $\|\T \eta[t_0]\|_{L_{I_0}^q L_x^\infty}$. We set $F:=(0,- Pf)$.
By the Duhamel principle we have
\begin{equation*}
\eta[t]=\int_{t_*}^t \Psi(t; s, F(s)) ds.
\end{equation*}
By uniqueness we have $\eta(t) = -\int_{t_0}^{t_*} \psi(t; s, F(s)) ds$.
Thus
\begin{align*}
\T \eta[t_0]&=-P \p_t\left(\int_{t_0}^{t_*}\psi(t,s, F(s)) ds\right) =-\int_{t_0}^{t_*} P \p_t \psi(t, s, F(s)) ds.
\end{align*}
It follows from Theorem \ref{decayth} that
\begin{align*}
\|P \p_t \psi(t,s, F(s))\|_{L_x^\infty} &\les \left((1+|t-s|)^{-\frac{2}{q}}+d(|t-s|)\right)\sum_{m=0}^2\|\p^m Pf(s)\|_{L_x^1}\\
& \les \left((1+|t-s|)^{-\frac{2}{q}}+d(|t-s|)\right) \|f\|_{L_x^1}.
\end{align*}
Thus, in view of the Hardy-Littlewood-Sobolev inequality, (\ref{correccondi}) and  Hausdorff Young inequality we obtain
\begin{equation}\label{t1}
\|\T \eta[t_0]\|_{L^q_{I_0} L_x^\infty}
\les \|f\|_{L^{q'}_{I_0}L_x^1}+\left\|\int_{t_0}^{t_*} d(|t-s|)\|f(s)\|_{L_x^1} ds\right\|_{L^q_{I_0}}
\les \|f\|_{L^{q'}_{I_0} L_x^1}.
\end{equation}
Combining (\ref{10.4.2}), (\ref{10.4.3}) and  (\ref{t1}), we therefore obtain (\ref{gl1}).

It remains to prove (\ref{dpsi}). Let $\ti \phi$ be a solution of $\Box_\bg \ti \phi=0$ in $I_*$.
Then there holds the energy estimate $\|\bd \ti\phi(t)\|_{L^2(\Sigma)}
\les \|\bd\ti \phi(t_0)\|_{L^2(\Sigma)}$ for $t\in [t_0, t_*]$. Let $t_0\le t'<t_*$.  Similar to the derivation of
(\ref{expr1}), we have on  $I=[t',t_*]$ that
\begin{align*}
\int_{I\times {\mathbb R}^3} \p_t\ti \phi \Box_\bg \eta
&=\int_{\Sigma_{t'}} \left(\p_t \ti \phi \p_t \eta +\delta^{ij} \bd_i \ti \phi\bd_j \eta\right)
- \int_{I\times {\mathbb R}^3} Q[\ti\phi, \eta]^{\a\b}\pi^{(\bT)}_{\a\b},
\end{align*}
which together with $\Box_\bg \eta =-P f$ and the definition of $M$ gives
\begin{align*}
\int_{\Sigma_{t'}} \left(\p_t \ti \phi \p_t \eta +\delta^{ij} \bd_i \ti \phi\bd_j \eta\right)
&\les \|P \p_t \ti \phi\|_{L^q_{I} L_x^\infty}\|f\|_{L^{q'}_{I} L_x^1}
+\|\pi^{(\bT)}\|_{L^1_{I}L_x^\infty}\|\bd\eta\|_{L_t^\infty L_x^2}\|\bd \ti \phi\|_{L_I^\infty L_x^2}\\
& \les \left( M\|f\|_{L^{q'}_{I} L_x^1}
+ \|\pi^{(\bT)}\|_{L^1_{I}L_x^\infty}  \|\bd\eta\|_{L_t^\infty L_x^2}\right) \|\bd \ti \phi(t')\|_{L_x^2}.
\end{align*}
Since $\bd \ti\phi(t')$ can be arbitrary, there is a universal constant $C_3$ such that
$$
\| \bd\eta(t')\|_{L_x^2} \le C_3  M\|f\|_{L^{q'}_I L_x^1}
+ C_3 \|\pi^{(\bT)}\|_{L^1_{I}L_x^\infty}  \|\bd\eta\|_{L_I^\infty L_x^2}.
$$
Recall that $t'\in [t_0, t_*)$ is arbitrary. Thus, if (\ref{smallas}) holds with $C_0\ge 2 C_3$ then
$$
\| \bd\eta\|_{L_{[t_0, t_*)}^\infty L_x^2} \le C_3 M\|f\|_{L^{q'}_{[t_0, t_*)} L_x^1}
+ \frac{1}{2} \|\bd\eta\|_{L_{[t_0, t_*)}^\infty L_x^2}.
$$
This implies (\ref{dpsi}) with $C_2=2 C_3$.  The proof of (\ref{pen}) is thus completed. We also have proved for any $t\in I_*$
 \begin{equation}\label{dsp2}
\|\bd\eta\|_{L^\infty_{[t, t_*)} L_x^2}\le C_2 M \|f\|_{L^{q'}_{[t_,t_*)} L_x^1}.
\end{equation}

Now we consider $\|P \p_i \psi\|_{L^q_{I_*} L_x^\infty}$. It suffices to estimate
\begin{equation*}
\I=\int_{I_*\times {\mathbb R}^3} f P \p_i \psi =\int_{I_*\times {\mathbb R}^3} \p_i \psi Pf
\end{equation*}
for any function $f$ satisfying $\|f\|_{L^{q'}_{I_*} L_x^1} \le 1$. Let $\eta$ be the solution of (\ref{psiastz}), then
$$
\I= -\int_{I_*\times {\mathbb R}^3} \p_i \psi\Box_\bg \eta.
$$
In view of (\ref{enbinlinear}), we have with $Z=\p_i$ that
\begin{align*}
\I &= -2\int_{\Sigma_0}\bT^\a P_\a + \int_{I_*\times {\mathbb R}^3} Q[\psi, \eta]^{\a\b}\pi^{(Z)}_{\a\b}.
\end{align*}
By direct calculation we can see that $\pi^{(Z)}=\bg\c\p g$. Thus it follows from the energy estimate (\ref{dsp2}) and (\ref{smallas}) that
\begin{align*}
\left|\int_{I_*\times {\mathbb R}^3} Q[\psi, \eta]^{\a\b}\pi^{(Z)}_{\a\b}\right|
&\les \|\ti\pi\|_{L^1_{I_*} L_x^\infty}\|\bd \eta\|_{L^\infty_{I_*}L_x^2}\|\bd\psi\|_{L^\infty_{I_*} L_x^2}
\les \|\bd\psi(0)\|_{L_x^2}\|f\|_{L^{q'}_{I_*}L_x^1}
\end{align*}
and
\begin{align*}
\left|\int_{\Sigma_0}\bT^\a P_\a\right|
\les\|\bd\psi(0)\|_{L^2}\|\bd \eta\|_{L^\infty_{I_*} L_x^2}
\les\|\bd\psi(0)\|_{L^2} \|f\|_{L_{I_*}^{q'}L_x^1}.
\end{align*}
Therefore $|\I|\les \|\bd\psi(0)\|_{L_x^2}\|f\|_{L^{q'}_{I_*} L_x^1}$.
Hence we can conclude that
\begin{equation*}
\|P \p_i \psi\|_{L^q_{I_*} L_x^\infty}\les \|\bd \psi(0)\|_{L_x^2}.
\end{equation*}
 The proof is thus complete.

 \section{\bf Appendix C: Proof of Proposition \ref{exten}}

The goal of this section is to prove Proposition \ref{exten}. Let us consider within the domain of the geodesic ball $B_{\tau_*}({\bf o})$ on $\Sigma_0$.
When there is a $v$-foliation defined in a neighborhood of ${\bf o}$ contained in $B_{\tau_*}({\bf o})$ with $v({\bf o})=0$ such that
each level set $S_v$ is diffeomorphic to ${\mathbb S}^2$, let $a$ be the corresponding lapse function defined by $a^{-1}=|\nab v|_g$
and let $\ga$ denote the induced metric of $g$ on $S_v$. Then the metric $g$ of $\Sigma_0$ in the neighborhood of ${\bf o}$ can be
written as
\begin{equation*}
a^2 dv^2+ \ga_{AB} d\omega^A d\omega^B.
\end{equation*}
Let $N$ be the outward unit normal vector field to each $S_v$. The second fundamental form $\theta$ of $S_v$
is defined by $\theta=\f12\Lie_N \ga$. We refer to \cite{JS, JS1} for a local existence of $v$-foliation in a small neighborhood of ${\bf o}$
such that
\begin{equation}\label{thedef}
\tr\theta+k_{NN}=\frac{2}{av}+\Tr k-V_4, \quad a({\bf o})=1.
\end{equation}
Under such a $v$-foliation, there hold the following  structure equations (See \cite[Chapter 3]{CK} and \cite{JS}):
 \begin{align}
 &a^{-1} \sD a =-N(\tr\theta)-|\theta|^2-R_{NN}\label{a_1}\\
 &\sn^B {\hat\theta}_{AB}=\f12 \sn_A \tr\theta+R_{NA}\label{a_2}\\
 &\sn_N \hat\theta_{AB}+\tr\theta \hat \theta_{AB}=-a^{-1}(\sn\hot\sn a)_{AB}-{\hat R}_{AB}\label{a_11}\\
&K=\frac{(\tr\theta)^2}{4}-\f12 |\hat \theta|^2+\frac{R-2 R_{NN}}{2}\label{a_12}
 \end{align}
where $K$ is the Gaussian curvature on  $S_v$, and $\hat R_{AB}=R_{AB}-\f12 \ga_{AB} R_{CD}\ga^{CD}$.
Here $R_{ij}$ denotes the Ricci curvature of $\Sigma_0$ induced by $g$. For a scalar function $f$, $(\sn\hot\sn f)_{AB}$ denotes the traceless part of $\sn^2_{AB} f$. 

Noting that with $\ti s$  the geodesic distance to ${\bf o}$, and $\sl g$ the induced metric of $g$ at the level sets of $\ti s$,
 we can write the metric $g$ in the geodesic ball $B_{\tau_*}({\bf o})\subset\Sigma_0$  as
$
d\ti s^2+\sl g_{AB} d\omega^A d\omega^B.
$
By $a^{-1}=|\nab v|_g$ and if $a^{-1}\le 2$,
$
 \frac{dv}{d\ti s}\le a^{-1}\le 2.
$
This implies  $v\le 2\tau_*\les \la^{1-8\ep_0}$ in $B_{\tau_*}({\bf o})$.

 Let us fix $0<\ep<s-2 $, and $\frac{2}{p}=1-\ep$.
 To prove Proposition \ref{exten}, we make the bootstrap assumptions on $\itt_{v_*}=\cup_{0\le v\le v_*} S_v$
 \begin{align}
 &|a-1|<\frac{1}{2},\label{a_0} \\
&  \|v^{\f12}\hat \theta\|_{L_v^4 L_\omega^4}\le \la^{-\frac{1}{4}}, \quad \|v^\f12 \sn a\|_{L_v^4 L_\omega^{4}}\le \la^{-\frac{1}{4}}\label{a_6} \\
&|\cga-\ga^{(0)}|+\|\p(\cga-\ga^{(0)})\|_{L_\omega^p(S_v)} \le \la^{-\ep_0}\label{a_7}
 \end{align}
 where $\gac:=v^{-2}\ga$ is the rescaled metric on $S_v$ and  ${\gamma}^{(0)}$ is the
canonical metric on ${\Bbb S}^2$.  (\ref{a_7}) are the comparisons of components of the metrics under the transport coordinates. With $v_*\les\tau_*$,
(\ref{a_0}) and (\ref{a_7}) are improved to (\ref{a_3}) and  (\ref{a_5}) respectively. (\ref{a_6}) is improved in (\ref{theta_1}) and (\ref{ellp_2}). By continuity  argument,  this shows  that if  $v_*\le \frac{4\tau_*}{5}$, there exists $S_v$ foliation  such that  (\ref{thedef}) holds till $v\ge v_*$, and that (\ref{w7.13.1}) holds true.

Similar to Lemma \ref{inii},  we can derive by local expansion (see  \cite{Wangthesis}),
\begin{lemma}\label{inii_7}
(i)
\begin{equation*}
\lim_{v\rightarrow 0}\sn a= 0, \quad \lim_{v\rightarrow 0} \|\hat \theta\|_{L^\infty_\omega}<\infty.
\end{equation*}

(ii)  Relative to the transport coordinates,  there hold
\begin{equation}
\lim_{v\rightarrow 0}\stackrel{\circ}
\gamma_{ab}={\gamma}_{ab}^{(0)},\qquad  \lim_{v\rightarrow
0}\p_c\!\!\stackrel{\circ}\ga_{ab}=\p_c{\ga}_{ab}^{(0)}\label{w8.7.1}
\end{equation}
where $a, b, c=1,2$.
\end{lemma}

We will constantly employ the facts
\begin{align}
&\p_v \sqrt{|\ga|}=a \tr \,\theta \sqrt{|\ga|}, \label{thet1}\\
&a\tr\, \theta-\frac{2}{v}=-\G,\,\,  \G=a (V_4- \tr \, k), \label{G1}
\end{align}
where $\tr\, k= \ga^{AB}k_{AB}$; the equation (\ref{G1}) is a consequence of (\ref{thedef}).

By (\ref{a_0}) and  the first assumption in (\ref{a_7}), (\ref{trc_1})-(\ref{w7.28.1}) hold.
We will derive the following useful consequences of (\ref{a_0})-(\ref{a_7}).

\begin{lemma}\label{mtieq}
For $0\le 1-\frac{2}{q}<s-2$, there hold
\begin{align}
&C^{-1} \le \sqrt{|\ga|}v^{-2} \le C, \label{a_8}\\
&\|v \bp^2 \phi, \bp \phi\|_{L_v^2 L_\omega^q}+\|v^\f12 \bp \phi\|_{L_v^\infty L_\omega^{2q}}+\|\bp \phi\|_{L_v^2 L_\omega^\infty}\les \la^{-\f12},\label{a_9_w}\\
&\|v (\emph{\bd} \pi,\mbox{Ric}), \pi\|_{L_v^2 L_\omega^q}+\|v^\f12 \pi\|_{L_v^\infty L_\omega^{2q}}+\|\pi\|_{L_v^2 L_\omega^\infty}\les \la^{-\f12}.\label{a_9}
\end{align}
\end{lemma}

\begin{proof}
We make bootstrap assumption on $\G$
\begin{equation}\label{g5.3}
\|\G\|_{L_\omega^\infty L_v^2}\les \la^{-\f12+2\ep_0},
\end{equation}
which will be improved by the last inequality of (\ref{a_9_w}).

Now we prove (\ref{a_8}).  By using (\ref{thet1}) and (\ref{G1})
\begin{align}\label{tr_5.1}
\p_v(v^{-2}\sqrt{|\ga|})=-\G \c v^{-2}\sqrt{|\ga|}.
\end{align}
Using (\ref{g5.3}), (\ref{tr_5.1}), and the fact that $v^{-2} \sqrt{|\ga|}\rightarrow \sqrt{|\ga^{(0)}|}$ as $v\rightarrow 0$, we obtain
\begin{equation*}
\left|\log (v^{-2}\sqrt{|\ga|}) - \log \sqrt{|\ga^{(0)}}\right| \les \int_0^v |\G| d v' \les \la^{-2\ep_0}.
\end{equation*}
Thus (\ref{a_8}) is proved.

With the help of (\ref{a_8}), by using the $L^2$ estimates of energy (\ref{eng0}), (\ref{trc_1}) and (\ref{trc_2})
\begin{equation}\label{w7.31.1}
\|v \bp^2 \phi\|_{L_v^2 L_\omega^2}+\| \bp\phi\|_{L_v^2 L_\omega^2}+\|v^\f12 \bp \phi\|_{L^\infty L^4_\omega}\les \la^{-\f12}.
\end{equation}
This proves the first three  inequalities in (\ref{a_9_w}) for the case $q=2$.

For $0<1-\frac{2}{q}<s-2$,   using (\ref{w7.28.1}) with $f=\bp^2 \phi$, we can derive
\begin{equation*}
\|v \bp^2 \phi\|_{L_v^2 L_\omega^q}\les \|\bp^2 \phi\|_{H^{s-2}}.
\end{equation*}
Using Proposition \ref{eng3}, we obtain the first inequality of (\ref{a_9_w}) that
\begin{equation}\label{5.3.6}
\|v \bp^2 \phi\|_{L_v^2 L_\omega^q}\les \la^{-\f12}.
\end{equation}
 By using  (\ref{sobinf}), (\ref{5.3.6}) and the second inequality in (\ref{w7.31.1})
 \begin{equation}\label{w7.31.3}
\|\bp \phi\|_{L_v^2 L_\omega^\infty}\les \|v \sn \bp \phi\|_{L_v^2 L_\omega^q}+\|\bp \phi\|_{L_v^2 L_\omega^2}\les \la^{-\f12}.
 \end{equation}
 By (\ref{sob}) and (\ref{w7.31.1}),
\begin{equation}\label{w7.31.2}
\| \bp \phi\|_{L_v^2 L_\omega^q}\les \|v\sn \bp \phi\|_{L_v^2 L_\omega^2}+\|\bp \phi\|_{L_v^2 L_\omega^2}\les \la^{-\f12}.
\end{equation}
The estimate on $\|v^\f12\bp\phi\|_{L^\infty L_\omega^{2q}}$ in (\ref{a_9_w}) can be proved by using (\ref{tran_sob}), (\ref{5.3.6}) and (\ref{w7.31.3}).
Hence the proof for (\ref{a_9_w}) is complete.

It only remains to consider the estimates for $\bd \pi$ and $Ric$  in (\ref{a_9}), since all other estimates follow from (\ref{a_9_w}) and $|\pi|\les |\bp\phi|$.

Note that with $\frac{1}{q'}=\frac{1}{q}-\frac{1}{4}$ and $\frac{1}{b}=\frac{q}{4}-\frac{1}{2}$,  by (\ref{a_9_w}), we can derive
\begin{equation}\label{w7.27.1}
\|v^\f12 \bp \phi\|_{L_v^4 L_\omega^{q'}}\le v_*^{\frac{1}{4}}\|v^{1-\frac{q}{4}}\bp \phi\|_{L_v^b L_\omega^{q'}}\les \|v^\f12 \bp \phi\|_{L_v^\infty L_\omega^{2q}}^{2-\frac{q}{2}}\|\bp\phi\|_{L_v^2 L_\omega^\infty}^{\frac{q}{2}-1}{v_*}^\frac{1}{4}\les \la^{-\f12}{v_*}^{\frac{1}{4}}.
\end{equation}

  In view of (\ref{ricc_def}), and writing $\mbox{Ric}$ under coordinates,
$
|\bd \pi, \mbox{Ric}|\les |\bd\ti\pi|+|\ti\pi|\c |(\pi, \theta, \sn\log a)|.
$
Also using $|\bd\ti\pi|\les |\bp\phi\c \bp\phi|+|\bp^2 \phi|,$ we can derive
\begin{equation*}
\|v(\bd \pi, \mbox{Ric})\|_{L_v^2 L_\omega^q}\les \|v\bp^2 \phi\|_{L_v^2 L_\omega^q}+\|v|\bp \phi|\c |(\bp \phi,\theta,\sn\log a)| \|_{L_v^2 L_\omega^q}\les \la^{-\f12}
\end{equation*}
which follows from  (\ref{G1}), (\ref{a_6}), (\ref{a_0}), (\ref{a_9_w}) and (\ref{w7.27.1}).
\end{proof}

As a consequence of (\ref{a_7}), (\ref{sob}) and (\ref{sobinf})
\begin{lemma}
Let  $2\le q\le p$.  For any $S_v$ tangent vector field $F$, there holds
\begin{equation}\label{eqv_2}
\|v\sn F\|_{L_v^2 L_\omega^q}+\|F\|_{L_v^2 L_\omega^q}\approx \sum_A\|\sn(F^A)\|_{L_v^2 L_x^q}+\|F\|_{L_v^2 L_\omega^q}
\end{equation}
where $F^A$ is the component relative to the coordinate frame on $S_v$.
This same equivalence holds for any $S_v$ tangent tensor fields.
\end{lemma}

We frequently employ the following formula  for  smooth scalar functions $f$,
\begin{align}
\p_v \int_{S_v} v^{-m} f d \mu_\ga&=\int v^{-m}\left(\p_v f +\left(a\tr\theta-\frac{m}{v}\right)f\right) d \mu_\ga\nn \\
&=\int v^{-m}\left(\p_v f+\left(\frac{2-m}{v}-\G\right) f\right) d \mu_\ga\label{4.25.2}
\end{align}
where $\G=a (V_4-\tr k) $. (\ref{4.25.2}) is a direct consequence of (\ref{thet1}) and (\ref{G1}).

\begin{lemma}
There holds for the lapse function $a$, the parabolic equation
\begin{equation}\label{aa_1}
a^2\left(\sD \log a+|\sn \log a|^2\right) =\frac{2}{v}\p_v \log a+a^2 \mathfrak{E}, \quad \log a(0, \c)=0
\end{equation}
where
\begin{equation}\label{ff}
\mathfrak{E}=a^{-1} \p_v (V_4-{\emph\tr}\, k)-|\hat \theta|^2+\frac{2}{va}(V_4-{\emph \tr}\, k)-\frac{(V_4-{\emph \tr}\, k)^2}{2}+R_{NN}.
\end{equation}
Equivalent to (\ref{aa_1}), there holds for the lapse function $a$ the following equation
\begin{equation}\label{4.27.1}
a^{-1} \sD a=-\frac{2}{av}\p_v(a^{-1})+\mathfrak{E}
\end{equation}
\end{lemma}

Indeed, by using (\ref{thedef}) we have
\begin{align*}
a^{-1} \sD a&=-N \left(\frac{2}{va}+\tr k-V_4\right)-\f12 \left(\frac{2}{av}+\tr k-V_4\right)^2-|\hat\theta|^2+R_{NN}\\
&=-a^{-1}\left(-\frac{2}{v^2a }+\frac{2}{v} \p_v(a^{-1})\right)-\frac{2}{(va)^2}+\mathfrak{E}.
\end{align*}
This gives (\ref{4.27.1}). With the help of
\begin{equation*}
\sD \log a=\sn(a^{-1} \sn a)=-|\sn \log a|^2+a^{-1} \sD a
\end{equation*}
we can obtain (\ref{aa_1}).

\subsection{$L^2$ estimates}
In Lemma \ref{ellp_4.26} and Lemma \ref{ellp_5.1}, we provide preliminary estimates for lapse function $a$.

\begin{lemma}\label{ellp_4.26}
\begin{align}
&\|\sn\log a\|_{L_v^2 L_\omega^2}+\|v^{-\f12} \log a\|_{L_v^\infty L_\omega^2}+\|v^{-1}\log a\|_{L_v^2 L_\omega^2}\les \la^{-\f12},\label{ellp_1}\\
& \|v^{-\f12} |a^{-1}-1|\|_{L_v^\infty L_\omega^2}+\|v^{-1}|a^{-1}-1|\|_{L_v^2 L_\omega^2}\les \la^{-\f12}, \label{ellp_0}\\
& \|v^{-\f12}\log a\|_{L_v^4 L_\omega^4}+\|v^{-\f12} |a^{-1}-1|\|_{L_v^4 L_\omega^4}\les \la^{-\frac{1}{4}-2\ep_0}. \label{ellp_8}
\end{align}
\end{lemma}

\begin{proof}
We first consider (\ref{ellp_1}). Multiplying (\ref{aa_1}) by  $v^{-2}\log a$ and integrating on $S_v$ , also by integration by parts, we deduce
\begin{align*}
\int_{S_v}& \left(\frac{\p_v |\log a|^2}{v}+a^2 |\sn \log a|^2(1+\log a)\right) v^{-2} d\mu_\ga
=-\int_{S_v} v^{-2} a^2 \mathfrak{E}\log a d\mu_\ga.
\end{align*}
By using (\ref{4.25.2}) with $m=3$, we have
\begin{align*}
\p_v\int_{S_v}& \frac{|\log a|^2}{v^3} d\mu_\ga+\int_{S_v} v^{-2} \left( \left(\frac{\log a}{v}\right)^2+a^2 |\sn \log a|^2(1+\log a)\right)d\mu_\ga \\
&=-\int_{S_v} v^{-2}\left(a^2 \mathfrak{E}+v^{-1}\G\c \log a\right)\log a d\mu_\ga.
\end{align*}
We integrate this identity in $v$ from $v=0$ to $v=\tau$ and use Lemma \ref{inii_7}.
Note that (\ref{a_0}) implies $1+\log a>\frac{3}{10}$. Thus, by taking supremum in $0\le v\le v_*$ and using the Cauchy-Schwartz
inequality and (\ref{a_0}), we derive
\begin{align}
& \|\sn \log a\|_{L_v^2 L_\omega^2}^2 +\|v^{-\f12}\log a\|_{L_v^\infty L_\omega^2}^2+\left\|\frac{\log a}{v}\right\|_{L_v^2 L_\omega^2}^2\nn\\
& \les \|v \mathfrak{E}\|_{L_v^2 L_\omega^2}^2+\|\G\|_{L_v^2 L_\omega^2} \|v^{-\f12} \log a\|_{L_v^4 L_\omega^4}^2 \label{snaa_2}
\end{align}
Applying (\ref{sob}) to $S_v$-tangent tensor field or scalar function $W$
\begin{align}\label{sobv}
\|v^{-\f12} W\|_{L_v^4 L_\omega^4}^2\les (\|\sn W\|_{L_v^2 L_\omega^2} +\|v^{-\f12} W\|_{L_v^\infty L_\omega^2})\|v^{-\f12}W\|_{L_v^\infty L_\omega^2} (v_*)^\f12.
\end{align}
Hence by  (\ref{a_9}) which gives $\|\G\|_{L_v^2 L_\omega^2}\les \la^{-\f12}$ and in view of $v_*\les\la^{1-8\ep_0}$,
\begin{align}
\|\G\|_{L_v^2 L_\omega^2} \|v^{-\f12} W\|_{L_v^4 L_\omega^4}^2
&\les \la^{-4\ep_0} \left(\|\sn W\|_{L_v^2 L_\omega^2} +\|v^{-\f12} W\|_{L_v^\infty L_\omega^2}\right)\|v^{-\f12}W\|_{L_v^\infty L_\omega^2}\label{4.25.3}
\end{align}
Then the last term in (\ref{snaa_2}) can be treated as
\begin{align*}
\|\G\|_{L_v^2 L_\omega^2} \|v^{-\f12} \log a\|_{L_v^4 L_\omega^4}^2
&\les \la^{-4\ep_0} \left(\|\sn \log a\|_{L_v^2 L_\omega^2} +\|v^{-\f12} \log a\|_{L_v^\infty L_\omega^2}\right)\|v^{-\f12}\log a\|_{L_v^\infty L_\omega^2}
\end{align*}
This term can be absorbed by the left hand side of (\ref{snaa_2})

Note that by using (\ref{a_6}) and (\ref{a_9}),
\begin{align}
\|v \mathfrak{E}\|_{L_v^2 L_\omega^2}
&\les \left\|v\Big(\nab_N (V_4-\tr k)+\frac{2}{va}(V_4-\tr k)-\f12(V_4-\tr k)^2 -|\hat\theta|^2+R_{NN}\Big)\right\|_{L_v^2 L_\omega^2}\nn\\
&\les \|v (\nab_N (V_4-\tr k), R_{NN})\|_{L_v^2 L_\omega^2}+\|V_4-\tr k\|_{L_v^2 L_\omega^2}+\|v(V_4-\tr k)^2, v|\hat \theta|^2\|_{L_v^2 L_\omega^2}\nn\\
&\les \la^{-\f12}\label{fl2}
\end{align}
Hence, in view of (\ref{snaa_2}),  we obtain (\ref{ellp_1}).
By using (\ref{ellp_1}), we obtain by using (\ref{sobv}) that
\begin{equation}\label{4.27.2}
\|v^{-\f12}\log a\|_{L_v^4 L_\omega^4}^4\les \la^{-1-8\ep_0}.
\end{equation}

Next, we prove (\ref{ellp_0}).  We multiply (\ref{4.27.1}) by $a(1-a^{-1})v^{-2}$ and integrating on $S_v$
\begin{align*}
\int_{S_v} -v^{-2}\sn (1-a^{-1}) \sn a d\mu_\ga
&=\int_{S_v} \left(v^{-1} \p_v\left( (a^{-1}-1)^2 \right)+a\mathfrak{E}(1-a^{-1})\right)v^{-2}  d\mu_\ga
\end{align*}
By using (\ref{4.25.2}) with $m=3$,
\begin{align*}
-\int_{S_v} a^{-2}|\sn a|^2 v^{-2} d\mu_\ga&=\p_v \int_{S_v} v^{-3} (a^{-1}-1)^2 d\mu_\ga +\int_{S_v} \left(\frac{1}{v}+\G\right) (a^{-1}-1)^2 v^{-3} d\mu_\ga\\
&+\int_{S_v} \mathfrak{E}(1-a^{-1}) v^{-2}a d\mu_\ga.
\end{align*}
Integrating  on $(0, v]$ for $v\le  v_* $,
\begin{align}
\|v^{-\f12} (a^{-1}-1)\|_{L_v^\infty L_\omega^2}^2&+\|v^{-1}(a^{-1}-1)\|_{L_v^2 L_\omega^2}^2+\|\sn \log a\|_{L_v^2 L_\omega^2}^2\label{4.27.3}\\
&\les \|v \mathfrak{E}\|_{L_v^2 L_\omega^2}+\|\G\|_{L_v^2 L_\omega^2}\|\frac{a^{-1}-1}{v^\f12}\|_{L_v^4 L_\omega^4}^2\nn
\end{align}
Apply  (\ref{4.25.3}) to  $W=a^{-1}-1$,  the last term is absorbed by (\ref{4.27.3}). By using (\ref{fl2}) we complete the proof of (\ref{ellp_0}).

Similar  to (\ref{4.27.2}), the last inequality in (\ref{ellp_8}) can be proved by using (\ref{ellp_0}), the first inequality in (\ref{ellp_1})  and (\ref{sobv}).
\end{proof}

We now prove

\begin{lemma}\label{ellp_5.1}
\begin{align}
&\|v\sn^2\log a\|_{L_v^2 L_\omega^2}+\|v^\f12 \sn\log a\|_{L_v^\infty L_\omega^2}\les \la^{-\f12}, \quad\|v^\f12\sn\log a\|_{L_v^4 L_\omega^4}\les \la^{-2\ep_0-\frac{1}{4}} \label{ellp_2}
\end{align}
\end{lemma}

The last inequality  improves the second assumption in (\ref{a_6}). We will rely on the following equation to prove Lemma \ref{ellp_5.1}.

\begin{lemma}
\begin{align}
& a^2\left(\sD \sn \log a+\Big(-\frac{3(\tr\theta)^2}{4}+\frac{|\hat \theta|^2}{2}-\frac{R-2R_{NN}}{2}\Big) \sn \log a
+2 \sn \sn_B \log a \sn_B \log a\right)\nn\\
&+2 a^2 \sn \log a \left( \sD \log a+|\sn \log a|^2\right)
=\frac{2\p_v \sn \log a}{v}+a\left(\frac{2}{v}  \hat \theta +\f12 \G\tr\theta\right)\sn \log a+\sn (a^2 \mathfrak{E}). \label{ellp_3}
\end{align}
\end{lemma}

(\ref{ellp_3}) follows by differentiating (\ref{aa_1}) and using
the following commutation relations on $\Sigma$.
\begin{equation}\label{comm_4.25.1}
[\sn, \p_v]f= a\left(\frac{1}{2}\tr\theta \sn f+\hat\theta \sn f\right)
\end{equation}
and
\begin{equation}\label{comm_4.25}
\sn_A \sD f-\sD \sn_A f=-K\sn_A f.
\end{equation}
where $K$ is the Gaussian curvature.

(\ref{comm_4.25.1}) can be obtained in view of \cite[Corollary 3.2.3.1, Page 64]{CK}. (\ref{comm_4.25}) can be obtained by using
\begin{align*}
\sn \sD f&=\sn_C\sn_A\sn_C f+[\sn_A, \sn_C]\sn_C f=\sn_C \sn_A \sn_C f-R_{AC}\sn_C f
\end{align*}
and $
R_{AC}=\delta_{AC} K$.

\begin{proof}[Proof of Lemma \ref{ellp_5.1}]
Let us make a bootstrap assumption
\begin{equation}\label{bba3}
\|v\sn^2\log a\|_{L_v^2 L_\omega^2}\le \la^{-\f12+2\ep_0}
\end{equation}
which will be improved by the first inequality in (\ref{ellp_2}).

Multiply (\ref{ellp_3}) by $\sn_A  \log a$, followed by integration by part on $S_v$ with area form $d\mu_\ga$,
\begin{align*}
& \int_{S_v} \left(a^2 \left(|\sn^2 \log a|^2+2|\sn\log a|^4\right)+v^{-1}\p_v |\sn \log a|^2+\frac{3a^2}{4}(\tr\theta)^2|\sn \log a|^2\right)\\
&=\int_{S_v} \left(- 4a^2 \sn^2 \log a\sn\log a \sn\log a +a^2\left(\frac{|\hat \theta|^2}{2}+\frac{2R_{NN}-R}{2}\right) |\sn \log a|^2\right)\\
&-\int_{S_v} \left(\frac{2a}{v} \hat \theta \sn \log a+\frac{a}{2}\G\tr\theta \sn \log a+\sn (a^2 \mathfrak{E})\right)\sn \log a.
\end{align*}
Now we multiply by $v^2$ and integrate from $0$ to $\tau$, using Lemma \ref{inii_7} and (\ref{thedef})
\begin{align*}
\int_{S_\tau}& v^{-1}|\sn \log a|^2 d\mu_\ga+ \int_0^\tau\int_{S_v}\left[a^2\left(| \sn^2 \log a|^2+2|\sn \log a|^4\right) +G_1 |\sn \log a|^2\right] d\mu_\ga dv\\
&=\int_0^\tau \int_{S_v}  \left(-4 a^2 \sn^2 \log a \sn \log a \c \sn \log a+a^2\frac{|\hat \theta|^2+2R_{NN}-R}{2}|\sn \log a|^2 \right)d\mu_\ga dv\\
&-\int_0^\tau \int_{S_v} \left(\frac{2a}{v}\hat \theta \sn \log a+\f12 a\tr\theta \G \sn \log a+\sn (a^2 \mathfrak{E})\right)\sn \log a d\mu_\ga d v,
\end{align*}
where the coefficient $G_1=\frac{3a^2}{4}(\tr\theta)^2 -v^{-1}(a\tr\theta-\frac{1}{v})$.

By using $a\tr\theta=\frac{2}{v}-\G$, we can write
\begin{align*}
G_1&=\frac{3}{4} \left(\frac{2}{v}-\G\right)^2-\frac{1}{v}\left(\frac{1}{v}-\G\right)
=\frac{2}{v^2}-\frac{2\G}{v}+\frac{3\G^2}{4}.
\end{align*}
We can drop the last term in $G_1$ due to non-negativity, and incorporate the second into terms on the right hand side.
This implies,
\begin{align*}
& \int_{S_\tau} v^{-1}|\sn \log a|^2 d\mu_\ga
+\int_0^\tau\int_{S_v} \left(a^2\left(|\sn^2 \log a|^2+2|\sn \log a|^4\right)+\frac{2}{v^2} |\sn \log a|^2\right) d\mu_\ga d v\\
&\les \int_0^\tau\int_{S_v}\left( a^2|\sn\log a|^2 \left(|\sn^2 \log a|+|\hat \theta|^2+|\G|^2+\frac{|\hat \theta|+|\G|}{av} +|\frac{R-2R_{NN}}{2}|\right)
+|(a^2 \mathfrak{E})\sD \log a|\right),
\end{align*}
where we performed the integration by part on $S_v$ to treat the term of $\mathfrak{E}$.

By (\ref{a_0}), (\ref{a_8}), we can deduce by the H\"{o}lder and the Cauchy-Schwartz inequalities that
\begin{align}
\int &(|\sn^2\log a|^2+|\sn\log a|^4)+ \frac{1}{v^2}|\sn\log a|^2+\int_{S_\tau}v^{-1}|\sn \log a|^2\nn\\
&\les \|v (R, R_{NN}, |\hat\theta, \G|^2), v\sn^2\log a\|_{L_v^2 L_\omega^2}\|v^{\f12}\sn\log a\|^2_{L_v^4 L_\omega^4}+\|v\mathfrak{E}\|_{L_v^2 L_\omega^2}^2\nn\\
&+\|v(\hat \theta+\G)\c \sn\log a\|_{L_v^2 L_\omega^2}\|\sn\log a\|_{L_v^2 L_\omega^2}.\label{bb4.25}
\end{align}

We now treat  the term
$$
\A_1=\| v\sn^2\log a, v(R, R_{NN}, |\hat \theta, \G|^2)\|_{L_v^2 L_\omega^2}\|v^{\f12}\sn\log a\|^2_{L_v^4 L_\omega^4}.
$$
By (\ref{a_9}) and (\ref{a_6}), we have
\begin{equation}\label{5_3.1}
\|v(R, R_{NN}, |\hat \theta, \G|^2)\|_{L_v^2 L_\omega^2}\les \la^{-\f12}.
\end{equation}
Applying (\ref{sobv}) to $W=v \sn\log a$,
we obtain
\begin{equation}\label{4.25.5}
\|v^\f12 \sn\log a\|_{L_v^4 L_\omega^4}^2\les {v_*}^\f12 \|{v^\f12}\sn\log a\|_{L_v^\infty L_\omega^2}(\|{v^\f12}\sn\log a\|_{L_v^\infty L_\omega^2}+\|v\sn^2\log a\|_{L_v^2 L_\omega^2}).
\end{equation}
Hence, in view of (\ref{bba3}) and (\ref{5_3.1}), we conclude
\begin{equation}\label{a.4.25}
\A_1\les  \la^{-2\ep_0}( \|{v^\f12}\sn\log a\|_{L_v^\infty L_\omega^2}^2+\| v\sn^2\log a\|_{L_v^2 L_\omega^2}^2).
\end{equation}
 (\ref{a_6}) and  the second inequality in (\ref{a_9}) imply
\begin{equation}\label{4.25.4}
\|v^\f12 (|\hat\theta|+|\G|)\|_{L_v^4 L_\omega^4}\les \la^{-\frac{1}{4}}.
\end{equation}
We consider $\A_2=\|v(\hat \theta+\G)\c \sn\log a\|_{L_v^2 L_\omega^2}\|\sn\log a\|_{L_v^2 L_\omega^2}$
by using (\ref{4.25.4}) and (\ref{4.25.5})
\begin{align*}
\A_2&\les \la^{-2\ep_0}\|{v^\f12}\sn\log a\|_{L_v^\infty L_\omega^2}^\f12(\|{v^\f12}\sn\log a\|_{L_v^\infty L_\omega^2}+\|v\sn^2\log a\|_{L_v^2 L_\omega^2})^\f12\|\sn \log a\|_{L_v^2 L_\omega^2}.
\end{align*}
Thus $\A_1$ and $\A_2$ can be absorbed by the left side of (\ref{bb4.25}).
By using (\ref{fl2}), we complete the proof of the first pair of inequalities in (\ref{ellp_2}). The last inequality follows in view of (\ref{4.25.5}). (\ref{ellp_2}) improves (\ref{bba3}). Thus we complete the proof.
\end{proof}

\subsubsection{ Estimates on $\hat\theta$}

We will rely on the Hodge system (\ref{a_2}) and the transport equation (\ref{a_11}) to derive estimates on $\theta$.
Similar to Lemma \ref{hdgm1}, under the assumption of (\ref{a_7}) there holds the following  Calderon-Zygmund
estimates on $\cup_{0\le v\le v_*}S_v$,

\begin{proposition}
Let $\D$ be the Hodge operator $\D_1$ or $\D_2$.
For $2\le q\le p$, there holds for $S_v$ tangent tensor fields $F$ in the domain of $\D$,
\begin{equation}\label{ellp_6}
\|v \sn F\|_{L_\omega^q}+\|F\|_{L_\omega^q}\les \|v\D F\|_{L_\omega^q}.
\end{equation}
For scalar functions $f$,
\begin{equation}\label{cz_518}
\|v \sn^2 f\|_{L_\omega^q}+\|\sn f\|_{L_\omega^q}\les \|v \sD f\|_{L_\omega^q}.
\end{equation}
\end{proposition}

Now we provide some estimates on $\hat \theta$. The second estimate in (\ref{theta_1}) improves the first assumption in (\ref{a_6}).

\begin{proposition}\label{w7.24.4}
\begin{align}
&\|v\sn \hat \theta\|_{L_v^2 L_\omega^2}+\|\hat \theta\|_{L_v^2 L_\omega^2}\les \la^{-\f12}, \label{elli_1}\\
&\|v^\f12 \hat \theta\|_{L_v^\infty L_\omega^2}\les \la^{-\f12}, \quad \|v^\f12 \hat \theta\|_{L_v^4 L_\omega^4}\les \la^{-\frac{1}{4}-2\ep_0}. \label{theta_1}
\end{align}
\end{proposition}

\begin{proof}
By using (\ref{ellp_6}), (\ref{a_2}) and also (\ref{a_8}), (\ref{a_0}), we obtain
\begin{align*}
\|v\sn \hat \theta\|_{L_v^2 L_\omega^2}&+\|\hat \theta\|_{L_v^2 L_\omega^2}\les \|\sn_A\tr\theta, R_{NA}\|_{L_v^2 L_x^2}\\
&\les \|v^{-1}\sn a, \sn \tr k, \sn V_4,R_{AN} \|_{L_v^2 L_x^2}\les \la^{-\f12}.
\end{align*}
For the last inequality, we employed (\ref{ellp_1}) and (\ref{a_9}).

To derive an estimate on $\|v^\f12 \hat \theta\|_{L_v^\infty L_\omega^2}$, we rely on (\ref{a_11}).
 We first have with $p=2$  and by using (\ref{a_8}), (\ref{thet1}) and Lemma \ref{inii_7} that
\begin{align}
\|v^\f12 \hat \theta\|_{L_\omega^p L_v^\infty}
&\les \left\|v^{-\frac{3}{2}}\int_0^v |(\sn\hot \sn a, a\hat R_{AB})| d\mu_\ga dv'\right\|_{L_\omega^p}\label{theta_2}
\end{align}
which implies, in view of Lemma \ref{ellp_5.1} and (\ref{a_9})
\begin{equation*}
\|v^\f12 \hat \theta\|_{L_\omega^2 L_v^\infty}\les \la^{-\f12}.
\end{equation*}
Using (\ref{sobv}) and (\ref{elli_1}),  we then have
\begin{equation*}
\|v^\f12 \hat \theta\|_{L_v^4 L_\omega^4}^2\les \la^{-4\ep_0-\f12}.
\end{equation*}
Thus the proof of (\ref{theta_1}) is completed.
\end{proof}

For future reference, we give the following result.

\begin{lemma}
There holds the following decomposition for the Gaussian curvature
\begin{equation}\label{kep}
K=\frac{1}{a^2 v^2}+\ckk K+\sn \pi,
\end{equation}
where $\ckk K$ is a scalar function which verifies
 \begin{equation}\label{ckkk_4}
\|v^\frac{3}{2}\ckk K\|_{L_\omega^{q}}\les \la^{-1/2}, \mbox{ with } 2<q\le p.
\end{equation}
\end{lemma}

\begin{proof}
Let us recall (see \cite[Page 167]{CK})
\begin{equation*}
K=-\frac{1}{4} \tr\chi \tr\chib+\f12 \chih\c \chibh-\f12 \ga^{AC}\ga^{BD} \bR_{ADCB}
\end{equation*}
Recall from Lemma \ref{decom_lem} (ii) that
$
\bR_{ABAB}=\div \pi+\chi\c \pi +\pi\c\pi,
$ also using (\ref{G1}) and $\theta=k+\chi$ by (\ref{amc}),
we can rewrite
\begin{align}
K-\frac{1}{a^2 v^2}&=\frac{1}{4}(\tr\theta)^2-\frac{1}{a^2 v^2}+\pi\c\pi-\f12|\hat\theta|^2+\chi\c \pi+\sn\pi\nn\\
&=\pi\c\pi+|\hat\theta|^2+\hat\theta\c\pi+\frac{\pi}{av}+\sn\pi\label{kep1}
\end{align}
where the last identity is a symbolic expression.

Let $\ckk K$  be the part $\pi\c \pi+|\hat\theta|^2+\hat\theta\c\pi+\frac{\pi}{av}$ in (\ref{kep1}).
 To see (\ref{ckkk_4}),   by H\"{o}lder inequality, (\ref{theta_1}) and (\ref{a_9}) we have
\begin{align*}
&\left\|v^{\frac{3}{2}}\Big(\pi\c\pi+|\hat\theta|^2+\hat\theta\c \pi+\frac{\pi}{av}\Big)\right\|_{L_\omega^{q}}
\les\|v^\f12\pi\|_{L_\omega^q}+v_*^\f12\|v^\f12(\pi, \hat \theta)\|_{L_\omega^{2q}}^2\les \la^{-\f12}
\end{align*}
as desired.
\end{proof}

\subsection{\bf $L^p$ estimates}

Next, we derive estimates of higher regularity than (\ref{ellp_2}). We will employ Littlewood Paley decomposition $\slP_\mu$ on $(S_v, \ga)$.
Let $m$ be a function in the Schwartz class defined on $[0,\infty)$ having finite number of vanishing moments
and set $m_\mu(\tau):=\mu^2 m(\mu^2 \tau)$ for any dyadic numbers $\mu>0$. The geometric Littlewood-Paley
projection $\slP_\mu$ associated with $\ga$ is defined by
\begin{equation}\label{glp}
\slP_\mu H=\int_0^\infty m_\mu(\tau) U(\tau)H d\tau
\end{equation}
for any $S$-tangent tensor $H$, where $U(\tau) H$ denotes the solution of the heat flow
\begin{equation*}
\frac{d}{d\tau}U(\tau)H=v^2\sD U(\tau)H, \quad\quad U(0)H=H.
\end{equation*}
where $\sD$ is the Laplace-Beltrami operator of the induced metric $\ga$ on $S_v$.
One can refer to \cite{KRsurf} for various properties of $U(\tau)$ and $\slP_\mu$. In particular,
it has been shown that one can always find an $m$ such that the associated $\slP_\mu$ satisfies
$\sum_{\mu>0} \slP_\mu^2=Id$. We will also employ the fact (see \cite{KRsurf})
\begin{equation}\label{sdp}
\sD \slP_\mu=v^{-2} \mu^2 \widetilde{\slP}_\mu,
\end{equation}
where $\widetilde\slP_\mu$ is a Littlewood Paley decomposition associated to a different symbol $\ti m$.

We recall from \cite{KRsurf} that  for $S_v$ tangent tensor $F$,
\begin{align}
&\|U(\tau)F\|_{L^q_\omega}\les \|F\|_{L_\omega^q},  \quad \mbox{ for } 2\le q\le \infty, \label{u1} \\
&\|\sn U(\tau)F\|_{L^2_\omega}\les \|\sn F\|_{L_\omega^2}, \label{u2}\\
&\|v\sn U(\tau) F\|_{L^2_\omega}\les \tau^{-\f12} \|F\|_{L_\omega^2}, \label{u3}\\
&\|v^2\sD U(\tau) F\|_{L^2_\omega}\les \tau^{-1} \|F\|_{L_\omega^2}.\label{u4}
\end{align}
For scalar functions $f$,  by using (\ref{cz_518}), (\ref{u1})-(\ref{u4}), with the help  of Sobolev inequalities (\ref{sobinf}) and (\ref{sob}), we obtain the sharp Bernstein  inequality
\begin{equation}\label{bern_2}
\|U(\tau)f\|_{L_\omega^\infty}\les \tau^{-1/2} \|f\|_{L_\omega^2},\quad \|U(\tau)f\|_{L_\omega^2}\les \tau^{-\f12} \|f\|_{L_\omega^1}.
\end{equation}
As consequences of (\ref{u1})-(\ref{bern_2}), there hold the finite band property (\ref{glp2}) and (\ref{glp3}),
and the Bernstein inequality (\ref{bs}) for the geometric Littlewood-Paley projection.

\begin{lemma}\label{glpp}
The geometric Littlewood-Paley projections $\slP_\mu$ has the following properties for any smooth $S_v$ tangent tensor $F$,
\begin{eqnarray}
\|\slP_I F\|_{L^q}\les \|F\|_{L^q},
&&\|\sn \slP_\mu F\|_{L_\omega^2}\les \|\sn F\|_{L_\omega^2},\label{glp1}\\
\|v \sn \slP_\mu F\|_{L_\omega^2}+\|v\slP_\mu \sn F\|_{L_\omega^2}\les \mu \|F\|_{L_\omega^2},
&&\|v^2 \sD \slP_\mu F\|_{L_\omega^q}\les \mu^2 \|F\|_{L_\omega^q}, \label{glp2}\\
\|\slP_\mu F\|_{L_\omega^2}\les \mu^{-1} \|v \sn F\|_{L_\omega^2}, && \|\slP_\mu F\|_{L^q_\omega}\les \mu^{-2}\|v^2 \sD F\|_{L^q_\omega}. \label{glp3}
\end{eqnarray}
where $I$ is an interval in $\{ 2^m, m\in \Bbb Z\}$ and $ 1\le q\le \infty$. Moreover
\begin{eqnarray}
\|\slP_\mu F\|_{L^q_\omega}\les (\mu^{1-\frac{2}{q}}+1)\|F\|_{L_\omega^2},
&&\|\slP_\mu f\|_{L^\infty_\omega}\les (\mu+1)\|f\|_{L_\omega^2},\label{bs}
\end{eqnarray}
where $2\le q<\infty$ and $f$ is a scalar function.
\end{lemma}


Let $\slE_\mu$ be the standard Littlewood-Paley projection defined on ${\Bbb S}^2$, and $\sum_\mu\slE_\mu=Id$.  Under the assumption (\ref{a_7}),
 for $S_v$ tangent vector field $F$, there holds with $\sum \slP_\mu^2=Id$  the equivalence relation (see \cite{Wang09})
\begin{equation}\label{eqv}
\|\mu^\ep \slP_\mu F\|_{l_\mu^2L^2(S_v)}\approx \sum_{A}\|\mu^\ep \slE_\mu (F^A)\|_{l_\mu^2 L^2(S_v)}.
\end{equation}
where $F^A$ denotes the components of $F$ under the transport coordinate frame. The same equivalence relation holds for $S_v$-tangent tensor fields $F$.

\begin{lemma}\label{embd_1}
For scalar functions $f$, with $2<q<\infty$,  there holds
\begin{equation*}
\|f\|_{L_\omega^q}\les \|\mu^{1-\frac{2}{q}} \slE_\mu f\|_{l_\mu^2 L_\omega^2}
\end{equation*}
\end{lemma}

\begin{proof}
We recall the Littlewood-Paley inequality for smooth functions $f$ for $1<q'<\infty$
\begin{equation}\label{lpineq}
\|f\|_{L_\omega^{q'}}\approx \|(\slE_\mu f)_{l_\mu^2}\|_{L_\omega^{q'}}.
\end{equation}
Using (\ref{lpineq}), (\ref{sob}), the finite band inequality and the Minkowski inequality, we can infer
\begin{equation*}
\|f\|_{L_\omega^q}\approx \|(\slE_\mu f)_{l_\mu^2}\|_{L_\omega^q}\les \|\mu^{1-\frac{2}{q}} \slE_{\mu} f\|_{l_\mu^2 L_\omega^2}
\end{equation*}
as desired.
\end{proof}

Applying the Littlewood-Paley projections $P_\ell$ in ${\Bbb R}^3$ with $\sum_{\ell}P_\ell=Id$,
and using the finite band property, we can derive for scalar function $f$ and $\mu>1$ that
\begin{equation}\label{w7.26.1}
\mu^\ep \|\slE_\mu f\|_{L_v^2 L_x^2}\le \left(\sum_{\ell>\mu} \left(\frac{\mu}{\ell}\right)^\ep
+\sum_{1<\ell\le \mu}\left(\frac{\ell}{\mu}\right)^{1-\ep} \right)\|\ell^\ep P_\ell f\|_{L^2(\Sigma_0)}
+\mu^{-1+\ep}\|f\|_{L^2(\Sigma_0)}.
\end{equation}

\begin{lemma}\label{com_5.24}
For $\mu>1$ and $\ell>1$, there holds with $c>0$ sufficiently close to $0$,
\begin{equation*}
\mu^\ep \|\slP_\mu \sn \slP^2_\ell f\|_{L_\omega^2}
\les \min\left(\left(\frac{\mu}{\ell}\right)^\ep, \left(\frac{\ell}{\mu}\right)^{2-\ep}\right)
\|\ell^\ep\sn\slP_\ell f\|_{L_\omega^2}+\mu^{\frac{1}{2}-\frac{2}{p}+c}\ell^{-c}\|\sn f\|_{L_\omega^2}.
\end{equation*}
\end{lemma}

\begin{proof}
If $\mu<\ell$, using (\ref{glp1})
\begin{align}
\mu^\ep \|\slP_\mu \sn \slP^2_\ell f\|_{L_\omega^2}\les (\frac{\mu}{\ell})^\ep\|\ell^\ep \sn \slP^2_\ell f\|_{L_\omega^2}
\les (\frac{\mu}{\ell})^\ep \|\ell^\ep \sn \slP_\ell f\|_{L_\omega^2}. \label{i5.4}
\end{align}
In the case that $\mu>\ell$,  by (\ref{sdp}) we have
\begin{align}
\|\mu^\ep \slP_\mu \sn \slP_\ell^2 f\|_{L_\omega^2}&\les \|\mu^{\ep-2} v^2\sD \ti \slP_\mu \sn \slP_\ell^2 f\|_{L_\omega^2}\nn\\
&\les v^2\|\mu^{\ep-2} \ti\slP_\mu[\sD, \sn]\slP_\ell^2 f\|_{L_\omega^2}+v^2\|\mu^{\ep -2} \ti\slP_\mu \sn \sD \slP_\ell^2 f\|_{L_\omega^2} \label{i_5.2}
\end{align}
By using (\ref{comm_4.25}) and $|a-1|\le \f12$,
\begin{align}
I_\mu &:=\mu^{\ep-2}v^2\|\ti\slP_\mu[\sD, \sn]\slP_\ell^2 f\|_{L_\omega^2}\nn\\
&\les \mu^{\ep-2}v^2 \left\|\ti\slP_\mu\left(\left(K-\frac{1}{a^2 v^2}\right) \sn \slP_\ell^2f\right)\right\|_{L_\omega^2}
+\mu^{\ep-2}\|\ti\slP_\mu \sn \slP_\ell^2 f\|_{L_\omega^2}\label{i5.24}.
\end{align}
In view of $K=\frac{1}{a^2 v^2}+\ckk K+\sn \pi$ from (\ref{kep}),
we will show
\begin{align}
&\mu^{\ep-2}v^2\|\ti\slP_\mu(\ckk K\c \sn\slP_\ell^2 f)\|_{L_\omega^2}\les \la^{-4\ep_0} \mu^{\ep-2+\frac{2}{p}} \|\sn f\|_{L_\omega^2}\label{w7.23.2}\\
&\mu^{\ep-2} v^2 \|\ti {\sl P}_\mu(\sn \pi\c \sn P_\ell^2 f)\|_{L_\omega^2}\les \la^{-4\ep_0} \mu^{\f12-\frac{2}{p}} \|\sn f\|_{L_\omega^2}\label{w7.23.1}
\end{align}
which, together with (\ref{i5.24}), imply with $c>0$  sufficiently close to $0$
\begin{equation}\label{i_5.1}
I_\mu\les \mu^{\f12-\frac{2}{p}}\|\sn f\|_{L_\omega^2}\les \mu^{\f12-\frac{2}{p}+c}\ell^{-c}\|\sn f\|_{L_\omega^2}.
\end{equation}
To see (\ref{w7.23.2}), with  $\frac{1}{q}=\frac{1}{2}+\frac{1}{p}$, we derive from the Bernstein inequality (\ref{bs}) that
\begin{align}
\mu^{\ep-2}v^2\|\ti\slP_\mu(\ckk K\c \sn\slP_\ell^2 f)\|_{L_\omega^2}&\les \mu^{\ep-2+\frac{2}{p}}v^2 \|\ckk K\c \sn \slP_\ell^2 f\|_{L_\omega^{q}}\nn\\
&\les \mu^{\ep-2+\frac{2}{p}}v^2\|\ckk K\|_{L_\omega^p}\|\sn \slP_\ell^2 f\|_{L_\omega^2}\nn.
\end{align}
By using  (\ref{ckkk_4}) and (\ref{glp1}), we can obtain (\ref{w7.23.2}).

Noting that
$\sn\pi \c\sn \slP_\ell^2 f=\sn(\pi\c \sn \slP_\ell^2 f)- \pi\c \sn^2 \slP_\ell^2 f$,
we consider (\ref{w7.23.1})  by writing
\begin{align*}
\mu^{\ep-2} v^2 \|\ti {\sl P}_\mu(\sn \pi\c \sn \slP_\ell^2 f)\|_{L_\omega^2}
&\le \mu^{\ep-2} v^2 \left(\|\ti\slP_\mu\sn (\pi\c \sn \slP_\ell^2 f)\|_{L_\omega^2}+\|\ti\slP_\mu(\pi\c \sn^2 \slP_\ell^2 f)\|_{L_\omega^2}\right)\\
&=I^{(1)}_\mu+I^{(2)}_\mu.
\end{align*}
Let $\frac{1}{p'}=1-\frac{1}{p}$. By H\"{o}lder inequality, (\ref{sob}) and  (\ref{a_9})
\begin{align*}
I^{(1)}_\mu&\les\mu^{\ep-1} v \|\pi\c \sn \slP_\ell^2 f\|_{L_\omega^2}
\les \mu^{\ep-1}\|\sn \slP_\ell^2 f\|_{L_\omega^{2p'}}\|v\pi\|_{L_\omega^{2p}}\\
&\les \mu^{\ep-1} \la^{-4\ep_0} \left(\|v\sn^2\slP_\ell^2 f\|_{L_\omega^2}^{1-\frac{1}{p'}}\|\sn \slP_\ell^2 f\|_{L_\omega^2}^{\frac{1}{p'}}
+\|\sn \slP_\ell^2 f\|_{L_\omega^2}\right)\\
&\les \la^{-4\ep_0}\mu^{\ep-1}\ell^{1-\frac{1}{p'}}\|\sn f\|_{L_\omega^2},
\end{align*}
where we employed the fact that $\frac{1}{2}-\frac{2}{p}>-\frac{1}{p}$  and
\begin{equation}\label{w8.6.2}
\|v \sn^2 \slP_\ell^2 f\|_{L_\omega^2}\les \ell \|\sn f\|_{L_\omega^2}.
\end{equation}
To derive (\ref{w8.6.2}), we used (\ref{cz_518}), (\ref{sdp}), the finite band property and the second inequality in (\ref{glp1}).

Since $\ep=1-\frac{2}{p}$,  and $\mu>\ell>1$,
\begin{equation*}
I^{(1)}_\mu\les \la^{-4\ep_0} \ell^{\frac{1}{p}}\mu^{-\frac{2}{p}} \|\sn f\|_{L_\omega^2}\les \la^{-4\ep_0} \mu^{\f12-\frac{2}{p}}\|\sn f\|_{L_\omega^2}.
\end{equation*}
Using (\ref{bs}), (\ref{a_9}) and (\ref{w8.6.2})
\begin{align}
I^{(2)}_\mu &\les \mu^{\ep-2} v^2 \|\pi\c \sn^2 \slP_\ell^2 f\|_{L_\omega^2}
\les \mu^{\ep-2}\mu^\f12  v \|\pi\|_{L_\omega^4} \|v\sn^2 \slP^2_\ell f\|_{L_\omega^2}\label{w7.23.3}\\
&\les \mu^{\ep-\frac{3}{2}}\ell \la^{-4\ep_0} \|\sn f\|_{L_\omega^2}
\les \mu^{\f12-\frac{2}{p}}\la^{-4\ep_0} \|\sn f\|_{L_\omega^2}. \nn
\end{align}
  (\ref{w7.23.1}) can be obtained by combining the estimates for $I^{(1)}_\mu$ and  $I^{(2)}_\mu$.

For the other term in (\ref{i_5.2}), using (\ref{sdp}) and (\ref{glp1})
\begin{align}
\mu^\ep\mu^{-2} v^2\|\ti\slP_\mu \sn \sD \slP_\ell^2 f\|_{L_\omega^2}
\les  \mu^{\ep-2}\ell^2 \|\ti\slP_\mu \sn \ti \slP_\ell \slP_\ell f\|_{L_\omega^2}
\les \left(\frac{\ell}{\mu}\right)^{2-\ep} \ell^{\ep}\|\sn \slP_\ell f\|_{L_\omega^2}\label{i5.5}.
\end{align}
Lemma \ref{com_5.24} follows by combining (\ref{i5.4}), (\ref{i_5.2}), (\ref{i_5.1}) and (\ref{i5.5}).
\end{proof}

Now we prove the following product estimates.

\begin{lemma}\label{prodan}
For scalar functions $f$ and $g$ on $S_v$ there hold the product estimates
\begin{align}
&\|\mu^\ep \slE_\mu (f G)\|_{l_\mu^2 L^2(S_v)}\les (\|\sn f\|_{L^2(S_v)}+\|f\|_{L^\infty})\|G\|_{\dot{H}^\ep(S_v)}, \label{prd4}\\
&\|\mu^\ep \slE_\mu(f G)\|_{l_\mu^2 L_\omega^2}\les \|\sn f\|_{L_x^2}\| G\|_{L_\omega^p}+\|\sn G\|_{L_x^2}\| f\|_{ L_\omega^p}.\label{prd5}
\end{align}
For $S_v$-tangent tensor field $F$, there holds
\begin{align}
&\|\mu^\ep \slE_\mu(|F|^2)v\|_{l_\mu^2 L_v^2 L_\omega^2}
\les \left(\|v \sn F\|_{L_v^2 L_\omega^2}+\|F\|_{L_v^2 L_\omega^2}\right)\|v F\|_{L_v^\infty L_\omega^p}, \label{prd5.1}
\end{align}
where $p=\frac{2}{1-\ep}$, and $0<\ep<s-2$.
\end{lemma}

\begin{proof}
We first consider (\ref{prd4}). For $f$ and $G$  smooth scalar functions, we write by trichotomy
\begin{equation}\label{fg}
\slE_\mu (f\c G)= \slE_\mu(\slE_\mu G\c\slE_{\le \mu}f)+\slE_\mu(\slE_\mu f\c\slE_{\le \mu}G)+\slE_\mu\sum_{\ell>\mu}(\slE_\ell f\c \slE_\ell G).
\end{equation}
By Bernstein inequality and finite band property,
\begin{align}
\mu^\ep\|\slE_\mu(\slE_\mu f\c \slE_{\le \mu}  G)\|_{L_x^2}
&\les \mu^\ep \|\slE_{\le \mu} G\|_{L_x^\infty}\|\slE_\mu f\|_{L_x^2}
\les \sum_{1<\ell<\mu}(\frac{\ell}{\mu})^{1-\ep}\|\ell^\ep \slE_\ell G\|_{L_x^2}\|\p_\omega \slE_\mu f\|_{L_\omega^2}\nn
\end{align}
Taking $l_\mu^2$ yields
\begin{equation*}
\|\mu^\ep\slE_\mu(\slE_\mu f\c \slE_{\le \mu} G)\|_{l_\mu^2 L_x^2}\les \|\sn f\|_{L^2(S_v)} \|G\|_{\dot{H}^\ep(S_v)}
\end{equation*}
where we employed  $\|\p_\omega f\|_{L_\omega^2}\les \|\sn f\|_{L_x^2}$  which can be derived using (\ref{a_7}).
For the first term on the right of (\ref{fg}), we deduce
\begin{align*}
\mu^\ep\|\slE_\mu(\slE_{\le \mu} f\c \slE_\mu G))\|_{L_x^2}&\les\mu^\ep\|\slE_{\le \mu} f\|_{L^\infty(S_v)}\|\slE_\mu G\|_{L_x^2}.
\end{align*}
By taking $l_\mu^2$, we obtain
\begin{align*}
\|\mu^\ep\slE_\mu(\slE_{\le \mu} f \c \slE_\mu G))\|_{l_\mu^2L_x^2}&\les \|f\|_{L^\infty(S_v)}\| G\|_{\dot{H}^\ep (S_v)},
\end{align*}
where we used $\|f_{\le \mu}\|_{L^\infty(S_v)}\les \|f\|_{L^\infty(S_v)}$.
For the last term of (\ref{fg}),  we use the Bernstein inequality and the finite band property to obtain
\begin{align*}
\mu^\ep \Big\|\sum_{\ell>\mu}\slE_\mu(\slE_\ell f\c \slE_\ell G) \Big\|_{L_x^2}
&\les \sum_{\ell>\mu} \left(\frac{\mu}{\ell}\right)^{\ep+1}\|\p_\omega \slE_\ell f\|_{L_\omega^2}\|\ell^\ep \slE_\ell G\|_{L_x^2}.
\end{align*}
By taking $l_\mu^2$, we have
\begin{equation*}
\Big \|\mu^\ep\sum_{\ell>\mu}\slE_\mu(\slE_\ell f \c \slE_\ell G)\|_{l_\mu^2L_x^2} \les \|\sn f\|_{L^2(S_v)}\|G\|_{\dot{H}^\ep(S_v)}.
\end{equation*}
Thus the proof of (\ref{prd4}) is completed.

Now we prove (\ref{prd5}).
With $\frac{1}{q}=\frac{1}{2}+\frac{1}{p}$, we use Bernstein inequality and finite band property to obtain
\begin{align*}
\|\mu^\ep \slE_\mu(\slE_\mu G\c\slE_{\le \mu} f)\|_{L^2_\omega}&\les \mu^{\ep+\frac{2}{q}-1}\|\slE_\mu G\c \slE_{\le \mu} f\|_{L_\omega^{q}}\\
&\les \mu^{\ep+\frac{2}{p}-1}\|\slE_{\le \mu} f\|_{L_\omega^p}\|\slE_\mu\p_\omega G\|_{L_\omega^2}.
\end{align*}
Noting that $\frac{2}{p}=1-\ep$, by taking $l_\mu^2$, using (\ref{lpineq}) and (\ref{a_7}), we obtain
\begin{equation}\label{3.17.4}
\|\mu^\ep \slE_\mu( \slE_\mu G\c \slE_{\le \mu} f)\|_{l_\mu^2L_\omega^2}\les \|\sn G\|_{L_x^2}\|f\|_{L_\omega^p}.
\end{equation}
The second term  of (\ref{fg}) can be treated in the same way.

For the last term of (\ref{fg}), we employ Bernstein inequality, finite band property to obtain
\begin{align*}
\Big\|\mu^\ep \slE_\mu\sum_{\ell>\mu}(\slE_\ell f\c \slE_\ell G)\Big\|_{L_\omega^2}
&\les \mu^{\ep+\frac{2}{q}-1}\|\slE_\ell f\c \slE_ \ell G\|_{L_\omega^{q}}
\les \sum_{\ell>\mu}\mu^{\ep+\frac{2}{p}}\ell^{-1} \| \p_\omega \slE_\ell G\|_{L_\omega^2} \|\slE_\ell f\|_{L_\omega^p}.
\end{align*}
By taking $l_\mu^2$, we can obtain
\begin{equation}\label{3.17.3}
\Big\|\mu^\ep \slE_\mu\sum_{\ell>\mu}(\slE_\ell f\c \slE_\ell G)\Big\|_{l_\mu^2 L_\omega^2}
\les\|f\|_{L_\omega^p} \|\sn G\|_{L_x^2}.
\end{equation}
Combining (\ref{3.17.4}) with (\ref{3.17.3}), we obtain (\ref{prd5}) for scalar functions $f$ and $G$.
Finally, if $F$ is a $S_v$-tangent vector field,  (\ref{prd5.1}) can be obtained by using (\ref{eqv_2}) and (\ref{prd5}).
\end{proof}

We will employ the following result to improve the results in Lemma \ref{ellp_5.1}.

\begin{lemma}
\begin{equation}\label{w7.24.3}
\|\mu^\ep \slE_\mu((\bp \phi)^2)\|_{l_\mu^2 L_v^2 L_x^2}\les \la^{-\f12-4\ep_0},
\quad \|\mu^\ep \slE_\mu(\bp^2 \phi)\|_{l_\mu^2 L_v^2 L_x^2}\les \la^{-\f12}.
\end{equation}
\end{lemma}

\begin{proof}
The first inequality is a consequence of (\ref{prd5}) and (\ref{a_9_w}).  Note that  (\ref{w7.26.1}) implies $\|\mu^\ep \slE_\mu(\bp^2 \phi)\|_{l_\mu^2 L_v^2 L_x^2}\les \|\bp^2 \phi\|_{H^{\ep}}$.
The second inequality is a consequence of Proposition \ref{eng3} under the rescaled coordinates.
\end{proof}

We now consider the parabolic equations for smooth scalar functions $\Psi$ on the domain $\cup_{0\le v\le v_*}S_v$,
\begin{equation}\label{4.26.1}
\sD\Psi=\frac{2\p_v \Psi}{v}+\F, \quad \Psi(0)=0.
\end{equation}
By assuming the following commutator estimate

\begin{lemma}\label{comm_4.26}
For scalar functions $f$, there holds with $\ep=1-\frac{2}{p}$,
\begin{equation}\label{cca_2}
\|\mu^\ep [\slP_\mu, \p_v]f\|_{l_\mu^2 L_v^2 L_\omega^2}
\les \la^{-4\ep_0} \left(\|v^\f12\sn f\|_{L_v^\infty L_\omega^2} +\|v \sn^2 f\|_{L_v^2 L_\omega^2}\right).
\end{equation}
\end{lemma}

we first prove

\begin{proposition}\label{ellp_4}
Let   $\ep=1-\frac{2}{p}$.  For $\Psi$ verifying (\ref{4.26.1}) there holds
\begin{align*}
\|\mu^\ep v^{-\f12} \sn \slP_\mu \Psi\|_{L_\mu^2 L_v^\infty L_x^2}&+\|\mu^\ep \sn^2 \slP_\mu \Psi\|_{L_v^2 l_\mu^2 L_x^2}+\|\mu^\ep v^{-1} \sn \slP_\mu \Psi\|_{l_\mu^2 L_v^2 L_x^2}\nn\\
&\les \|\mu^\ep \slP_\mu \F\|_{l_\mu^2 L_v^2 L_x^2}+\la^{-4\ep_0}\| \F\|_{L_v^2 L_x^2}.
\end{align*}
\end{proposition}

\begin{proof}
We apply the Littlewood-Paley projection $\slP_\mu$ to obtain
\begin{equation*}
\sD \slP_\mu\Psi=\frac{2\p_v \slP_\mu\Psi}{v}+\frac{2[\slP_\mu, \p_v]\Psi}{v}+\slP_\mu\F.
\end{equation*}
By using (\ref{comm_4.25}), this implies
\begin{align*}
\sD\sn \slP_\mu \Psi&=K\c \sn \slP_\mu \Psi+\frac{2\p_v \sn \slP_\mu \Psi }{v}+\frac{2[\sn, \p_v] \slP_\mu \Psi}{v}+\frac{2\sn [\slP_\mu, \p_v]\Psi}{v}+\sn \slP_\mu \F.
\end{align*}
Now we multiply the above equation by $\sn \slP_\mu \Psi$ followed with integration by part on $S_v$,
\begin{align}
\int_{S_v} -|\sn^2 \slP_\mu \Psi|^2
&=\int_{S_v} \left(K|\sn \slP_\mu \Psi|^2+\frac{2[\sn, \p_v]\slP_\mu \Psi}{v}\c \sn \slP_\mu \Psi + \frac{\p_v |\sn \slP_\mu \Psi|^2}{v}\right) \label{3.17.1}\\
&+\int_{S_v} \left(\frac{2}{v}\sn[\slP_\mu, \p_v]\Psi\c \sn \slP_\mu \Psi- \slP_\mu \F \sD \slP_\mu \Psi\right)\nn.
\end{align}
Noting that by using (\ref{comm_4.25.1}), we have
\begin{equation*}
[\sn, \p_v]\slP_\mu \Psi=\f12 a \tr\theta \sn \slP_\mu\Psi+a\hat \theta \sn \slP_\mu \Psi.
\end{equation*}
Denote by $\I$ the three terms on the right hand side of (\ref{3.17.1}). In view of (\ref{4.25.2})  and (\ref{a_12}),
\begin{align*}
\I&=\int_{S_v} \left(\frac{(\tr\theta)^2}{4}-\f12 |\hat\theta|^2+\frac{R-2R_{NN}}{2}\right)|\sn \slP_\mu \Psi|^2  d \mu_\ga
+ \p_v\int_{S_v} v^{-1}|\sn \slP_\mu \Psi|^2 d\mu_\ga\\
&+\int_{S_v} \left[- v^{-1}\left(\frac{1}{v}-\G\right) |\sn \slP_\mu\Psi|^2 +v^{-1}\left(a\tr\theta \sn \slP_\mu \Psi
+2 a\hat \theta\sn \slP_\mu \Psi\right)\sn \slP_\mu \Psi\right] d\mu_\ga  .
\end{align*}
Now we integrate $\I$ in $(0, v)$ for $0<v\le v_*$,  also using (\ref{G1})
\begin{equation}\label{i1}
\int_0^v\I dv'= \int_{S_v} v^{-1}|\sn \slP_\mu \Psi|^2
+\int_0^{v}\int_{S_{v'}}\left(\frac{2}{{v'}^2}+ \frac{\G^2}{4a^2}\right) |\sn \slP_\mu \Psi|^2 d\mu_\ga dv'+\J,
\end{equation}
where the error term $\J$ is written schematically as
\begin{align}
\J&=\int_0^{v}\int_{S_{v'}}\left(|R-2R_{NN}|, |\hat \theta|^2, {v'}^{-1}(|\hat \theta|+|\G|)\right)|\sn \slP_\mu \Psi|^2  d\mu_\ga dv'\nn\\
&+\int_0^{v}\int_{S_{v'}} {v'}^{-2}|a^{-2}-1| |\sn \slP_\mu \Psi|^2 d\mu_\ga dv'\nn.
\end{align}
For $\J$, we proceed by using H\"{o}lder inequality,
\begin{equation*}
\J\les \|vR_{ij},\hat\theta,\G, v |\hat \theta|^2, (a^{-2}-1)v^{-1}\|_{L_v^2 L_\omega^2}\|v|\sn \slP_\mu \Psi|^2\|_{L_v^2 L_\omega^2}.
\end{equation*}
By using (\ref{sobv}),
\begin{align*}
\|v^\f12\sn \slP_\mu \Psi\|_{L_v^4 L_\omega^4}\les(\|v \sn^2 \slP_\mu \Psi\|_{L_v^2 L_\omega^2}+\|v^\f12 \sn \slP_\mu \Psi\|_{L_v^\infty L_\omega^2}) v_*^\frac{1}{4}.
\end{align*}
Combining the above inequalities with  (\ref{elli_1}),  (\ref{ellp_0}), (\ref{theta_1}) and (\ref{a_9}), we can obtain
\begin{equation*}
\J\les \la^{-4\ep_0}( \|v \sn^2 \slP_\mu \Psi\|_{L_v^2 L_\omega^2}+\|v^\f12 \sn \slP_\mu \Psi\|_{L_v^\infty L_\omega^2})^2.
\end{equation*}

At last, by applying Lemma \ref{comm_4.26} to $f=\Psi$ and combined with an integration by part,
 we can obtain
 \begin{align}
\|\mu^\ep v^{-\f12} \sn \slP_\mu \Psi\|_{L_\mu^2 L_v^\infty L_x^2}&+\|\mu^\ep \sn^2 \slP_\mu \Psi\|_{L_v^2 l_\mu^2 L_x^2}+\|\mu^\ep v^{-1} \sn \slP_\mu \Psi\|_{l_\mu^2 L_v^2 L_x^2}\nn\\
&\les \|\mu^\ep \slP_\mu \F\|_{l_\mu^2 L_v^2 L_x^2}+\la^{-4\ep_0}\| \F\|_{L_v^2 L_x^2}.\label{w7.24.1}
\end{align}
By repeating the above procedure to (\ref{4.26.1}), we have
\begin{equation*}
\| v^{-\f12} \sn\Psi\|_{L_v^\infty L_x^2}+\| \sn^2  \Psi\|_{L_v^2L_x^2}+\|v^{-1} \sn\Psi\|_{L_v^2 L_x^2}\les \|\F\|_{L_v^2 L_x^2}.
\end{equation*}
Proposition \ref{ellp_4} follows by substituting the above inequality to (\ref{w7.24.1}).
\end{proof}

Now we prove Lemma \ref{comm_4.26}.
\begin{proof}[Proof of Lemma \ref{comm_4.26}]
To see (\ref{cca_2}), we employ \cite[Page 64, Corollary 3.2.3.2]{CK} to obtain the following commutation formula for  smooth functions $f$ on $\Sigma_0$ that
\begin{align}\label{ff_1}
[\p_v, \sD]f&=-a\tr\theta\sD f-2a \hat \theta \c \sn^2 f-(2a R_{NA}+a \sn_A \tr\theta+2 \hat \theta_{AB}\sn_B a)\cdot \sn_A f.
\end{align}
By (\ref{glp}), the definition of the geometric Littlewood-Paley decomposition, we can write
\begin{equation*}
[\slP_\mu, \p_v] f=\int_0^\infty m_\mu(\tau) [U(\tau), \p_v]f d\tau
\end{equation*}
By Duhamel principle, we have
\begin{equation*}
[U(\tau), \p_v] f=\int_0^\tau U(\tau-\tau')[v^2\sD, \p_v] U(\tau') f d\tau'.
\end{equation*}
Thus we can obtain in view of (\ref{ff_1}) that
\begin{align*}
[\slP_\mu, \p_v]f&=v^2 \int_0^\infty m_\mu(\tau) \int_0^\tau U(\tau-\tau')[(a\tr\theta-\frac{2}{v})\sD U(\tau') f \\
&-2a \hat \theta \c \sn^2 U(\tau')f -(2a R_{NA}+a \sn_A \tr\theta+2 \hat \theta_{AB}\sn_B a)\cdot \sn_A U(\tau') f]
\end{align*}
We consider the term with $R^\sharp_A=2a R_{NA}+a \sn_A \tr\theta+2 \hat \theta_{AB}\sn_B a$. Let
\begin{equation*}
I_\mu=v^2\int_0^\infty m_\mu(\tau) \int_0^\tau U(\tau-\tau')\left( R^\sharp \cdot \sn U(\tau') f\right) d\tau'.
\end{equation*}
By using (\ref{bern_2}) and (\ref{u2})
\begin{align*}
\|I_\mu\|_{L_\omega^2}&\les v^2\int_0^\infty| m_\mu (\tau)| \int_0^\tau \frac{1}{\sqrt{\tau-\tau'}}\|R^\sharp\|_{L_\omega^2}\|\sn U(\tau') f \|_{L_\omega^2}d\tau'\\
&\les v_*^\f12\int_0^\infty \tau^\f12|m_\mu(\tau) |\|vR^\sharp \|_{L_\omega^2}\|v^\f12\sn f\|_{L_\omega^2}.
\end{align*}
By using (\ref{a_9}), (\ref{G1}), (\ref{ellp_1}), (\ref{ellp_8}) and (\ref{theta_1}),
\begin{align*}
\|v R^\sharp\|_{L_v^2 L_\omega^2}&\les \|v R\|_{L_v^2 L_\omega^2}+\|v\sn(\G/a), \sn (a^{-1})\|_{L_v^2 L_\omega^2}+\|v^\f12( \hat \theta, \sn\log a)\|_{L_v^4 L_\omega^4}^2\les \la^{-\f12}.
\end{align*}
Hence, also using $\int_0^\infty| m_\mu(\tau) |\tau^\f12\les \mu^{-1}$, we obtain
\begin{equation}\label{est_I}
\|\mu^\ep I_\mu\|_{l_\mu^2 L_v^2 L_\omega^2}\les v_*^\f12\la^{-\f12}\|v^\f12 \sn f\|_{L_v^\infty L_\omega^2}.
\end{equation}

Now we consider
\begin{align*}
{\emph {II}}_\mu&:=v^2 \int_0^\infty m_\mu(\tau) \int_0^\tau U(\tau-\tau')[2a \hat \theta \sn^2 U(\tau')f] d\tau'.
\end{align*}
With the help of (\ref{cz_518}),  for $0\le v \le v_*$, we  have on $S_v$,
\begin{equation}\label{flowbound}
\|\sn^2 U(\tau) f\|_{L^2(S_v)}\les\|\sD U(\tau) f\|_{L^2(S_v)}\les \|\sn^2 f\|_{L^2(S_v)}
\end{equation}
where we employed (\ref{u1}) and $[\sD, U(\tau)]=0$.

By using (\ref{bern_2}), (\ref{flowbound}), and (\ref{theta_1})
\begin{align*}
\|{\emph II}_\mu\|_{L_v^2 L_\omega^2}
&\les v_*^\f12 \int_0^\infty |m_\mu(\tau)|\int_0^\tau \frac{1}{\sqrt{\tau-\tau'}} d\tau' d\tau
\|v^\f12 a \hat \theta\|_{L_v^\infty L_\omega^2}\|v\sn^2 f\|_{L_v^2 L_\omega^2} \\
&\les v_*^\f12\mu^{-1} \la^{-\f12}\|v \sn^2 f\|_{L_v^2 L_\omega^2}.
\end{align*}
Thus
\begin{align*}
\|\mu^\ep {\emph II}_\mu\|_{l_\mu^2 L_v^2 L_\omega^2}\les v_*^\f12 \la^{-\f12}\|v \sn^2 f\|_{L_v^2 L_\omega^2}
\end{align*}

Similar to ${\emph II}_\mu$, we now consider
\begin{equation*}
{\emph {III}}_\mu: =v^2 \int_0^\infty m_\mu(\tau) \int_0^\tau U(\tau-\tau')(a\tr\theta-\frac{2}{v})\sD U(\tau') f d\tau'
\end{equation*}
Using (\ref{a_9}) and $ a\tr\theta-\frac{2}{v}=-\G$, we
conclude that
\begin{equation*}
\|\mu^\ep {\emph {III}}_\mu\|_{l_\mu^2 L_v^2 L_\omega^2}\les v_*^\f12 \la^{-\f12} \|v \sn^2 f\|_{L_v^2 L_\omega^2}.
\end{equation*}
Thus  the proof of Lemma \ref{comm_4.26} is completed.
\end{proof}

Now we  show

\begin{proposition}\label{ellp_7}
For  $\ep=1-\frac{2}{p}$, there holds
\begin{align*}
\|\mu^\ep v^{-\f12} \sn \slP_\mu \log a\|_{L_\mu^2 L_v^\infty L_x^2}
&+\|\mu^\ep \sn^2 \slP_\mu \log a\|_{L_v^2 l_\mu^2 L_x^2}+\|\mu^\ep v^{-1} \sn \slP_\mu \log a\|_{l_\mu^2 L_v^2 L_x^2}
\les \la^{-\f12}.
\end{align*}
\end{proposition}

To prove Proposition \ref{ellp_7}, let  us make bootstrap assumption that
\begin{equation}\label{aa_2}
|a-1|\le \la^{-\ep_0}
\end{equation}
and
\begin{equation}\label{aa_p}
\|v\sn a\|_{L_v^\infty L_\omega^p}+\|v \hat \theta\|_{L_v^\infty L_\omega^p}\le 1.
\end{equation}
We will improve it to
\begin{equation}\label{aa_3}
|a-1|\les \la^{-4\ep_0}
\end{equation}
\begin{equation}\label{aa_4}
\|v \sn a\|_{L_v^\infty L_\omega^p}+\|v \hat \theta\|_{L_v^\infty L_\omega^p}\les \la^{-4\ep_0}.
\end{equation}
which are contained  in Proposition \ref{w8.6.4} and Proposition \ref{3.17.5}.

In view of (\ref{aa_1}), we can apply Proposition \ref{ellp_4} to $\Psi=\log a$ with
$$\F=a^2 \mathfrak{E}-(a^2-1) \sD \log a-|\sn a|^2,$$
where $\mathfrak{E}$ has been defined in (\ref{ff}). We will prove  the following result for $\F$.
\begin{lemma}\label{err_1}
For  $\ep=1-\frac{2}{p}$,
\begin{align}
&\|v \F\|_{L_v^2 L_\omega^2}+\|\mu^\ep \slE_\mu \F\|_{L_v^2 l_\mu^2 L_x^2}\les \la^{-\f12}. \label{w7.24.2}
\end{align}
\end{lemma}

The inequality in Proposition \ref{ellp_7} follows immediately in view of (\ref{eqv}), Lemma \ref{err_1} and  Proposition \ref{ellp_4}.

\begin{proof}
In view of (\ref{ff}), by using (\ref{a_9}) and (\ref{theta_1}) we can derive
\begin{equation*}
\|v\mathfrak{E}\|_{L_v^2 L_\omega^2}\les \|R_{NN}, N (\tr k, V_4)\|_{L_v^2 L_x^2}+\|v^\f12 (\hat \theta, \pi)\|_{L_v^4 L_\omega^4}^2+\|\pi\|_{L_v^2 L_\omega^2}\les \la^{-\f12}.
\end{equation*}
Hence the first inequality (\ref{w7.24.2}) follows by using (\ref{ellp_2}). We prove the second one by showing
\begin{align}
&\|\mu^\ep \slE_\mu \F_1\|_{l_\mu^2 L_v^2 L_x^2}\les\la^{-\ep_0} \|\mu^\ep \sn^2 \slP_\mu \log a\|_{l_\mu^2L_v^2 L_x^2}, \label{f4.26.1}\\
&\|\mu^\ep \slE_\mu (|\sn a|^2)\|_{l_\mu^2 L_v^2 L_x^2}\les \la^{-\f12}, \label{f4.26.2}\\
&\|\mu^\ep \slE_\mu (a^2\mathfrak{E})\|_{l_\mu^2 L_v^2 L_x^2}\les \la^{-\f12}, \label{4.26.3}
\end{align}
where $\F_1=(a^2-1) \sD \log a$.

We apply (\ref{prd4}) to $f=a^2-1$ and $G= \sD \log a$ to obtain that
\begin{equation}\label{4.26.4}
\|\mu^\ep \slE_\mu \F_1\|_{l_\mu^2 L_v^2 L_x^2}
\les \left( \|v\sn(a^2-1)\|_{L_v^\infty L^2_\omega}+\|a^2-1\|_{L^\infty}\right)\|\mu^\ep \slE_\mu \sD \log a\|_{l_\mu^2 L_v^2 L_x^2}.
\end{equation}
By (\ref{ellp_2}) and (\ref{aa_2}), we have
\begin{equation*}
\|v\sn(a^2-1)\|_{L_v^\infty L^2_\omega}+\|a^2-1\|_{L^\infty}\les \la^{-\ep_0}.
\end{equation*}
By using (\ref{eqv}), we derive from (\ref{4.26.4}) that
\begin{equation}\label{4.26.7}
\|\mu^\ep \slE_\mu \F_1\|_{l_\mu^2 L_v^2 L_x^2}\les\la^{-\ep_0}\|\mu^\ep \slE_\mu \sD \log a\|_{l_\mu^2 L_v^2 L_x^2}\les \la^{-\ep_0}\|\mu^\ep \slP_\mu \sD \log a\|_{l_\mu^2 L_v^2 L_x^2}.
\end{equation}
By using $[\sD, \slP_\mu]=0$, we obtain (\ref{f4.26.1}).

We now consider (\ref{f4.26.2}) by applying (\ref{prd5.1}) to $F=\sn a$,
\begin{equation}\label{4.26.5}
\|\mu^\ep \slE_\mu(|\sn a|^2) v\|_{l_\mu^2 L_v^2 L_\omega^2}\les \|v \sn a\|_{L_v^\infty L_\omega^p}\|\sn a, v \sn a\|_{L_v^2 L_\omega^2}\les \la^{-\f12}
\end{equation}
where we employed (\ref{ellp_1}), Lemma \ref{ellp_5.1} and (\ref{aa_p}). Thus we proved (\ref{f4.26.2}).

 Next we prove (\ref{4.26.3}). Let us  first write
 $\mathfrak{E}=\ckk F+\frac{2}{va} (V_4-\tr k)$.
 By finite band property, (\ref{a_9}) and (\ref{ellp_2})
\begin{align}
& \left\|\mu^\ep \slE_\mu \left(\frac{a}{v}(V_4-\tr k)\right)\right\|_{l_\mu^2 L_v^2 L_x^2}\nn\\
&\les \|\sn( a (V_4-\tr k))\|_{L_v^2 L_x^2}\les\|\sn a\c \pi\|_{L_v^2 L_x^2}+\|\sn \pi\|_{L_v^2 L_x^2} \les \la^{-\f12}. \label{4.26.8}
 \end{align}
 Similar to (\ref{4.26.7}),
 \begin{align}
& \|\mu^\ep \slE_\mu\left((a^2-1)\ckk F\right)\|_{l_\mu^2 L_v^2 L_x^2}\les \la^{-\ep_0}\c\|\mu^\ep \slE_\mu \ckk F\|_{l_\mu^2 L_v^2 L_x^2}.\label{w7.31.4}
\end{align}
(\ref{4.26.3}) will follow by combining (\ref{4.26.8}), (\ref{w7.31.4}) with the following result.
 \begin{align}
&\|\mu^\ep \slE_\mu \ckk F\|_{l_\mu^2 L_v^2 L_x^2}\les \la^{-\f12}. \label{4.26.6}
 \end{align}
To see (\ref{4.26.6}), we  first  derive with $p=\frac{2}{1-\ep}$, by using (\ref{prd5.1}),
\begin{align*}
&\|\mu^\ep \slE_\mu (|\hat \theta|^2, (V_4-\tr k)^2)v\|_{l_\mu^2 L_v^2 L_\omega^2}\\
&\les \|v(\hat \theta, \tr k, V_4)\|_{L_v^\infty L_\omega^p}\|\sn \hat \theta, v^{-1}\hat \theta, \sn \tr k, \sn V_4\|_{L_v^2 L_x^2}\les\la^{-\f12},
\end{align*}
where we employed (\ref{aa_p}), Proposition \ref{w7.24.4} and Lemma \ref{mtieq} for deriving the last inequality.
 It only remains  to show
\begin{align}
&\|\mu^\ep \slE_\mu R_{NN}\|_{l_\mu^2 L_v^2 L_x^2}+\|\mu^\ep \slE_\mu(\nab_N (\tr k, V_4))\|_{l_\mu^2 L_v^2 L_x^2}\les \la^{-\f12}.\label{ep_2}
\end{align}

Indeed, with $f(\phi)$ being products of the factors among $\{ g^{(i)}(\phi), i=0,1,2\}$
$$R_{NN}=N^m N^n R_{mn}(g)=N^m N^n\c f(\phi) \c( \p^2\phi+(\p \phi)^2)$$
 we  write symbolically that
 \begin{equation}\label{rnn}
 R_{NN}=(f\c N)\c \p^2 \phi+(f\c N)(\p \phi)^2.
 \end{equation}
 By using (\ref{prd4}), (\ref{a_9_w}) and (\ref{w7.24.3}),
\begin{align}
&\|\mu^\ep \slE_\mu \left((f\c N) \p^2 \phi\right)\|_{l_\mu^2 L_v^2 L_x^2}+\|\mu^\ep \slE_\mu\left((f\c N)(\p \phi)^2\right)\|_{l_\mu^2 L_v^2 L_x^2}\label{epRnn}\\
&\les \left(\|\sn (f\c N) \|_{L_v^\infty L_x^2}+\|f\c N\|_{L_v^\infty L_x^\infty}\right)
\left(\|\mu^\ep \slE_\mu(\p^2 \phi)\|_{l_\mu^2 L_v^2 L_x^2}+\|\mu^\ep \slE_\mu\left((\p \phi)^2\right)\|_{l_\mu^2 L_v^2 L_x^2} \right)\nn\\
&\les \la^{-\f12}\nn
\end{align}
where we used $\sn N=\theta$, (\ref{G1}) and (\ref{theta_1}). Thus  the estimate for $R_{NN}$ in (\ref{ep_2}) is proved.

Note that  with  $\Pi_{m'}^m=\delta_{m'}^m-N_{m'}N^m$
\begin{align*}
\nab_N \tr k&=N^i\nab_i( k_{mn} \Pi_{m'}^m \Pi_{n'}^\nu g^{m' n'})
=N^i\Pi_{m'}^m \Pi_{n'}^n g^{m'n'}\nab_i k_{mn}+N^i k_{mn}\nab_i(\Pi_{\mu'}^\mu \Pi_{\nu'}^\nu g^{m'n'}).
\end{align*}
Also in view of  $L^\mu=N^\mu+\bT^\mu$ and
$
\nab_N(V_4)= N^i \nab_i V_\mu L^\mu+N^i V_\mu \nab_i L^\mu,
$
we have the symbolic identity
\begin{align}
\nab_N(V_4), \nab_N \tr k&=(N^i\c f)(\p \bp \phi)+ f\c(\bp \phi)^2 \c N^i+f\c\bp \phi\c \nab_N N\label{5.27.2}.
\end{align}
The first two terms can be treated as (\ref{epRnn}).
We treat the last term in (\ref{5.27.2}) by using (\ref{prd5}), (\ref{thetan}) and (\ref{aa_p})
\begin{align*}
&\|\mu^\ep \slE_\mu(f\bp \phi\c \nab_N N)\|_{l_\mu^2 L_v^2 L_x^2}\nn\\
&\les  \|\sn(f \bp \phi)\|_{L_v^2 L_x^2}\|v\nab_N N\|_{L^\infty L^p_\omega}+\|\sn^2\log a, v^{-1}\sn\log a\|_{L_v^2 L_x^2}\|v \bp\phi\|_{L_v^\infty L_\omega^p}\les \la^{-\f12},
\end{align*}
where we also used  (\ref{a_9_w}), (\ref{ellp_1}) and (\ref{ellp_2}) to derive the last inequality. Thus the proof of (\ref{ep_2}) is complete.
\end{proof}

Finally, we need to improve the assumptions (\ref{aa_2}), (\ref{aa_p}) to (\ref{aa_3}) and (\ref{aa_4}), which  are contained in Proposition \ref{3.17.5} and  the following result.

\begin{proposition}\label{w8.6.4}
There hold for the lapse function of the $v$-foliation that
 \begin{equation}\label{lpa}
  \|v^\frac{1}{2} \sn a \|_{L_v^\infty L_\omega^p}+\|\sn a, v \sn^2  a\|_{L_v^2 L_\omega^p}\les \la^{-\f12}
 \end{equation}
\begin{align}
|a-1|\les \la^{-4\ep_0}, \|v^{-\f12}(a-1)\|_{L^\infty}+\|\sn \log a\|_{L_v^2 L_x^\infty}\les \la^{-\f12}.\label{lpa1}
\end{align}
\end{proposition}

Indeed, in view of Lemma \ref{embd_1} and (\ref{eqv}),
\begin{equation}\label{lp5.24}
 \|\sn\log a\|_{L_\omega^p}\les  \|\mu^\ep\slE_\mu \sn\log a\|_{l_\mu^2 L_\omega^2}\approx  \|\mu^\ep \slP_\mu \sn\log a\|_{l_\mu^2 L_\omega^2}
\end{equation}
Applying the Littlewood-Paley decomposition $\slP_\ell$ such that  $\sum \slP^2_\ell=Id$, followed with using Lemma \ref{com_5.24} and Minkowski inequality, we obtain
\begin{equation*}
v^\f12 \|\sn \log a\|_{L_\omega^p}\les \|v^\f12 \mu^\ep \sn\slP_\mu \log a\|_{l_\mu^2 L_v^\infty L_\omega^2} +\|v^\f12\sn\log a\|_{L_v^\infty L_\omega^2}
\end{equation*}
Using Proposition \ref{ellp_7} and Lemma \ref{ellp_5.1}, we can obtain the first inequality in (\ref{lpa}).

By (\ref{cz_518}), Lemma \ref{embd_1}, we deduce that
\begin{equation*}
\|v \sn^2\log a\|_{L_v^2 L_\omega^p}\les\|\mu^\ep \slE_\mu \sD \log a\|_{l_\mu^2 l_v^2 L_x^2}\les \|\mu^\ep \slP_\mu \sD \log a\|_{l_\mu^2 L_v^2 L_\omega^2}\les \la^{-\f12}
\end{equation*}
where we employed (\ref{eqv})  and Proposition \ref{ellp_7}.

Using (\ref{ellp_1}), (\ref{ellp_2}) and (\ref{sob}), we obtain
$
\|\sn a\|_{L_v^2 L_\omega^p}\les \la^{-\f12}$.
Hence, we get  (\ref{lpa}).
The last inequality  of (\ref{lpa1}) follows from  (\ref{lpa}) and (\ref{ellp_1}) by using (\ref{sobinf}).

By (\ref{sobinf}), (\ref{lpa}) and (\ref{ellp_0}),
\begin{align*}
\|v^{-\f12} (a-1)\|_{L^\infty}\les \|v^\f12 \sn a\|_{ L_v^\infty L_\omega^p}+\|v^{-\f12}(a-1)\|_{L_v^\infty L_\omega^2}\les \la^{-\f12}
\end{align*}
This gives
\begin{equation}\label{a0.0}
|a-1|\les \la^{-\f12}v_*^\f12 \les \la^{-4\ep_0} \le \frac{1}{4}.
\end{equation}
Hence we finish the proof of (\ref{lpa1}) and the first inequality in (\ref{a_3}).

Next we give improved estimates on $\hat\theta$  including the one in (\ref{a_3}).

\begin{proposition}\label{3.17.5}
For $\frac{2}{p}=1-\ep$, there hold
\begin{equation}\label{theta_p}
\|v^\f12 \hat \theta\|_{L_v^\infty L_\omega^{p}}\les \la^{-\f12},
\end{equation}
\begin{equation}\label{theta_3}
\|v(\sn \hat\theta , \sn \tr\theta)\|_{L_v^2 L_\omega^p}\les \la^{-\f12}, \quad \|\hat \theta\|_{L_v^2 L_x^\infty}\les \la^{-\f12}.
\end{equation}
\end{proposition}

\begin{proof}
Recall from (\ref{theta_2}) that
\begin{equation}\label{t_317_1}
\|v^\f12 \hat\theta\|_{L_\omega^p L_v^\infty}\les \|v \sn^2 a\|_{L_\omega^p L_v^2}+\| v \hat R_{AB}\|_{L_\omega^p L_v^2}
\end{equation}
 by using (\ref{lpa}) and (\ref{a_9}),
\begin{align*}
\|v^\f12 \hat\theta\|_{L_\omega^p L_v^\infty}&\les \la^{-\f12}.
\end{align*}

Now we consider (\ref{theta_3}).
 By using (\ref{thedef}), (\ref{a_9}) and (\ref{lpa}) we obtain
\begin{equation*}
\|v\sn \tr\theta\|_{L_v^2 L_\omega^p}\les \|\sn a\|_{L_v^2L_\omega^p}+\|v \sn (\tr k, V_4)\|_{L_v^2 L_\omega^p}\les \la^{-\f12}.
\end{equation*}
Hence, by (\ref{a_2}), (\ref{ellp_6}) and (\ref{a_9})
\begin{align*}
\|v\sn \hat \theta\|_{L_v^2 L_\omega^p}&\les \|v\sn \tr\theta \|_{L_v^2 L_\omega^p}+\|v R_{NA}\|_{L_v^2 L_\omega^p}\les \la^{-\f12}.
\end{align*}
The last inequality of (\ref{theta_3}) follows from the above estimate, (\ref{theta_p}) and  (\ref{sobinf}).
\end{proof}

 We can replace $\chi$ by $\theta$, and $\frac{1}{t-u}$ by $\frac{1}{v}$ in the proof for (\ref{8.0.3}). By using Proposition \ref{3.17.5} and (\ref{w8.7.1}), we derive
\begin{equation*}
|\cga-\ga^{(0)}|+\|\p(\cga-\ga^{(0)})\|_{L_\omega^{p}(S_v)} \les \la^{-4\ep_0}
\end{equation*}
which improves (\ref{a_7}).

Finally, we prove (\ref{4a_6}). In view of (\ref{thet1}),  $\p_v \varphi=a\tr\theta-\frac{2}{v}$. Also using (\ref{comm_4.25.1}), we obtain the transport equation
\begin{equation}\label{w8.7.2}
\p_v \sn \varphi+\frac{a\tr\theta}{2}\sn \varphi=-a\hat \theta\c \sn \varphi+\sn(a\tr\theta-\frac{2}{v}).
\end{equation}
Using (\ref{theta_3}), (\ref{a_9}) and (\ref{w8.7.1}), applying  to (\ref{w8.7.2}) an argument similar as Lemma \ref{tsp2}, we can obtain
\begin{equation*}
\|\tir^\f12 \sn \varphi\|_{L_v^\infty L_\omega^p}\les \|v\sn\mathfrak{\G}\|_{L_\omega^p L_v^2}\les \|v \sn a \c \pi\|_{L_v^2 L_\omega^p}+\|v \sn \pi\|_{L_v^2 L_\omega^p}\les \la^{-\f12}.
\end{equation*}
To derive the last inequality, we employed (\ref{a_9}) and (\ref{ellp_2}).
Thus we complete the proof of Proposition \ref{exten}.


\end{document}